\documentclass[12pt,a4paper]{article}

\usepackage{amsfonts,amsmath,amsthm}
\usepackage{graphicx}
\usepackage{amssymb}
\usepackage{mathrsfs}
\usepackage[top=2cm, bottom=2cm, left=2cm, right=2cm]{geometry}
\usepackage[all]{xy}
\usepackage{tikz-cd}
\usepackage{hyperref}
\usepackage{authblk}
\usepackage{enumitem}

\usepackage[normalem]{ulem}
\usepackage{framed}

\allowdisplaybreaks

\newtheorem{theorem}{Theorem}[section]
\newtheorem{proposition}[theorem]{Proposition}
\newtheorem{corollary}[theorem]{Corollary}
\newtheorem{definition}[theorem]{Definition}
\newtheorem{lemma}[theorem]{Lemma}

\newtheorem{thIntro}{Theorem}
\newtheorem{coroIntro}[thIntro]{Corollary}

\theoremstyle{remark}
\newtheorem{remark}[theorem]{Remark}
\newtheorem{example}[theorem]{Example}

\newcommand{\flecheIso}[4]{                     
            \begin{array}{rcl} #1 & \overset{\sim}{\longrightarrow} & #2 \\   %
                         #3 &\longmapsto & #4          %
            \end{array}}

\newcommand{\fonc}[5]{                     
            \begin{array}{crll}#1 :& #2 & \rightarrow & #3 \\   %
                         &#4 &\mapsto & #5          %
            \end{array}}

\newcommand{\foncIso}[5]{                     
            \begin{array}{crll}#1 :& #2 & \overset{\sim}{\rightarrow} & #3 \\   %
                         &#4 &\mapsto & #5          %
            \end{array}}
            
\newcommand{\smallsquare}{%
  \vcenter{\hbox{\:\scalebox{0.35}{$\blacksquare$}\:}}%
}

\newcommand{\Nak}{\mathsf{N}}
\newcommand{\cat}{\mathcal{C}}
\newcommand{\zcat}{\mathcal{Z}(\cat)}
\newcommand{\dD}{D}

\newcommand{\one}{\boldsymbol{1}}
\newcommand{\projone}{P_{\boldsymbol{1}}}

\DeclareMathOperator{\Ext}{Ext}

\DeclareMathOperator{\Hom}{Hom}

\title{\bf An adjunction theorem for Davydov--Yetter cohomology and infinitesimal braidings}
\author[1]{M. Faitg}
\author[2]{A.M. Gainutdinov}
\author[3]{C. Schweigert}
\affil[1]{\small \textit{Institut de Mathématiques de Toulouse, Universit\'e Paul Sabatier, 118 route de Narbonne, 
\newline F-31062 Toulouse, France.}}
\affil[2]{\small \textit{Institut Denis Poisson, CNRS, Université de Tours, Universit\'e d’Orl\'eans, Parc de Grandmont, 37200 Tours, France}}
\affil[3]{\small \textit{Fachbereich Mathematik, Universität Hamburg, Bundesstra{\ss}e 55, 20146 Hamburg, Germany}}
\date{}

\begin{document}

\begin{flushright}
\textsf{ZMP-HH/24-26
\\Hamburger Beitr\"age zur Mathematik Nr. 981}
\end{flushright}

{\let\newpage\relax\maketitle}

\maketitle

\vspace{-2em}

\noindent {\small {\em E-mail adresses:} matthieu.faitg@math.univ-toulouse.fr, azat.gainutdinov@cnrs.fr,

\hspace{5.95em}christoph.schweigert@uni-hamburg.de}

\bigskip

\begin{abstract}
Davydov-Yetter cohomology $H_{\mathrm{DY}}^{\bullet}(F)$ is  associated to a monoidal functor $F: \mathcal{C} \to \mathcal{D}$ between $\Bbbk$-linear monoidal categories where $\Bbbk$ is a field, and its second degree classifies the infinitesimal deformations of the monoidal structure of $F$. Our main result states that if $F$ admits a right adjoint $R$, then there is an object $\Gamma$ in the Drinfeld center $\mathcal{Z}(\mathcal{C})$ defined in terms of $R$ such that the Davydov-Yetter cohomology of $F$ can be expressed as the Davydov-Yetter cohomology of the identity functor on $\mathcal{C}$ with the coefficient $\Gamma$.

\indent We apply this result in the case when the product functor $\otimes: \mathcal{C} \boxtimes\mathcal{C} \to\mathcal{C}$ has a monoidal structure given by a braiding $c$ on $\mathcal{C}$ and determine explicitly the coefficient $\Gamma$ as a coend object in $\mathcal{Z}(\mathcal{C}) \boxtimes \mathcal{Z}(\mathcal{C})$. The motivation is that $H^{\bullet}_{\mathrm{DY}}(\otimes)$ contains a ``space of infinitesimal braidings tangent to $c$'' in a way that we describe precisely. For $\mathcal{C} = H\text{-}\mathrm{mod}$, where $H$ is a finite-dimensional Hopf algebra over a field $\Bbbk$, this is the Zariski tangent space to the affine variety of R-matrices for $H$. In the case of perfect $\Bbbk$, we  give a dimension formula for this space as an explicit end involving only (low-degree) relative Ext's of the standard adjunction between $\zcat$ and~$\cat$.
As a further application of the adjunction theorem, we describe deformations of the restriction functor associated to a Hopf subalgebra and a Drinfeld twist. Both applications are illustrated in the example of bosonization of exterior algebras.
\end{abstract}

\newpage

\tableofcontents

\section{Introduction}

We continue to explore  properties of Davydov--Yetter (DY) cohomology of monoidal functors coming from its relation with relative homological algebra discovered in \cite{GHS,FGS}. This is part of a research program whose goal is to derive efficient tools for the computation of DY cohomology from the properties of relative Ext groups. In the present work we investigate how adjoint functors interact with DY cohomology groups and deduce an {\em adjunction theorem for DY cohomology}.

\indent Let $\Bbbk$ be a field. Recall that given a $\Bbbk$-linear monoidal functor $F : \mathcal{C} \to \mathcal{D}$ between $\Bbbk$-linear monoidal categories $\mathcal{C}, \mathcal{D}$ one can define the DY cohomology spaces $H^n_{\mathrm{DY}}(F)$ for all $n \geq 0$ \cite{davydov, CY}. The space $H^2_{\mathrm{DY}}(F)$ classifies infinitesimal deformations of the {\em monoidal structure} $F^{(2)}_{X,Y} : F(X) \otimes F(Y) \overset{\sim}{\to} F(X \otimes Y)$ while $H^3_{\mathrm{DY}}(F)$ is responsible for obstructions of lifting deformations to higher order. We briefly review the deformation of monoidal structures and DY cohomology in \S\ref{subsectionDefMonStruct}. 
Many interesting categorical structures  can be repackaged into monoidal structures of well-chosen functors and their deformation theory is thus controlled by Davydov--Yetter cohomology. For instance a result of Joyal--Street \cite{JS} gives a correspondence between braidings on a monoidal category $\mathcal{C}$ and equivalence classes of monoidal structures for the product functor $\otimes : \mathcal{C} \times \mathcal{C} \to \mathcal{C}$. Hence, as noted by Yetter \cite{yetter1}, the DY cohomology of $\otimes$ is related to infinitesimal deformations of braidings. Another example is the one-to-one correspondence between mixed associators for a $\mathcal{C}$-module category $(\mathcal{M},\rhd)$ and monoidal structures for the action functor $\rho : \mathcal{C} \to \mathrm{Fun}(\mathcal{M})$, $X \mapsto X \rhd -$. Hence the DY cohomology of $\rho$ is related to deformations of mixed associators. 

\smallskip

\indent In this paper, we focus on the application to
deformations of braidings. Using our adjunction theorem we give a precise description of
the vector space of infinitesimal deformations of a given braiding on a  (finite) tensor category $\mathcal{C}$ in terms of
 certain 2nd relative Ext groups involving non-trivial coefficients, as we explain now.

\indent In \cite{GHS} a version of DY cohomology {\em with coefficients} was introduced -- they are pairs of objects in the centralizer of the functor $F$, which is a monoidal category denoted by $\mathcal{Z}(F)$, see \S \ref{relExtGroupsDYCohomology}. For $F = \mathrm{Id}_{\mathcal{C}}$ then $\mathcal{Z}(F)$ is the Drinfeld center $\mathcal{Z}(\mathcal{C})$. 
In this case, the coefficients are very much similar to those in the standard Hochschild theory, and can be realised as internal {\em bimodules} over an algebra in $\mathcal{C}\boxtimes\mathcal{C}$. Indeed, it is known  that $H_{\mathrm{DY}}^{\bullet}(\mathrm{Id}_{\mathcal{C}})$ is isomorphic to the Hochschild cohomology of the algebra object $\mathscr{A} = \int^{X \in \mathcal{C}} X \boxtimes X^{\vee} \in \mathcal{C} \boxtimes \mathcal{C}$, \cite[Prop.\,7.22.7]{EGNO}, and furthermore the category of internal  $\mathscr{A}$-bimodules is equivalent  to $\mathcal{Z}(\mathcal{C})$ by~\cite[\S 3.4]{EO}, \cite[\S 4.4]{shimizu}.

\indent We denote the DY cohomology of $F$ with coefficients $\mathsf{V}, \mathsf{W} \in \mathcal{Z}(F)$ by $H_{\mathrm{DY}}^{\bullet}(F;\mathsf{V},\mathsf{W})$. For the trivial coefficients $\mathsf{V}, \mathsf{W} = \boldsymbol{1}$ we have $H^{\bullet}_{\mathrm{DY}}(F) = H^{\bullet}_{\mathrm{DY}}(F; \boldsymbol{1},\boldsymbol{1})$. Although the infinitesimal deformations of $F^{(2)}$ are classified by the cohomology with trivial coefficients, non-trivial coefficients play an extremely important role in this work. As we explain below, our adjunction theorem allows to reduce the study of any right exact functor $F$ to the study of the identity functor on its source category, at the price of a non-trivial coefficient.

\indent We recall from \cite[Cor.\,4.7]{FGS} that under certain  assumptions on the linear monoidal categories $\mathcal{C}, \mathcal{D}$ and the monoidal functor $F : \mathcal{C} \to \mathcal{D}$, we have an isomorphism between the DY cohomology of $F$ and the relative Ext groups associated to the forgetful functor $\mathcal{Z}(F) \to \mathcal{D}$:
\begin{equation}\label{DYasExtIntro}
H^{\bullet}_{\mathrm{DY}}(F;\mathsf{V},\mathsf{W}) \cong \Ext^{\bullet}_{\mathcal{Z}(F),\mathcal{D}}(\mathsf{V},\mathsf{W}), \text{ and in particular } H^{\bullet}_{\mathrm{DY}}(F) \cong \Ext^{\bullet}_{\mathcal{Z}(F),\mathcal{D}}(\boldsymbol{1},\boldsymbol{1}).
\end{equation}
While \eqref{DYasExtIntro} was proven for exact functors $F$ between finite tensor categories in \cite{FGS}, the more general existence results of coends contained in the Appendix \ref{subsectionNatAndCoendsFiniteCat} allow us to relax these assumptions to right-exact functors (see \S\ref{relExtGroupsDYCohomology}).

\indent The computation of $\Ext_{\mathcal{Z}(F),\mathcal{D}}^{\bullet}(\boldsymbol{1},\boldsymbol{1})$ in \eqref{DYasExtIntro} remains a non-trivial task which requires in particular to have a good understanding of the category $\mathcal{Z}(F)$. The case $F = \mathrm{Id}_{\mathcal{C}}$ is somewhat special because the category $\mathcal{Z}(\mathcal{C})$ is well-known and in \cite{FGS} we have obtained more results for this choice of $F$, like a method to construct DY cocycles explicitly. The application of our adjunction theorem is precisely to {\em replace computations in $\mathcal{Z}(F)$ by computations in the more familiar category $\mathcal{Z}(\mathcal{C})$}, by trading $F$ for a certain coefficient object thanks to an adjunction that we now discuss.

\indent Note first that the monoidal functor $F$ can be ``lifted'' at the level of centralizers:
\begin{equation}\label{morphismAdjunctionZFZCIntro}
\xymatrix@C=4em@R=.7em{
\mathcal{Z}(\mathcal{C}) \ar@/^.7em/[dd]^{\mathcal{U}_{\mathcal{C}}} \ar[r]^{\widetilde{F}} &
\mathcal{Z}(F) \ar@/^.7em/[dd]^{\mathcal{U}_F}\\
\dashv & \dashv\\
\ar@/^.7em/[uu]^{\mathcal{F}_{\mathcal{C}}} \mathcal{C} \ar[r]_F & \ar@/^.7em/[uu]^{\mathcal{F}_F} \mathcal{D}
} \end{equation}
where $\mathcal{U}_{\mathcal{C}}$, $\mathcal{U}_F$ are the forgetful functors and $\mathcal{F}_{\mathcal{C}}$, $\mathcal{F}_F$ are their left adjoints (see \S\ref{relExtGroupsDYCohomology}). Given a half-braiding $\lambda : V \otimes - \Rightarrow - \otimes V$ one obtains from $F(\lambda)$ and the monoidal structure of $F$ a natural isomorphism $F(V) \otimes F(-) \Rightarrow F(-) \otimes F(V)$ which is is a half-braiding relative to $F$, whence defining $\widetilde{F}$. An important point is that the pair of functors $(F,\widetilde{F})$ is compatible with the adjunctions, meaning that apart the equality $\mathcal{U}_F \, \widetilde{F} =F \, \mathcal{U}_{\mathcal{C}}$ we have a natural isomorphism
\[ \widetilde{F} \, \mathcal{F}_{\mathcal{C}} \cong \mathcal{F}_F \, F \]
which is compatible with the adjunction data. The pair $(F,\widetilde{F})$ is an example of what we call a {\em strong morphism of adjunctions}  (Definition \ref{defMorphismOfResolventPairs}).

\indent Strong morphisms of adjunctions are relevant to us for two reasons. First, they preserve relatively projective resolutions (Proposition \ref{propStrongMorphismsPreserveEverything}): in the particular case \eqref{morphismAdjunctionZFZCIntro}, the functor $\widetilde{F}$ transforms a resolution of $\mathsf{V} \in \mathcal{Z}(\mathcal{C})$ into a resolution of $\widetilde{F}(\mathsf{V}) \in \mathcal{Z}(F)$. Second, if $F$ admits a right adjoint $R : \mathcal{D} \to \mathcal{C}$ then a classical theorem on adjoint lifting \cite{johnstone} gives a right adjoint $\widetilde{R}$ of $\widetilde{F}$ (see \S\ref{sectionLiftingAdjunctions} and \S\ref{sectionChangeOfFunctor}). It follows from these two facts that 
$$
\Ext_{\mathcal{Z}(F), \mathcal{D}}(\boldsymbol{1},\boldsymbol{1}) \cong \Ext_{\mathcal{Z}(\mathcal{C}), \mathcal{C}}\bigl(\boldsymbol{1}, \widetilde{R}(\boldsymbol{1}) \bigr)
$$
and by \eqref{DYasExtIntro} we get the {\em adjunction theorem for DY cohomology}:

\begin{thIntro}\label{th1Intro}
Let $\mathcal{C}, \mathcal{D}$ be $\Bbbk$-linear monoidal abelian categories. Assume moreover that $\mathcal{C}$ is finite and rigid and $\otimes_{\mathcal{D}}$ is right-exact in each variable. Then for any $\Bbbk$-linear monoidal functor $F : \mathcal{C} \to \mathcal{D}$ which has a right adjoint $R : \mathcal{D} \to \mathcal{C}$ we have
\begin{equation}\label{isoTh1Intro}
H^{\bullet}_{\mathrm{DY}}(F) \cong H^{\bullet}_{\mathrm{DY}}\bigl(\mathrm{Id}_{\mathcal{C}}; \boldsymbol{1}, \widetilde{R}(\boldsymbol{1}) \bigr)
\end{equation}
where $ \widetilde{R} : \mathcal{Z}(F) \to \mathcal{Z}(\mathcal{C})$ is the lift of $R$.
\end{thIntro}

 In the case the category $\mathcal{D}$ is also finite, then the right adjoint of $F : \mathcal{C} \to \mathcal{D}$ exists if and only if $F$ is right exact \cite[Cor.\,1.9]{DSPS}.
 Hence, the DY cohomology of the identity functor {\em with coefficients} determines the DY cohomology of any right-exact functor out of $\mathcal{C}$. Theorem~\ref{th1Intro} is efficient in practice because only the {\em second} coefficient in the right-hand side of \eqref{isoTh1Intro} is non-trivial. Hence to compute $H^{\bullet}_{\mathrm{DY}}\bigl(\mathrm{Id}_{\mathcal{C}}; \boldsymbol{1}, \widetilde{R}(\boldsymbol{1}) \bigr) \cong \Ext_{\mathcal{Z}(\mathcal{C}),\mathcal{C}}\bigl(\boldsymbol{1},\widetilde{R}(\boldsymbol{1})\bigr)$ it is enough to find {\em once and for all} a relatively projective resolution of $\boldsymbol{1} \in \mathcal{Z}(\mathcal{C})$. Then one must determine the object $\widetilde{R}(\boldsymbol{1}) \in \mathcal{Z}(\mathcal{C})$ and compute the resulting cohomology. This last cohomological computation is expected to be less hard than finding a resolution of $\boldsymbol{1} \in \mathcal{Z}(F)$ for any functor $F$ if we were to use formula \eqref{DYasExtIntro}. Moreover the cohomological computation can be replaced by the following dimension formula (Cor.~\ref{coroTangentSpaceChangeOfFunctor}):
\[ \dim H^n_{\mathrm{DY}}(F) = \dim \Hom_{\mathcal{Z}(\mathcal{C})}( \mathsf{K}, \mathsf{M} ) - \dim \Hom_{\mathcal{Z}(\mathcal{C})}( \mathsf{P}, \mathsf{M} ) + \dim \Hom_{\mathcal{Z}(\mathcal{C})}( \boldsymbol{1}, \mathsf{M} ) \]
where $\mathsf{P}$ is  is the relatively projective cover of $\boldsymbol{1} \in \mathcal{Z}(\mathcal{C})$ \cite[\S 2.3]{FGS} (or any relative projective object covering $\boldsymbol{1}$), $\mathsf{K} = \mathrm{ker}(\mathsf{P} \twoheadrightarrow \boldsymbol{1})$ and $\mathsf{M} = \widetilde{R}(\boldsymbol{1}) \otimes (\mathsf{K}^{\vee})^{n-1}$, for $n \geq 2$. Note that everything in the right-hand side is computed in $\mathcal{Z}(\mathcal{C})$.

\indent Theorem \ref{th1Intro} is inspired by \cite[\S 4.3]{GHS} where it was observed that the DY cohomology of the fiber functor $H\text{-}\mathrm{mod} \to \mathrm{vect}_{\Bbbk}$ of the category of modules over a finite-dimensional $\Bbbk$-algebra $H$ can be equivalently described through the DY cohomology of the identity functor at the price of a coefficient $H^*$ with the $D(H)$-module structure given by coregular $H$-action and adjoint $H^*$-action. This is generalized in \S\ref{subsectionDYTwistedFiber} where we consider the DY cohomology of the restriction to a Hopf subalgebra functor endowed with a monoidal structure involving a Drinfeld twist.

\smallskip

\indent Our next results concern a functor whose deformations are related to deformations of braidings. To explain this, let $c = \bigl( c_{X,Y} : X \otimes Y \overset{\sim}{\to} Y \otimes X \bigr)_{X,Y \in \mathcal{C}}$ be a braiding on the $\Bbbk$-linear monoidal category $\mathcal{C}$ and consider the product functor $P = \otimes : \mathcal{C} \boxtimes \mathcal{C} \to \mathcal{C}$ endowed with the following monoidal structure coming from $c$:
\begin{align}
\begin{split}\label{monStructMonProductFromBraidingIntro}
&P(X_1 \boxtimes Y_1) \otimes P(X_2 \boxtimes Y_2) = X_1 \otimes Y_1 \otimes X_2 \otimes Y_2\\
&\xrightarrow{\mathrm{id}_{X_1} \otimes c_{Y_1,X_2} \otimes \mathrm{id}_{Y_2}} P\bigl( (X_1 \boxtimes Y_1) \otimes (X_2 \boxtimes Y_2) \bigr) = X_1 \otimes X_2 \otimes Y_1 \otimes Y_2.
\end{split}
\end{align}
Denote by $P_c$ this monoidal functor. Actually, any monoidal structure on $P$ which satisfies certain unitality conditions is isomorphic to \eqref{monStructMonProductFromBraidingIntro} for some braiding $c$ on $\mathcal{C}$. This correspondence between (equivalence classes of) unital monoidal structures for $P$ and braidings on $\mathcal{C}$ was found by Joyal and Street when they studied multiplications in a monoidal category \cite[\S 5]{JS}. In \S\ref{subsectionDYcohomologyTangentBraidings} we establish an infinitesimal version of this correspondence. Let $\mathbf{T}_c\mathrm{Br}(\mathcal{C})$ be the vector space of {\em infinitesimal braidings tangent to $c$}. As the name and notation indicate, these are natural transformations $t = \bigl( t_{X,Y} : X \otimes Y \to Y \otimes X \bigr)_{X,Y \in \mathcal{C}}$ such that $c + \epsilon t$ is a braiding on the category $\mathcal{C} \otimes_{\Bbbk} \Bbbk[\epsilon]/(\epsilon^2)$ with scalars for Hom spaces extended to the dual numbers $\Bbbk[\epsilon]/(\epsilon^2)$.
\begin{thIntro}\label{thmDYCohomAndTangentBraidingsProductIntro}
Let  $\mathcal{C}$ be a $\Bbbk$-linear monoidal category and for any braiding $c \in \mathrm{Br}(\mathcal{C})$ let $P_c = \otimes : \mathcal{C} \boxtimes \mathcal{C} \to \mathcal{C}$ be the monoidal product functor endowed with the monoidal structure \eqref{monStructMonProductFromBraidingIntro}. We have
\[ H^2_{\mathrm{DY}}(P_c) \cong \mathbf{T}_c\mathrm{Br}(\mathcal{C}) \oplus H^2_{\mathrm{DY}}(\mathrm{Id}_{\mathcal{C}}) \oplus H^2_{\mathrm{DY}}(\mathrm{Id}_{\mathcal{C}}). \]
\end{thIntro}

 Note that we do not require finiteness assumption on $\mathcal{C}$ in this theorem. This generalizes a result of Yetter \cite{yetter1}, who noted that the Joyal--Street correspondence gives at the infinitesimal level a relation between infinitesimal braidings on $\mathcal{C}$ and the 2nd DY cohomology of $\otimes_{\mathcal{C}}$. However he did not describe a complement of $\mathbf{T}_c\mathrm{Br}(\mathcal{C})$ in $H^2_{\mathrm{DY}}(P_c)$.

\indent In the context of Vassiliev invariants and deformations of symmetric categories, a slightly different notion of ``infinitesimal braidings'' has appeared \cite[\S 4]{cartier}, \cite[\S XX.4]{kassel}. More precisely, when the braiding $c$ on $\mathcal{C}$ is symmetric, infinitesimal braidings in this other sense can be seen as a subspace (strict in general) of the vector space $\mathbf{T}_c\mathrm{Br}(\mathcal{C})$ defined above, see Remark~\ref{remarkCartierInfBraidings}.

It follows from Theorem \ref{thmDYCohomAndTangentBraidingsProductIntro} and Ocneanu rigidity \cite[\S 7]{ENO}, \cite[\S 3.5]{GHS} that infinitesimal deformations of a braiding on a finite tensor category $\mathcal{C}$ might exist only in the case  of {\em non}-semisimple $\mathcal{C}$. A deep and active trend in quantum topology is the construction of topological invariants from non-semisimple ribbon categories, like Lyubashenko's representations of mapping class groups \cite{lyuMCG} or renormalized link invariants based on modified traces \cite{GKP}. Since the braiding is one of the main ingredients in these constructions, it could be interesting to see how the non-semisimple invariants deform along the infinitesimal deformations of braidings (or higher order deformations).

 Theorem \ref{thmDYCohomAndTangentBraidingsProductIntro} motivates the computation of the DY cohomology of the monoidal functor $P_c$ whose monoidal structure is defined by the braiding $c$. Our third main result is the application of Theorem \ref{th1Intro} to $P_c$, which amounts to describe the object $\widetilde{R}(\boldsymbol{1}) \in \mathcal{Z}(\mathcal{C} \boxtimes \mathcal{C})$ in this case. Namely, in \S\ref{adjunctionThmMonProduct}, we prove:
\begin{thIntro}\label{thIntroCoeffMonProduct}
1. For the choice $F = P_c$, the object $\widetilde{R}(\boldsymbol{1})$ appearing in Theorem \ref{th1Intro} is equal to 
$\bigl( \mathscr{A}, \lambda^{(+),(-)} \bigr)$
where $\mathscr{A} = \int^{X \in \mathcal{C}} X^{\vee} \boxtimes X \in \mathcal{C} \boxtimes \mathcal{C}$ and $\lambda^{(+),(-)}$ is a half-braiding defined in \eqref{halfBraidingOnBulk} by using the braiding $c$ and its inverse in $\mathcal{C}$.
\\2. If the ground field $\Bbbk$ is perfect, $\bigl( \mathscr{A}, \lambda^{(+),(-)} \bigr)$ can be expressed as the coend
\[ 
\Gamma = \int^{X \in \mathcal{C}} \bigl(X^{\vee}, c_{X^{\vee},-}\bigr) \boxtimes \bigl(X, c_{-,X}^{-1}\bigr) \]
through the equivalence $\mathcal{Z}(\mathcal{C} \boxtimes \mathcal{C}) \cong \mathcal{Z}(\mathcal{C}) \boxtimes \mathcal{Z}(\mathcal{C})$.
\end{thIntro}
\noindent The assumption on $\Bbbk$ in item 2 is used in our proof of the equivalence $\mathcal{Z}(\mathcal{C} \boxtimes \mathcal{C}) \cong \mathcal{Z}(\mathcal{C}) \boxtimes \mathcal{Z}(\mathcal{C})$ in Lemma \ref{lemmaDrinfeldCenterOfDeligneProduct}. It ensures that the Deligne product of two resolvent pairs is again a resolvent pair (App.\,\ref{subsectionDeligneResolventPairs}).

\indent  When $\Bbbk$ is perfect, the combination of Theorems \ref{thIntroCoeffMonProduct} and \ref{thmDYCohomAndTangentBraidingsProductIntro} gives a dimension formula:
\begin{equation}\label{dimensionFormulaTangentSpaceIntro}
\dim \mathbf{T}_c\mathrm{Br}(\mathcal{C}) = \dim \Ext^2_{\mathcal{Z}(\mathcal{C}) \boxtimes \mathcal{Z}(\mathcal{C}), \mathcal{C} \boxtimes \mathcal{C}}\bigl(\boldsymbol{1}, \Gamma \bigr) - 2 \dim \Ext^2_{\mathcal{Z}(\mathcal{C}), \mathcal{C}}(\boldsymbol{1}, \boldsymbol{1})
\end{equation}
A nice feature of this formula, which makes it efficient in practice, is that it is enough to know a relatively projective resolution of $\boldsymbol{1} \in\mathcal{Z}(\mathcal{C})$ to compute these relative Ext's (it is even enough knowing the first four terms of the resolution). Indeed a resolution of $\boldsymbol{1} \boxtimes \boldsymbol{1} \in \mathcal{Z}(\mathcal{C}) \boxtimes \mathcal{Z}(\mathcal{C})$ can be obtained by taking the product of the resolution of $\boldsymbol{1} \in\mathcal{Z}(\mathcal{C})$ with itself, thanks to the general results in Appendix \ref{subsectionDeligneResolventPairs} about Deligne product of resolvent pairs.

\indent Furthermore, in \S\ref{sec:end-formula} and  under the assumption that $\Bbbk$ is perfect, we  apply a K\"unneth formula and rewrite the dimension formula~\eqref{dimensionFormulaTangentSpaceIntro} via an explicit end involving only $\Ext^n_{\zcat,\cat}$ at $n=1,2$.
For example, if $\cat$ is  unimodular with a symmetric braiding $c$, we obtain the following end-formula
\begin{multline}
    \label{eq:dim-Tc-Br-Ext-end-sym}
\dim \mathbf{T}_c\mathrm{Br}(\mathcal{C}) =
\dim \int_{P \in \mathrm{Proj}(\mathcal{C})}  \Ext^1_{\mathcal{Z}(\mathcal{C}),\mathcal{C}}\bigl(\boldsymbol{1},(P^{\vee}, c_{P^{\vee},-})\bigr) 
 \otimes \Ext^1_{\mathcal{Z}(\mathcal{C}),\mathcal{C}}\bigl(\boldsymbol{1},(P,c_{P,-}) \bigr)\\
+ 
2 \dim \Ext^2_{\mathcal{Z}(\mathcal{C}),\mathcal{C}}\bigl(\boldsymbol{1},(P_{\one},c_{P_{\one},-}) \bigr)
- 2 \dim \Ext^2_{\mathcal{Z}(\mathcal{C}),\mathcal{C}}(\boldsymbol{1},\one) \ ,
\end{multline} 
where $P_{\one}$ is the projective cover of $\one$ in $\cat$, and $(P,c_{P,-})$ denotes the corresponding object in $\zcat$,
while for the general case expression see Corollary~\ref{cor:HDY-end-proj}.

\smallskip

\indent In \S\ref{sectionFinDimHopf} we specialize our results to the case $\mathcal{C} = H\text{-}\mathrm{mod}$ where $H$ is a finite-dimensional Hopf algebra over a field $\Bbbk$. In this case it is well-known that a braiding on $H\text{-}\mathrm{mod}$ is equivalent to an $R$-matrix in $H^{\otimes 2}$. When they are expressed in a basis of $H$, the defining conditions of an $R$-matrix are polynomial equations, so we have an affine variety of $R$-matrices for $H$. Similarly, an infinitesimal braiding tangent to $c$ on $H\text{-}\mathrm{mod}$ is equivalent to a vector in the Zariski tangent space of the $R$-matrix associated to $c$. Note that $\mathcal{Z}(\mathcal{C}) \cong D(H)\text{-}\mathrm{mod}$ and $\mathcal{Z}(\mathcal{C} \boxtimes \mathcal{C}) \cong D(H \otimes H)\text{-}\mathrm{mod}$. Through these isomorphisms, the object $\bigl( \mathscr{A}, \lambda^{(+),(-)} \bigr)$ in Theorem~\ref{thIntroCoeffMonProduct} is equal to the dual vector space $H^*$ endowed with a $D(H \otimes H)$-module structure based on the coregular actions of $H$ on $H^*$, \textit{cf}. Proposition \ref{propAdjunctionTheoremForMonProductInAmod} to see how the $R$-matrix enters in this action. Hence for $\mathcal{C} = H\text{-}\mathrm{mod}$, the dimension formula~\eqref{dimensionFormulaTangentSpaceIntro} takes the following form:
\begin{coroIntro} Let $H$ be a (quasi-triangular) Hopf algebra over a field $\Bbbk$ with an R-matrix $R$.
We then have
\begin{equation}\label{dimensionFormulaTangentSpaceRMatIntro}
\dim \mathbf{T}_R\mathrm{RMat}(H) = \dim \Ext^2_{D(H \otimes H),H \otimes H}\bigl(\Bbbk, H^* \bigr) - 2 \dim \Ext^2_{D(H),H}(\Bbbk,\Bbbk)
\end{equation}
where $\mathrm{RMat}(H)$ is the affine variety of $R$-matrices on $H$, $\mathbf{T}_R\mathrm{RMat}(H)$ is its Zariski tangent space at the point $R$, $H^*$ has the $D(H \otimes H)$-module structure \eqref{actionDAAOnADual} and $\Bbbk$ is the ground field with the trivial module structure.
\end{coroIntro}

In \S\ref{subsectionExampleBk} we consider the example of $\mathcal{C} = B_k\text{-}\mathrm{mod}$, where $B_k = \Lambda\mathbb{C}^k \rtimes \mathbb{C}[\mathbb{Z}/2\mathbb{Z}]$ is the bosonization of the exterior algebra seen as a Hopf superalgebra. It admits a triangular $R$-matrix $R_0 \in B_k^{\otimes 2}$ and thus a symmetric braiding in $B_k\text{-}\mathrm{mod}$. The end formulas in Corollary \ref{cor:HDY-end-proj}  reveal that the Zariski tangent space to $R_0$ has dimension $k^2$ and also allow for a quick computation of the dimensions of the DY cohomology spaces of $\otimes$ endowed with the monoidal structure coming from the symmetric braiding. For completeness we also obtain these results by using the formula \eqref{dimensionFormulaTangentSpaceRMatIntro}, which requires the analysis of the action of $D(B_k \otimes B_k)$ on the coefficient $\widetilde{R}(\mathbb{C}) = B_k^*$. Moreover we provide an explicit basis of $\mathbf{T}_{R_0}\mathrm{RMat}(B_k)$ and promote this space to a $k^2$-parameter family of genuine $R$-matrices for $B_k$, which is in agreement with \cite{PV}.

\indent Finally, in \S\ref{subsectionDYTwistedFiber} we apply Theorem \ref{th1Intro} to restriction functors. Let $J \in H^{\otimes 2}$ be a Drinfeld twist. Suppose that $K$ is a Hopf subalgebra of $H_J$, which has the product of $H$ but its coproduct is altered by $J$. Then $J$ induces a monoidal structure on the restriction functor $F = \mathrm{Res}^H_K : H\text{-}\mathrm{mod} \to K\text{-}\mathrm{mod}$. We describe the centralizer category $\mathcal{Z}(\mathrm{Res}^H_K)$ as a category of finite-dimensional modules over a ``twisted Drinfeld double'' $D(H_J,K)$, while $\mathcal{Z}(\mathcal{C}) = \mathcal{Z}(H\text{-}\mathrm{mod})$ is of course isomorphic to $D(H)\text{-}\mathrm{mod}$. The object $\widetilde{R}(\boldsymbol{1}) \in D(H)\text{-}\mathrm{mod}$ is $\Hom_K(H,\Bbbk)$ as a vector space. We show that the subalgebra $H \subset D(H)$ acts by coregular action on $\widetilde{R}(\boldsymbol{1})$ while $H^* \subset D(H)$ acts by an adjoint action twisted by $J$; see Corollary \ref{coroChangeFunctorThmForResFunctor}. In the case of quasi-triangular $H$, restriction functor $\mathrm{Res}^{H \otimes H}_H$ induced by the coproduct $\Delta : H \to H \otimes H$ and Drinfeld twist defined from the $R$-matrix, then the corresponding coefficient $\widetilde{R}(\boldsymbol{1})$ recovers the coefficient of Theorem \ref{thIntroCoeffMonProduct}, see Example \ref{exampleComputingCoeffPullbackCoproduct}. An example is provided in \S\ref{subsectionResFunctorBk} for restriction functors $B_n\text{-}\mathrm{mod} \to B_k\text{-}\mathrm{mod}$ with $k \leq n$.

\medskip

\noindent \textbf{Acknowledgements.} We thank A. Davydov for stimulating discussions. This work was done when M.F. was supported by the ANR LabEx CIMI within the French State funding ``Investissements
d’Avenir''. A.M.G is grateful to
Hamburg University for its kind hospitality in 2023 and also thanks the Humboldt Foundation for their financial support during the visit. C.S. is supported in part by the Collaborative Research Centre CRC 1624 ``Higher structures, moduli spaces and integrability'' - 506632645, and the Excellence Cluster EXC 2121 ``Quantum Universe'' - 390833306.

\section{Adjunction theorem for relative Ext groups}\label{sectionAdjThmRelExt}
This section contains  general facts on resolvent pairs which will later be applied to the specific resolvent pair \eqref{morphismAdjunctionZFZCIntro} whose relative Ext groups give the DY cohomology of tensor functors. We first show in \S\ref{sectionRelatingResolventPairs} that appropriate pairs of functors $(F,\Phi)$ connecting two resolvent pairs as displayed in \eqref{diagramRelatingResolventPairs} are compatible with relatively projective resolutions. If the functor $\Phi$ has a right adjoint, we deduce an adjunction formula for the relative Ext groups associated to these resolvent pairs. Then in \S\ref{sectionLiftingAdjunctions} we note from Beck's theorem \cite[Th.\,1]{beck} that any resolvent pair is monadic. We recall the theorem of adjoint lifting along monads \cite{johnstone} which allows one to construct a right adjoint of the top functor $\Phi$ from a right adjoint of the bottom functor $F$. This is relevant in practice because $F$ is easier to manipulate than $\Phi$.

\subsection{Resolvent pairs and relative Ext groups}
Let $\mathcal{A}, \mathcal{B}$ be abelian categories and
\begin{equation}\label{adjunction}
\xymatrix@R=.7em{
\mathcal{A}\ar@/^.7em/[dd]^{\mathcal{U}}\\
\dashv\\
\ar@/^.7em/[uu]^{\mathcal{F}}\mathcal{B}
}
\end{equation}
be a pair of adjoint functors where $\mathcal{F}$ is left adjoint to $\mathcal{U}$. The adjunction \eqref{adjunction} is called a {\em resolvent pair} if the functor $\mathcal{U}$ is additive, exact and faithful. Then $\mathcal{F}$ is automatically an additive functor \cite[\S IV.1, Th.\,3]{MLCat}. Recall from \cite[Chap. IX]{macLane} (also see \cite[\S 2.1]{FGS} for a quick but complete review) that under these assumptions we have the relative Ext groups
\[ \Ext^n_{\mathcal{A},\mathcal{B}}(V,W) \]
for all $n \geq 0$ and $V,W \in \mathcal{A}$. Despite the notation they depend on the adjunction $\mathcal{F} \dashv \mathcal{U}$ and not just on the categories $\mathcal{A}, \mathcal{B}$. They are computed by applying the functor $\Hom_{\mathcal{A}}(-,W)$ to a {\em relatively projective resolution} $0 \leftarrow V \leftarrow P_0 \leftarrow P_1 \leftarrow \ldots$ and taking the cohomology of the resulting cochain complex. The {\em bar resolution} of $V \in \mathcal{A}$ is the relatively projective resolution given by
\begin{equation}\label{generalBarResolution}
\mathrm{Bar}_{\mathcal{A},\mathcal{B}}^{\bullet}(V) = \left(0 \longleftarrow V \xleftarrow{\:\varepsilon_V\:} G(V) \xleftarrow{\:d_1^V\:} G^2(V) \xleftarrow{\:d_2^V\:} \ldots \right) 
\end{equation}
where $G = \mathcal{F}\mathcal{U} : \mathcal{A} \to \mathcal{A}$ is the comonad associated to the adjunction $\mathcal{F} \dashv \mathcal{U}$, $\varepsilon : G \Rightarrow \mathrm{Id}_{\mathcal{A}}$ is its counit and
\begin{equation}\label{generalBarDiff}
d_n^V = \sum_{i=0}^n (-1)^i \,\partial^V_{n,i} \qquad \text{with } \:\partial^V_{n,i} = G^{n-i}\bigl( \varepsilon_{G^i(V)} \bigr).
\end{equation}
\indent Note that if the categories $\mathcal{A}$, $\mathcal{B}$ are $\Bbbk$-linear, where $\Bbbk$ is a field, then the relative Ext groups $\Ext^n_{\mathcal{A},\mathcal{B}}(V,W)$ are actually $\Bbbk$-vector spaces.

\subsection{Relating resolvent pairs and their relative $\mathrm{Ext}$ groups}\label{sectionRelatingResolventPairs}
Let $\mathcal{A}, \mathcal{A}'$ be abelian categories with enough projectives and $\Phi : \mathcal{A} \to \mathcal{A}'$ be an exact functor which has a right adjoint $\Psi : \mathcal{A}' \to \mathcal{A}$. Recall that for usual Ext groups it is in general {\em not} true that
\begin{equation}\label{isoAdjunctionForUsualExt}
\Ext^{\bullet}_{\mathcal{A}'}(\Phi(V),V') \cong \Ext^{\bullet}_{\mathcal{A}}(V, \Psi(V'))
\end{equation}
where $V \in \mathcal{A}$ and $V' \in \mathcal{A}'$. This isomorphism of abelian groups holds true if and only if $\Phi$ preserves projective objects. Indeed, if $\Phi$ preserves projectives and $0 \longleftarrow V \overset{d_0}{\longleftarrow} P_1 \overset{d_1}{\longleftarrow} P_2 \overset{d_2}{\longleftarrow}~\ldots$ is a projective resolution of $V$ in $\mathcal{A}$ then $0 \longleftarrow \Phi(V) \xleftarrow{\Phi(d_0)} \Phi(P_1) \xleftarrow{\Phi(d_1)} \Phi(P_2) \xleftarrow{\Phi(d_2)}~\ldots$ is a projective resolution of $\Phi(V)$ in $\mathcal{A}'$ by exactness of $\Phi$. Thus by adjunction we have a commutative diagram
\begin{equation}\label{adjunctionGivesIsoOfCochainComplexes}
\xymatrix@C=3em{
0 \ar[r] & \Hom_{\mathcal{A}'}(\Phi(P_1), V') \ar[d]^-{\cong} \ar[r]^-{\Phi(d_1)^*} & \Hom_{\mathcal{A}'}(\Phi(P_2), V') \ar[d]^-{\cong} \ar[r]^-{\Phi(d_2)^*} &\ldots\\
0 \ar[r] & \Hom_{\mathcal{A}}(P_1, \Psi(V')) \ar[r]_-{d_1^*} & \Hom_{\mathcal{A}}(P_2,\Psi(V')) \ar[r]_-{d_2^*} &\ldots
} \end{equation}
which gives the isomorphism \eqref{isoAdjunctionForUsualExt}. Conversely if \eqref{isoAdjunctionForUsualExt} holds and $P \in \mathcal{A}$ is a projective object then for all $V' \in \mathcal{A}'$ we have $\Ext^1_{\mathcal{A}'}(\Phi(P),V') \cong \Ext^1_{\mathcal{A}}(P, \Psi(V')) = 0$, proving that $\Phi(P)$ is projective. 

\smallskip

In this section we prove an isomorphism of the form \eqref{isoAdjunctionForUsualExt} for {\em relative} Ext groups of two resolvent pairs. The point is that if these resolvent pairs are related by a well-behaved pair of functors then we can show that their relative cohomologies are related as well. More precisely:

\begin{definition}\label{defMorphismOfResolventPairs}
Consider the following diagram of categories and functors:
\begin{equation}\label{diagramRelatingResolventPairs}
\xymatrix@C=4em@R=.7em{
\mathcal{A}\ar@/^.7em/[dd]^{\mathcal{U}} \ar[r]^{\Phi} &
\mathcal{A}' \ar@/^.7em/[dd]^{\mathcal{U}'}\\
\dashv & \dashv\\
\ar@/^.7em/[uu]^{\mathcal{F}}\mathcal{B} \ar[r]_F & \ar@/^.7em/[uu]^{\mathcal{F}'}\mathcal{B}'
}
\end{equation}
where the columns are adjunctions.
\\1. If there is a natural isomorphism $\alpha : F \mathcal{U} \overset{\sim}{\implies} \mathcal{U}'  \Phi$ then we say that $(F, \Phi,\alpha) : (\mathcal{F} \dashv \mathcal{U}) \to (\mathcal{F}' \dashv \mathcal{U}')$ is a morphism of adjunctions.
\\2. For such $(F, \Phi,\alpha)$ let $\beta^{\alpha} : \mathcal{F}' F \Rightarrow \Phi \mathcal{F}$ be the natural transformation defined by
\begin{equation}\label{defBetaStrongMorphism}
\forall \, X \in \mathcal{B}, \quad \beta_X^{\alpha} : \mathcal{F}'F(X) \xrightarrow{\mathcal{F}'F(\eta_X)} \mathcal{F}'F\mathcal{U}\mathcal{F}(X) \xrightarrow{\mathcal{F}'(\alpha_{\mathcal{F}(X)})} \mathcal{F}'\mathcal{U}'\Phi\mathcal{F}(X) \xrightarrow{\varepsilon'_{\Phi \mathcal{F}(X)}}  \Phi\mathcal{F}(X)
\end{equation}
where $\varepsilon$ is the counit of $\mathcal{F} \dashv \mathcal{U}$ and $\eta'$ is the unit of $\mathcal{F}' \dashv \mathcal{U}'$. If $\beta^{\alpha}$ is an isomorphism then we say that the morphism of adjunctions $(F, \Phi,\alpha)$ is strong.
\\3. If the columns in the diagram \eqref{diagramRelatingResolventPairs} are resolvent pairs and the functors $F$, $\Phi$ are additive then we say that a triple $(F, \Phi,\alpha)$ as in item 1 is a morphism of resolvent pairs.
\end{definition}

The next lemma characterizes the formula defining $\beta^{\alpha}$ in \eqref{defBetaStrongMorphism}. It follows from naturality and the defining properties of units and counits (given e.g. in \cite[\S IV.1]{MLCat}).
\begin{lemma}\label{lemmaCompatibleAdjunctionData}
Let $(F,\Phi,\alpha)$ be a morphism of adjunctions as in Definition \ref{defMorphismOfResolventPairs} and let $\beta : \mathcal{F}'F \Rightarrow \Phi\mathcal{F}$ be any natural transformation. Consider the following diagrams $(D_1)$ and $(D_2)$
\[ \xymatrix@C=3em@R=.2em{
&F(X) \ar[r]^-{\eta'_{F(X)}} \ar[dd]_-{F(\eta_X)} & \mathcal{U}'\mathcal{F}'F(X) \ar[dd]^-{\mathcal{U}'(\beta_X)} \\
(D_1)\!\!\!\!\!\!\!\!\!\!\!\!\!\!\!\!\!\!&&\\
&F\mathcal{U} \mathcal{F}(X) \ar[r]_-{\alpha_{\mathcal{F}(X)}} & \mathcal{U}' \Phi  \mathcal{F}(X)
}
\qquad\qquad
\xymatrix@C=3em@R=.2em{
&\mathcal{F}'F\mathcal{U}(V) \ar[r]^-{\beta_{\mathcal{U}(V)}} \ar[dd]_-{\mathcal{F}'(\alpha_V)}& \Phi\mathcal{F}\mathcal{U}(V) \ar[dd]^-{\Phi(\varepsilon_V)} \\
(D_2)\!\!\!\!\!\!\!\!\!\!\!\!\!\!\!\!\!\!&&\\
&\mathcal{F}' \mathcal{U}' \Phi(V) \ar[r]_-{\varepsilon'_{\Phi(V)}} & \Phi(V)
} \]
where $\eta$ and $\varepsilon$ (resp. $\eta'$ and $\varepsilon'$) are the unit and counit of the adjunction $\mathcal{F} \dashv \mathcal{U}$ (resp. $\mathcal{F}' \dashv \mathcal{U}'$). The following statements are equivalent:
\begin{enumerate}[itemsep=-.1em,topsep=.1em]
\item The diagram $(D_1)$ commutes
for all $X \in \mathcal{B}$.
\item The diagram $(D_2)$ commutes for all $V \in \mathcal{A}$.
\item The natural transformation $\beta$ is equal to $\beta^{\alpha}$ defined in \eqref{defBetaStrongMorphism}.
\end{enumerate}
\end{lemma}

\begin{remark}
1. Lemma \ref{lemmaCompatibleAdjunctionData} is actually true even if $\alpha$ is not an isomorphism, but just any natural transformation  $F \mathcal{U} \Rightarrow \mathcal{U}' \Phi$. Yet another property equivalent to the statements in Lemma \ref{lemmaCompatibleAdjunctionData} is $\alpha_V = \mathcal{U}'\Phi(\varepsilon_V) \circ \mathcal{U}'(\beta_{\mathcal{U}(V)}) \circ \eta'_{F\mathcal{U}(V)}$ for all $V \in \mathcal{A}$.
\\2. In \cite[\S IV.7]{MLCat} a ``map of adjunctions'' is defined to be a pair of functors $(F,\Phi)$ as in \eqref{diagramRelatingResolventPairs} such that $F(\eta_X) = \eta'_{F(X)}$ for all $X \in \mathcal{B}$ or equivalently such that $\Phi(\varepsilon_V) = \varepsilon'_{\Phi(V)}$ for all $V \in \mathcal{A}$. By Lemma \ref{lemmaCompatibleAdjunctionData} a ``map of adjunctions'' is a morphism of adjunctions $(F,\Phi,\alpha)$ such that $\alpha = \mathrm{id}$ and $\beta^{\alpha} = \mathrm{id}$.
\\3. Morphisms of adjunctions of the form $(F,\Phi,\mathrm{id})$ are discussed e.g. in \cite{SS} and \cite{zaganidis}.
\\4. Item 1 in Definition \ref{defMorphismOfResolventPairs} yields a category whose objects are adjunctions (\textit{i.e.} pairs of functors together with a unit and counit). The composition of morphisms is $(F_2,\Phi_2,\alpha^2) \circ (F_1,\Phi_1,\alpha^1) = (F_2 F_1, \,\Phi_2 \Phi_1, \,\alpha^{12} )$ with $\alpha^{12} = \alpha^2_{\Phi_1} \circ F_2(\alpha^1)$. One shows easily that $\beta^{\alpha^{12}} = \Phi_2(\beta^{\alpha^1}) \circ \beta^{\alpha^2}_{F_1}$ (it suffices to check that any one of the diagrams in Lemma \ref{lemmaCompatibleAdjunctionData} commute with this choice of $\beta$). In particular, if $\beta^{\alpha^1}$ and $\beta^{\alpha^2}$ are isomorphisms then so is $\beta^{\alpha^{12}}$. Hence the composition of two strong morphisms of adjunctions is again a strong morphism of adjunctions.
\end{remark}

\indent A morphism of resolvent pair $(F,\Phi,\alpha)$ as in item 3 of Definition \ref{defMorphismOfResolventPairs} allows us to relate the comonads $G = \mathcal{F}\mathcal{U}$ and $G' = \mathcal{F}'\mathcal{U}'$ on $\mathcal{A}$ and $\mathcal{A}'$, from which the bar resolutions \eqref{generalBarResolution} are defined. Indeed consider the natural transformation $\gamma : G' \Phi \Rightarrow \Phi G$ with components
\begin{equation}\label{morphismGammaOfComonads}
\gamma_V : G' \Phi(V) = \mathcal{F}' \mathcal{U}' \Phi(V) \xrightarrow{\mathcal{F}'(\alpha_V^{-1})} \mathcal{F}' F \mathcal{U}(V) \xrightarrow{\beta^{\alpha}_{\mathcal{U}(V)}} \Phi \mathcal{F} \mathcal{U}(V) = \Phi G(V)
\end{equation}
for all $V \in \mathcal{A}$, where $\beta^{\alpha}$ is defined in \eqref{defBetaStrongMorphism}. By diagram $(D_2)$ in Lemma \ref{lemmaCompatibleAdjunctionData} it satisfies
\begin{equation}\label{compatibilityGammaCounit}
\forall \, V \in \mathcal{A}, \quad \varepsilon'_{\Phi(V)} = \Phi(\varepsilon_V) \circ \gamma_V.
\end{equation}
For $n \in \mathbb{N}$ define $\gamma^{(n)} : G'^{\,n} \Phi \Rightarrow \Phi G^n$ by $\gamma^{(0)}_V = \mathrm{id}_{\Phi(V)}$ and
then by induction
\begin{equation}\label{morphismGammaOfComonadsIterated}
\gamma^{(n+1)}_V : G'^{\,n+1} \Phi(V) \xrightarrow{G'(\gamma^{(n)}_V)} G' \Phi G^n(V) \xrightarrow{\gamma_{G^n(V)}} \Phi G^{n+1}(V).
\end{equation}
In particular $\gamma^{(1)}_V = \gamma_V$.

\begin{lemma}\label{preservingCounitsImpliesPreservingResolutions}Let $(F,\Phi,\alpha)$ be a morphism of resolvent pairs as in Def.~\ref{defMorphismOfResolventPairs}(3). For all $V \in \mathcal{A}$, the morphisms $\gamma^{(n)}_V$ provide a morphism of chain complexes
\[ \gamma^{(\bullet)}_V : \mathrm{Bar}_{\mathcal{A}',\mathcal{B}'}^{\bullet}\bigl( \Phi(V) \bigr) \to \Phi\bigl( \mathrm{Bar}_{\mathcal{A},\mathcal{B}}^{\bullet}(V) \bigr). \]
\end{lemma}
\begin{proof}
Let $\partial^V_{n,i}$ and $\partial^{\Phi(V)}_{n,i}$ be the coface morphisms of the two bar complexes, see \eqref{generalBarDiff}. We prove by induction on $n$ that
\[ \forall\, i \in \{ 0, \ldots, n \}, \quad \gamma^{(n)}_V \circ \partial^{\Phi(V)}_{n,i} = \Phi(\partial^V_{n,i}) \circ \gamma^{(n+1)}_V. \]
The case $n=0$ is \eqref{compatibilityGammaCounit}. Assume that this is true for some $n \geq 0$. Note by definition of the coface maps that $\partial^{\Phi(V)}_{n+1,i} = G'(\partial^{\Phi(V)}_{n,i})$ and $\partial^V_{n+1,i} = G (\partial^V_{n,i})$. Then for any $i \in \{ 0, \ldots, n \}$ we have by \eqref{morphismGammaOfComonadsIterated}, naturality of $\gamma$ and the induction hypothesis
\begin{align*}
\gamma^{(n+1)}_V \circ \partial^{\Phi(V)}_{n+1,i} &= \gamma_{G^{n}(V)} \circ G'(\gamma^{(n)}_V) \circ G'(\partial^{\Phi(V)}_{n,i}) = \gamma_{G^{n}(V)} \circ G'\Phi(\partial^V_{n,i}) \circ G'\bigl( \gamma^{(n+1)}_V \bigr)\\
&= \Phi G (\partial^V_{n,i}) \circ \gamma_{G^{n+1}(V)} \circ G'\bigl( \gamma^{(n+1)}_V \bigr) = \Phi(\partial^V_{n+1,i}) \circ \gamma^{(n+2)}_V.
\end{align*}
The case $i = n+1$ is treated separately without using the induction hypothesis. Indeed, by naturality of $\varepsilon'$, \eqref{compatibilityGammaCounit} and \eqref{morphismGammaOfComonadsIterated} we have
\begin{align*}
\gamma^{(n+1)}_V \circ \partial^{\Phi(V)}_{n+1,n+1} &= \gamma^{(n+1)}_V \circ \varepsilon'_{G'^{\, n+1}\Phi(V)} = \varepsilon'_{\Phi G^{n+1}(V)} \circ G'(\gamma^{(n+1)}_V)\\
&= \Phi\bigl( \varepsilon_{G^{n+1}(V)} \bigr) \circ \gamma_{G^{n+1}(V)} \circ G'(\gamma^{(n+1)}_V) = \Phi(\partial^V_{n+1,n+1}) \circ \gamma^{(n+2)}_V.
\end{align*}
Since $\Phi$ is additive it follows that $\gamma^{(n)}_V \circ \Phi(d^V_{n}) = d^{\Phi(V)}_{n} \circ \gamma^{(n+1)}_V$ for all $n$, where $d^V$ and $d^{\Phi(V)}$ are the differentials \eqref{generalBarDiff} of $\mathrm{Bar}_{\mathcal{A},\mathcal{B}}^{\bullet}(V)$ and $ \mathrm{Bar}_{\mathcal{A}',\mathcal{B}'}^{\bullet}\bigl( \Phi(V) \bigr)$.
\end{proof}

\indent Recall the notion of a strong morphism of adjunctions (item 2 in Definition \ref{defMorphismOfResolventPairs}).
\begin{theorem}\label{propExtGroupsForMorphismsOfResolventPairs}
Let $(F, \Phi,\alpha)$ be a morphism of resolvent pairs as in Def.~\ref{defMorphismOfResolventPairs}(3) and assume that $\Phi : \mathcal{A} \to \mathcal{A}'$ has a right adjoint $\Psi$. Then there are morphisms of abelian groups
\[ \forall \, n \geq 0, \quad \Ext^n_{\mathcal{A}, \mathcal{B}}(V, \Psi(V')) \to \Ext^n_{\mathcal{A}', \mathcal{B}'}(\Phi(V), V') \]
for all $V \in \mathcal{A}$, $V' \in \mathcal{A}'$. If $(F,\Phi,\alpha)$ is strong then they are isomorphisms of abelian groups.
\end{theorem}
\begin{proof}
By Lemma \ref{preservingCounitsImpliesPreservingResolutions} and adjunction $\Phi \dashv \Psi$ we have a commutative diagram
\begin{equation}\label{diagramProofAdjThm}
\xymatrix@C=4em{
\ldots \ar[r]^-{(d^V_{n-1})^*} & \Hom_{\mathcal{A}}(G^n(V), \Psi(V')) \ar[r]^-{(d^V_n)^*} \ar[d]^-{\cong} & \Hom_{\mathcal{A}}(G^{n+1}(V),\Psi(V')) \ar[r]^-{(d^V_{n+1})*} \ar[d]^-{\cong} &\ldots\\
\ldots \ar[r]^-{\Phi(d^V_{n-1})^*} & \Hom_{\mathcal{A}'}(\Phi G^n(V), V') \ar[d]^-{(\gamma^{(n)})^*} \ar[r]^-{\Phi(d^V_n)^*} & \Hom_{\mathcal{A}'}(\Phi G^{n+1}(V), V') \ar[d]^-{(\gamma^{(n+1)})^*} \ar[r]^-{\Phi(d^V_{n+1})^*} &\ldots\\
\ldots \ar[r]^-{(d^{\Phi(V)}_{n-1})^*} & \Hom_{\mathcal{A}'}(G'^{\, n}\Phi(V), V') \ar[r]^-{(d^{\Phi(V)}_n)^*} & \Hom_{\mathcal{A}'}(G'^{\, n+1}\Phi(V), V') \ar[r]^-{(d^{\Phi(V)}_{n+1})^*} &\ldots
} \end{equation}
The cohomology of the first row is $\Ext^{\bullet}_{\mathcal{A}, \mathcal{B}}(V, \Psi(V'))$ while the cohomology of the third row is $\Ext^n_{\mathcal{A}', \mathcal{B}'}(\Phi(V), V')$, so we get the desired morphisms of abelian groups. For the second claim note that if $(F,\Phi,\alpha)$ is strong then the natural transformation $\gamma : G'\Phi \Rightarrow \Phi G$ is an isomorphism whose inverse is given by $\gamma_V^{-1} = \mathcal{F}'(\alpha_V) \circ (\beta^{\alpha}_{\mathcal{U}(V)})^{-1}$. It follows by induction \eqref{morphismGammaOfComonadsIterated} that $\gamma^{(n)}$ is an isomorphism for all $n$. Hence the columns in the diagram above are isomorphisms.
\end{proof}

The next proposition means that strong morphisms of resolvent pairs are compatible with all the definitions for relative cohomology. Compare with the introductory discussion at the beginning of this section and note that exactness assumptions are not required in the relative case.

\begin{proposition}\label{propStrongMorphismsPreserveEverything}
Let $(F,\Phi,\alpha)$ be a {\em strong} morphism of resolvent pairs as in Def.~\ref{defMorphismOfResolventPairs}(3).
\\1. If a morphism $f$ in $\mathcal{A}$ is allowable for $\mathcal{F} \dashv \mathcal{U}$ then $\Phi(f)$ is allowable for $\mathcal{F}' \dashv \mathcal{U}'$.
\\2. If $P \in \mathcal{A}$ is relatively projective for  $\mathcal{F} \dashv \mathcal{U}$ then $\Phi(P)$ is relatively projective for $\mathcal{F}' \dashv \mathcal{U}'$.
\\3. If $0 \longleftarrow V \overset{\delta_0}{\longleftarrow} P_1 \overset{\delta_1}{\longleftarrow} P_2 \overset{\delta_2}{\longleftarrow} \ldots$ is a relatively projective resolution of $V \in \mathcal{A}$ then $0 \longleftarrow \Phi(V) \xleftarrow{\Phi(\delta_0)} \Phi(P_1) \xleftarrow{\Phi(\delta_1)} \Phi(P_2) \xleftarrow{\Phi(\delta_2)} \ldots$ is a relatively projective resolution of $\Phi(V) \in \mathcal{A}'$.
\end{proposition}
\begin{proof}
1. We recall that $f \in \Hom_{\mathcal{A}}(V,W)$ is called {\em allowable} if there exists $s \in \Hom_{\mathcal{B}}\bigl(\mathcal{U}(W), \mathcal{U}(V)\bigr)$ such that $\mathcal{U}(f) \circ s \circ \mathcal{U}(f) = \mathcal{U}(f)$. Define
\[ s' : \mathcal{U}'(\Phi(W)) \xrightarrow{\alpha^{-1}_W} F(\mathcal{U}(W)) \xrightarrow{F(s)} F(\mathcal{U}(V)) \xrightarrow{\alpha_V} \mathcal{U}'(\Phi(V)). \]
By naturality of $\alpha$ we have $\mathcal{U}'(\Phi(f)) \circ s' \circ \mathcal{U}'(\Phi(f)) = \mathcal{U}'(\Phi(f))$ and thus $\Phi(f)$ is allowable.
\\2. The object $P$ is relatively projective if and only if it is a direct summand of $G(V)$ for some $V \in \mathcal{A}$ \cite[Prop.~2.17(1)]{FGS}. Note that by the strongness assumption $\gamma_V : \Phi G(V) \to G' \Phi(V)$ from \eqref{morphismGammaOfComonads} is an iso. Hence, since $\Phi$ is additive, $\Phi(P)$ is a direct summand of $\Phi G(V) \cong G'\Phi(V)$ and is thus relatively projective.
\\3. By items 1 and 2 we know that $\Phi(P_n)$ is a relatively projective object and that $\Phi(\delta_n)$ is allowable for each $n$. It remains to prove that the image under $\Phi$ of the resolution is an exact sequence. This is based on Lemma \ref{preservingCounitsImpliesPreservingResolutions} which here gives an isomorphism of cochain complexes due to the strongness assumption. The comparison theorem of relatively projective resolutions \cite[Th.\,IX.4.3]{macLane} ensures that $\mathrm{id}_V$ can be lifted to morphisms of chain complexes in two ways:
\[ \xymatrix{
0 & \ar[l] V \ar@{=}[d] & \ar[l]_-{\delta_0} \ar[d]^{f_1} P_1 & \ar[l]_-{\delta_1} \ar[d]^{f_2} P_2 & \ar[l]_-{\delta_2} \ldots\\
0 & \ar[l] V & \ar[l]^-{\varepsilon_V} G(V) & \ar[l]^-{d_1^V} G^2(V) & \ar[l]^-{d^2_V} \ldots
} \qquad\quad \xymatrix{
0 & \ar[l] V \ar@{=}[d] & \ar[l]_-{\varepsilon_V} \ar[d]^{g_1} G(V) & \ar[l]_-{d^V_1} \ar[d]^{g_2} G^2(V) & \ar[l]_-{d^V_2} \ldots\\
0 & \ar[l] V & \ar[l]^-{\delta_0} P_1 & \ar[l]^-{\delta_1} P_2 & \ar[l]^-{\delta_2} \ldots
} \]
Note that $g_{\bullet} \circ f_{\bullet}$ and $\mathrm{id}_{P_{\bullet}}$ are two lifts of $\mathrm{id}_V$ along the resolution $V \leftarrow P_{\bullet}$. Hence, still by the comparison theorem \cite[Th.\,IX.4.3]{macLane}, these two lifts are chain homotopic. Applying $\Phi$ to the homotopy it follows that $\Phi(g_{\bullet}) \circ \Phi(f_{\bullet})$ and $\mathrm{id}_{\Phi(P_{\bullet})}$ are homotopic. The same argument gives that $\Phi(f_{\bullet}) \circ \Phi(g_{\bullet})$ and $\mathrm{id}_{\Phi G^{\bullet}(V)}$ are homotopic. Hence $\Phi(f_{\bullet}) : \Phi(P_{\bullet}) \to \Phi G^{\bullet}(V)$ and $\Phi(g_{\bullet}) : \Phi G^{\bullet}(V) \to \Phi(P_{\bullet})$ are inverse to each other up to homotopy. We thus have quasi-isomorphisms
\[ \Phi(P_{\bullet}) \cong \Phi\bigl( \mathrm{Bar}_{\mathcal{A},\mathcal{B}}^{\bullet}(V) \bigr) \cong \mathrm{Bar}_{\mathcal{A}',\mathcal{B}'}^{\bullet}\bigl(\Phi(V)\bigr) \]
where $P_0 = V$. Since the bar resolution is exact, the complex $\Phi(V) \leftarrow \Phi(P_{\bullet})$ is exact as well.
\end{proof}

\subsection{Lifting adjunctions along resolvent pairs}\label{sectionLiftingAdjunctions}
Recall diagram \eqref{diagramRelatingResolventPairs}. In order to apply Theorem \ref{propExtGroupsForMorphismsOfResolventPairs}, one would like to deduce the existence of a right adjoint of $\Phi$ from the existence of a right adjoint to $F$. We will see that this is indeed possible, by first showing that every resolvent pair is monadic and then by using the technology of lifting right adjoints along categories of modules over monads which is reviewed below from \cite{johnstone} (and also \cite[\S 3.5]{BLV}). We start with a few preliminaries.

\smallskip

\indent Let $\mathcal{C}$ be a category and $\mathbb{T} = (T, \mu, \eta)$ be a monad on $\mathcal{C}$, with underlying functor $T : \mathcal{C} \to \mathcal{C}$, multiplication $\mu : T T \Rightarrow T$ and unit $\eta : \mathrm{Id}_{\mathcal{C}} \Rightarrow T$ \cite[Chap.\,VI]{MLCat}. Let $\mathbb{T}\text{-}\mathrm{mod}$ be the category of $\mathbb{T}$-modules (a.k.a.\,$\mathbb{T}$-algebras, a.k.a.\,Eilenberg--Moore category), whose objects are pairs $(V, r)$ where $V \in \mathcal{C}$ and $r \in \Hom_{\mathcal{C}}(T(V), V)$ satisfies $r \circ T(r) = r \circ \mu_V$ and $r \circ \eta_V = \mathrm{id}_V$. A morphism $f : (V,r) \to (W,u)$ in $\mathbb{T}\text{-}\mathrm{mod}$ is $f \in \Hom_{\mathcal{C}}(V,W)$ such that $f \circ r = u \circ T(f)$. This yields the adjunction
\begin{equation}\label{standardAdjunctionMonad}
\xymatrix@R=.7em{
\mathbb{T}\text{-}\mathrm{mod} \ar@/^.7em/[dd]^{\mathcal{U}_{\mathbb{T}}} & \qquad \mathcal{F}_{\mathbb{T}}(X) = (T(X), \mu_X), \quad \mathcal{F}_{\mathbb{T}}(f) = T(f)\\
\dashv  &\\
\ar@/^.7em/[uu]^{\mathcal{F}_{\mathbb{T}}}\mathcal{C} & \qquad \mathcal{U}_{\mathbb{T}}(V,r) = V, \quad \mathcal{U}_{\mathbb{T}}(f) = f
}
\end{equation}
such that $T = \mathcal{U}_{\mathbb{T}} \circ \mathcal{F}_{\mathbb{T}}$. The unit of this adjunction is just the unit $\eta$ of the monad while its counit $\varepsilon$ is given by
\begin{equation}\label{counitGivesTheAction}
\forall \, (V,r) \in \mathbb{T}\text{-}\mathrm{mod}, \quad \varepsilon_{(V,r)} = r \in \Hom_{\mathbb{T}\text{-}\mathrm{mod}}\bigl( \mathcal{F}_{\mathbb{T}}(V), (V,r) \bigr).
\end{equation}

\smallskip

\indent Let $\mathcal{F} : \mathcal{B} \to \mathcal{A}$ and $\mathcal{U} : \mathcal{A} \to \mathcal{B}$ be a pair of adjoint functors: $\mathcal{F} \dashv \mathcal{U}$. Then we have the monad $\mathbb{T} = \bigl(\mathcal{U}\mathcal{F}, \:\mathcal{U}(\varepsilon_{\mathcal{F}(-)}), \:\eta \bigr)$ on $\mathcal{B}$, where $\eta : \mathrm{Id}_{\mathcal{B}} \Rightarrow \mathcal{U}\mathcal{F}$ and $\varepsilon : \mathcal{F}\mathcal{U} \Rightarrow \mathrm{Id}_{\mathcal{A}}$ are the unit and counit of the adjunction $\mathcal{F} \dashv \mathcal{U}$. There exists a unique functor $K : \mathcal{A} \to \mathbb{T}\text{-}\mathrm{mod}$, called {\em comparison functor} \cite[\S VI.3]{MLCat}, such that $\mathcal{U}_{\mathbb{T}} K = \mathcal{U}$ and $K \mathcal{F} = \mathcal{F}_{\mathbb{T}}$. It is given by
\begin{equation}\label{comparisonFunctor}
K(V) = \bigl( \mathcal{U}(V), \mathcal{U}(\varepsilon_V) \bigr), \qquad K(f) = \mathcal{U}(f)
\end{equation}
on an object $V$ and a morphism $f$ in $\mathcal{A}$ \cite[\S VI.3]{MLCat}. The adjunction $\mathcal{F} \dashv \mathcal{U}$ is called {\em monadic} if $K$ is an equivalence of categories.

\begin{proposition}\label{propResolventPairsAreMonadic}
1. Any resolvent pair is a monadic adjunction.
\\2. Conversely, if $\mathbb{T} = (T,\mu,\eta)$ is a monad on an abelian category $\mathcal{C}$ such that $T : \mathcal{C} \to \mathcal{C}$ is additive then the adjunction \eqref{standardAdjunctionMonad} is a resolvent pair.
\\3. Let $\xymatrix{\mathcal{A} \ar@<5pt>[r]^{\mathcal{U}}_{\text{\rotatebox{90}{$\dashv$}}} & \ar@<5pt>[l]^{\mathcal{F}} \mathcal{B}}$ be a resolvent pair and $\mathbb{T}$ the associated monad on $\mathcal{B}$. Then
\vspace{-.7em}
\[ \Ext^{\bullet}_{\mathcal{A},\mathcal{B}}(V,W) \cong \Ext^{\bullet}_{\mathbb{T}\text{-}\mathrm{mod},\mathcal{B}}\bigl( K(V), K(W) \bigr) \]
for all $V,W \in \mathcal{A}$, where $K : \mathcal{A} \to \mathbb{T}\text{-}\mathrm{mod}$ is the comparison functor.
\end{proposition}
\begin{proof}
1. Take a resolvent pair $\mathcal{F} \dashv \mathcal{U}$ as in \eqref{adjunction}. Note that the functor $\mathcal{U}$ is faithful by definition and recall that any faithful functor reflects epimorphisms and monomorphisms.\footnote{Indeed if $f$ is a morphism in $\mathcal{A}$ such that $\mathcal{U}(f)$ is an epimorphism then for any morphisms $g,g'$ such that $g \circ f = g' \circ f$ we have $\mathcal{U}(g) \circ \mathcal{U}(f) = \mathcal{U}(g') \circ \mathcal{U}(f)$, whence $\mathcal{U}(g) = \mathcal{U}(g')$ and faithfulness of $\mathcal{U}$ gives $g=g'$ and thus $f$ is epi. A similar argument applies to monomorphisms.} Let $f$ be a morphism in $\mathcal{A}$ such that $\mathcal{U}(f)$ is an isomorphism. Then in particular $\mathcal{U}(f)$ is both a monomorphism and an epimorphism. By the preliminary remark, $f$ is both a monomorphism and an epimorphism. Since $\mathcal{A}$ is an abelian category, it follows that $f$ is an isomorphism \cite[\S VIII.3]{MLCat}. Hence:
\begin{itemize}[itemsep=-.1em, topsep=.2em]
\item $\mathcal{U}$ reflects isomorphisms.
\item $\mathcal{A}$ has all coequalizers (by existence of cokernels in an abelian category).
\item $\mathcal{U}$ preserves coequalizers (because it is exact and in particular it preserves cokernels).
\end{itemize}
We can thus apply Beck's monadicity theorem \cite[Th.\,1]{beck} by which $K$ is an equivalence.
\\2. $\mathbb{T}\text{-}\mathrm{mod}$ is an abelian category under these assumptions \cite[Prop.\,5.3]{EM}. The functor $\mathcal{U}_{\mathbb{T}}$ is exact due to the definition of kernels and cokernels in $\mathbb{T}\text{-}\mathrm{mod}$, and it is obviously faithful and additive.
\\3. By item 2 the relative Ext groups on the right-hand side make sense. From the defining properties of the comparison functor $K$ we see that $\bigl( \mathrm{Id}_{\mathcal{B}}, K, \mathrm{id} \bigr)$ is a morphism of resolvent pairs in the sense of Def.~\ref{defMorphismOfResolventPairs}. The diagram (D1) in Lemma \ref{lemmaCompatibleAdjunctionData} commutes with the choice $\beta = \mathrm{id}$ since $F = \mathrm{Id}$ in the present situation. Hence $\beta^{\mathrm{id}} = \mathrm{id}$ by Lemma \ref{lemmaCompatibleAdjunctionData} and in particular the morphism of adjunctions $\bigl( \mathrm{Id}_{\mathcal{B}}, K, \mathrm{id} \bigr)$ is strong. One can choose the quasi-inverse $\bar K : \mathbb{T}\text{-}\mathrm{mod} \to \mathcal{A}$ so that it is a right-adjoint of $K$ \cite[\S IV.4]{MLCat}. Then by Theorem \ref{propExtGroupsForMorphismsOfResolventPairs} we have $\Ext^{\bullet}_{\mathcal{A},\mathcal{B}}(V,W) \cong \Ext^{\bullet}_{\mathcal{A},\mathcal{B}}\bigl( V, \bar K K(W) \bigr) \cong \Ext^{\bullet}_{\mathbb{T}\text{-}\mathrm{mod},\mathcal{B}}\bigl( K(V), K(W) \bigr)$.
\end{proof}
\noindent It follows from Proposition \ref{propResolventPairsAreMonadic} that we can restrict ourselves to resolvent pairs arising from linear monads.

\smallskip

\indent Let $\mathcal{C}$, $\mathcal{D}$ be categories (not necessarily abelian), $\mathbb{T} = (T, \mu, \eta)$ be a monad on $\mathcal{C}$ and $\mathbb{T}' = (T', \mu', \eta')$ be a monad on $\mathcal{D}$. 
\begin{definition}{\em \cite[\S 1]{street} }\label{defMorphismMonads}
A morphism of monads $\mathbb{T} \to \mathbb{T}'$ is a pair $(F,\zeta)$ where $F : \mathcal{C} \to \mathcal{D}$ is a functor and  $\zeta : T' F \Rightarrow F T$ is a natural transformation such that
\begin{equation}\label{conditionsZetaLiftF}
\forall \, X \in \mathcal{C}, \qquad 
\begin{array}{l}
\zeta_X \circ \mu'_{F(X)} = F(\mu_X) \circ \zeta_{T(X)} \circ T'(\zeta_X),\\[.5em]
\zeta_X \circ \eta'_{F(X)} = F(\eta_X).
\end{array}
\end{equation}
We say that a morphism of monads $(F,\zeta)$ is strong if $\zeta$ is an isomorphism.
\end{definition}
\noindent This yields a category $\mathbf{Mnd}$ whose objects are monads. The composition of $(G,\omega) : \mathbb{T}' \to \mathbb{T}''$ with $(F,\zeta) : \mathbb{T} \to \mathbb{T}'$ is $(GF, G(\zeta) \circ \omega_{F(-)}) : \mathbb{T} \to \mathbb{T}''$.

\smallskip

The next lemma relates morphisms of adjunctions (Def. \ref{defMorphismOfResolventPairs}) with morphisms of monads; it is a slight adaptation of \cite[Lem.\,1]{johnstone}, \cite[\S 3.5]{BLV} or \cite[Prop.\,2.2.4]{zaganidis}.
\begin{lemma}\label{liftedFunctor}
1. Let $(F,\zeta)$ be a morphism of monads $\mathbb{T} \to \mathbb{T}'$. Then there is a lifted functor $\widetilde{F}_{\zeta} : \mathbb{T}\text{-}\mathrm{mod} \to \mathbb{T}'\text{-}\mathrm{mod}$ defined by
\begin{equation}\label{liftFunctorAlongMonads}
\widetilde{F}_{\zeta}(V,r) = \bigl( F(V), F(r) \circ \zeta_V \bigr), \qquad \widetilde{F}_{\zeta}(f) = F(f)
\end{equation}
and $(F,\widetilde{F}_{\zeta},\mathrm{id})$ is a morphism between the adjunctions defined by $\mathbb{T}$ and $\mathbb{T}'$ (see \eqref{standardAdjunctionMonad}).  Moreover if $(F,\zeta)$ is strong then $(F,\widetilde{F}_{\zeta},\mathrm{id})$ is strong (item 2 in Def. \ref{defMorphismOfResolventPairs}).
\\2. Conversely, consider the diagram
\[ \xymatrix@C=4em@R=.7em{
\mathbb{T}\text{-}\mathrm{mod}\ar@/^.7em/[dd]^{\mathcal{U}_{\mathbb{T}}} \ar[r]^{\Phi} &
\mathbb{T}'\text{-}\mathrm{mod} \ar@/^.7em/[dd]^{\mathcal{U}_{\mathbb{T}'}}\\
\dashv & \dashv\\
\ar@/^.7em/[uu]^{\mathcal{F}_{\mathbb{T}}}\mathcal{C} \ar[r]_F & \ar@/^.7em/[uu]^{\mathcal{F}_{\mathbb{T}'}}\mathcal{D}
} \]
and assume that $(F,\Phi,\alpha)$ is a morphism of adjunctions. Then there exists a morphism of monads $(F,\zeta) : \mathbb{T} \to \mathbb{T}'$ such that $\Phi \cong \widetilde{F}_{\zeta}$. Moreover if $(F,\Phi,\alpha)$ is strong then $(F,\zeta)$ is strong.
\end{lemma}
\begin{proof}
1. The conditions \eqref{conditionsZetaLiftF} ensure precisely that $F(r) \circ \zeta_V$ is a $\mathbb{T}'$-module structure on $F(V)$ and hence $\widetilde{F}_{\zeta}$ is well-defined. It is readily seen that $F \,\mathcal{U}_{\mathbb{T}} = \mathcal{U}_{\mathbb{T}'} \widetilde{F}_{\zeta}$, so $(F,\widetilde{F}_{\zeta},\mathrm{id})$ is a morphism of adjunctions. For the last claim note that $\widetilde{F}_{\zeta}\mathcal{F}_{\mathbb{T}}(X) = \bigl( FT(X), F(\mu_X) \circ \zeta_{T(X)}\bigr)$ so that $\beta^{\mathrm{id}} : \mathcal{F}_{\mathbb{T}'}F \Rightarrow \widetilde{F}_{\zeta} \mathcal{F}_{\mathbb{T}}$ from \eqref{defBetaStrongMorphism} is given by
\begin{equation}\label{betaForMonads}
\beta^{\mathrm{id}}_X = \varepsilon'_{\widetilde{F}_{\zeta}\mathcal{F}_{\mathbb{T}}(X)} \circ \mathcal{F}_{\mathbb{T}'}F(\eta_X) = F(\mu_X) \circ \zeta_{T(X)} \circ T'F(\eta_X) =F(\mu_X) \circ FT(\eta_X) \circ \zeta_X = \zeta_X
\end{equation}
where for the second equality we used \eqref{counitGivesTheAction} and the definition of $\mathcal{F}_{\mathbb{T}'}$ on morphisms, the third equality is by naturality of $\zeta$ and the last equality uses the unit axiom for monads.
\\2. We are given a natural isomorphism $\alpha : F \mathcal{U}_{\mathbb{T}} \overset{\sim}{\implies} \mathcal{U}_{\mathbb{T}'} \Phi$ from which we can define a natural transformation $\beta^{\alpha} : \mathcal{F}_{\mathbb{T}'} F \Rightarrow \Phi \mathcal{F}_{\mathbb{T}}$ as in \eqref{defBetaStrongMorphism}. For all $X \in \mathcal{C}$ let
\begin{equation*}
\zeta_X : T'F(X) = \mathcal{U}_{\mathbb{T}'}\mathcal{F}_{\mathbb{T}'}F(X) \xrightarrow{\mathcal{U}_{\mathbb{T}'}(\beta^{\alpha}_X)} \mathcal{U}_{\mathbb{T}'} \Phi \mathcal{F}_{\mathbb{T}}(X) \xrightarrow{\alpha_{\mathcal{F}_{\mathbb{T}}(X)}^{-1}} F\mathcal{U}_{\mathbb{T}}\mathcal{F}_{\mathbb{T}}(X) = FT(X).
\end{equation*}
We want to show that $\zeta$ satisfies \eqref{conditionsZetaLiftF}. Note first by \eqref{counitGivesTheAction} and by naturality of $\alpha$ that
\begin{equation}\label{firstTechnicalFactProofLemmaMorphMon}
F(\mu_X) \circ \alpha_{\mathcal{F}_{\mathbb{T}}T(X)}^{-1} = F\mathcal{U}_{\mathbb{T}}(\varepsilon_{\mathcal{F}_{\mathbb{T}}(X)}) \circ \alpha_{\mathcal{F}_{\mathbb{T}}\mathcal{U}_{\mathbb{T}}\mathcal{F}_{\mathbb{T}}(X)}^{-1} = \alpha_{\mathcal{F}_{\mathbb{T}}(X)}^{-1} \circ \mathcal{U}_{\mathbb{T}'}\Phi(\varepsilon_{\mathcal{F}_{\mathbb{T}}(X)})
\end{equation}
for all $X \in \mathcal{C}$. Moreover by items 2 and 3 in Lemma \ref{lemmaCompatibleAdjunctionData} we have
\begin{equation}\label{secondTechnicalFactProofLemmaMorphMon}
\Phi(\varepsilon_{\mathcal{F}_{\mathbb{T}}(X)}) \circ \beta^{\alpha}_{T(X)} = \Phi(\varepsilon_{\mathcal{F}_{\mathbb{T}}(X)}) \circ \beta^{\alpha}_{\mathcal{U}_{\mathbb{T}}\mathcal{F}_{\mathbb{T}}(X)} = \varepsilon'_{\Phi\mathcal{F}_{\mathbb{T}}(X)} \circ \mathcal{F}_{\mathbb{T}'}(\alpha_{\mathcal{F}_{\mathbb{T}}(X)}).
\end{equation}
Hence
\begin{align*}
&F(\mu_X) \circ \zeta_{T(X)} \circ T'(\zeta_X) = F(\mu_X) \circ \alpha_{\mathcal{F}_{\mathbb{T}}T(X)}^{-1} \circ \mathcal{U}_{\mathbb{T}'}(\beta^{\alpha}_{T(X)}) \circ T'(\alpha_{\mathcal{F}_{\mathbb{T}}(X)}^{-1}) \circ T'\mathcal{U}_{\mathbb{T}'}(\beta^{\alpha}_X)\\
&\overset{\eqref{firstTechnicalFactProofLemmaMorphMon}}{=} \alpha_{\mathcal{F}_{\mathbb{T}}(X)}^{-1} \circ \mathcal{U}_{\mathbb{T}'}\Phi(\varepsilon_{\mathcal{F}_{\mathbb{T}}(X)}) \circ \mathcal{U}_{\mathbb{T}'}(\beta^{\alpha}_{T(X)}) \circ T'(\alpha_{\mathcal{F}_{\mathbb{T}}(X)}^{-1}) \circ T'\mathcal{U}_{\mathbb{T}'}(\beta^{\alpha}_X)\\
&\overset{\eqref{secondTechnicalFactProofLemmaMorphMon}}{=} \alpha_{\mathcal{F}_{\mathbb{T}}(X)}^{-1} \circ \mathcal{U}_{\mathbb{T}'}(\varepsilon'_{\Phi\mathcal{F}_{\mathbb{T}}(X)}) \circ \mathcal{U}_{\mathbb{T}'}\mathcal{F}_{\mathbb{T}'}(\alpha_{\mathcal{F}_{\mathbb{T}}(X)}) \circ T'(\alpha_{\mathcal{F}_{\mathbb{T}}(X)}^{-1}) \circ T'\mathcal{U}_{\mathbb{T}'}(\beta^{\alpha}_X)\\
&= \alpha_{\mathcal{F}_{\mathbb{T}}(X)}^{-1} \circ \mathcal{U}_{\mathbb{T}'}(\varepsilon'_{\Phi\mathcal{F}_{\mathbb{T}}(X)}) \circ \mathcal{U}_{\mathbb{T}'}\mathcal{F}_{\mathbb{T}'}\mathcal{U}_{\mathbb{T}'}(\beta^{\alpha}_X) = \alpha_{\mathcal{F}_{\mathbb{T}}(X)}^{-1} \circ \mathcal{U}_{\mathbb{T}'}(\beta^{\alpha}_X) \circ \mathcal{U}_{\mathbb{T}'}(\varepsilon'_{\mathcal{F}_{\mathbb{T}'}F(X)}) = \zeta_X \circ \mu'_{F(X)}
\end{align*}
where the second-to-last equality is by naturality of $\varepsilon'$ while the last equality is by \eqref{standardAdjunctionMonad} and \eqref{counitGivesTheAction}. Multiplying the equality (D1) in Lemma \ref{lemmaCompatibleAdjunctionData} by $\alpha^{-1}$ we also have $\zeta_X \circ \eta'_{F(X)} = F(\eta_X)$ for all $X \in~\mathcal{C}$. Hence $\zeta$ satisfies the condition \eqref{conditionsZetaLiftF} and we have the functor $\widetilde{F}_{\zeta}$ defined as in \eqref{liftFunctorAlongMonads} above. We claim that $\alpha$ provides a natural isomorphism $\widetilde{F}_{\zeta} \overset{\sim}{\Rightarrow} \Phi$, \textit{i.e.} $\alpha_{(V,r)} \in \Hom_{\mathbb{T}'\text{-}\mathrm{mod}}\bigl( \widetilde{F}_{\zeta}(V,r), \Phi(V,r) \bigr)$ for all $(V,r) \in \mathbb{T}\text{-}\mathrm{mod}$. Indeed, denote $\Phi(V,r) = (W,s)$ for some $W \in \mathcal{D}$. Then
\begin{align*}
s \circ T'(\alpha_{(V,r)}) &= \mathcal{U}_{\mathbb{T}'}(\varepsilon'_{\Phi(V,r)}) \circ \mathcal{U}_{\mathbb{T}'}\mathcal{F}_{\mathbb{T}'}(\alpha_{(V,r)}) = \mathcal{U}_{\mathbb{T}'}\Phi(\varepsilon_{(V,r)}) \circ \mathcal{U}_{\mathbb{T}'}(\beta^{\alpha}_V)\\
&= \alpha_{(V,r)} \circ F\mathcal{U}_{\mathbb{T}}(\varepsilon_{(V,r)}) \circ \alpha^{-1}_{\mathcal{F}_{\mathbb{T}}(V)} \circ \mathcal{U}_{\mathbb{T}'}(\beta^{\alpha}_V) = \alpha_{(V,r)} \circ F(r) \circ \zeta_V.
\end{align*}
where the first equality is by \eqref{counitGivesTheAction} and \eqref{standardAdjunctionMonad}, the second is by item 2 in Lemma \ref{lemmaCompatibleAdjunctionData}, the third is by naturality of $\alpha$ and the last is by \eqref{counitGivesTheAction} and definition of $\zeta$. Finally if $(F,\Phi,\alpha)$ is strong then $\beta^{\alpha}$ is an isomorphism and we see from its definition that $\zeta$ is also an isomorphism.
\end{proof}

\begin{example}
Let $\mathcal{C},\mathcal{D}$ be monoidal categories, assumed to be strict for simplicity. Let $(A,m_A,1_A)$ be an associative algebra in $\mathcal{C}$, where the multiplication $m_A : A \otimes A \to A$  and the unit $1_A : \boldsymbol{1}_{\mathcal{C}} \to A$ satisfy the obvious axioms. It defines a monad $\mathbb{T}_A = (T_A,\mu^A,\eta^A)$ with $T_A = A \otimes -$, $\mu^A_X = m_A \otimes \mathrm{id}_X$ and $\eta^A_X = 1_A \otimes \mathrm{id}_X$. Similarly let $(B,m_B,1_B)$ be an algebra in $\mathcal{D}$ and $\mathbb{T}_B = (B \otimes -,\mu^B,\eta^B)$ be the associated monad. Let $F : \mathcal{C} \to \mathcal{D}$ be a monoidal functor. Then $F(A)$ is an algebra in $\mathcal{D}$ with product $m_{F(A)} = F(m_A) \circ F^{(2)}_{A,A}$ and unit $1_{F(A)} = F(1_A)$. For $f : B \to F(A)$ define
\[ \zeta^f_X : T_BF(X) = B \otimes F(X) \xrightarrow{f \otimes \mathrm{id}_{F(X)}} F(A) \otimes F(X) \xrightarrow{F^{(2)}_{A,X}} F(A \otimes X) = FT_A(X). \]
Denote by $\Hom_{\mathbf{Mnd}}^{(F,-)}(\mathbb{T}_A, \mathbb{T}_B)$ the set of monad morphisms (Def.~\ref{defMorphismMonads}) whose underlying functor is exactly the monoidal functor $F$. Then we have two maps
\[ \xymatrix{
\Hom_{\mathbf{Mnd}}^{(F,-)}(\mathbb{T}_A, \mathbb{T}_B) \ar@<-2pt>[r]_-{\Omega} & \ar@<-3pt>[l]_-{\Theta} \in \Hom_{\mathrm{alg}}(B,F(A))}, \qquad \Omega(\zeta) = \zeta_{\boldsymbol{1}}, \qquad \Theta(f) = \zeta^f \]
where $\Hom_{\mathrm{alg}}$ means morphisms of algebras in $\mathcal{D}$. Clearly $\Omega \circ \Theta = \mathrm{id}$. But the converse, \textit{i.e.} that $\Theta \circ \Omega = \mathrm{id}$ is not in general true: one checks easily that $\zeta$ is in the image of $\Theta \circ \Omega$ if and only it satisfies $\zeta_{X \otimes Y} \circ (\mathrm{id}_B \otimes F^{(2)}_{X,Y}) = F^{(2)}_{A \otimes X,Y} \circ (\zeta_X \otimes \mathrm{id}_{F(Y)})$ for all $X,Y \in \mathcal{C}$. If we look at $\mathcal{C}$ (resp. $\mathcal{D}$) as a right $\mathcal{C}$-module category by $X \lhd Y = X \otimes Y$ (resp. $V \lhd Y = V \otimes F(Y)$) then this extra condition on $\zeta$ is equivalent to the fact that $\zeta : T_BF \Rightarrow FT_A$ is a morphism of $\mathcal{C}$-module functors \cite[Def.\,7.2.2]{EGNO} from $(T_BF,s)$ to $(FT_A,t)$, where $s_{X,Y} = \mathrm{id}_B \otimes F^{(2)}_{X,Y}$ and $t_{X,Y} = F^{(2)}_{A \otimes X,Y}$ for all $X,Y \in \mathcal{C}$. If we call such monad morphisms $\mathcal{C}$-equivariant, then the conclusion is that we have a bijection between algebra morphisms $B \to F(A)$ in $\mathcal{D}$ and $\mathcal{C}$-equivariant monad morphisms $\mathbb{T}_A \to \mathbb{T}_B$ with underlying functor $F$.
\end{example}

\indent Now we recall the lifting theorem for right adjoints \cite[Th.\,4]{johnstone} (a short summary is also given in \cite[\S 3.5]{BLV}). We keep the notations introduced before Lemma \ref{liftedFunctor}. Assume that $F$ has a right adjoint $R : \mathcal{D} \to \mathcal{C}$. Denote by $e : \mathrm{Id}_{\mathcal{C}} \Rightarrow R F$ and $h : F R \Rightarrow \mathrm{Id}_{\mathcal{D}}$ the unit and counit of the adjunction $F \dashv R$. If $\zeta : T' F \overset{\sim}{\implies} F T$ {\em is an isomorphism} we can define $\xi : T R \Rightarrow R T'$ by
\begin{equation}\label{defIsoNatXi}
\xi_Y : T R(Y) \xrightarrow{e_{TR(Y)}} RFTR(Y) \xrightarrow {R(\zeta_{R(Y)}^{-1})} RT'FR(Y) \xrightarrow{RT'(h_{Y})} RT'(Y).
\end{equation}
A tedious computation using the naturality of $e$, $h$, $\zeta$ and the unit-counit equations for $e,h$ shows that
\[ \forall \, Y \in \mathcal{D}, \qquad \xi_{Y} \circ \mu_{R(Y)} = R(\mu'_{Y}) \circ \xi_{T'(Y)} \circ T(\xi_{Y}) \:\text{ and }\: \xi_{Y} \circ \eta_{R(Y)} = R(\eta'_{Y}). \]
Hence $(R,\xi)$ is a morphism of monads $\mathbb{T}' \to \mathbb{T}$ and by item 1 in Lemma \ref{liftedFunctor} there is a lifted functor $\widetilde{R} = \widetilde{R}_{\xi} : \mathbb{T}'\text{-}\mathrm{mod} \to \mathbb{T}\text{-}\mathrm{mod}$ defined by
\begin{equation}\label{rightAdjointLift}
\widetilde{R}(W,u) = \bigl( R(W), R(u) \circ \xi_W \bigr), \qquad \widetilde{R}(g) = R(g).
\end{equation}

\begin{lemma}{\em \cite[Th.\,4]{johnstone}}\label{liftRightAdjoint}
Let $(F,\zeta) : \mathbb{T} \to \mathbb{T}'$ be a strong morphism of adjunctions (Def.~\ref{defMorphismMonads}) and $\widetilde{F}_{\zeta} : \mathbb{T}\text{-}\mathrm{mod} \to \mathbb{T}'\text{-}\mathrm{mod}$ be the associated lifted functor \eqref{liftFunctorAlongMonads}. If $R$ is a right adjoint of $F$ then $\widetilde{R}$ defined in \eqref{rightAdjointLift} is a right adjoint of $\widetilde{F}_{\zeta}$.
\end{lemma}
\begin{proof}
For convenience here are some details, taken from \cite[\S 3.5]{BLV}. Let $(V,r) \in \mathbb{T}\text{-}\mathrm{mod}$. Define $\widetilde{e}_{(V,r)} = e_V \in \Hom_{\mathcal{C}}(V, RF(V))$. Using the definition of $\xi$ \eqref{defIsoNatXi}, the naturality of $e$, $h$, $\zeta$ and the unit-counit equations for $e,h$, we get
\[ RF(r) \circ R(\zeta_V) \circ \xi_{F(V)} \circ T(e_V) = e_V \circ r \]
which means that $\widetilde{e}_{(V,r)} \in \Hom_{\mathbb{T}\text{-}\mathrm{mod}}\bigl((V,r), \widetilde{R}\widetilde{F}_{\zeta}(V,r)\bigr)$ and thus $\widetilde{e}$ is a natural transformation $\mathrm{Id}_{\mathbb{T}\text{-}\mathrm{mod}} \Rightarrow \widetilde{R}\widetilde{F}_{\zeta}$. Similarly, for $(W,u) \in \mathbb{T}'\text{-}\mathrm{mod}$ let $\widetilde{h}_{(W,u)} = h_W \in \Hom_{\mathcal{D}}(FR(W), W)$. Then actually $\widetilde{h}_{(W,u)} \in \Hom_{\mathbb{T}'\text{-}\mathrm{mod}}\bigl(\widetilde{F}_{\zeta}\widetilde{R}(W,u), (W,u)\bigr)$ and thus $\widetilde{h}$ is a natural transformation $\widetilde{F}_{\zeta}\widetilde{R} \Rightarrow \mathrm{Id}_{\mathbb{T}'\text{-}\mathrm{mod}}$. The pair $(\widetilde{e}, \widetilde{h})$ satisfies the unit-counit equations because so does the pair $(e,h)$, which implies that $\widetilde{R}$ is right adjoint to $\widetilde{F}_{\zeta}$ \cite[\S IV.1, Th. 2]{MLCat}.
\end{proof}

Thanks to the right adjoint construction in \eqref{rightAdjointLift} we are in the following situation when $(F,\zeta)$ is a strong morphism of monads:
\begin{equation}\label{diagramLiftAdjunction}
 \xymatrix@C=4em@R=.7em{
\mathbb{T}\text{-}\mathrm{mod}\ar@/^.7em/[dd]^{\mathcal{U}_{\mathbb{T}}} \ar@/^.7em/[r]^{\widetilde{F}_{\zeta}}_{\text{\normalsize \rotatebox{270}{$\dashv$}}} & \ar@/^.7em/[l]^{\widetilde{R}}
\mathbb{T}'\text{-}\mathrm{mod} \ar@/^.7em/[dd]^{\mathcal{U}_{\mathbb{T}'}}\\
\dashv & \dashv\\
\ar@/^.7em/[uu]^{\mathcal{F}_{\mathbb{T}}}\mathcal{C} \ar@/^.7em/[r]^{F}_{\text{\normalsize \rotatebox{270}{$\dashv$}}} & \ar@/^.7em/[l]^R \ar@/^.7em/[uu]^{\mathcal{F}_{\mathbb{T}'}}\mathcal{D}
}
\end{equation}

\begin{proposition}\label{isoExtForMonadicAdjunctions}
Let $\mathcal{C}$, $\mathcal{D}$ be abelian categories, $\mathbb{T} = (T, \mu,\eta)$ be a monad on $\mathcal{C}$ and $\mathbb{T}' = (T', \mu',\eta')$ be a monad on $\mathcal{D}$ such that $T$, $T'$ are additive functors. Let $(F,\zeta) : \mathbb{T} \to \mathbb{T}'$ be a morphism of monads such that
\begin{itemize}[topsep=.2em, itemsep=-.2em]
\item $F : \mathcal{C} \to \mathcal{D}$ is an additive functor which has a right adjoint $R$,
\item $\zeta : T'F \Rightarrow FT$ is an isomorphism.
\end{itemize}
Then
\[ \forall \, n \geq 0, \quad \Ext^n_{\mathbb{T}'\text{-}\mathrm{mod}, \mathcal{D}}\bigl(\widetilde{F}_{\zeta}(\mathsf{V}), \mathsf{W}\bigr) \cong \Ext^n_{\mathbb{T}\text{-}\mathrm{mod}, \mathcal{C}}\bigl(\mathsf{V}, \widetilde{R}(\mathsf{W})\bigr) \]
for all $\mathsf{V} \in \mathbb{T}\text{-}\mathrm{mod}$ and $\mathsf{W} \in \mathbb{T}'\text{-}\mathrm{mod}$, where $\widetilde{R}$ is the right adjoint of $\widetilde{F}_{\zeta}$ defined in \eqref{rightAdjointLift}.
\end{proposition}
\begin{proof}
It is clear that $\widetilde{F}_{\zeta}$ is additive because it is equal to $F$ on morphisms. Hence $(F, \widetilde{F}_{\zeta},\mathrm{id})$ is a strong morphism of resolvent pairs by item 1 in Lemma \ref{liftedFunctor} and we apply Theorem \ref{propExtGroupsForMorphismsOfResolventPairs}.
\end{proof}

To finish we describe explicitly the isomorphism in Proposition \ref{isoExtForMonadicAdjunctions}. Let $G_{\mathbb{T}} = \mathcal{F}_{\mathbb{T}} \mathcal{U}_{\mathbb{T}}$ (resp. $G_{\mathbb{T}'} = \mathcal{F}_{\mathbb{T}'} \mathcal{U}_{\mathbb{T}'}$) be the comonad on $\mathbb{T}\text{-}\mathrm{mod}$ (resp. on $\mathbb{T}'\text{-}\mathrm{mod}$). The conditions \eqref{conditionsZetaLiftF} on $\zeta$ imply $\zeta_X \in \Hom_{\mathbb{T}'\text{-}\mathrm{mod}}\bigl( \mathcal{F}_{\mathbb{T}'}F(X), \widetilde{F}_{\zeta}\mathcal{F}_{\mathbb{T}}(X) \bigr)$ for all $X \in \mathcal{C}$. Hence for all $\mathsf{V} = (V,r) \in \mathbb{T}\text{-}\mathrm{mod}$ we have
\begin{equation*}
\zeta_V = \zeta_{\mathcal{U}_{\mathbb{T}}(\mathsf{V})} \in \Hom_{\mathbb{T}'\text{-}\mathrm{mod}}\bigl( \mathcal{F}_{\mathbb{T}'}F\mathcal{U}_{\mathbb{T}}(\mathsf{V}), \widetilde{F}_{\zeta}\mathcal{F}_{\mathbb{T}}\mathcal{U}_{\mathbb{T}}(\mathsf{V}) \bigr) = \Hom_{\mathbb{T}'\text{-}\mathrm{mod}}\bigl( G_{\mathbb{T}'}\widetilde{F}_{\zeta}(\mathsf{V}), \widetilde{F}_{\zeta}G_{\mathbb{T}}(\mathsf{V}) \bigr)
\end{equation*}
where we used that $F\mathcal{U}_{\mathbb{T}} = \mathcal{U}_{\mathbb{T}'} \widetilde{F}_{\zeta}$. Since $(F, \widetilde{F}_{\zeta}, \alpha)$ is a morphism of resolvent pairs with $\alpha = \mathrm{id}$ (Def.~\ref{defMorphismOfResolventPairs}), the natural transformation $\gamma : G_{\mathbb{T}'} \widetilde{F}_{\zeta}\Rightarrow \widetilde{F}_{\zeta} G_{\mathbb{T}}$ from \eqref{morphismGammaOfComonads} is given by
\[ \forall \, \mathsf{V} = (V,r) \in \mathbb{T}\text{-}\mathrm{mod}, \quad \gamma_{\mathsf{V}} = \beta^{\alpha}_{\mathcal{U}_{\mathbb{T}}(\mathsf{V})} \circ \mathcal{F}_{\mathbb{T}'}(\alpha_{\mathsf{V}}^{-1}) = \beta^{\mathrm{id}}_V \overset{\eqref{betaForMonads}}{=} \zeta_V \]
The morphism $\gamma^{(n)}_{\mathsf{V}}$ defined inductively in \eqref{morphismGammaOfComonadsIterated} is thus equal to $\zeta^{(n)}_V : G_{\mathbb{T}'}^n\widetilde{F}_{\zeta}(\mathsf{V}) \to \widetilde{F}_{\zeta}G_{\mathbb{T}}^n(\mathsf{V})$ defined inductively by
\begin{equation}\label{isoGammaForFLiftAndComonads}
\zeta^{(n+1)}_V : G_{\mathbb{T}'}^{n+1}\widetilde{F}_{\zeta}(\mathsf{V}) \xrightarrow{G_{\mathbb{T}'}(\zeta^{(n)}_V) = T'(\zeta^{(n)}_V)} G_{\mathbb{T}'}\widetilde{F}_{\zeta}G_{\mathbb{T}}^n(\mathsf{V}) \xrightarrow{\zeta_{\mathcal{U}_{\mathbb{T}}G_{\mathbb{T}}^n(\mathsf{V})} = \zeta_{T^n(V)}} \widetilde{F}_{\zeta}G_{\mathbb{T}}^{n+1}(\mathsf{V}) \end{equation}
and $\zeta^{(0)}_V = \mathrm{id}_{\widetilde{F}_{\zeta}(\mathsf{V})}$ for all $\mathsf{V} = (V,r) \in \mathbb{T}\text{-}\mathrm{mod}$. Hence by \eqref{diagramProofAdjThm}, for all $\mathsf{V} = (V,r) \in \mathbb{T}\text{-}\mathrm{mod}$ and $\mathsf{W} \in \mathbb{T}'\text{-}\mathrm{mod}$, we have an isomorphism
\begin{align}
\begin{split}\label{isoAdjunctionWithBarComplexes}
\Hom_{\mathbb{T}\text{-}\mathrm{mod}}\!\left( G_{\mathbb{T}}^{n+1}(\mathsf{V}), \widetilde{R}(\mathsf{W}) \right) &\overset{\sim}{\longrightarrow} \Hom_{\mathbb{T}'\text{-}\mathrm{mod}}\!\left( \widetilde{F}_{\zeta}G_{\mathbb{T}}^{n+1}(\mathsf{V}), \mathsf{W} \right)\\
&\xrightarrow{(\zeta^{(n+1)}_V)^*} \Hom_{\mathbb{T}'\text{-}\mathrm{mod}}\!\left( G_{\mathbb{T}'}^{n+1}\widetilde{F}_{\zeta}(\mathsf{V}), \mathsf{W} \right)
\end{split}
\end{align}
which descends to the isomorphism of Proposition \ref{isoExtForMonadicAdjunctions}.

\section{Application to deformation of monoidal structures}
In this paper we only consider infinitesimal deformations. This section begins with a brief review of the deformation theory of monoidal structures (\S\ref{subsectionDefMonStruct}). We adopt a geometric language: given a functor $F$ and a monoidal structure $\theta$ for $F$, we look at deformations of $\theta$ as tangent vectors at the point $\theta$, although there is in general no underlying algebraic variety (or manifold) structure on the set of all monoidal structures for $F$. This point of view will be convenient in \S\ref{subsectionDYcohomologyTangentBraidings}. Then we recall the isomorphism between DY cohomology and relative Ext groups (\S\ref{relExtGroupsDYCohomology}), and actually re-prove it under weaker assumptions than in \cite{FGS}. This allows us to apply the general results of \S\ref{sectionAdjThmRelExt} and to obtain an {\em adjunction theorem for DY cohomology} (Theorem~\ref{thmChangeOfCoeffDY} in~\S\ref{sectionChangeOfFunctor}).

\subsection{Tangent spaces and DY cohomology}\label{subsectionDefMonStruct}
Let $\mathcal{C}$, $\mathcal{D}$ be monoidal categories, which we take strict for simplicity. The tensor unit object will be denoted by $\boldsymbol{1}$. Let $F : \mathcal{C} \to \mathcal{D}$ be a functor which satisfies $F(\boldsymbol{1}) = \boldsymbol{1}$. A {\em monoidal structure} for $F$ is a natural isomorphism $\theta : F \otimes F \overset{\sim}{\implies} F(- \otimes -)$ such that
\begin{equation}\label{conditionsMonStruct}
\begin{array}{c}
\theta_{X \otimes Y,Z} \circ (\theta_{X,Y} \otimes \mathrm{id}_{F(Z)}) = \theta_{X,Y\otimes Z} \circ (\mathrm{id}_{F(X)} \otimes \theta_{Y,Z})\\[.3em]
\text{and } \quad \theta_{X,\boldsymbol{1}} = \theta_{\boldsymbol{1},X} = \mathrm{id}_{F(X)}.
\end{array} \qquad\quad (\forall\, X,Y,Z \in \mathcal{C})
\end{equation}
The pair $(F,\theta)$ is called a {\em monoidal functor}.\footnote{In full generality one requires the existence of an isomorphism $F^{(0)}: \boldsymbol{1} \overset{\sim}{\to} F(\boldsymbol{1})$ and the second line in \eqref{conditionsMonStruct} becomes $\theta_{X,\boldsymbol{1}} \circ (\mathrm{id}_{F(X)} \otimes F^{(0)}) = \theta_{\boldsymbol{1},X} \circ (F^{(0)} \otimes \mathrm{id}_{F(X)}) = \mathrm{id}_{F(X)}$. In what follows we always assume that $F^{(0)} = \mathrm{id}_{F(\boldsymbol{1})}$ for simplicity. This is sufficient for our main application in \S\ref{sectionLaxMult} where we take $F = \otimes_{\mathcal{C}}$ with a monoidal structure coming from a braiding in $\mathcal{C}$.} When the monoidal structure $\theta$ of $F$ is fixed it is customary to denote it by $F^{(2)}$. Using $\theta$ repeatedly one obtains for each $n>0$ a natural isomorphism
\begin{equation}\label{higherMonStruct}
\theta^{(n)}_{X_1, \ldots, X_n} : F(X_1) \otimes \ldots \otimes F(X_n) \overset{\sim}{\to} F(X_1 \otimes \ldots \otimes X_n)
\end{equation}
with the convention that $\theta^{(1)} = \mathrm{id}_{F}$ and $\theta^{(2)} = \theta$; when the notation $F^{(2)}$ is used we write $F^{(n)}$ instead of $\theta^{(n)}$. There are many different explicit formulas for $\theta^{(n)}$ in terms of $\theta$, but they all give the same isomorphism in $\mathcal{D}$ due to the coherence theorem for monoidal structures \cite{epstein}.

\smallskip

\indent Given two monoidal functors $(F,\theta), (F',\theta') : \mathcal{C} \to \mathcal{D}$, a {\em monoidal natural transformation} $\omega : (F,\theta) \Rightarrow (F',\theta')$ is a natural transformation $\omega : F \Rightarrow F'$ such that
\[ \theta'_{X,Y} \circ \bigl( \omega_X \otimes \omega_Y \bigr) = \omega_{X \otimes Y} \circ \theta_{X,Y} \]
for all $X,Y \in \mathcal{C}$. Denote by $\mathrm{Mon}(F)$ the set of all monoidal structures for a given functor $F$. We say that $\theta_1,\theta_2 \in \mathrm{Mon}(F)$ are equivalent, denoted by $\theta_1 \sim \theta_2$, if there exists a monoidal natural isomorphism $(F,\theta_1) \overset{\sim}{\implies} (F,\theta_2)$. Consider the group $\mathrm{Aut}(F)$ of {\em all} natural automorphisms $F \Rightarrow F$. For $u \in \mathrm{Aut}(F)$ and $\theta \in \mathrm{Mon}(F)$ let $u\cdot \theta : F \otimes F \Rightarrow F(- \otimes -)$ be the natural isomorphism with components
\begin{equation}\label{actionAutOnMon}
(u \cdot \theta)_{X,Y} = u_{X \otimes Y} \circ \theta_{X,Y} \circ (u^{-1}_X \otimes u^{-1}_Y)
\end{equation}
It is straightforward to check that $u \cdot \theta \in \mathrm{Mon}(F)$. Hence $\mathrm{Aut}(F)$ acts on $\mathrm{Mon}(F)$ and (by definition) the equivalence class $[\theta]$ for the relation $\sim$ is the orbit of $\theta$ under the action of $\mathrm{Aut}(F)$:
\begin{equation}\label{equivalenceClassesAsOrbits}
\mathrm{Mon}(F)/\!\!\sim \:= \mathrm{Mon}(F)/\mathrm{Aut}(F). 
\end{equation}

\smallskip

\indent In the situation when $\mathcal{D}$ (the target category of $F$) is linear, we introduce a vector space which is ``tangent'' to a given point $\theta \in \mathrm{Mon}(F)$. More precisely let $\Bbbk$ be a field and assume that $\mathcal{D}$ is a $\Bbbk$-linear category whose monoidal product is $\Bbbk$-bilinear on morphisms. Consider the category $\mathcal{D}_{\epsilon}$ defined by
\begin{equation}\label{categoryExtendedDualNumbers}
\mathrm{Ob}(\mathcal{D}_{\epsilon}) = \mathrm{Ob}(\mathcal{D}),\qquad \mathrm{Hom}_{\mathcal{D}_{\epsilon}}(X,Y) = \mathrm{Hom}_{\mathcal{D}}(X,Y) \otimes_{\Bbbk} \Bbbk[\epsilon]/(\epsilon^2).
\end{equation}
In short, the objects are unchanged and the Hom spaces have their scalars extended to the ring of dual numbers. Thus, a morphism in $\mathcal{D}_{\epsilon}$ is an expression $a + \epsilon b$ where $a,b$ are morphisms in $\mathcal{D}$. The composition and monoidal product are bilinearly extended in the obvious way:
\begin{align}
\begin{split}\label{compositionExtendedDualNumbers}
(a + \epsilon b) \circ (a' + \epsilon b') &= a \circ a' + \epsilon (a \circ b' + b \circ a'),\\
(a + \epsilon b) \otimes_{\epsilon} (a' + \epsilon b') &= a \otimes a' + \epsilon (a \otimes b' + b \otimes a')
\end{split}
\end{align}
and $(\mathcal{D}_{\epsilon},\otimes_{\epsilon}, \boldsymbol{1})$ becomes a $\Bbbk[\epsilon]/(\epsilon^2)$-linear monoidal category. For any natural transformation $f : F(-) \otimes F(-) \Rightarrow F(- \otimes -)$, not necessarily monoidal, we have the collection of morphisms $\theta + \epsilon f = \bigl( \theta_{X,Y} + \epsilon f_{X,Y} \bigr)_{X,Y \in \mathcal{C}}$ in $\mathcal{D}_{\epsilon}$. We define {\em the tangent space at $\theta$} as
\begin{equation}\label{tangentSpaceMonStruct}
\mathbf{T}_{\theta}\mathrm{Mon}(F) = \bigl\{ f : F(-) \otimes F(-) \Rightarrow F(- \otimes -) \, \big|\, \theta + \epsilon f \text{ satisfies \eqref{conditionsMonStruct}} \bigr\}.
\end{equation}
Explicitly a natural transformation $f$ is in $\mathbf{T}_{\theta}\mathrm{Mon}(F)$ if and only if
\begin{align}
\begin{split}\label{cocycleConditionDefMonStruct}
&\theta_{X_1,X_2\otimes X_3} \circ (\mathrm{id}_{F(X_1)} \otimes f_{X_2,X_3}) - f_{X_1 \otimes X_2,X_3} \circ (\theta_{X_1,X_2} \otimes \mathrm{id}_{F(X_3)})\\
+\:& f_{X_1,X_2\otimes X_3} \circ (\mathrm{id}_{F(X_1)} \otimes \theta_{X_2,X_3}) - \theta_{X_1 \otimes X_2,X_3} \circ (f_{X_1,X_2} \otimes \mathrm{id}_{F(X_3)})= 0
\end{split}
\end{align}
for all $X_1,X_2, X_3 \in \mathcal{C}$ and
\begin{equation}\label{unitConstraintInfinitesimal}
f_{X,\boldsymbol{1}} = f_{\boldsymbol{1},X} = 0
\end{equation}
for all $X \in \mathcal{C}$. Since the conditions \eqref{cocycleConditionDefMonStruct} and \eqref{unitConstraintInfinitesimal} 
are linear, $\mathbf{T}_{\theta}\mathrm{Mon}(F)$ is a $\Bbbk$-vector subspace in $\mathrm{Nat}\bigl( F \otimes F, F(- \otimes -) \bigr)$.

\smallskip

\indent We now want to give a definition for $\mathbf{T}_{[\theta]}\bigl( \mathrm{Mon}(F)/\!\!\sim \bigr)$. As a motivation, recall that when a Lie group $G$ acts on a manifold $M$ then under suitable assumptions the quotient set $M/G$ is a manifold (see e.g. \cite[Chap.\,21]{lee}) and one has
\[ \mathbf{T}_{G \cdot m}(M/G) = (\mathbf{T}_mM)/(\mathbf{T}_{m}(G \cdot m)) \]
where $\mathbf{T}$ denote tangent spaces at the prescribed points. Combining this with \eqref{equivalenceClassesAsOrbits} we would like to define $\mathbf{T}_{[\theta]}\bigl( \mathrm{Mon}(F)/\!\!\sim \bigr)$ as $\mathbf{T}_{\theta}\mathrm{Mon}(F)/\mathbf{T}_{\theta}\bigl( \mathrm{Aut}(F) \cdot \theta \bigr)$ but we must give a sense to the denominator. Note that for all $v \in \mathrm{Nat}(F,F)$, $u=\mathrm{id}_F + \epsilon v$ can be thought as the first term of the Taylor expansion of a curve passing through the point $\mathrm{id} \in \mathrm{Aut}(F)$ and we have
\[ \bigl( (\mathrm{id}_F + \epsilon v) \cdot \theta \bigr)_{X,Y} \overset{\eqref{actionAutOnMon}}{=} \theta_{X,Y} + \epsilon \bigl[ v_{X \otimes Y} \circ \theta_{X,Y} - \theta_{X,Y} \circ (\mathrm{id}_{F(X)} \otimes v_Y) - \theta_{X,Y} \circ (v_X \otimes \mathrm{id}_{F(Y)}) \bigr]. \]
The elements arising as the coefficients of $\epsilon$ in these ``Taylor expansions'' form a reasonable ansatz for $\mathbf{T}_{\theta}\bigl( \mathrm{Aut}(F) \cdot \theta \bigr)$. These heuristic remarks lead to the precise desired definition:
\begin{equation}\label{defTangentSpaceQuotient}
\mathbf{T}_{[\theta]}\bigl( \mathrm{Mon}(F)/\!\!\sim \bigr) = \bigl( \mathbf{T}_{\theta} \mathrm{Mon}(F) \bigr)/\!\equiv_{\theta}
\end{equation}
where we declare that $f \equiv_{\theta} g$ if there exists $v \in \mathrm{Nat}(F,F)$ such that
\begin{equation}\label{defEquivOnTangentSpace}
f_{X,Y} - g_{X,Y} = \theta_{X,Y} \circ (\mathrm{id}_{F(X)} \otimes v_Y) - v_{X \otimes Y} \circ \theta_{X,Y} + \theta_{X,Y} \circ (v_X \otimes \mathrm{id}_{F(Y)})
\end{equation}
for all $X,Y \in \mathcal{C}$. Let us show that, up to isomorphism, this definition does not depend on the representative $\theta$. For $u \in \mathrm{Aut}(F)$ and $f \in \mathbf{T}_{\theta} \mathrm{Mon}(F)$ define $u \cdot f$ as in \eqref{actionAutOnMon}. There is an isomorphism of vector spaces
\begin{equation}\label{canonicalIsoTangentSpaceQuotient}
\ell_u : \mathbf{T}_{\theta} \mathrm{Mon}(F) \to \mathbf{T}_{u \cdot \theta} \mathrm{Mon}(F), \qquad f \mapsto u \cdot f.
\end{equation}
For $v \in \mathrm{Nat}(F,F)$ set $(u \cdot v)_X = u_X \circ v_X \circ u^{-1}_X$, which defines $u \cdot v \in \mathrm{Nat}(F,F)$. If $f$ and $g$ are as in \eqref{defEquivOnTangentSpace} a simple computation reveals that
\begin{align*}
&(u\cdot f)_{X,Y} - (u \cdot g)_{X,Y}\\
=\:&(u \cdot \theta)_{X,Y} \circ \bigl(\mathrm{id}_{F(X)} \otimes (u\cdot v)_Y \bigr) - (u\cdot v)_{X \otimes Y} \circ (u \cdot \theta)_{X,Y} + (u \cdot \theta)_{X,Y} \circ \bigl( (u \cdot v)_X \otimes \mathrm{id}_{F(Y)} \bigr).
\end{align*}
Hence $f \equiv_{\theta} g$ implies $u \cdot f \equiv_{u \cdot \theta} u \cdot g$, so that $\ell_u$ descends into a linear map 
\[ \overline{\ell}_u : \mathbf{T}_{\theta} \mathrm{Mon}(F)/\!\equiv_{\theta} \:\: \to \mathbf{T}_{u \cdot \theta} \mathrm{Mon}(F)/\!\equiv_{u \cdot \theta} \]
whose inverse is $\overline{\ell}_{u^{-1}}$.

Here is a rephrasing of the above definitions for multilinear functors:\footnote{We consider multilinear instead of just linear functors because in \S\ref{subsectionDYcohomologyTangentBraidings} we will apply Lemma \ref{lemmaDeformationMultilinearFunctor} to $F = \otimes$.}
\begin{lemma}\label{lemmaDeformationMultilinearFunctor}
Let $\mathcal{C}^1, \ldots, \mathcal{C}^n, \mathcal{D}$ be $\Bbbk$-linear monoidal categories and $F : \mathcal{C}^1 \times \ldots \times \mathcal{C}^n \to \mathcal{D}$ be a functor which is $\Bbbk$-linear in each variable. Define a $\Bbbk[\epsilon]/(\epsilon^2)$-linear functor $F_{\epsilon} : \mathcal{C}^1_{\epsilon} \times \ldots \times \mathcal{C}^n_{\epsilon} \to \mathcal{D}_{\epsilon}$ by linear extension of $F$ in each variable. Then for any $\theta \in \mathrm{Mon}(F)$:
\\1. We have $f \in \mathbf{T}_{\theta}\mathrm{Mon}(F)$ if and only if $\theta + \epsilon f \in \mathrm{Mon}(F_{\epsilon})$.
\\2. For all $f,g \in \mathbf{T}_{\theta}\mathrm{Mon}(F)$, we have $f \equiv_{\theta} g$ if and only if there exists $v \in \mathrm{Nat}(F,F)$ such that $\mathrm{id}_F + \epsilon v : (F_{\epsilon}, \theta + \epsilon f) \Rightarrow (F_{\epsilon}, \theta + \epsilon g)$ is a monoidal natural isomorphism.
\end{lemma}
\begin{proof}
1. Indeed by definition $\theta + \epsilon f \in \mathrm{Mon}(F_{\epsilon})$ if and only if $\theta + \epsilon f$ satisfies \eqref{conditionsMonStruct}, which means that $f \in \mathbb{T}_{\theta}\mathrm{Mon}(F)$ by \eqref{tangentSpaceMonStruct}. The only point of this lemma is that when $F$ and its source category are linear (or more generally $F$ is multilinear), then one can define the extension $F_{\epsilon}$ and talk about monoidal structures for $F_{\epsilon}$.
\\2. Straightforward computation.
\end{proof}

\medskip

\indent We now recall the cohomological interpretation of the above discussion, following \cite{davydov,CY,yetter1}. Let $(F,\theta) : \mathcal{C} \to \mathcal{D}$ be a monoidal functor and assume that $\mathcal{D}$ is $\Bbbk$-linear. For each $n \geq 0$ define a $\Bbbk$-vector space
\[ C^n_{\mathrm{DY}}(F,\theta) = \left\{ \begin{array}{c}
\text{natural transformations of the form}\\ \bigl(f_{X_1, \ldots, X_n} : F(X_1) \otimes \ldots \otimes F(X_n) \to F(X_1 \otimes \ldots \otimes X_n)\bigr)_{X_1,\ldots,X_n \in \mathcal{C}} \end{array} \right\}. \]
For all $n \geq 0$ and $0 \leq i \leq n+1$ let
\begin{equation}\label{cofaceDYNonStrict}
\partial_i^n : C^n_{\mathrm{DY}}(F,\theta) \to C^{n+1}_{\mathrm{DY}}(F,\theta)
\end{equation}
such that $\partial_i^n(f)_{X_1, \ldots, X_{n+1}}$ is equal to 
\[ \begin{cases}
\theta_{X_1, X_2 \otimes \ldots \otimes X_{n+1}} \circ \bigl( \mathrm{id}_{F(X_1)} \otimes f_{X_2, \ldots, X_{n+1}} \bigr) & \text{if } i=0 \\
f_{X_1, \ldots,X_i \otimes X_{i+1}, \ldots, X_{n+1}} \circ \bigl( \mathrm{id}_{F(X_1) \otimes \ldots \otimes F(X_{i-1})} \otimes \theta_{X_i, X_{i+1}} \otimes \mathrm{id}_{F(X_{i+2}) \otimes \ldots \otimes F(X_{n+1})} \bigr) & \text{if } 1 \leq i \leq n\\
\theta_{X_1 \otimes \ldots \otimes X_n, X_{n+1}} \circ \bigl( f_{X_1, \ldots, X_n} \otimes \mathrm{id}_{F(X_{n+1})} \bigr) & \text{if } i=n+1
\end{cases} \]
Also for $0 \leq i \leq n-1$ let $s^n_i : C^n_{\mathrm{DY}}(F,\theta) \to C^{n-1}_{\mathrm{DY}}(F,\theta)$ be given by
\begin{equation}\label{cofaceDY}
s^n_i(f)_{X_1, \ldots, X_{n-1}} = f_{X_1, \ldots, X_i, \boldsymbol{1}, X_{i+1}, \ldots, X_n}.
\end{equation}
These linear maps satisfy the cosimplicial identities recalled in Appendix \ref{appendixNormalization}. In particular we have the differential $\delta^n = \sum_{i=0}^{n+1}(-1)^i\partial^n_i$ and $\bigl(C^{\bullet}_{\mathrm{DY}}(F,\theta), \delta \bigr)$ is called the {\em Davydov--Yetter (DY) complex} of $(F,\theta)$. We write the subspaces of cocycles and coboundaries in degree $n$ as $Z^n_{\mathrm{DY}}(F,\theta)$ and $B^n_{\mathrm{DY}}(F,\theta)$, respectively. The DY cohomology for $F$ in degree $n$ is $H^n_{\mathrm{DY}}(F,\theta) = Z^n_{\mathrm{DY}}(F,\theta)/B^n_{\mathrm{DY}}(F,\theta)$.

\smallskip

\indent An element $f \in C^2_{\mathrm{DY}}(F,\theta)$ is in $Z^2_{\mathrm{DY}}(F,\theta)$ if and only if it satisfies \eqref{cocycleConditionDefMonStruct}. The condition \eqref{unitConstraintInfinitesimal} in the definition of $\mathbf{T}_{\theta}\mathrm{Mon}(F)$ corresponds precisely to the normalization defined in Appendix \ref{appendixNormalization}. Indeed, due to \eqref{cofaceDY}, the normalized complex for DY is
\begin{equation}\label{normalizedSubcompDY}
N\!C^n_{\mathrm{DY}}(F,\theta) = \bigl\{ f \in C^n_{\mathrm{DY}}(F,\theta) \, \big| \, f_{X_1, \ldots, X_n} = 0 \text{ if } X_i = \boldsymbol{1} \text{ for some } i \bigr\}.
\end{equation}
As a result
\[ \mathbf{T}_{\theta}\mathrm{Mon}(F) = N\!Z^2_{\mathrm{DY}}(F,\theta) \]
where $N\!Z^{\bullet}_{\mathrm{DY}}(F,\theta) = Z^{\bullet}_{\mathrm{DY}}(F,\theta) \cap N\!C^{\bullet}_{\mathrm{DY}}(F,\theta)$ is the subspace of cocycles in the normalized DY complex. Moreover \eqref{defEquivOnTangentSpace} is equivalent to $f - g = \delta^1(v)$, which means that $f \equiv_{\theta} g$ if and only if $f$ and $g$ are cohomologous. One can assume that $v \in N\!C^1_{\mathrm{DY}}(F,\theta)$ by item 3 in Proposition \ref{propPropertiesN}. Thus if we denote by  $N\!H^{\bullet}_{\mathrm{DY}}(F,\theta)$ the cohomology of the normalized DY complex \eqref{normalizedSubcompDY}, we have
\begin{equation}\label{tangentSpaceAndCohom}
\mathbf{T}_{[\theta]}\bigl( \mathrm{Mon}(F)/\!\!\sim \bigr) = N\!H^2_{\mathrm{DY}}(F,\theta) \cong H^2_{\mathrm{DY}}(F,\theta)
\end{equation}
where the second isomorphism is due to Corollary \ref{coroProjQuasiIso}, which asserts that the inclusion $N\!C^{\bullet}_{\mathrm{DY}}(F,\theta) \to C^{\bullet}_{\mathrm{DY}}(F,\theta)$ is an isomorphism in cohomology.

\smallskip

\noindent \textbf{Notation.} When a monoidal structure on $F$ is fixed we denote it by $F^{(2)}$. Then we write $H^{\bullet}_{\mathrm{DY}}(F)$ instead of $H^{\bullet}_{\mathrm{DY}}(F,F^{(2)})$.

\smallskip

We finish with some explicit formulas for the normalization projector $\mathcal{N} : C^{\bullet}_{\mathrm{DY}}(F) \to N\!C^{\bullet}_{\mathrm{DY}}(F)$ defined in general in Appendix \ref{appendixNormalization}, which is a quasi-isomorphism. In degree $2$ it is given by
\[ \mathcal{N}^2(f)_{X,Y} = f_{X,Y} - \mathrm{id}_{F(X)} \otimes f_{\boldsymbol{1},Y} - f_{X \otimes Y,\boldsymbol{1}} + \mathrm{id}_{F(X)} \otimes f_{Y,\boldsymbol{1}} \]
thanks to \eqref{normalizationProjLowDegrees}. Moreover if $f$ is a cocycle then $\mathcal{N}^2(f) = f - \delta^1(f_{\boldsymbol{1},-})$ by \eqref{normalizationLowDegrees}. For $f \in C^3_{\mathrm{DY}}(\mathcal{C})$ we have
\begin{align*}
\mathcal{N}^3(f)_{X,Y,Z} = f_{X,Y,Z} &- \mathrm{id}_X \otimes f_{\boldsymbol{1},Y,Z} - f_{X \otimes Y,\boldsymbol{1},Z} - f_{X, Y \otimes Z,\boldsymbol{1}} + \mathrm{id}_X \otimes f_{Y,\boldsymbol{1},Z}\\
&+ f_{X \otimes Y,Z,\boldsymbol{1}} + \mathrm{id}_X \otimes f_{\boldsymbol{1},Y \otimes Z,\boldsymbol{1}} - \mathrm{id}_X \otimes f_{Y,Z,\boldsymbol{1}}
\end{align*}
thanks to \eqref{normalizationProjLowDegrees}. Moreover if $f$ is a cocycle then $\mathcal{N}^3(f) = f + \delta^2(-f_{\boldsymbol{1},-,-} + f_{-,\boldsymbol{1},-} - \mathrm{id} \otimes f_{\boldsymbol{1},\boldsymbol{1},-})$ by \eqref{normalizationLowDegrees}.

\begin{remark}
All this subsection remains valid if we consider {\em lax} monoidal structures (\textit{i.e.} not assumed to be isomorphisms). However all our results in the sequel need monoidal structures which are isomorphisms, and ``monoidal structure'' is always understood in this sense.
\end{remark}

\subsection{Relative Ext groups and DY cohomology with coefficients}\label{relExtGroupsDYCohomology}
\indent Let $\mathcal{C},\mathcal{D}$ be monoidal categories and $F : \mathcal{C} \to \mathcal{D}$ a monoidal functor with a given monoidal structure $F^{(2)}_{X_1,X_2} : F(X_1) \otimes F(X_2) \to F(X_1 \otimes X_2)$. We denote by $\mathcal{Z}(F)$ the {\em centralizer of the functor $F$} \cite[\S 3.1]{shimizuRibbon}, \cite[Def.\,3.1]{GHS}; its objects are pairs $\mathsf{V} = (V,\rho)$ where $\rho : V \otimes F(-) \overset{\sim}{\Rightarrow} F(-) \otimes V$ satisfies
\[ \forall \, X,Y \in \mathcal{C},\quad \rho_{X \otimes Y} \circ (\mathrm{id}_V \otimes F^{(2)}_{X,Y}) = (F^{(2)}_{X,Y} \otimes \mathrm{id}_V) \circ (\mathrm{id}_{F(X)} \otimes \rho_Y) \circ (\rho_X \otimes \mathrm{id}_{F(Y)}). \]
It is a monoidal category with $(V,\rho^V) \otimes (W,\rho^W) = \bigl(V \otimes W, (\rho^V \otimes \mathrm{id}_W) \circ (\mathrm{id}_V \otimes \rho^W) \bigr)$. For $F = \mathrm{Id}_{\mathcal{C}}$ we recover the {\em Drinfeld center} $\mathcal{Z}(\mathcal{C})$.

\smallskip

\indent Assume that $\mathcal{D}$ is $\Bbbk$-linear, for a field $\Bbbk$, with $\Bbbk$-bilinear monoidal product. There is a generalization of DY cohomology {\em with coefficients}, which are objects from $\mathcal{Z}(F)$ \cite[Def.\,3.3]{GHS}. For coefficients $\mathsf{V}, \mathsf{W} \in \mathcal{Z}(F)$, the cochain space of degree $n$, denoted by $C^n_{\mathrm{DY}}(F;\mathsf{V}, \mathsf{W})$, consists of natural transformations of the form
\begin{equation}\label{componentsDYCochains}
f_{X_1,\ldots,X_n} : V \otimes F(X_1) \otimes \ldots \otimes F(X_n) \to F(X_1 \otimes \ldots \otimes X_n) \otimes W.
\end{equation}
Then $C^{\bullet}_{\mathrm{DY}}(F;\mathsf{V}, \mathsf{W})$ is a simplicial vector space. The coface maps with coefficients are given in \cite[Def.\,3.3]{GHS} when $F$ is strict. For non-strict $F$ the coface maps without coefficients (or with $\mathsf{V}=\mathsf{W}=\boldsymbol{1}$) are given below \eqref{cofaceDYNonStrict}. It is straightforward to combine these two definitions for the general case; we will however not need the explicit formulas in the sequel and thus omit them. We denote by $H^{\bullet}_{\mathrm{DY}}(F;\mathsf{V}, \mathsf{W})$ the associated cohomology. Note that {\em trivial coefficients} recover the cohomology from \S\ref{subsectionDefMonStruct}: $H^{\bullet}_{\mathrm{DY}}(F;\boldsymbol{1}, \boldsymbol{1}) = H^{\bullet}_{\mathrm{DY}}(F)$.

\smallskip

\indent Recall that a $\Bbbk$-linear category is called {\em finite} if it is equivalent to the category of finite-dimensional modules over a finite-dimensional $\Bbbk$-algebra \cite[\S 1.8]{EGNO}; such a category is in particular abelian. A monoidal category $\mathcal{C}$ is called {\em rigid} if every object has left and right dual objects $X^{\vee}$ and $^{\vee}\!X$, meaning that there exist morphisms
\[ \mathrm{ev}_X : X^{\vee} \otimes X \to \boldsymbol{1}, \quad \mathrm{coev}_X : \boldsymbol{1} \to X \otimes X^{\vee}, \quad \widetilde{\mathrm{ev}}_X : X \otimes {^{\vee}\!X} \to \boldsymbol{1}, \quad \widetilde{\mathrm{coev}}_X : \boldsymbol{1} \to {^{\vee}\!X} \otimes X \]
satisfying the usual ``zig-zag'' axioms \cite[\S 2.10]{EGNO}. If $\mathcal{C}$ is rigid and $F : \mathcal{C} \to \mathcal{D}$ is a monoidal functor, define for all $X \in \mathcal{C}$
\begin{align}
\begin{split}\label{dualityForImagesOfF}
&\mathrm{ev}_{F(X)} : F(X^{\vee}) \otimes F(X) \xrightarrow{F^{(2)}_{X,X^{\vee}}} F(X^{\vee} \otimes X) \xrightarrow{F(\mathrm{ev}_X)} \boldsymbol{1},\\
&\mathrm{coev}_{F(X)} : \boldsymbol{1} \xrightarrow{F(\mathrm{coev}_X)} F(X \otimes X^{\vee}) \xrightarrow{(F^{(2)}_{X,X^{\vee}})^{-1}} F(X) \otimes F(X^{\vee}).
\end{split}
\end{align}
This turns $F(X^{\vee})$ into the left dual of $F(X)$. One can similarly define $\widetilde{\mathrm{ev}}_{F(X)}$ and $\widetilde{\mathrm{coev}}_{F(X)}$ which turns $F({^{\vee}\!X})$ into the right dual of $F(X)$. Note that there are compatibilities like
\[ \mathrm{ev}_{F(Y)} \circ (\mathrm{id}_{F(Y^{\vee})} \otimes \mathrm{ev}_{F(X)} \otimes \mathrm{id}_{F(Y)}) = \mathrm{ev}_{F(X \otimes Y)} \circ (F^{(2)}_{Y^{\vee}, X^{\vee}} \otimes F^{(2)}_{X,Y}). \]

\indent From now on we make the following assumptions on $\mathcal{C}$, $\mathcal{D}$ and $F$:
\begin{equation}\label{assumptionsFCD}
\left\{ \begin{array}{l}
\mathcal{C}, \mathcal{D} \text{ are } \Bbbk\text{-linear abelian categories,}\\
\mathcal{C}, \mathcal{D} \text{ have } \Bbbk\text{-bilinear monoidal products (strict for simplicity),}\\
\mathcal{C}\text{ is finite and rigid},\\
\otimes_{\mathcal{D}} : \mathcal{D} \times \mathcal{D} \to \mathcal{D}\text{ is right exact in each variable,}\\
F : \mathcal{C} \to \mathcal{D} \text{ is a } \Bbbk\text{-linear monoidal right-exact functor.}
\end{array} \right.
\end{equation}
We will show that these are sufficient conditions to construct a left adjoint $\mathcal{F}_F$ of the forgetful functor $\mathcal{U}_F : \mathcal{Z}(F) \to \mathcal{D}$, thus yielding a resolvent pair $\mathcal{F}_F \dashv \mathcal{U}_F$ between $\mathcal{Z}(F)$ and $\mathcal{D}$. In the special case where $F$ is a strict exact functor between finite tensor categories $\mathcal{C},\mathcal{D}$ we related in \cite[\S 4.2]{FGS} the DY cohomology of $F$ with the corresponding relative Ext groups $\Ext_{\mathcal{Z}(F),\mathcal{D}}^{\bullet}$, building upon the main result of \cite{GHS} on DY cohomology and comonad cohomology. This still holds under the more general assumptions \eqref{assumptionsFCD}:

\begin{theorem}\label{thmDYCohomRelExt}
Under the assumptions \eqref{assumptionsFCD} we have for all $\mathsf{V}, \mathsf{W} \in \mathcal{Z}(F)$ an isomorphism of cochain complexes
\begin{equation}\label{equivalenceDYCochainAndBarRes}
C^{\bullet}_{\mathrm{DY}}(F;\mathsf{V}, \mathsf{W}) \cong \Hom_{\mathcal{Z}(F)}\left( \mathrm{Bar}^{\bullet}_{\mathcal{Z}(F),\mathcal{D}}(\mathsf{V}), \mathsf{W} \right)
\end{equation}
where $\mathrm{Bar}^{\bullet}$ is defined in \eqref{generalBarResolution}--\eqref{generalBarDiff}. Therefore there is an isomorphism of $\Bbbk$-vector spaces
\begin{equation}\label{isoDYRelExt}
H^{\bullet}_{\mathrm{DY}}(F; \mathsf{V}, \mathsf{W}) \cong \Ext_{\mathcal{Z}(F), \mathcal{D}}^{\bullet}(\mathsf{V}, \mathsf{W}).
\end{equation}
\end{theorem}

\noindent The proof of this theorem in \cite[\S 4.2]{FGS} uses the explicit construction of the left adjoint $\mathcal{F}_F$, based on the description of $\mathcal{Z}(F)$ as a category of modules over some monad $Z_F$ (see \S\ref{sectionLiftingAdjunctions} for a short reminder on monads). This last fact is well-known: \cite[Th.\,5.12]{BV} can be used for $F = \mathrm{Id}_{\mathcal{C}}$, \cite[\S 3.3]{shimizuRibbon} considers a related situation and \cite[\S 3.3]{GHS} gives a sketch of proof for general strict $F$. However here we work with weaker assumptions on $\mathcal{C}$, $\mathcal{D}$ and $F$. In particular, $F$ is not assumed to be strict as a monoidal functor, because of the application in \S\ref{sectionTangentSpaceBraiding} for the tensor product functor with the monoidal structure given by a braiding on $\mathcal{C}$. Hence some (rather straightforward) adaptations are required in order to prove Theorem \ref{thmDYCohomRelExt} and we discuss some details for convenience.

\smallskip

\indent Note first that the only reason for the stronger assumptions regarding rigidity, finiteness and exactness in \cite{GHS,FGS} was to ensure for all $W \in \mathcal{D}$ the existence of the following coend by invoking \cite[Cor.\,5.1.8]{KL}: 
\begin{equation}\label{defMonadZF}
Z_F(W) = \int^{X \in \mathcal{C}} F(X^{\vee}) \otimes W \otimes F(X).
\end{equation}
But we can use the more general Corollary \ref{coendsAsCokerAndConsequence}, which proves the existence of $Z_F(W)$ under the assumptions \eqref{assumptionsFCD}. Denote by
\[ i^F_X(W) : F(X^{\vee}) \otimes W \otimes F(X) \to Z_F(W) \]
its universal dinatural transformation. Being a coend with parameters, the assignment $W \mapsto Z_F(W)$ is a functor $Z_F : \mathcal{D} \to \mathcal{D}$ \cite[\S IX.7]{MLCat}. Explicitly, if $f \in \Hom_{\mathcal{D}}(W,W')$ then $Z_F(f)$ is defined by the commutative diagram
\begin{equation}\label{ZFonMorphisms}
\xymatrix@C=9em{
F(X^{\vee}) \otimes W \otimes F(X) \ar[r]^{\mathrm{id}_{F(X^{\vee})} \otimes f \otimes \mathrm{id}_{F(X)}} \ar[d]_{i^F_X(W)} & F(X^{\vee}) \otimes W' \otimes F(X)\ar[d]^{i^F_X(W')}\\
Z_F(W) \ar[r]_{\exists!\, Z_F(f)} &  Z_F(W')
}
\end{equation}
for all $X \in \mathcal{C}$. The functor $Z_F$ has the structure of a monad whose multiplication $\mu^F : Z^2_F \Rightarrow Z_F$ is described as follows. By the theorem for iterated coends \cite[\S IX.8]{MLCat} we have $Z^2_F(W) = \int^{X,Y \in \mathcal{C}} F(Y^{\vee}) \otimes F(X^{\vee}) \otimes W \otimes F(X) \otimes F(Y)$ and let
\begin{equation}\label{defUniversalDinatDoubleCoend}
i^{F,\,(2)}_{X,Y}(W) = i^F_Y\bigl( Z_F(W) \bigr) \circ \bigl(\mathrm{id}_{F(Y^{\vee})} \otimes i^F_X(W) \otimes \mathrm{id}_{F(Y)}\bigr)
\end{equation}
be the associated universal dinatural transformation. Then by universality there is a commutative diagram
\begin{equation}\label{defMuMonadZF}
\xymatrix@C=8em{
F(Y^{\vee}) \otimes F(X^{\vee}) \otimes W \otimes F(X) \otimes F(Y) \ar[r]^-{F^{(2)}_{Y^{\vee},X^{\vee}} \,\otimes\, \mathrm{id}_W \,\otimes\, F^{(2)}_{X,Y}} \ar[d]_{i^{F,(2)}_{X,Y}(W)}& F\bigl((X \otimes Y)^{\vee}\bigr) \otimes W \otimes F(X \otimes Y) \ar[d]^{i^F_{X \otimes Y}(W)}\\
Z^2_F(W) \ar[r]_{\exists!\, \mu^F_W} & Z_F(W)
} \end{equation}
for all $X, Y \in \mathcal{C}$ and $W \in \mathcal{D}$. This defines a natural transformation $\mu^F : Z_F^2 \Rightarrow Z_F$. Also define $\eta^F : \mathrm{Id}_{\mathcal{D}} \Rightarrow Z_F$ by 
$\eta^F_W = i_{\boldsymbol{1}}^F(W)$ for all $W \in \mathcal{D}$. It follows from the axioms of $F^{(2)}$  in \eqref{conditionsMonStruct} that $\mu^F$ is associative and has the unit $\eta^F$. As a result we have the monad $\bigl( Z_F, \mu^F, \eta^F \bigr)$, its category of modules $Z_F\text{-}\mathrm{mod}$ and the adjunction (see \eqref{standardAdjunctionMonad})
\begin{equation}\label{adjunctionZfModD}
\xymatrix@R=.7em{
\bigl( Z_F(W), \mu^F_W \bigr) &Z_F\text{-}\mathrm{mod}\ar@/^.7em/[dd]^{\mathcal{U}_{Z_F}} & (V,r)\ar@{|->}[dd]\\
&\dashv&\\
W \ar@{|->}[uu] &\ar@/^.7em/[uu]^{\mathcal{F}_{Z_F}}\mathcal{D}&V
} \end{equation}
For $F = \mathrm{Id}_{\mathcal{C}}$ we get the standard central monad and write $Z_{\mathcal{C}}$, $i^{\mathcal{C}}_X(V)$, $\mu^{\mathcal{C}}$ and $\eta^{\mathcal{C}}$.

\smallskip

\indent Furthermore, there is an isomorphism of categories $Z_F\text{-}\mathrm{mod} \cong \mathcal{Z}(F)$ constructed as follows. If $(W,r)$ is a $Z_F$-module, we can define for all $X \in \mathcal{C}$
\begin{align}
\begin{split}\label{halfBraidingFromModuleStructure}
\rho(r)_X : W \otimes F(X) &\xrightarrow{\mathrm{coev}_{F(X)} \,\otimes\, \mathrm{id}}  F(X) \otimes F(X^{\vee}) \otimes W \otimes F(X)\\
&\xrightarrow{\mathrm{id} \,\otimes\, i^F_X(W)} F(X) \otimes Z_F(W) \xrightarrow{\mathrm{id} \,\otimes\, r} F(X) \otimes W
\end{split}
\end{align}
with $\mathrm{coev}_{F(X)}$ from \eqref{dualityForImagesOfF}. Long but straightforward computations prove that $\rho(r) : W \otimes F(-) \Rightarrow F(-) \otimes W$ is a half-braiding relative to $F$ and thus $\bigl(W,\rho(r)\bigr) \in \mathcal{Z}(F)$. Conversely if $(W,\rho) \in \mathcal{Z}(F)$ define $r(\rho) : Z_F(W) \to W$ by the commutative diagram
\begin{equation}\label{moduleStructureViaHalfBraiding}
\xymatrix@C=4.5em{
F(X^{\vee}) \otimes W \otimes F(X) \ar[r]^-{\mathrm{id}_{F(X^{\vee})} \otimes \rho_X} \ar[d]_{i_X^F(W)} & F(X^{\vee}) \otimes F(X) \otimes W \ar[d]^{\mathrm{ev}_{F(X)} \otimes \mathrm{id}_W}\\
Z_F(W) \ar[r]_{\exists! \, r(\rho)} & W
} \end{equation}
for all $X \in \mathcal{C}$ and with $\mathrm{ev}_{F(X)}$ from \eqref{dualityForImagesOfF}; then $\bigl( W, r(\rho) \bigr) \in Z_F\text{-}\mathrm{mod}$. Through this isomorphism, the adjunction \eqref{adjunctionZfModD} becomes
\begin{equation}\label{adjunctionZF}
\xymatrix@R=.7em{
\bigl( Z_F(W), \rho^{Z_F(W)} \bigr) &\mathcal{Z}(F)\ar@/^.7em/[dd]^{\mathcal{U}_F} & (V,\rho^V)\ar@{|->}[dd]\\
&\dashv&\\
W \ar@{|->}[uu] &\ar@/^.7em/[uu]^{\mathcal{F}_F}\mathcal{D}&V
}
\end{equation}
with the half braiding $\rho^{Z_F(W)} = \rho(\mu^F_W) : Z_F(W) \otimes F(-) \overset{\sim}{\implies} F(-) \otimes Z_F(W)$ defined by \eqref{halfBraidingFromModuleStructure}. We note that the construction in \eqref{moduleStructureViaHalfBraiding} can be interpreted {\it a posteriori} as the comparison functor $K : \mathcal{Z}(F) \to Z_F\text{-}\mathrm{mod}$ (defined in general in \eqref{comparisonFunctor}) associated to the adjunction \eqref{adjunctionZF}.

\smallskip

\indent Since the functor $Z_F$ is additive the adjunction \eqref{adjunctionZfModD} is a resolvent pair thanks to item 2 in Proposition \ref{propResolventPairsAreMonadic}. Due to the isomorphism  $Z_F\text{-}\mathrm{mod}\cong \mathcal{Z}(F)$ explained above, the same is true for the adjunction \eqref{adjunctionZF} . Hence, under the assumptions \eqref{assumptionsFCD}, we have the relative Ext groups $\mathrm{Ext}^{\bullet}_{\mathcal{Z}(F),\mathcal{D}}$. Of course 
\begin{equation}\label{isoRelExtCentralizerOrModulesCentralMonad}
\forall \, \mathsf{V}, \mathsf{W} \in \mathcal{Z}(F), \quad \mathrm{Ext}^{\bullet}_{\mathcal{Z}(F),\mathcal{D}}(\mathsf{V},\mathsf{W}) \cong \mathrm{Ext}^{\bullet}_{Z_F\text{-}\mathrm{mod},\mathcal{D}}\bigl( K(\mathsf{V}), K(\mathsf{W}) \bigr)
\end{equation}
where $K$ is the isomorphism $\mathcal{Z}(F) \overset{\sim}{\to} Z_F\text{-}\mathrm{mod}$.
\begin{proof}[Proof of Theorem \ref{thmDYCohomRelExt}]
Let $G_F= \mathcal{F}_F \mathcal{U}_F$ be the comonad on $\mathcal{Z}(F)$ and recall from \eqref{generalBarResolution} that $\mathrm{Bar}^n_{\mathcal{Z}(F),\mathcal{D}}(\mathsf{V}) = G_F^{n+1}(\mathsf{V})$. We construct an isomorphism 
\begin{equation*}
\Gamma^n : C^n_{\mathrm{DY}}(F;\mathsf{V}, \mathsf{W}) \cong \Hom_{\mathcal{Z}(F)}\left( G_F^{n+1}(\mathsf{V}), \mathsf{W} \right)
\end{equation*}
by inserting the monoidal structure of $F$ in \cite[eq.\,(48)]{FGS}, where $F$ was assumed to be strict. Write $\mathsf{V} = (V,\rho^V)$ and $\mathsf{W} = (W,\rho^W)$. Using the Fubini theorem for coends \cite[\S IX.8]{MLCat}, note that $G_F^n(\mathsf{V})$ is
\begin{equation}\label{ZFIterated}
Z_F^n(V) = \int^{X_1, \ldots,X_n \in \mathcal{C}} F(X_n^{\vee}) \otimes \ldots \otimes F(X_1^{\vee}) \otimes V \otimes F(X_1) \otimes \ldots \otimes F(X_n)
\end{equation}
as an object in $\mathcal{D}$. Let $i_{X_1,\ldots,X_n}^{F,(n)}(V) : F(X_n^{\vee}) \otimes \ldots \otimes F(X_1^{\vee}) \otimes V \otimes F(X_1) \otimes \ldots \otimes F(X_n) \to Z_F^n(V)$ be the universal dinatural transformation defined inductively by $i^{F,(1)}_X(V) = i^F_X(V)$ and
\begin{align}
\begin{split}\label{dinatTransfoZFIterated}
i^{F,(n+1)}_{X_1,\ldots,X_{n+1}}(V) &= i_{X_{n+1}}^F\bigl( Z_F^n(V) \bigr) \circ \bigl( \mathrm{id}_{F(X_{n+1}^{\vee})} \otimes i^{F,(n)}_{X_1,\ldots,X_n}(V) \otimes \mathrm{id}_{F(X_{n+1})} \bigr)\\
&= i^{F,(n)}_{X_2,\ldots,X_{n+1}}\bigl( Z_F(V) \bigr) \circ \bigl( \mathrm{id}_{F(X_{n+1}^{\vee}) \otimes \ldots \otimes F(X_2^{\vee})} \otimes i^F_{X_1}(V) \otimes \mathrm{id}_{F(X_2) \otimes \ldots \otimes F(X_{n+1})} \bigr).
\end{split}
\end{align}
Take $f \in C^n_{\mathrm{DY}}(F;\mathsf{V}, \mathsf{W})$, which has components  \eqref{componentsDYCochains}. Using the universality of $Z^n_F(V)$, we define $\Gamma^n(f)$ by declaring that (meaning is explained after the picture)
\begin{equation}\label{isoDYComplexBarComplex}
\begingroup%
  \makeatletter%
  \providecommand\color[2][]{%
    \errmessage{(Inkscape) Color is used for the text in Inkscape, but the package 'color.sty' is not loaded}%
    \renewcommand\color[2][]{}%
  }%
  \providecommand\transparent[1]{%
    \errmessage{(Inkscape) Transparency is used (non-zero) for the text in Inkscape, but the package 'transparent.sty' is not loaded}%
    \renewcommand\transparent[1]{}%
  }%
  \providecommand\rotatebox[2]{#2}%
  \newcommand*\fsize{\dimexpr\f@size pt\relax}%
  \newcommand*\lineheight[1]{\fontsize{\fsize}{#1\fsize}\selectfont}%
  \ifx\svgwidth\undefined%
    \setlength{\unitlength}{418.06758392bp}%
    \ifx\svgscale\undefined%
      \relax%
    \else%
      \setlength{\unitlength}{\unitlength * \real{\svgscale}}%
    \fi%
  \else%
    \setlength{\unitlength}{\svgwidth}%
  \fi%
  \global\let\svgwidth\undefined%
  \global\let\svgscale\undefined%
  \makeatother%
  \begin{picture}(1,0.34251224)%
    \lineheight{1}%
    \setlength\tabcolsep{0pt}%
    \put(0,0){\includegraphics[width=\unitlength,page=1]{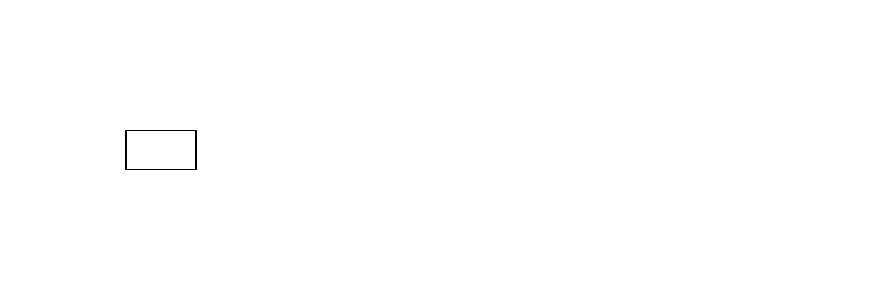}}%
    \put(0.1511211,0.16166111){\color[rgb]{0,0,0}\makebox(0,0)[lt]{\lineheight{1.25}\smash{\begin{tabular}[t]{l}$\Gamma^n(f)$\end{tabular}}}}%
    \put(0,0){\includegraphics[width=\unitlength,page=2]{isoGamma.pdf}}%
    \put(0.1914553,0.12786913){\color[rgb]{0,0,0}\makebox(0,0)[lt]{\lineheight{1.25}\smash{\begin{tabular}[t]{l}$_{Z_F^{n+1}(V)}$\end{tabular}}}}%
    \put(0,0){\includegraphics[width=\unitlength,page=3]{isoGamma.pdf}}%
    \put(0.10541351,0.0680123){\color[rgb]{0,0,0}\makebox(0,0)[lt]{\lineheight{1.25}\smash{\begin{tabular}[t]{l}$i^{(n+1)}_{X_1,\ldots,X_{n+1}}(V)$\end{tabular}}}}%
    \put(0,0){\includegraphics[width=\unitlength,page=4]{isoGamma.pdf}}%
    \put(0.17720549,0.2327843){\color[rgb]{0,0,0}\makebox(0,0)[lt]{\lineheight{1.25}\smash{\begin{tabular}[t]{l}$_{W}$\end{tabular}}}}%
    \put(-0.00041109,0.00354536){\color[rgb]{0,0,0}\makebox(0,0)[lt]{\lineheight{1.25}\smash{\begin{tabular}[t]{l}$_{F(X_{n+1}^{\vee})}$\end{tabular}}}}%
    \put(0.1010823,0.0035126){\color[rgb]{0,0,0}\makebox(0,0)[lt]{\lineheight{1.25}\smash{\begin{tabular}[t]{l}$_{F(X_1^{\vee})}$\end{tabular}}}}%
    \put(0.18032442,0.00373911){\color[rgb]{0,0,0}\makebox(0,0)[lt]{\lineheight{1.25}\smash{\begin{tabular}[t]{l}$_{V}$\end{tabular}}}}%
    \put(0.21085182,0.0035126){\color[rgb]{0,0,0}\makebox(0,0)[lt]{\lineheight{1.25}\smash{\begin{tabular}[t]{l}$_{F(X_1)}$\end{tabular}}}}%
    \put(0.28339222,0.00351262){\color[rgb]{0,0,0}\makebox(0,0)[lt]{\lineheight{1.25}\smash{\begin{tabular}[t]{l}$_{F(X_{n+1})}$\end{tabular}}}}%
    \put(0.25044484,0.02933619){\color[rgb]{0,0,0}\makebox(0,0)[lt]{\lineheight{1.25}\smash{\begin{tabular}[t]{l}$\ldots$\end{tabular}}}}%
    \put(0.08536611,0.0285904){\color[rgb]{0,0,0}\makebox(0,0)[lt]{\lineheight{1.25}\smash{\begin{tabular}[t]{l}$\ldots$\end{tabular}}}}%
    \put(0,0){\includegraphics[width=\unitlength,page=5]{isoGamma.pdf}}%
    \put(0.72867857,0.06775569){\color[rgb]{0,0,0}\makebox(0,0)[lt]{\lineheight{1.25}\smash{\begin{tabular}[t]{l}$f_{X_1,\ldots,X_n}$\end{tabular}}}}%
    \put(0,0){\includegraphics[width=\unitlength,page=6]{isoGamma.pdf}}%
    \put(0.6923271,0.00754433){\color[rgb]{0,0,0}\makebox(0,0)[lt]{\lineheight{1.25}\smash{\begin{tabular}[t]{l}$_{V}$\end{tabular}}}}%
    \put(0.73657912,0.00754433){\color[rgb]{0,0,0}\makebox(0,0)[lt]{\lineheight{1.25}\smash{\begin{tabular}[t]{l}$_{F(X_1)}$\end{tabular}}}}%
    \put(0.85609396,0.00777087){\color[rgb]{0,0,0}\makebox(0,0)[lt]{\lineheight{1.25}\smash{\begin{tabular}[t]{l}$_{F(X_n)}$\end{tabular}}}}%
    \put(0.80005114,0.03310052){\color[rgb]{0,0,0}\makebox(0,0)[lt]{\lineheight{1.25}\smash{\begin{tabular}[t]{l}$\ldots$\end{tabular}}}}%
    \put(0,0){\includegraphics[width=\unitlength,page=7]{isoGamma.pdf}}%
    \put(0.88099063,0.15223564){\color[rgb]{0,0,0}\makebox(0,0)[lt]{\lineheight{1.25}\smash{\begin{tabular}[t]{l}$\rho^V_{X_{n+1}}$\end{tabular}}}}%
    \put(0,0){\includegraphics[width=\unitlength,page=8]{isoGamma.pdf}}%
    \put(0.89513831,0.1152291){\color[rgb]{0,0,0}\makebox(0,0)[lt]{\lineheight{1.25}\smash{\begin{tabular}[t]{l}$_{V}$\end{tabular}}}}%
    \put(0,0){\includegraphics[width=\unitlength,page=9]{isoGamma.pdf}}%
    \put(0.72992655,0.11409813){\color[rgb]{0,0,0}\makebox(0,0)[lt]{\lineheight{1.25}\smash{\begin{tabular}[t]{l}$_{F(X_1 \otimes \ldots \otimes X_n)}$\end{tabular}}}}%
    \put(0,0){\includegraphics[width=\unitlength,page=10]{isoGamma.pdf}}%
    \put(0.92590509,0.00649821){\color[rgb]{0,0,0}\makebox(0,0)[lt]{\lineheight{1.25}\smash{\begin{tabular}[t]{l}$_{F(X_{n+1})}$\end{tabular}}}}%
    \put(0,0){\includegraphics[width=\unitlength,page=11]{isoGamma.pdf}}%
    \put(0.93154376,0.22830847){\color[rgb]{0,0,0}\makebox(0,0)[lt]{\lineheight{1.25}\smash{\begin{tabular}[t]{l}$_{W}$\end{tabular}}}}%
    \put(0,0){\includegraphics[width=\unitlength,page=12]{isoGamma.pdf}}%
    \put(0.58935802,0.00572653){\color[rgb]{0,0,0}\makebox(0,0)[lt]{\lineheight{1.25}\smash{\begin{tabular}[t]{l}$_{F(X_1^{\vee})}$\end{tabular}}}}%
    \put(0,0){\includegraphics[width=\unitlength,page=13]{isoGamma.pdf}}%
    \put(0.5123545,0.00308469){\color[rgb]{0,0,0}\makebox(0,0)[lt]{\lineheight{1.25}\smash{\begin{tabular}[t]{l}$_{F(X_n^{\vee})}$\end{tabular}}}}%
    \put(0.41774362,0.00254567){\color[rgb]{0,0,0}\makebox(0,0)[lt]{\lineheight{1.25}\smash{\begin{tabular}[t]{l}$_{F(X_{n+1}^{\vee})}$\end{tabular}}}}%
    \put(0,0){\includegraphics[width=\unitlength,page=14]{isoGamma.pdf}}%
    \put(0.38504646,0.07162048){\color[rgb]{0,0,0}\makebox(0,0)[lt]{\lineheight{1.25}\smash{\begin{tabular}[t]{l}$=$\end{tabular}}}}%
    \put(0.55916047,0.0335153){\color[rgb]{0,0,0}\makebox(0,0)[lt]{\lineheight{1.25}\smash{\begin{tabular}[t]{l}$\ldots$\end{tabular}}}}%
    \put(0.74733842,0.19456152){\color[rgb]{0,0,0}\makebox(0,0)[lt]{\lineheight{1.25}\smash{\begin{tabular}[t]{l}$\ldots$\end{tabular}}}}%
    \put(0,0){\includegraphics[width=\unitlength,page=15]{isoGamma.pdf}}%
    \put(0.68286601,0.15146548){\color[rgb]{0,0,0}\makebox(0,0)[lt]{\lineheight{1.25}\smash{\begin{tabular}[t]{l}$(F^{(n)}_{X_1,\ldots,X_n})^{-1}$\end{tabular}}}}%
  \end{picture}%
\endgroup%

\end{equation}
for all $X_1,\ldots,X_{n+1} \in \mathcal{C}$. We use the same graphical conventions as in \cite[\S 3]{FGS}: we read diagrams from bottom to top, the caps correspond to the evaluation morphisms $\mathrm{ev}_{F(X)}$ defined in \eqref{dualityForImagesOfF} and $F^{(n)}$ was defined in \eqref{higherMonStruct}. The inverse $\Gamma^{-1}$ is constructed similarly, by inserting the iterated monoidal structure $F^{(n)}$ in \cite[eq.\,(49)]{FGS}. One can check by tedious but straightforward computations that $\Gamma$ commutes with the coface maps.
\end{proof}

\subsection{Adjunction theorem for DY cohomology}\label{sectionChangeOfFunctor}
Let $\mathcal{C}$, $\mathcal{D}$ and $F : \mathcal{C} \to \mathcal{D}$ be as in \eqref{assumptionsFCD}. Recall from \S\ref{relExtGroupsDYCohomology} the monads $Z_{\mathcal{C}} : \mathcal{C} \to \mathcal{C}$ and $Z_F : \mathcal{D} \to \mathcal{D}$. Here we establish a strong morphism (Def.~\ref{defMorphismMonads}) between these monads and deduce from Prop.~\ref{isoExtForMonadicAdjunctions} and Thm.~\ref{thmDYCohomRelExt} our main result Thm.~\ref{thmChangeOfCoeffDY} called {\it Adjunction Theorem for DY cohomology.}

\smallskip

\indent For $V \in \mathcal{C}$, $W \in \mathcal{D}$ and all $X \in \mathcal{C}$ we continue to denote by $i^{\mathcal{C}}_X(V) : X^{\vee} \otimes V \otimes X \to Z_{\mathcal{C}}(V)$ and $i^F_X(W) : F(X^{\vee}) \otimes W \otimes F(X) \to Z_F(W)$ the universal dinatural transformations of the coends $Z_{\mathcal{C}}(V)$ and $\mathcal{Z}_F(W)$. From the monoidal structure $F^{(2)} : F \otimes F \overset{\sim}{\implies} F(- \otimes -)$ we can define $F^{(3)} : F \otimes F \otimes F \overset{\sim}{\implies} F(- \otimes - \otimes -)$; the exact formula of $F^{(3)}$ is irrelevant due to the coherence theorem for monoidal structures \cite{epstein}. By universality of $i^F\bigl(F(V)\bigr)$ we have a commutative diagram
\begin{equation}\label{diagramDefiningZetaForCentralMonads}
\xymatrix@C=4em{
F(X^{\vee}) \otimes F(V) \otimes F(X) \ar[d]_{i^F_X(F(V))}\ar[r]^-{F^{(3)}_{X^{\vee}, V, X}} & F\bigl( X^{\vee} \otimes V \otimes X \bigr) \ar[d]^{F(i^{\mathcal{C}}_X(V))}\\
Z_F  F(V) = \int^{X \in \mathcal{C}} F(X^{\vee}) \otimes F(V) \otimes F(X) \ar[r]_-{\exists!\,\zeta_V} & F Z_{\mathcal{C}}(V) = F\left( \int^{X \in \mathcal{C}} X^{\vee} \otimes V \otimes X \right)
} \end{equation}
It is easily seen from \eqref{ZFonMorphisms} that the collection of isomorphisms $(\zeta_V)_{V \in \mathcal{C}}$ assembles into a natural isomorphism $\zeta : Z_F \, F \Rightarrow F \, Z_{\mathcal{C}}$.

\begin{lemma}\label{lemmaZFFisoFZC}
$(F,\zeta)$ defined in \eqref{diagramDefiningZetaForCentralMonads} is a strong morphism of monads $Z_{\mathcal{C}} \to Z_F$ (Def. \ref{defMorphismMonads}). Hence by Lemma \ref{liftedFunctor} we have the functor
\begin{equation}\label{liftedFunctorCentralMonads}
\widetilde{F} = \widetilde{F}_{\zeta} : Z_{\mathcal{C}}\text{-}\mathrm{mod} \to Z_F\text{-}\mathrm{mod}, \quad (V,r) \mapsto \bigl( F(V), F(r) \circ \zeta_V \bigr), \quad f \mapsto F(f).
\end{equation}
\end{lemma}
\begin{proof}
Since $F : \mathcal{C} \to \mathcal{D}$ is a right exact functor, Corollary \ref{coendsAsCokerAndConsequence} ensures $\zeta_V$ is an isomorphism for any $V \in \mathcal{C}$. To show the first equality in \eqref{conditionsZetaLiftF} we compute
\begin{align*}
&\zeta_V \circ \mu^F_{F(V)} \circ i^{F,(2)}_{X,Y}(F(V)) \overset{\eqref{defMuMonadZF}}{=} \zeta_V \circ i_{X \otimes Y}^F(F(V)) \circ \bigl( F^{(2)}_{Y^{\vee},X^{\vee}} \otimes \mathrm{id}_{F(V)} \otimes F^{(2)}_{X,Y} \bigr)\\
&\overset{\eqref{diagramDefiningZetaForCentralMonads}}{=} F\bigl(i^{\mathcal{C}}_{X \otimes Y}(V)\bigr) \circ F^{(3)}_{Y^{\vee} \otimes X^{\vee},V, X \otimes Y} \circ \bigl( F^{(2)}_{Y^{\vee},X^{\vee}} \otimes \mathrm{id}_{F(V)} \otimes F^{(2)}_{X,Y} \bigr)\\
&\overset{\eqref{defMuMonadZF}}{=} F(\mu^{\mathcal{C}}_V) \circ F\bigl( i^{\mathcal{C}}_Y\bigl( Z_{\mathcal{C}}(V) \bigr) \bigr) \circ F\bigl( \mathrm{id}_{Y^{\vee}} \otimes i^{\mathcal{C}}_X(V) \otimes \mathrm{id}_Y \bigr) \circ F^{(3)}_{Y^{\vee} \otimes X^{\vee},V, X \otimes Y} \circ \bigl( F^{(2)}_{Y^{\vee},X^{\vee}} \otimes \mathrm{id}_{F(V)} \otimes F^{(2)}_{X,Y} \bigr)\\
&= F(\mu^{\mathcal{C}}_V) \circ F\bigl( i^{\mathcal{C}}_Y\bigl( Z_{\mathcal{C}}(V) \bigr) \bigr) \circ F\bigl( \mathrm{id}_{Y^{\vee}} \otimes i^{\mathcal{C}}_X(V) \otimes \mathrm{id}_Y \bigr) \circ F^{(3)}_{Y^{\vee}, X^{\vee} \otimes V \otimes X, Y}\circ \bigl( \mathrm{id}_{F(Y^{\vee})} \otimes F^{(3)}_{X^{\vee},V,X} \otimes \mathrm{id}_{F(Y)} \bigr)\\
&= F(\mu^{\mathcal{C}}_V) \circ F\bigl( i^{\mathcal{C}}_Y\bigl( Z_{\mathcal{C}}(V) \bigr) \bigr) \circ F^{(3)}_{Y^{\vee}, Z_{\mathcal{C}}(V), Y} \circ \bigl( \mathrm{id}_{F(Y^{\vee})} \otimes F\bigl(i^{\mathcal{C}}_X(V)\bigr) \otimes \mathrm{id}_F(Y) \bigr)\\
&\qquad\qquad\qquad\qquad\qquad\qquad\qquad\qquad\qquad\:\:\,\circ \bigl( \mathrm{id}_{F(Y^{\vee})} \otimes F^{(3)}_{X^{\vee},V,X} \otimes \mathrm{id}_{F(Y)} \bigr)\\
&\overset{\eqref{defMuMonadZF}}{=} F(\mu^{\mathcal{C}}_V) \circ \zeta_{Z_{\mathcal{C}}(V)} \circ i^F_Y\bigl( FZ_{\mathcal{C}}(V) \bigr) \circ \left( \mathrm{id}_{F(Y^{\vee})} \otimes \left( \zeta_V \circ i_X^F(F(V)) \right) \otimes \mathrm{id}_{F(Y)} \right)\\
&\overset{\eqref{ZFonMorphisms}}{=}F(\mu^{\mathcal{C}}_V) \circ \zeta_{Z_{\mathcal{C}}(V)} \circ Z_F(\zeta_V) \circ i_Y^F\bigl( Z_FF(V) \bigr) \circ \left( \mathrm{id}_{F(Y^{\vee})} \otimes i_X^F(F(V)) \otimes \mathrm{id}_{F(Y)} \right)\\
&\overset{\eqref{defUniversalDinatDoubleCoend}}{=}F(\mu^{\mathcal{C}}_V) \circ \zeta_{Z_{\mathcal{C}}(V)} \circ Z_F(\zeta_V) \circ i^{F,(2)}_{X,Y}(F(V))
\end{align*}
where the unlabelled equalities are respectively derived from the coherence condition for monoidal structures and from naturality of $F^{(3)}$. The second equality in \eqref{conditionsZetaLiftF} is readily true.
\end{proof}

Now assume that $F$ has a right adjoint $R : \mathcal{D} \to \mathcal{C}$. {\em The existence of $R$ is automatic if $\mathcal{D}$ is finite as a $\Bbbk$-linear category} by \cite[Cor.\,1.9]{DSPS}, which asserts that a right-exact $\Bbbk$-linear functor between finite $\Bbbk$-linear categories admits a right adjoint. Since $\zeta$ is an isomorphism, Lemma \ref{liftRightAdjoint} applies and the adjunction $F \dashv R$ can be lifted to an adjunction $\widetilde{F}_{\zeta} \dashv \widetilde{R}$ between the categories of modules over $Z_{\mathcal{C}}$ and $Z_F$. Using the isomorphisms $Z_{\mathcal{C}}\text{-}\mathrm{mod} \cong \mathcal{Z}(\mathcal{C})$ and $Z_F\text{-}\mathrm{mod} \cong \mathcal{Z}(F)$ explained in \S \ref{relExtGroupsDYCohomology} this can be rephrased as follows
\begin{equation}\label{diagramAdjunctionsBetweenCentralizers}
\xymatrix@C=4em@R=.7em{
\mathcal{Z}(\mathcal{C}) \ar@/^.7em/[dd]^{\mathcal{U}_{\mathcal{C}}} \ar@/^.7em/[r]^{\widetilde{F}}_{\text{\normalsize \rotatebox{270}{$\dashv$}}} & \ar@/^.7em/[l]^{\widetilde{R}}
\mathcal{Z}(F) \ar@/^.7em/[dd]^{\mathcal{U}_F}\\
\dashv & \dashv\\
\ar@/^.7em/[uu]^{\mathcal{F}_{\mathcal{C}}}\mathcal{C} \ar@/^.7em/[r]^{F}_{\text{\normalsize \rotatebox{270}{$\dashv$}}} & \ar@/^.7em/[l]^R \ar@/^.7em/[uu]^{\mathcal{F}_F}\mathcal{D}
} \end{equation}
Using \eqref{moduleStructureViaHalfBraiding}, \eqref{liftedFunctorCentralMonads} and \eqref{halfBraidingFromModuleStructure} one can compute that the functor $\widetilde{F}$ in \eqref{diagramAdjunctionsBetweenCentralizers} is given on objects by $\widetilde{F}(V, \rho) = \bigl( F(V), \rho^F \bigr)$ with
\[ \rho^F_X : F(V) \otimes F(X) \xrightarrow{F^{(2)}_{V,X}} F(V \otimes X) \xrightarrow{F(\rho_X)} F(X \otimes V) \xrightarrow{(F^{(2)}_{X,V})^{-1}} F(X) \otimes F(V) \]
for all $X \in \mathcal{C}$ and on morphisms by $\widetilde{F}(f) = F(f)$. Although we do not need it in the sequel, we note that $\widetilde{F}$ is monoidal. Indeed if $\mathsf{V}_1 = (V_1,\rho_1), \mathsf{V}_2 = (V_2,\rho_2) \in \mathcal{Z}(\mathcal{C})$ then using \eqref{conditionsMonStruct} it is straightforward to check that $F^{(2)}_{V_1,V_2} \in \Hom_{\mathcal{C}}\bigl( F(V_1) \otimes F(V_2), F(V_1 \otimes V_2) \bigr)$ is actually in $\Hom_{\mathcal{Z}(\mathcal{C})}\bigl( \widetilde{F}(\mathsf{V}_1) \otimes \widetilde{F}(\mathsf{V}_2), \widetilde{F}(\mathsf{V}_1 \otimes \mathsf{V}_2) \bigr)$, \textit{i.e.} it commutes with the half-braidings. Hence we can define $\widetilde{F}^{(2)}_{\mathsf{V}_1,\mathsf{V}_2} = F^{(2)}_{V_1,V_2}$.

\begin{theorem}\label{thmChangeOfCoeffDY}
Recall the assumptions on $\mathcal{C}$, $\mathcal{D}$ and $F$ made in \eqref{assumptionsFCD} and assume moreover that $F$ admits a right adjoint $R$ (which is for instance the case when $\mathcal{D}$ is finite as $\Bbbk$-linear category). Then with the notations from \eqref{diagramAdjunctionsBetweenCentralizers} we have
\[ H^{\bullet}_{\mathrm{DY}}\bigl(F; \widetilde{F}(\mathsf{V}), \mathsf{W}\bigr) \cong H^{\bullet}_{\mathrm{DY}}\bigl(\mathcal{C}; \mathsf{V}, \widetilde{R}(\mathsf{W})\bigr). \] 
for all $\mathsf{V} \in \mathcal{Z}(\mathcal{C})$ and $\mathsf{W} \in \mathcal{Z}(F)$. In particular,
\[ H^{\bullet}_{\mathrm{DY}}(F) \cong H^{\bullet}_{\mathrm{DY}}\bigl(\mathcal{C}; \boldsymbol{1}, \widetilde{R}(\boldsymbol{1})\bigr). \]
\end{theorem}
\begin{proof}
We have
\begin{equation}\label{isosProofThmAdjunction}
H^{\bullet}_{\mathrm{DY}}\bigl(F; \widetilde{F}(\mathsf{V}), \mathsf{W}\bigr) \cong \Ext_{\mathcal{Z}(F), \mathcal{D}}^{\bullet}\bigl(\widetilde{F}(\mathsf{V}), \mathsf{W}\bigr) \cong \Ext_{\mathcal{Z}(\mathcal{C}), \mathcal{C}}^{\bullet}\bigl(\mathsf{V}, \widetilde{R}(\mathsf{W}) \bigr) \cong H^{\bullet}_{\mathrm{DY}}\bigl(\mathcal{C}; \mathsf{V}, \widetilde{R}(\mathsf{W})\bigr).
\end{equation}
The first isomorphism uses Theorem \ref{thmDYCohomRelExt}, the second uses \eqref{isoRelExtCentralizerOrModulesCentralMonad} and Proposition \ref{isoExtForMonadicAdjunctions} and the third uses again Theorem \ref{thmDYCohomRelExt} but now for the identity functor $\mathrm{Id}_{\mathcal{C}}$. For the last claim in the theorem, take $\mathsf{V} = \mathsf{W} = \boldsymbol{1}$ and note that $\widetilde{F}(\boldsymbol{1}) = \boldsymbol{1}$.
\end{proof}
Recall from \eqref{defTangentSpaceQuotient} the space $\mathbb{T}_{[\theta]}\bigl(\mathrm{Mon}(F)/\!\!\sim \bigr)$ which is ``tangent'' to the equivalence class $[\theta]$ of a monoidal structure $\theta$ for $F$.
\begin{corollary}\label{coroTangentSpaceChangeOfFunctor}
1. Let $\theta = F^{(2)}$ be the given monoidal structure of $F$. Then under the assumptions and notations of Theorem \ref{thmChangeOfCoeffDY} we have
\[ \mathbb{T}_{[\theta]}\bigl(\mathrm{Mon}(F)/\!\!\sim \bigr)  \cong  \mathrm{Ext}^2_{\mathcal{Z}(\mathcal{C}),\mathcal{C}}\bigl( \boldsymbol{1},\widetilde{R}(\boldsymbol{1}) \bigr). \]
2. If $0 \to \mathsf{K} \to \mathsf{P} \to \boldsymbol{1} \to 0$ is an allowable exact sequence in $\mathcal{Z}(\mathcal{C})$ with $\mathsf{P}$ relatively projective then for all $ n \geq 2$ we have
\[ \dim H^n_{\mathrm{DY}}(F) = \dim \Hom_{\mathcal{Z}(\mathcal{C})}( \mathsf{K}, \mathsf{M} ) - \dim \Hom_{\mathcal{Z}(\mathcal{C})}( \mathsf{P}, \mathsf{M} ) + \dim \Hom_{\mathcal{Z}(\mathcal{C})}( \boldsymbol{1}, \mathsf{M} ) \]
 where $\mathsf{M} = \widetilde{R}(\boldsymbol{1}) \otimes (\mathsf{K}^{\vee})^{\otimes (n-1)}$.
\end{corollary}
\begin{proof}
1. Follows from \eqref{tangentSpaceAndCohom} and Theorem \ref{thmChangeOfCoeffDY}.
\\2. Follows from  Theorem \ref{thmChangeOfCoeffDY} and the dimension formula in \cite[Cor.~4.10]{FGS}.
\end{proof}

We can describe rather explicitly the functor $\widetilde{R} : \mathcal{Z}(F) \to \mathcal{Z}(\mathcal{C})$.
\begin{proposition}
If $(W,\lambda) \in \mathcal{Z}(F)$ then $\widetilde{R}(W,\lambda) = \bigl( R(W), \lambda^R \bigr)$ where $\lambda^R_X : R(W) \otimes X \to X \otimes R(W)$ is the image of $\lambda_X : W \otimes F(X) \to F(X) \otimes W$ through the following sequence of linear maps for all $X \in \mathcal{C}$:
\begin{align*}
\Hom_{\mathcal{D}}\bigl(W \otimes F(X), F(X) \otimes W \bigr)
&\overset{\sim}{\longrightarrow} \Hom_{\mathcal{D}}\bigl(F(X^{\vee}) \otimes  W \otimes F(X), W\bigr) \quad \text{(duality \eqref{dualityForImagesOfF})}\\
&\xrightarrow{(\mathrm{id} \otimes h_W \otimes \mathrm{id})^*} \Hom_{\mathcal{D}}\bigl(F(X^{\vee}) \otimes  FR(W) \otimes F(X), W\bigr) \\
&\overset{\sim}{\longrightarrow} \Hom_{\mathcal{D}}\bigl(F(X^{\vee} \otimes  R(W) \otimes X), W\bigr) \quad \text{(monoidality of }F\text{)}\\
&\overset{\sim}{\longrightarrow} \Hom_{\mathcal{C}}\bigl(X^{\vee} \otimes  R(W) \otimes X, R(W) \bigr) \quad \text{(adjunction } F\dashv R \text{)}\\
&\overset{\sim}{\longrightarrow} \Hom_{\mathcal{C}}\bigl(R(W) \otimes X, X \otimes   R(W) \bigr) \quad \text{(duality in }\mathcal{C}\text{)}
\end{align*}
where $h : FR \Rightarrow \mathrm{Id}_{\mathcal{D}}$ is the counit of $F \dashv R$. On morphisms we simply have $\widetilde{R}(f) = R(f)$.
\end{proposition}
\begin{proof}
Let us first use the categories of modules $Z_{\mathcal{C}}\text{-}\mathrm{mod}$ and $Z_F\text{-}\mathrm{mod}$, which are isomorphic to $\mathcal{Z}(\mathcal{C})$ and $\mathcal{Z}(F)$. Let $(W,r) \in Z_F\text{-}\mathrm{mod}$; then by Lemma \ref{liftRightAdjoint}, $\widetilde{R}(W,r) = \bigl( R(W), R(r) \circ \xi_W \bigr)$ with $\xi_W : Z_{\mathcal{C}} R(W) \to RZ_F(W)$ defined in \eqref{defIsoNatXi}. Let $e : \mathrm{Id}_{\mathcal{C}} \Rightarrow RF$ and $h : FR \Rightarrow \mathrm{Id}_{\mathcal{D}}$ denote the unit and counit of $F \dashv R$. By naturality of $e$, by definition of $\zeta$ in \eqref{diagramDefiningZetaForCentralMonads} and by definition of $Z_F(h_W)$ in \eqref{ZFonMorphisms}, $\xi_W$ is characterized by the following commutative diagram
\begin{equation}\label{descriptionXiForCentralMonads}
\xymatrix@C=3.5em{
X^{\vee} \otimes R(W) \otimes X \ar[rr]^-{e_{X^{\vee} \otimes R(W) \otimes X}} \ar[dd]_{i^{\mathcal{C}}_X(R(W))} && RF\bigl( X^{\vee} \otimes R(W) \otimes X \bigr) \ar[d]^{R(F^{(3)}_{X^{\vee},R(W),X})^{-1}}\\
&& R\bigl( F(X^{\vee}) \otimes FR(W) \otimes F(X) \bigr)\ar[d]^{R(\mathrm{id}_{F(X^{\vee})} \otimes h_W \otimes \mathrm{id}_{F(X)})}\\
Z_{\mathcal{C}}R(W) \ar[r]_{\exists!\,\xi_W} & RZ_F(W) & \ar[l]^-{R(i^F_X(W))} R\bigl( F(X^{\vee}) \otimes W \otimes F(X) \bigr)
} \end{equation}
Now one can check that the functor $\mathcal{Z}(F) \overset{\eqref{moduleStructureViaHalfBraiding}}{\longrightarrow} Z_F\text{-}\mathrm{mod} \overset{\widetilde{R}}{\longrightarrow} Z_{\mathcal{C}}\text{-}\mathrm{mod} \overset{\eqref{halfBraidingFromModuleStructure}}{\longrightarrow} \mathcal{Z}(\mathcal{C})$ is indeed given by the announced formula.
\end{proof}

\indent To finish, we describe the isomorphism in Theorem \ref{thmChangeOfCoeffDY} at the level of the DY cochain complexes. Recall from \eqref{ZFIterated}--\eqref{dinatTransfoZFIterated} the iterated coend $Z_F^n(V)$ and its universal dinatural transformation $i^{F,(n)}_{X_1,\ldots,X_n}(V)$. Let $G_{\mathcal{C}}$ (resp. $G_F$) be the comonad on $\mathcal{Z}(\mathcal{C}) \cong Z_{\mathcal{C}}\text{-}\mathrm{mod}$ (resp. on $\mathcal{Z}(F) \cong Z_F\text{-}\mathrm{mod}$). Write $\mathsf{V} = (V,\rho^V) \in \mathcal{Z}(\mathcal{C})$ and $\mathsf{W} = (W,\rho^W) \in \mathcal{Z}(F)$. Note that the natural isomorphism $\zeta^{(n)}_V : G_F^nF(V) \to FG_{\mathcal{C}}^n(V)$ generally defined in \eqref{isoGammaForFLiftAndComonads} is characterized by the commutative diagram
\[ \xymatrix@C=8em{
F(X_n^{\vee})\ldots F(X_1^{\vee}) F(V) F(X_1) \ldots F(X_n^{\vee})  \ar[r]^-{F^{(2n+1)}_{X_n^{\vee}, \ldots, X_1^{\vee}, V, X_1, \ldots X_n}} \ar[d]^{i^{F,(n)}_{X_1,\ldots,X_n}(F(V))} & F\bigl( X_n^{\vee} \ldots X_1^{\vee} V X_1 \ldots X_n \bigr) \ar[d]_-{F\left(i^{\mathcal{C},(n)}_{X_1,\ldots,X_n}(V)\right)}\\
Z_F^nF(V) \ar[r]_-{\exists!\, \zeta^{(n)}_V} & FZ_{\mathcal{C}}^n(V)
} \]
where on the top row we omit $\otimes$ and on the bottom row we use that $\widetilde{F}_{\zeta}G_{Z_{\mathcal{C}}}^n(\mathsf{V}) = FZ_{\mathcal{C}}^n(V)$ and $G_{Z_F}^n\widetilde{F}_{\zeta}(\mathsf{V}) = Z_F^nF(V)$  as underlying objects in $\mathcal{D}$. 
Then we can define an isomorphism between DY cochain complexes by requiring that the following diagram commutes in all degrees $n \geq 0$:
\[ \xymatrix@C=3em{
\Hom_{\mathcal{Z}(\mathcal{C})}\!\left( G_{\mathcal{C}}^{n+1}(\mathsf{V}), \widetilde{R}(\mathsf{W}) \right) \ar[r]^-{\cong}_{\eqref{isoAdjunctionWithBarComplexes}} \ar[d]^-{\cong}_{\eqref{isoDYComplexBarComplex}} & \Hom_{\mathcal{Z}(F)}\!\left( G_F^{n+1}\widetilde{F}_{\zeta}(\mathsf{V}), \mathsf{W} \right) \ar[d]^-{\cong}_{\eqref{isoDYComplexBarComplex}} \\
C^n_{\mathrm{DY}}\bigl( \mathcal{C}; \mathsf{V},  \widetilde{R}(\mathsf{W}) \bigr) \ar[r]^-{\exists!\,\cong} & C^n_{\mathrm{DY}}\bigl( F; \widetilde{F}(\mathsf{V}), \mathsf{W} \bigr)
} \]
Explicitly, it is given as follows on the components of the natural transformations:
\begin{align*}
&\Hom_{\mathcal{C}}\bigl(V \otimes X_1 \otimes \ldots \otimes X_n, X_1 \otimes \ldots \otimes X_n \otimes R(W) \bigr)\\
\cong\:& \Hom_{\mathcal{C}}\bigl(X_n^{\vee} \otimes \ldots \otimes X_1^{\vee} \otimes V \otimes X_1 \otimes \ldots \otimes X_n, R(W) \bigr)\qquad \text{(duality)}\\
\cong\:& \Hom_{\mathcal{D}}\bigl( F\bigl(X_n^{\vee} \otimes \ldots \otimes X_1^{\vee} \otimes V \otimes X_1 \otimes \ldots \otimes X_n\bigr), W \bigr)\qquad \text{(adjunction } F \dashv R)\\
\cong\:& \Hom_{\mathcal{D}}\bigl( F(X_n^{\vee} \otimes \ldots \otimes X_1^{\vee}) \otimes F(V) \otimes F(X_1) \otimes \ldots \otimes F(X_n), W \bigr) \qquad \text{(monoidality of } F)\\
\cong\:&\Hom_{\mathcal{D}}\bigl( F(V) \otimes F(X_1) \otimes \ldots \otimes F(X_n), F(X_1 \otimes \ldots \otimes X_n) \otimes W \bigr) \qquad \text{(duality)}
\end{align*}
where for the second duality transformation we use \eqref{dualityForImagesOfF}.

\section{Tangent space to a braiding}\label{sectionTangentSpaceBraiding}
Let $\mathcal{C}$ be a monoidal category, which we assume strict for simplicity. Recall that a {\em braiding} on $\mathcal{C}$ is a natural isomorphism $c = \bigl(c_{X,Y} : X \otimes Y \overset{\sim}{\longrightarrow} Y \otimes X\bigr)_{X,Y \in \mathcal{C}}$ which satisfies
\begin{equation}\label{axiomsBraiding}
c_{X \otimes Y,Z} = (c_{X,Z} \otimes \mathrm{id}_Y) \circ (\mathrm{id}_X \otimes c_{Y,Z}), \qquad c_{X, Y \otimes Z} = (\mathrm{id}_Y \otimes c_{X,Z}) \circ (c_{X,Y} \otimes \mathrm{id}_Z)
\end{equation}
for all $X,Y,Z \in \mathcal{C}$. We denote by $\mathrm{Br}(\mathcal{C})$ the set of all braidings on $\mathcal{C}$ (not taken up to equivalence). Note that a braiding automatically satisfies $c_{X,\boldsymbol{1}} = c_{\boldsymbol{1},X} = \mathrm{id}_X$.

\smallskip

\indent Now assume that $\mathcal{C}$ is $\Bbbk$-linear and that $\otimes$ is $\Bbbk$-bilinear on morphisms, where $\Bbbk$ is a field. For $c \in \mathrm{Br}(\mathcal{C})$ there is a space of ``tangent vectors'' to $c$:
\begin{definition}\label{defTangentBraiding}
An infinitesimal braiding tangent to $c$ is a natural transformation $t = \bigl( t_{X,Y} : X \otimes Y \to Y \otimes X \bigr)_{X,Y \in \mathcal{C}}$ such that for all $X,Y,Z \in \mathcal{C}$ we have
\begin{align*}
t_{X \otimes Y, Z} &= (c_{X,Z} \otimes \mathrm{id}_Y) \circ (\mathrm{id}_X \otimes t_{Y,Z}) + (t_{X,Z} \otimes \mathrm{id}_Y) \circ (\mathrm{id}_X \otimes c_{Y,Z})\\
t_{X,Y \otimes Z} &= (\mathrm{id}_Y \otimes c_{X,Z}) \circ (t_{X,Y} \otimes \mathrm{id}_Z) + (\mathrm{id}_Y \otimes t_{X,Z}) \circ (c_{X,Y} \otimes \mathrm{id}_Z).
\end{align*}
We denote by $\mathbf{T}_c\mathrm{Br}(\mathcal{C})$ the $\Bbbk$-vector space of infinitesimal braidings tangent to $c$.
\end{definition}
\noindent As the name and notation suggests, this definition is obtained as follows: $t$ is an infinitesimal braiding on $\mathcal{C}$ tangent to $c$ if and only if $c+\epsilon t$ is a braiding on $\mathcal{C}_{\epsilon}$, where $\mathcal{C}_{\epsilon} = \mathcal{C} \otimes_{\Bbbk} \Bbbk[\epsilon]/(\epsilon^2)$ is defined in \S\ref{subsectionDefMonStruct}.
\begin{remark}\label{remarkCartierInfBraidings}
1. The name ``infinitesimal braiding'' is already used in the context of deformation of {\em symmetric} monoidal categories, in relation with Vassiliev invariants \cite[\S 4]{cartier}, \cite[Def.\, XX.4.1]{kassel}. If $c$ is a symmetric braiding on $\mathcal{C}$ and $\bigl( h_{X,Y} \in \mathrm{End}_{\mathcal{C}}(X \otimes Y) \bigr)_{X,Y \in \mathcal{C}}$ is an infinitesimal braiding in the sense of \cite[\S 4]{cartier}, then $\bigl( c_{X,Y} \circ h_{X,Y} \bigr)_{X,Y \in \mathcal{C}}$ satisfies Definition \ref{defTangentBraiding}. The converse is not true in general because the property $c_{X,Y} \circ h_{X,Y} = h_{Y,X} \circ c_{X,Y}$ required in \cite[\S 4]{cartier} is not implied by our definition.
\\2. A slight variation of Definition \ref{defTangentBraiding} appears in \cite[Def.\,1.1]{ABSW} under the name pre-Cartier braidings: $\bigl( h_{X,Y} \in \mathrm{End}_{\mathcal{C}}(X \otimes Y) \bigr)_{X,Y \in \mathcal{C}}$ is a pre-Cartier braiding if and only if $\bigl( c_{X,Y} \circ h_{X,Y} \bigr)_{X,Y \in \mathcal{C}}$ is an infinitesimal braiding in our sense.
\end{remark}

\smallskip

\indent Our first goal and the topic of \S\ref{sectionLaxMult} and \S\ref{subsectionDYcohomologyTangentBraidings} is to explain the relation between $\mathbf{T}_c\mathrm{Br}(\mathcal{C})$ and the DY cohomology of the monoidal product $\otimes : \mathcal{C} \times \mathcal{C} \to \mathcal{C}$ endowed with a monoidal structure coming from the braiding $c$. This is based on work by Joyal and Street \cite{JS} who established a bijection between braidings and so-called {\em multiplications} on a monoidal category. Yetter already realized in \cite[Th.\,3.9]{yetter1} that deformations of multiplications yield deformations of braidings but did not describe how the space of infinitesimal deformations of a given braiding fits into the DY cohomology of the corresponding multiplication, as we do in Corollary~\ref{coroDimFormulaTangentSpaceBraiding}. Let us mention that categories with several multiplications are explored in \cite{BFSV}.

\smallskip

\indent The second goal, achieved in \S\ref{adjunctionThmMonProduct}, is to apply our general adjunction theorem for DY cohomology (Thm.\,\ref{thmChangeOfCoeffDY}) to the functor $F = \otimes : \mathcal{C} \boxtimes \mathcal{C} \to \mathcal{C}$ endowed with the monoidal structure induced by a braiding. We first compute the corresponding coefficient in Theorem~\ref{thmChangeFunctorMonProduct} and this gives us a formula for the dimension of $\mathbf{T}_c\mathrm{Br}(\mathcal{C})$ in terms of 
 $\Ext^2_{\zcat\boxtimes\zcat,\cat\boxtimes\cat}$, see Corollary \ref{coroDimTcBrCInTermsOfRelExt}, which is computable from a relatively projective resolution of $\boldsymbol{1} \in \mathcal{Z}(\mathcal{C})$. In \S\ref{sec:end-formula}, we furthermore apply a K\"unneth formula to rewrite this dimension formula in terms of the `standard' adjunction between $\zcat$ and $\cat$, \textit{i.e.}\ involving only relative $\Ext^n_{\zcat,\cat}$ at $n=1,2$, into what we call \textit{the end formula} of the tangent space to a braiding, see  Corollary~\ref{cor:HDY-end-proj}.

\subsection{Multiplications and braidings}\label{sectionLaxMult}
\indent In this section we define the notion of a weak-unital multiplication on a monoidal category. This is a slight generalization of the concept of multiplication on a monoidal category introduced by Joyal and Street \cite[\S 5]{JS}. We then explain what the correspondence between multiplications and braidings obtained in \cite[Prop.\, 5.3]{JS} becomes for weak-unital multiplications. The point of introducing weak-unital multiplications is that a DY deformation of a multiplication yields only a weak-unital multiplication on $\mathcal{C}_{\epsilon} = \mathcal{C} \otimes_{\Bbbk} \Bbbk[\epsilon]/(\epsilon^2)$, as we will see in the next section (Remark \ref{remarkDeformationOfMultIsWeak}).

\smallskip

\indent Let $\mathcal{C} = (\mathcal{C},\otimes,\boldsymbol{1})$ be a monoidal category, assumed to be strict for simplicity. Then $\mathcal{C} \times \mathcal{C}$ is a monoidal category, with product $(X_1,Y_1)\otimes (X_2,Y_2) = (X_1 \otimes X_2, Y_1 \otimes Y_2)$ and unit object $(\boldsymbol{1},\boldsymbol{1})$.

\begin{definition}
1. A weak-unital multiplication on $\mathcal{C}$ is a monoidal functor $\Phi : \mathcal{C} \times \mathcal{C} \to \mathcal{C}$ such that $\Phi(-,\boldsymbol{1}) = \Phi(\boldsymbol{1},-) = \mathrm{Id}_{\mathcal{C}}$ as functors (but not necessarily as {\em monoidal} functors).
\\2. A multiplication on $\mathcal{C}$ is a monoidal functor $\Phi : \mathcal{C} \times \mathcal{C} \to \mathcal{C}$ such that $\Phi(-,\boldsymbol{1}) = \Phi(\boldsymbol{1},-) = \mathrm{Id}_{\mathcal{C}}$ as {\em monoidal} functors \cite[\S 5]{JS}, meaning that their monoidal structures are equal.\footnote{Actually Joyal and Street use monoidal natural isomorphisms instead of equalities in these conditions, but this is not the important point. Here we use equalities simply to have shorter formulas.}
\end{definition}
\noindent By definition a weak-unital multiplication $\Phi$ comes with a monoidal structure $\Phi^{(2)}$, which is a natural isomorphism
\[ \Phi^{(2)}_{(X_1,Y_1), (X_2,Y_2)} : \Phi(X_1, Y_1) \otimes \Phi(X_2, Y_2) \overset{\sim}{\longrightarrow} \Phi(X_1 \otimes X_2, Y_1 \otimes Y_2) \]
for all $(X_i,Y_i) \in \mathcal{C} \times \mathcal{C}$ such that
\begin{equation}\label{conditionMonStructTensorProduct}
\begin{array}{l}\displaystyle
\quad\Phi^{(2)}_{(X_1 \otimes X_2,Y_1 \otimes Y_2), (X_3,Y_3)} \circ \bigl( \Phi^{(2)}_{(X_1,Y_1), (X_2,Y_2)} \otimes \mathrm{id}_{X_3 \otimes Y_3} \bigr)\\[.5em]
=\Phi^{(2)}_{(X_1,Y_1), (X_2 \otimes X_3, Y_2 \otimes Y_3)} \circ \bigl( \mathrm{id}_{X_1 \otimes Y_1} \otimes \Phi^{(2)}_{(X_2,Y_2), (X_3,Y_3)} \bigr)
\end{array}
\qquad \forall \, (X_i,Y_i) \in \mathcal{C} \times \mathcal{C}
\end{equation}
and\footnote{Recall from \S\ref{subsectionDefMonStruct} that for simplicity we consider monoidal functors which satisfy $F(\boldsymbol{1}) = \boldsymbol{1}$. Hence here we assume that $\Phi(\boldsymbol{1},\boldsymbol{1}) = \boldsymbol{1}$ and \eqref{normalisationMonStructMult} makes sense. This is enough for our purposes because very soon $\Phi$ will become $\otimes$ as a functor. In full generality a given isomorphism $\Phi^{(0)} : \Phi(\boldsymbol{1},\boldsymbol{1}) \to \boldsymbol{1}$ has to be used.}
\begin{equation}\label{normalisationMonStructMult}
\Phi^{(2)}_{(\boldsymbol{1},\boldsymbol{1}),(X,Y)} = \Phi^{(2)}_{(X,Y),(\boldsymbol{1},\boldsymbol{1})} = \mathrm{id}_{X \otimes Y}.
\end{equation}
Since the monoidal structures of $\Phi(-,\boldsymbol{1})$ and $\Phi(\boldsymbol{1},-)$ are given by $\Phi(-,\boldsymbol{1})^{(2)}_{X_1,X_2} = \Phi^{(2)}_{(X_1,\boldsymbol{1}),(X_2,\boldsymbol{1})}$ and $\Phi(\boldsymbol{1},-)^{(2)}_{Y_1,Y_2} = \Phi^{(2)}_{(\boldsymbol{1},Y_1),(\boldsymbol{1},Y_2)}$, we see that a multiplication is a weak-unital multiplication which moreover satisfies the following {\em unitality condition}:
\begin{equation}\label{conditionUnitMult}
\Phi^{(2)}_{(X_1,\boldsymbol{1}),(X_2,\boldsymbol{1})} = \mathrm{id}_{X_1 \otimes X_2} \quad \text{ and } \quad \Phi^{(2)}_{(\boldsymbol{1},Y_1),(\boldsymbol{1},Y_2)} = \mathrm{id}_{Y_1 \otimes Y_2}.
\end{equation}

\begin{proposition}
Let $\Phi$ be a weak-unital multiplication on $\mathcal{C}$. We have a natural isomorphism $\gamma : \otimes \Rightarrow \Phi$ given by
\[ \gamma_{(X,Y)} : X \otimes Y = \Phi(X,\boldsymbol{1}) \otimes \Phi(\boldsymbol{1},Y) \xrightarrow{\Phi^{(2)}_{(X,\boldsymbol{1}), (\boldsymbol{1},Y)}} \Phi(X,Y). \]
Let
\begin{align*}
\phi_{(X_1,Y_1),(X_2,Y_2)} : X_1 \otimes Y_1 \otimes X_2 \otimes Y_2 &\xrightarrow{\gamma_{(X_1,Y_1)} \otimes \gamma_{(X_2,Y_2)}} \Phi(X_1,Y_1) \otimes \Phi(X_2,Y_2)\\
&\xrightarrow{\Phi^{(2)}_{(X_1,Y_1),(X_2,Y_2)}} \Phi(X_1 \otimes X_2, Y_1 \otimes Y_2)\\
&\xrightarrow{\gamma_{(X_1 \otimes X_2,Y_1 \otimes Y_2)}^{-1}} X_1 \otimes X_2 \otimes Y_1 \otimes Y_2.
\end{align*}
Then $\phi$ is a monoidal structure for $\otimes$ and $\gamma$ is a monoidal natural isomorphism $(\otimes,\phi) \Rightarrow \Phi$. Moreover $\Phi$ is a multiplication if and only if $(\otimes, \phi)$ is a multiplication.
\end{proposition}
\begin{proof}
This actually follows from a general easy fact: assume that we have a monoidal functor $(F,\theta) : \mathcal{X} \to \mathcal{Y}$, any functor $G : \mathcal{X} \to \mathcal{Y}$ and a natural isomorphism $\gamma : G \Rightarrow F$. Then $G$ inherits a monoidal structure given by
\[ G(X) \otimes G(X') \xrightarrow{\gamma_X \otimes \gamma_{X'}} F(X) \otimes F(X') \xrightarrow{\theta_{X,X'}} F(X \otimes X') \xrightarrow{\gamma^{-1}_{X \otimes X'}} G(X \otimes X') \]
and $\gamma$ becomes a monoidal natural isomorphism. For the last claim, note that
\[ \phi_{(X_1,\boldsymbol{1}),(X_2,\boldsymbol{1})} = \Phi^{(2)}_{(X_1,\boldsymbol{1}),(X_2,\boldsymbol{1})} \quad \text{and} \quad \phi_{(\boldsymbol{1},Y_1),(\boldsymbol{1},Y_2)} = \Phi^{(2)}_{(\boldsymbol{1},Y_1),(\boldsymbol{1},Y_2)} \]
because $\gamma_{(X,\boldsymbol{1})} = \mathrm{id}_X$ and $\gamma_{(\boldsymbol{1},Y)} = \mathrm{id}_Y$.
\end{proof}
\noindent Hence {\em we can restrict ourselves to weak-unital multiplications for which the underlying functor is $\otimes$}. In this case a weak-unital multiplication is just a monoidal structure for $\otimes$, {\it i.e.} a natural isomorphism
\[ \phi_{(X_1,Y_1),(Y_1,Y_2)} : X_1 \otimes Y_1 \otimes X_2 \otimes Y_2 \to X_1 \otimes X_2 \otimes Y_1 \otimes Y_2 \]
which satisfies \eqref{conditionMonStructTensorProduct} and \eqref{normalisationMonStructMult}. Similarly, with this point of view, a multiplication is just a monoidal structure of $\otimes$ which satisfies \eqref{conditionUnitMult}.

\begin{definition}\label{defUMonOtimes}
We denote by $\mathrm{Mon}(\otimes)$ the set of monoidal structures on $\otimes$. We say that an element in $\mathrm{Mon}(\otimes)$ is unital if it satisfies the unitality condition \eqref{conditionUnitMult} and denote by $\mathrm{UMon}(\otimes)$ the subset of such elements.
\end{definition}

\indent Let us now see more precisely how $\mathrm{UMon}(\otimes)$ fits into $\mathrm{Mon}(\otimes)$. We write $\mathrm{Mon}(\mathrm{Id}_{\mathcal{C}})$ for the {\em group} of monoidal structures on the identity functor; its elements are natural isomorphisms $\alpha_{X_1,X_2} : X_1 \otimes X_2 \to X_1 \otimes X_2$ such that $\alpha_{X_1 \otimes X_2,X_3} \circ (\alpha_{X_1,X_2} \otimes \mathrm{id}_{X_3}) = \alpha_{X_1,X_2 \otimes X_3} \circ (\mathrm{id}_{X_1} \otimes \alpha_{X_2,X_3})$.
\begin{lemma}\label{lemmaInteractionsMonIdWithMonOtimes}
1. For $\phi \in \mathrm{Mon}(\otimes)$ let
\begin{equation}\label{defp1p2}
p_1(\phi)_{X_1,X_2} = \phi_{(X_1,\boldsymbol{1}),(X_2,\boldsymbol{1})}, \qquad p_2(\phi)_{Y_1,Y_2} = \phi_{(\boldsymbol{1},Y_1),(\boldsymbol{1},Y_2)}.
\end{equation}
Then $p_1(\phi) \in \mathrm{Mon}(\mathrm{Id}_{\mathcal{C}})$ and $p_2(\phi) \in \mathrm{Mon}(\mathrm{Id}_{\mathcal{C}})$.
\\2. For $\alpha, \beta \in \mathrm{Mon}(\mathrm{Id}_{\mathcal{C}})$ and $\phi \in \mathrm{Mon}(\otimes)$ let $(\alpha \otimes \beta) \circ \phi$ be the natural transformations whose components $\bigl( (\alpha \otimes \beta) \circ \phi \bigr)_{(X_1,Y_1),(X_2,Y_2)}$ are given by the following composition:
\[ X_1 \otimes Y_1 \otimes X_2 \otimes Y_2 \xrightarrow{\phi_{(X_1,Y_1),(X_2,Y_2)}} X_1 \otimes X_2 \otimes Y_1 \otimes Y_2 \xrightarrow{\alpha_{X_1,X_2} \otimes \beta_{Y_1,Y_2}} X_1 \otimes X_2 \otimes Y_1 \otimes Y_2. \]
Then $(\alpha \otimes \beta) \circ \phi \in \mathrm{Mon}(\otimes)$.
\end{lemma}
\begin{proof}
Straightforward computations.
\end{proof}
\noindent With the notations of Lemma \ref{lemmaInteractionsMonIdWithMonOtimes}, $\mathrm{UMon}(\otimes) = \bigl\{ \phi \in \mathrm{Mon}(\otimes) \, \big| \, p_1(\phi) = p_2(\phi) = \mathrm{id} \bigr\}$. Also note that $p_1\bigl( (\alpha \otimes \beta) \circ \phi \bigr) = \alpha \circ p_1(\phi)$ and $p_2\bigl( (\alpha \otimes \beta) \circ \phi \bigr) = \beta \circ p_2(\phi)$. As a result there is a map
\begin{equation}\label{projMonOnUMon}
\mathcal{P} : \mathrm{Mon}(\otimes) \to \mathrm{UMon}(\otimes), \qquad \phi \mapsto \bigl( p_1(\phi)^{-1} \otimes p_2(\phi)^{-1} \bigr) \circ \phi
\end{equation}
which satisfies $\mathcal{P} \circ \mathcal{P} = \mathcal{P}$ and $\mathrm{UMon}(\otimes) = \bigl\{ \phi \in \mathrm{Mon}(\otimes) \, \big| \, \mathcal{P}(\phi) = \phi \bigr\}$. Moreover there is a bijection
\begin{equation}\label{decompositionMonOtimes}
\fonc{\mathcal{F}}{\mathrm{Mon}(\otimes)}{\mathrm{Mon}(\mathrm{Id}_{\mathcal{C}})^{\times 2} \times \mathrm{UMon}(\otimes)}{\phi}{\bigl( p_1(\phi), p_2(\phi), \mathcal{P}(\phi)  \bigr)}
\end{equation}
whose inverse is
\begin{equation}\label{decompositionMonOtimesInverse}
\mathcal{F}^{-1}(\alpha,\beta,\psi) = (\alpha \otimes \beta) \circ \psi.
\end{equation}

\smallskip

\indent Recall from \S\ref{subsectionDefMonStruct} the equivalence relation $\sim$ on $\mathrm{Mon}(F)$ and the description of equivalence classes of $\sim$. In particular for $F = \otimes$ and $F = \mathrm{Id}_{\mathcal{C}}$ we have the quotient spaces $\mathrm{Mon}(\otimes)/\!\!\sim$ and $\mathrm{Mon}(\mathrm{Id}_{\mathcal{C}})/\!\!\sim$. Let us say that a natural isomorphism $u : \otimes \Rightarrow \otimes$ is {\em unital} if
\begin{equation*}
\forall \, X \in \mathcal{C}, \quad u_{(X,\boldsymbol{1})} = u_{(\boldsymbol{1},X)} = \mathrm{id}_X.
\end{equation*}
For $\phi,\phi' \in \mathrm{UMon}(\otimes)$ we declare that $\phi \sim_{\mathrm{U}} \phi'$ if there exists a {\em unital} monoidal natural isomorphism $u : (\otimes,\phi) \Rightarrow (\otimes,\phi')$. This defines an equivalence relation on $\mathrm{UMon}(\otimes)$.

\begin{lemma}\label{lemmaDecompositionMon}
The map $\mathcal{F}$ defined in \eqref{decompositionMonOtimes} descends to a bijection
\[ \overline{\mathcal{F}} : \mathrm{Mon}(\otimes)/{\sim} \: \to \bigl( \mathrm{Mon}(\mathrm{Id}_{\mathcal{C}})/{\sim} \bigr)^{\times 2} \times \mathrm{UMon}(\otimes)/{\sim}_{\mathrm{U}}. \]
\end{lemma}
\begin{proof}
Let $\phi, \omega \in \mathrm{Mon}(\otimes)$ be such that $\phi \sim \omega$. Then there is a monoidal natural isomorphism $m : (\otimes,\phi) \Rightarrow (\otimes,\omega)$. Define natural isomorphisms
\begin{equation}\label{degeneraciesForAutOtimes}
p'_1(m), p'_2(m) : \mathrm{Id}_{\mathcal{C}} \Rightarrow \mathrm{Id}_{\mathcal{C}}, \qquad p'_1(m)_X = m_{(X,\boldsymbol{1})}, \quad p'_2(m)_X = m_{(\boldsymbol{1},X)}.
\end{equation}
It is easily seen that they are monoidal natural isomorphisms $p'_i(m) : \bigl(\mathrm{Id}_{\mathcal{C}}, p_i(\phi) \bigr) \Rightarrow \bigl( \mathrm{Id}_{\mathcal{C}}, p_i(\omega) \bigr)$ and hence $p_i(\phi) \sim p_i(\omega)$ for $i=1,2$. Next define a natural isomorphism
\begin{equation}\label{unitalProjectorForAutOtimes}
\mathcal{P}'(m) : \otimes \Rightarrow \otimes, \qquad \mathcal{P}'(m)_{(X,Y)} = \bigl( p'_1(m)^{-1}_X \otimes p'_2(m)^{-1}_Y \bigr) \circ m_{(X,Y)}.
\end{equation}
A straightforward computation reveals that it is a monoidal natural isomorphism $\mathcal{P}'(m) : \bigl( \otimes, \mathcal{P}(\phi) \bigr) \Rightarrow \bigl( \otimes, \mathcal{P}(\omega) \bigr)$. Moreover it is clearly unital and thus $\mathcal{P}(\phi) \sim_{\mathrm{U}} \mathcal{P}(\omega)$. This shows that the map $\overline{\mathcal{F}} : [\phi] \mapsto \bigl( [p_1(\phi)], [p_2(\phi)], [\mathcal{P}(\phi)]_{\mathrm{U}} \bigr)$ is well-defined.
\\Conversely, let $(\alpha,\beta,\phi), (\alpha',\beta',\phi') \in \mathrm{Mon}(\mathrm{Id}_{\mathcal{C}})^{\times 2} \times \mathrm{UMon}(\otimes)$ and assume that $\alpha \sim \alpha'$, $\beta \sim \beta'$, $\phi \sim_{\mathrm{U}} \phi'$. Then there exist monoidal natural isomorphisms $r : (\mathrm{Id}_{\mathcal{C}},\alpha) \Rightarrow (\mathrm{Id}_{\mathcal{C}}, \alpha')$, $s : (\mathrm{Id}_{\mathcal{C}},\beta) \Rightarrow (\mathrm{Id}_{\mathcal{C}}, \beta')$ and a unital monoidal natural isomorphism $u : (\otimes,\phi) \Rightarrow (\otimes,\phi')$. Consider the natural isomorphism
\[ (r \otimes s) \circ u : \otimes \Rightarrow \otimes, \qquad \bigl[ (r \otimes s) \circ u \bigr]_{(X,Y)} = (r_X \otimes s_Y) \circ u_{(X,Y)}. \]
A straightforward computation reveals that $(r \otimes s) \circ u$ is a monoidal natural isomorphism $\bigl(\otimes, (\alpha \otimes \beta) \circ \phi \bigr) \Rightarrow \bigl(\otimes, (\alpha' \otimes \beta') \circ \phi' \bigr)$, which means that $\mathcal{F}^{-1}(\alpha,\beta,\phi) \sim \mathcal{F}^{-1}(\alpha',\beta',\phi')$. As a result the map $\overline{\mathcal{F}^{-1}} : \bigl( [\alpha], [\beta], [\phi]_{\mathrm{U}} \bigr) \mapsto \bigl[ \mathcal{F}^{-1}(\alpha,\beta,\phi) \bigr]$ is well-defined and is inverse to $\overline{\mathcal{F}}$.
\end{proof}

\indent For $\phi \in \mathrm{Mon}(\otimes)$ let
\begin{equation}\label{braidingAssociatedToLaxMult}
B(\phi)_{X,Y} : X \otimes Y \xrightarrow{\phi_{(\boldsymbol{1},X),(Y,\boldsymbol{1})}} Y \otimes X \xrightarrow{\phi^{-1}_{(Y,\boldsymbol{1}),(\boldsymbol{1},X)}} Y \otimes X
\end{equation}
for all $X,Y \in \mathcal{C}$. This defines a natural isomorphism $B(\phi)$.

\begin{proposition}\label{factoLaxMult}{\em \cite[\S 5]{JS}}
\\1. For all $\phi \in \mathrm{Mon}(\otimes)$, $B(\phi)$ is a braiding on $\mathcal{C}$. Moreover if $\phi \sim \phi'$ then $B(\phi) = B(\phi')$.
\\2. Conversely let $c$ be a braiding on $\mathcal{C}$ and $\mathrm{id} \otimes c \otimes \mathrm{id}$ be the natural isomorphism with components
\begin{equation}\label{MonStructFromBraiding}
(\mathrm{id} \otimes c \otimes \mathrm{id})_{(X_1,Y_1),(X_2,Y_2)} = \mathrm{id}_{X_1} \otimes c_{Y_1,X_2} \otimes \mathrm{id}_{Y_2} : X_1 \otimes Y_1 \otimes X_2 \otimes Y_2 \to X_1 \otimes X_2 \otimes Y_1 \otimes Y_2.
\end{equation}
Then $\mathrm{id} \otimes c \otimes \mathrm{id} \in \mathrm{UMon}(\otimes)$.
\\3. For all $\phi \in \mathrm{UMon}(\otimes)$ we have $\phi \sim_{\mathrm{U}} \mathrm{id} \otimes B(\phi) \otimes \mathrm{id}$.
\\4. As a result there is a bijection
\[ \mathrm{UMon}(\otimes)/{\sim}_{\mathrm{U}} \to \mathrm{Br}(\mathcal{C}), \quad [\phi]_{\mathrm{U}} \mapsto B(\phi). \]
(where $[\:\:]_{\mathrm{U}}$ means the equivalence class with respect to $\sim_{\mathrm{U}}$) whose inverse is
\[ \mathrm{Br}(\mathcal{C}) \to \mathrm{UMon}(\otimes)/{\sim}_{\mathrm{U}}, \quad c \mapsto [\mathrm{id} \otimes c \otimes \mathrm{id}]_{\mathrm{U}} \]
\end{proposition}
\begin{proof}
The details of the computations are not given in \cite{JS}. We thus provide a detailed proof for the convenience of the reader.
\\1. Let us check the first axiom of a braiding recalled in \eqref{axiomsBraiding}. We use the usual diagrammatic calculus for strict monoidal categories, reading diagrams from bottom to top:
\begin{center}
\begingroup%
  \makeatletter%
  \providecommand\color[2][]{%
    \errmessage{(Inkscape) Color is used for the text in Inkscape, but the package 'color.sty' is not loaded}%
    \renewcommand\color[2][]{}%
  }%
  \providecommand\transparent[1]{%
    \errmessage{(Inkscape) Transparency is used (non-zero) for the text in Inkscape, but the package 'transparent.sty' is not loaded}%
    \renewcommand\transparent[1]{}%
  }%
  \providecommand\rotatebox[2]{#2}%
  \newcommand*\fsize{\dimexpr\f@size pt\relax}%
  \newcommand*\lineheight[1]{\fontsize{\fsize}{#1\fsize}\selectfont}%
  \ifx\svgwidth\undefined%
    \setlength{\unitlength}{469.51589477bp}%
    \ifx\svgscale\undefined%
      \relax%
    \else%
      \setlength{\unitlength}{\unitlength * \real{\svgscale}}%
    \fi%
  \else%
    \setlength{\unitlength}{\svgwidth}%
  \fi%
  \global\let\svgwidth\undefined%
  \global\let\svgscale\undefined%
  \makeatother%
  \begin{picture}(1,0.31242827)%
    \lineheight{1}%
    \setlength\tabcolsep{0pt}%
    \put(0,0){\includegraphics[width=\unitlength,page=1]{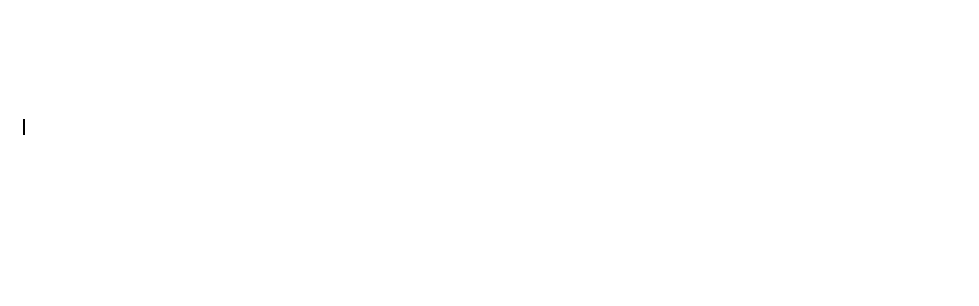}}%
    \put(0.01926555,0.19649477){\color[rgb]{0,0,0}\makebox(0,0)[lt]{\lineheight{1.25}\smash{\begin{tabular}[t]{l}$_Z$\end{tabular}}}}%
    \put(0,0){\includegraphics[width=\unitlength,page=2]{proofBraidingFromMonStruct.pdf}}%
    \put(0.00583141,0.15003467){\color[rgb]{0,0,0}\makebox(0,0)[lt]{\lineheight{1.25}\smash{\begin{tabular}[t]{l}$B(\phi)_{X \otimes Y,Z}$\end{tabular}}}}%
    \put(0,0){\includegraphics[width=\unitlength,page=3]{proofBraidingFromMonStruct.pdf}}%
    \put(0.01542379,0.108193){\color[rgb]{0,0,0}\makebox(0,0)[lt]{\lineheight{1.25}\smash{\begin{tabular}[t]{l}$_X$\end{tabular}}}}%
    \put(0,0){\includegraphics[width=\unitlength,page=4]{proofBraidingFromMonStruct.pdf}}%
    \put(0.05535854,0.10819301){\color[rgb]{0,0,0}\makebox(0,0)[lt]{\lineheight{1.25}\smash{\begin{tabular}[t]{l}$_Y$\end{tabular}}}}%
    \put(0,0){\includegraphics[width=\unitlength,page=5]{proofBraidingFromMonStruct.pdf}}%
    \put(0.0952933,0.10819301){\color[rgb]{0,0,0}\makebox(0,0)[lt]{\lineheight{1.25}\smash{\begin{tabular}[t]{l}$_Z$\end{tabular}}}}%
    \put(0,0){\includegraphics[width=\unitlength,page=6]{proofBraidingFromMonStruct.pdf}}%
    \put(0.0592003,0.19649477){\color[rgb]{0,0,0}\makebox(0,0)[lt]{\lineheight{1.25}\smash{\begin{tabular}[t]{l}$_X$\end{tabular}}}}%
    \put(0,0){\includegraphics[width=\unitlength,page=7]{proofBraidingFromMonStruct.pdf}}%
    \put(0.09913505,0.19649477){\color[rgb]{0,0,0}\makebox(0,0)[lt]{\lineheight{1.25}\smash{\begin{tabular}[t]{l}$_Y$\end{tabular}}}}%
    \put(0,0){\includegraphics[width=\unitlength,page=8]{proofBraidingFromMonStruct.pdf}}%
    \put(0.19497845,0.26837732){\color[rgb]{0,0,0}\makebox(0,0)[lt]{\lineheight{1.25}\smash{\begin{tabular}[t]{l}$_Z$\end{tabular}}}}%
    \put(0,0){\includegraphics[width=\unitlength,page=9]{proofBraidingFromMonStruct.pdf}}%
    \put(0.2001823,0.15321292){\color[rgb]{0,0,0}\makebox(0,0)[lt]{\lineheight{1.25}\smash{\begin{tabular}[t]{l}$B(\phi)_{X \otimes Y,Z}$\end{tabular}}}}%
    \put(0,0){\includegraphics[width=\unitlength,page=10]{proofBraidingFromMonStruct.pdf}}%
    \put(0.19912363,0.0458948){\color[rgb]{0,0,0}\makebox(0,0)[lt]{\lineheight{1.25}\smash{\begin{tabular}[t]{l}$_X$\end{tabular}}}}%
    \put(0,0){\includegraphics[width=\unitlength,page=11]{proofBraidingFromMonStruct.pdf}}%
    \put(0.31094099,0.04589478){\color[rgb]{0,0,0}\makebox(0,0)[lt]{\lineheight{1.25}\smash{\begin{tabular}[t]{l}$_Z$\end{tabular}}}}%
    \put(0,0){\includegraphics[width=\unitlength,page=12]{proofBraidingFromMonStruct.pdf}}%
    \put(0.18101371,0.087746){\color[rgb]{0,0,0}\makebox(0,0)[lt]{\lineheight{1.25}\smash{\begin{tabular}[t]{l}$\phi_{(\boldsymbol{1},X),(\boldsymbol{1},Y)}$\end{tabular}}}}%
    \put(0,0){\includegraphics[width=\unitlength,page=13]{proofBraidingFromMonStruct.pdf}}%
    \put(0.22893541,0.21873197){\color[rgb]{0,0,0}\makebox(0,0)[lt]{\lineheight{1.25}\smash{\begin{tabular}[t]{l}$\phi^{-1}_{(\boldsymbol{1},X),(\boldsymbol{1},Y)}$\end{tabular}}}}%
    \put(0,0){\includegraphics[width=\unitlength,page=14]{proofBraidingFromMonStruct.pdf}}%
    \put(0.25503228,0.0458948){\color[rgb]{0,0,0}\makebox(0,0)[lt]{\lineheight{1.25}\smash{\begin{tabular}[t]{l}$_Y$\end{tabular}}}}%
    \put(0,0){\includegraphics[width=\unitlength,page=15]{proofBraidingFromMonStruct.pdf}}%
    \put(0.25088713,0.26837732){\color[rgb]{0,0,0}\makebox(0,0)[lt]{\lineheight{1.25}\smash{\begin{tabular}[t]{l}$_X$\end{tabular}}}}%
    \put(0,0){\includegraphics[width=\unitlength,page=16]{proofBraidingFromMonStruct.pdf}}%
    \put(0.30679575,0.26837732){\color[rgb]{0,0,0}\makebox(0,0)[lt]{\lineheight{1.25}\smash{\begin{tabular}[t]{l}$_Y$\end{tabular}}}}%
    \put(0,0){\includegraphics[width=\unitlength,page=17]{proofBraidingFromMonStruct.pdf}}%
    \put(0.41062618,0.3019225){\color[rgb]{0,0,0}\makebox(0,0)[lt]{\lineheight{1.25}\smash{\begin{tabular}[t]{l}$_Z$\end{tabular}}}}%
    \put(0,0){\includegraphics[width=\unitlength,page=18]{proofBraidingFromMonStruct.pdf}}%
    \put(0.41477134,0.00915481){\color[rgb]{0,0,0}\makebox(0,0)[lt]{\lineheight{1.25}\smash{\begin{tabular}[t]{l}$_X$\end{tabular}}}}%
    \put(0,0){\includegraphics[width=\unitlength,page=19]{proofBraidingFromMonStruct.pdf}}%
    \put(0.52658867,0.0091548){\color[rgb]{0,0,0}\makebox(0,0)[lt]{\lineheight{1.25}\smash{\begin{tabular}[t]{l}$_Z$\end{tabular}}}}%
    \put(0,0){\includegraphics[width=\unitlength,page=20]{proofBraidingFromMonStruct.pdf}}%
    \put(0.39666136,0.051006){\color[rgb]{0,0,0}\makebox(0,0)[lt]{\lineheight{1.25}\smash{\begin{tabular}[t]{l}$\phi_{(\boldsymbol{1},X),(\boldsymbol{1},Y)}$\end{tabular}}}}%
    \put(0,0){\includegraphics[width=\unitlength,page=21]{proofBraidingFromMonStruct.pdf}}%
    \put(0.44458311,0.25227713){\color[rgb]{0,0,0}\makebox(0,0)[lt]{\lineheight{1.25}\smash{\begin{tabular}[t]{l}$\phi^{-1}_{(\boldsymbol{1},X),(\boldsymbol{1},Y)}$\end{tabular}}}}%
    \put(0,0){\includegraphics[width=\unitlength,page=22]{proofBraidingFromMonStruct.pdf}}%
    \put(0.47067996,0.00915481){\color[rgb]{0,0,0}\makebox(0,0)[lt]{\lineheight{1.25}\smash{\begin{tabular}[t]{l}$_Y$\end{tabular}}}}%
    \put(0,0){\includegraphics[width=\unitlength,page=23]{proofBraidingFromMonStruct.pdf}}%
    \put(0.46653475,0.3019225){\color[rgb]{0,0,0}\makebox(0,0)[lt]{\lineheight{1.25}\smash{\begin{tabular}[t]{l}$_X$\end{tabular}}}}%
    \put(0,0){\includegraphics[width=\unitlength,page=24]{proofBraidingFromMonStruct.pdf}}%
    \put(0.5224054,0.30227876){\color[rgb]{0,0,0}\makebox(0,0)[lt]{\lineheight{1.25}\smash{\begin{tabular}[t]{l}$_Y$\end{tabular}}}}%
    \put(0,0){\includegraphics[width=\unitlength,page=25]{proofBraidingFromMonStruct.pdf}}%
    \put(0.40723896,0.18495108){\color[rgb]{0,0,0}\makebox(0,0)[lt]{\lineheight{1.25}\smash{\begin{tabular}[t]{l}$\phi^{-1}_{(Z,\boldsymbol{1}),(\boldsymbol{1},X \otimes Y)}$\end{tabular}}}}%
    \put(0,0){\includegraphics[width=\unitlength,page=26]{proofBraidingFromMonStruct.pdf}}%
    \put(0.40723891,0.11786072){\color[rgb]{0,0,0}\makebox(0,0)[lt]{\lineheight{1.25}\smash{\begin{tabular}[t]{l}$\phi_{(\boldsymbol{1},X \otimes Y),(Z,\boldsymbol{1})}$\end{tabular}}}}%
    \put(0,0){\includegraphics[width=\unitlength,page=27]{proofBraidingFromMonStruct.pdf}}%
    \put(0.96963321,0.30328262){\color[rgb]{0,0,0}\makebox(0,0)[lt]{\lineheight{1.25}\smash{\begin{tabular}[t]{l}$_Y$\end{tabular}}}}%
    \put(0,0){\includegraphics[width=\unitlength,page=28]{proofBraidingFromMonStruct.pdf}}%
    \put(0.90128916,0.01481775){\color[rgb]{0,0,0}\makebox(0,0)[lt]{\lineheight{1.25}\smash{\begin{tabular}[t]{l}$_Y$\end{tabular}}}}%
    \put(0,0){\includegraphics[width=\unitlength,page=29]{proofBraidingFromMonStruct.pdf}}%
    \put(0.8373935,0.01483302){\color[rgb]{0,0,0}\makebox(0,0)[lt]{\lineheight{1.25}\smash{\begin{tabular}[t]{l}$_X$\end{tabular}}}}%
    \put(0,0){\includegraphics[width=\unitlength,page=30]{proofBraidingFromMonStruct.pdf}}%
    \put(0.88317912,0.05666893){\color[rgb]{0,0,0}\makebox(0,0)[lt]{\lineheight{1.25}\smash{\begin{tabular}[t]{l}$\phi_{(\boldsymbol{1},Y),(Z,\boldsymbol{1})}$\end{tabular}}}}%
    \put(0,0){\includegraphics[width=\unitlength,page=31]{proofBraidingFromMonStruct.pdf}}%
    \put(0.82824333,0.25618706){\color[rgb]{0,0,0}\makebox(0,0)[lt]{\lineheight{1.25}\smash{\begin{tabular}[t]{l}$\phi^{-1}_{(Z,\boldsymbol{1}),(\boldsymbol{1},X)}$\end{tabular}}}}%
    \put(0,0){\includegraphics[width=\unitlength,page=32]{proofBraidingFromMonStruct.pdf}}%
    \put(0.95719781,0.01481775){\color[rgb]{0,0,0}\makebox(0,0)[lt]{\lineheight{1.25}\smash{\begin{tabular}[t]{l}$_Z$\end{tabular}}}}%
    \put(0,0){\includegraphics[width=\unitlength,page=33]{proofBraidingFromMonStruct.pdf}}%
    \put(0.8499084,0.3035199){\color[rgb]{0,0,0}\makebox(0,0)[lt]{\lineheight{1.25}\smash{\begin{tabular}[t]{l}$_Z$\end{tabular}}}}%
    \put(0,0){\includegraphics[width=\unitlength,page=34]{proofBraidingFromMonStruct.pdf}}%
    \put(0.90581711,0.3035199){\color[rgb]{0,0,0}\makebox(0,0)[lt]{\lineheight{1.25}\smash{\begin{tabular}[t]{l}$_X$\end{tabular}}}}%
    \put(0,0){\includegraphics[width=\unitlength,page=35]{proofBraidingFromMonStruct.pdf}}%
    \put(0.82795671,0.18997894){\color[rgb]{0,0,0}\makebox(0,0)[lt]{\lineheight{1.25}\smash{\begin{tabular}[t]{l}$\phi_{(\boldsymbol{1},X),(Z,\boldsymbol{1})}$\end{tabular}}}}%
    \put(0,0){\includegraphics[width=\unitlength,page=36]{proofBraidingFromMonStruct.pdf}}%
    \put(0.88386528,0.12288857){\color[rgb]{0,0,0}\makebox(0,0)[lt]{\lineheight{1.25}\smash{\begin{tabular}[t]{l}$\phi^{-1}_{(Z,\boldsymbol{1}),(\boldsymbol{1},Y)}$\end{tabular}}}}%
    \put(0.14208122,0.15101823){\color[rgb]{0,0,0}\makebox(0,0)[lt]{\lineheight{1.25}\smash{\begin{tabular}[t]{l}$=$\end{tabular}}}}%
    \put(0.35844062,0.15620229){\color[rgb]{0,0,0}\makebox(0,0)[lt]{\lineheight{1.25}\smash{\begin{tabular}[t]{l}$=$\end{tabular}}}}%
    \put(0.78967366,0.16380713){\color[rgb]{0,0,0}\makebox(0,0)[lt]{\lineheight{1.25}\smash{\begin{tabular}[t]{l}$=$\end{tabular}}}}%
    \put(0,0){\includegraphics[width=\unitlength,page=37]{proofBraidingFromMonStruct.pdf}}%
    \put(0.746078,0.29872774){\color[rgb]{0,0,0}\makebox(0,0)[lt]{\lineheight{1.25}\smash{\begin{tabular}[t]{l}$_Y$\end{tabular}}}}%
    \put(0,0){\includegraphics[width=\unitlength,page=38]{proofBraidingFromMonStruct.pdf}}%
    \put(0.67834061,0.00276526){\color[rgb]{0,0,0}\makebox(0,0)[lt]{\lineheight{1.25}\smash{\begin{tabular}[t]{l}$_Y$\end{tabular}}}}%
    \put(0,0){\includegraphics[width=\unitlength,page=39]{proofBraidingFromMonStruct.pdf}}%
    \put(0.62243138,0.00363531){\color[rgb]{0,0,0}\makebox(0,0)[lt]{\lineheight{1.25}\smash{\begin{tabular}[t]{l}$_X$\end{tabular}}}}%
    \put(0,0){\includegraphics[width=\unitlength,page=40]{proofBraidingFromMonStruct.pdf}}%
    \put(0.66023061,0.04461644){\color[rgb]{0,0,0}\makebox(0,0)[lt]{\lineheight{1.25}\smash{\begin{tabular}[t]{l}$\phi_{(\boldsymbol{1},Y),(Z,\boldsymbol{1})}$\end{tabular}}}}%
    \put(0,0){\includegraphics[width=\unitlength,page=41]{proofBraidingFromMonStruct.pdf}}%
    \put(0.61230901,0.24908236){\color[rgb]{0,0,0}\makebox(0,0)[lt]{\lineheight{1.25}\smash{\begin{tabular}[t]{l}$\phi^{-1}_{(Z,\boldsymbol{1}),(\boldsymbol{1},X)}$\end{tabular}}}}%
    \put(0,0){\includegraphics[width=\unitlength,page=42]{proofBraidingFromMonStruct.pdf}}%
    \put(0.73424926,0.00276526){\color[rgb]{0,0,0}\makebox(0,0)[lt]{\lineheight{1.25}\smash{\begin{tabular}[t]{l}$_Z$\end{tabular}}}}%
    \put(0,0){\includegraphics[width=\unitlength,page=43]{proofBraidingFromMonStruct.pdf}}%
    \put(0.62640087,0.29849497){\color[rgb]{0,0,0}\makebox(0,0)[lt]{\lineheight{1.25}\smash{\begin{tabular}[t]{l}$_Z$\end{tabular}}}}%
    \put(0,0){\includegraphics[width=\unitlength,page=44]{proofBraidingFromMonStruct.pdf}}%
    \put(0.69900779,0.2985108){\color[rgb]{0,0,0}\makebox(0,0)[lt]{\lineheight{1.25}\smash{\begin{tabular}[t]{l}$_X$\end{tabular}}}}%
    \put(0,0){\includegraphics[width=\unitlength,page=45]{proofBraidingFromMonStruct.pdf}}%
    \put(0.62288655,0.18175632){\color[rgb]{0,0,0}\makebox(0,0)[lt]{\lineheight{1.25}\smash{\begin{tabular}[t]{l}$\phi^{-1}_{(Z,X),(\boldsymbol{1},Y)}$\end{tabular}}}}%
    \put(0,0){\includegraphics[width=\unitlength,page=46]{proofBraidingFromMonStruct.pdf}}%
    \put(0.6228865,0.11306854){\color[rgb]{0,0,0}\makebox(0,0)[lt]{\lineheight{1.25}\smash{\begin{tabular}[t]{l}$\phi_{(\boldsymbol{1},X),(Z,Y)}$\end{tabular}}}}%
    \put(0.57366236,0.16800185){\color[rgb]{0,0,0}\makebox(0,0)[lt]{\lineheight{1.25}\smash{\begin{tabular}[t]{l}$=$\end{tabular}}}}%
    \put(0,0){\includegraphics[width=\unitlength,page=47]{proofBraidingFromMonStruct.pdf}}%
    \put(0.21422077,0.12600812){\color[rgb]{0,0,0}\makebox(0,0)[lt]{\lineheight{1.25}\smash{\begin{tabular}[t]{l}$_X$\end{tabular}}}}%
    \put(0,0){\includegraphics[width=\unitlength,page=48]{proofBraidingFromMonStruct.pdf}}%
    \put(0.27012939,0.12600812){\color[rgb]{0,0,0}\makebox(0,0)[lt]{\lineheight{1.25}\smash{\begin{tabular}[t]{l}$_Y$\end{tabular}}}}%
    \put(0,0){\includegraphics[width=\unitlength,page=49]{proofBraidingFromMonStruct.pdf}}%
    \put(0.26214247,0.18990371){\color[rgb]{0,0,0}\makebox(0,0)[lt]{\lineheight{1.25}\smash{\begin{tabular}[t]{l}$_X$\end{tabular}}}}%
    \put(0,0){\includegraphics[width=\unitlength,page=50]{proofBraidingFromMonStruct.pdf}}%
    \put(0.31805112,0.18990371){\color[rgb]{0,0,0}\makebox(0,0)[lt]{\lineheight{1.25}\smash{\begin{tabular}[t]{l}$_Y$\end{tabular}}}}%
    \put(0,0){\includegraphics[width=\unitlength,page=51]{proofBraidingFromMonStruct.pdf}}%
    \put(0.42229875,0.08875543){\color[rgb]{0,0,0}\makebox(0,0)[lt]{\lineheight{1.25}\smash{\begin{tabular}[t]{l}$_X$\end{tabular}}}}%
    \put(0,0){\includegraphics[width=\unitlength,page=52]{proofBraidingFromMonStruct.pdf}}%
    \put(0.49387326,0.08932703){\color[rgb]{0,0,0}\makebox(0,0)[lt]{\lineheight{1.25}\smash{\begin{tabular}[t]{l}$_Y$\end{tabular}}}}%
    \put(0,0){\includegraphics[width=\unitlength,page=53]{proofBraidingFromMonStruct.pdf}}%
    \put(0.47788039,0.22334764){\color[rgb]{0,0,0}\makebox(0,0)[lt]{\lineheight{1.25}\smash{\begin{tabular}[t]{l}$_X$\end{tabular}}}}%
    \put(0,0){\includegraphics[width=\unitlength,page=54]{proofBraidingFromMonStruct.pdf}}%
    \put(0.53378904,0.22334764){\color[rgb]{0,0,0}\makebox(0,0)[lt]{\lineheight{1.25}\smash{\begin{tabular}[t]{l}$_Y$\end{tabular}}}}%
    \put(0,0){\includegraphics[width=\unitlength,page=55]{proofBraidingFromMonStruct.pdf}}%
    \put(0.68545089,0.08390272){\color[rgb]{0,0,0}\makebox(0,0)[lt]{\lineheight{1.25}\smash{\begin{tabular}[t]{l}$_Z$\end{tabular}}}}%
    \put(0,0){\includegraphics[width=\unitlength,page=56]{proofBraidingFromMonStruct.pdf}}%
    \put(0.74934649,0.08390272){\color[rgb]{0,0,0}\makebox(0,0)[lt]{\lineheight{1.25}\smash{\begin{tabular}[t]{l}$_Y$\end{tabular}}}}%
    \put(0,0){\includegraphics[width=\unitlength,page=57]{proofBraidingFromMonStruct.pdf}}%
    \put(0.92467074,0.16061322){\color[rgb]{0,0,0}\makebox(0,0)[lt]{\lineheight{1.25}\smash{\begin{tabular}[t]{l}$_Z$\end{tabular}}}}%
    \put(0,0){\includegraphics[width=\unitlength,page=58]{proofBraidingFromMonStruct.pdf}}%
    \put(0.42188148,0.15635851){\color[rgb]{0,0,0}\makebox(0,0)[lt]{\lineheight{1.25}\smash{\begin{tabular}[t]{l}$_Z$\end{tabular}}}}%
    \put(0,0){\includegraphics[width=\unitlength,page=59]{proofBraidingFromMonStruct.pdf}}%
    \put(0.48184753,0.15605925){\color[rgb]{0,0,0}\makebox(0,0)[lt]{\lineheight{1.25}\smash{\begin{tabular}[t]{l}$_X$\end{tabular}}}}%
    \put(0,0){\includegraphics[width=\unitlength,page=60]{proofBraidingFromMonStruct.pdf}}%
    \put(0.54168573,0.15635851){\color[rgb]{0,0,0}\makebox(0,0)[lt]{\lineheight{1.25}\smash{\begin{tabular}[t]{l}$_Y$\end{tabular}}}}%
    \put(0,0){\includegraphics[width=\unitlength,page=61]{proofBraidingFromMonStruct.pdf}}%
    \put(0.63752913,0.15156632){\color[rgb]{0,0,0}\makebox(0,0)[lt]{\lineheight{1.25}\smash{\begin{tabular}[t]{l}$_Z$\end{tabular}}}}%
    \put(0,0){\includegraphics[width=\unitlength,page=62]{proofBraidingFromMonStruct.pdf}}%
    \put(0.69715595,0.15119581){\color[rgb]{0,0,0}\makebox(0,0)[lt]{\lineheight{1.25}\smash{\begin{tabular}[t]{l}$_X$\end{tabular}}}}%
    \put(0,0){\includegraphics[width=\unitlength,page=63]{proofBraidingFromMonStruct.pdf}}%
    \put(0.75733338,0.15156634){\color[rgb]{0,0,0}\makebox(0,0)[lt]{\lineheight{1.25}\smash{\begin{tabular}[t]{l}$_Y$\end{tabular}}}}%
    \put(0,0){\includegraphics[width=\unitlength,page=64]{proofBraidingFromMonStruct.pdf}}%
    \put(0.86116379,0.22824107){\color[rgb]{0,0,0}\makebox(0,0)[lt]{\lineheight{1.25}\smash{\begin{tabular}[t]{l}$_Z$\end{tabular}}}}%
    \put(0,0){\includegraphics[width=\unitlength,page=65]{proofBraidingFromMonStruct.pdf}}%
    \put(0.91707239,0.22824105){\color[rgb]{0,0,0}\makebox(0,0)[lt]{\lineheight{1.25}\smash{\begin{tabular}[t]{l}$_X$\end{tabular}}}}%
    \put(0,0){\includegraphics[width=\unitlength,page=66]{proofBraidingFromMonStruct.pdf}}%
    \put(0.91638616,0.09493107){\color[rgb]{0,0,0}\makebox(0,0)[lt]{\lineheight{1.25}\smash{\begin{tabular}[t]{l}$_Z$\end{tabular}}}}%
    \put(0,0){\includegraphics[width=\unitlength,page=67]{proofBraidingFromMonStruct.pdf}}%
    \put(0.97229475,0.09493107){\color[rgb]{0,0,0}\makebox(0,0)[lt]{\lineheight{1.25}\smash{\begin{tabular}[t]{l}$_Y$\end{tabular}}}}%
    \put(0,0){\includegraphics[width=\unitlength,page=68]{proofBraidingFromMonStruct.pdf}}%
    \put(0.70941168,0.22025411){\color[rgb]{0,0,0}\makebox(0,0)[lt]{\lineheight{1.25}\smash{\begin{tabular}[t]{l}$_X$\end{tabular}}}}%
    \put(0,0){\includegraphics[width=\unitlength,page=69]{proofBraidingFromMonStruct.pdf}}%
    \put(0.63752913,0.22025411){\color[rgb]{0,0,0}\makebox(0,0)[lt]{\lineheight{1.25}\smash{\begin{tabular}[t]{l}$_Z$\end{tabular}}}}%
  \end{picture}%
\endgroup%

\end{center}
The first equality is by naturality of $B(\phi)$, the second equality is the definition of $B(\phi)$, the third equality uses \eqref{conditionMonStructTensorProduct} two times (once for the two lower coupons and once for the two upper coupons) and the fourth equality also uses \eqref{conditionMonStructTensorProduct} with $(X_1,Y_1) = (\boldsymbol{1},X)$, $(X_2,Y_2) = (Z,\boldsymbol{1})$, $(X_3,Y_3) = (\boldsymbol{1},Y)$ applied to the two middle coupons. Finally the last term is equal to $(B(\phi)_{X,Z} \otimes \mathrm{id}_Y) \circ (\mathrm{id}_X \otimes B(\phi)_{Y,Z})$ as desired. The other axiom of a braiding is shown similarly. The second claim is obtained by a straightforward computation.
\\2. It is an easy computation to check that $\mathrm{id} \otimes c \otimes \mathrm{id} \in \mathrm{Mon}(\otimes)$ using \eqref{axiomsBraiding}. Moreover the axioms of a braiding imply $c_{\boldsymbol{1},X} = c_{X,\boldsymbol{1}} = \mathrm{id}_X$, hence $\mathrm{id} \otimes c \otimes \mathrm{id} \in \mathrm{UMon}(\otimes)$.
\\3. Define $u_{(X,Y)} = \phi_{(X,\boldsymbol{1}),(\boldsymbol{1},Y)} : X \otimes Y \to X \otimes Y$. Then $u \in \mathrm{UAut}(\otimes)$. Note from \eqref{conditionMonStructTensorProduct} and the unitality condition \eqref{conditionUnitMult} that
\begin{align}
u_{(X,Y_1 \otimes Y_2)} &= \phi_{(X,Y_1),(\boldsymbol{1},Y_2)} \circ \bigl( u_{(X,Y_1)} \otimes \mathrm{id}_{Y_2} \bigr),\label{firstPropertyMonAutu}\\
u_{(X_1 \otimes X_2,Y)} &= \phi_{(X_1,\boldsymbol{1}),(X_2,Y)} \circ \bigl( \mathrm{id}_{X_1} \otimes u_{(X_2,Y)} \bigr).\label{secondPropertyMonAutu}
\end{align}
We have
\begin{align*}
&u_{(X_1 \otimes X_2,Y_1 \otimes Y_2)} \circ \bigl( \mathrm{id}_{X_1} \otimes B(\phi)_{Y_1,X_2} \otimes \mathrm{id}_{Y_2} \bigr)\\
\overset{\eqref{firstPropertyMonAutu},\eqref{braidingAssociatedToLaxMult}}=\:&\phi_{(X_1 \otimes X_2,Y_1),(\boldsymbol{1},Y_2)} \circ \bigl( u_{(X_1 \otimes X_2,Y_1)} \otimes \mathrm{id}_{Y_2} \bigr) \circ \bigl( \mathrm{id}_{X_1} \otimes u_{(X_2,Y_1)}^{-1} \otimes \mathrm{id}_{Y_2} \bigr)\\
&\qquad\qquad\qquad\qquad\qquad\qquad\qquad\qquad\:\circ \bigl( \mathrm{id}_{X_1} \otimes \phi_{(\boldsymbol{1},Y_1),(X_2,\boldsymbol{1})} \otimes \mathrm{id}_{Y_2} \bigr)\\
\overset{\eqref{secondPropertyMonAutu}}{=}\:& \phi_{(X_1 \otimes X_2,Y_1),(\boldsymbol{1},Y_2)} \circ \bigl( \phi_{(X_1,\boldsymbol{1}),(X_2,Y_1)} \otimes \mathrm{id}_{Y_2} \bigr) \circ \bigl( \mathrm{id}_{X_1} \otimes \phi_{(\boldsymbol{1},Y_1),(X_2,\boldsymbol{1})} \otimes \mathrm{id}_{Y_2} \bigr)\\
\overset{\eqref{conditionMonStructTensorProduct}}{=}\:&\phi_{(X_1 \otimes X_2,Y_1),(\boldsymbol{1},Y_2)} \circ \bigl( \phi_{(X_1,Y_1),(X_2,\boldsymbol{1})} \otimes \mathrm{id}_{Y_2} \bigr) \circ \bigl( u_{(X_1,Y_1)} \otimes \mathrm{id}_{X_2 \otimes Y_2} \bigr)\\
\overset{\eqref{conditionMonStructTensorProduct}}{=}\:&\phi_{(X_1,Y_1),(X_2,Y_2)} \circ \bigl( u_{(X_1,Y_1)} \otimes u_{(X_2,Y_2)} \bigr)
\end{align*}
Hence $u$ is a monoidal isomorphism $(\otimes, \mathrm{id} \otimes B(\phi) \otimes \mathrm{id}) \Rightarrow (\otimes,\phi)$.
\\4. It is readily seen that the composition $\mathrm{Br}(\mathcal{C}) \to \mathrm{UMon}(\otimes)/{\sim}_{\mathrm{U}} \to \mathrm{Br}(\mathcal{C})$ is the identity. Conversely the composition $\mathrm{UMon}(\otimes)/{\sim}_{\mathrm{U}} \to \mathrm{Br}(\mathcal{C}) \to \mathrm{UMon}(\otimes)/{\sim}_{\mathrm{U}}$ is given by $[\phi]_{\mathrm{U}} \mapsto [\mathrm{id} \otimes B(\phi) \otimes \mathrm{id}]_{\mathrm{U}}$, which is the identity by the previous item.
\end{proof}
Note that in item 1 of Proposition \ref{factoLaxMult} we simply assumed $\phi \in \mathrm{Mon}(\otimes)$, \textit{i.e.} $\phi$ is not assumed to be unital. Hence the following result makes sense:
\begin{corollary}\label{coroGeneralizedJSBijection}
There is a bijection $\mathrm{Mon}(\otimes)/{\sim} \to \bigl(\mathrm{Mon}(\mathrm{Id}_{\mathcal{C}})/{\sim}\bigr)^{\times 2} \times \mathrm{Br}(\mathcal{C})$ given by
\[ [\phi] \mapsto \bigl( \,[p_1(\phi)], \,[p_2(\phi)], \,B(\phi) \,\bigr) \]
where $p_1$, $p_2$ are defined in Lemma \ref{lemmaInteractionsMonIdWithMonOtimes}. Its inverse is $\bigl( \, [\alpha],\,[\beta],\ c\, \bigr) \mapsto \bigl[ (\alpha \otimes \beta) \circ (\mathrm{id} \otimes c \otimes \mathrm{id}) \bigr]$.
\end{corollary}
\begin{proof}
Combining Lemma \ref{lemmaDecompositionMon} with item 4 in Proposition \ref{factoLaxMult}, we know that the map $[\phi] \mapsto \bigl( \,[p_1(\phi)], \,[p_2(\phi)], \,B(\mathcal{P}(\phi)) \,\bigr)$ is a bijection $\mathrm{Mon}(\otimes)/{\sim} \:\: \to \bigl(\mathrm{Mon}(\mathrm{Id}_{\mathcal{C}})/{\sim}\bigr)^{\times 2} \times \mathrm{Br}(\mathcal{C})$. But it is easily seen that $B(\mathcal{P}(\phi)) = B(\phi)$.
\end{proof}
\noindent In particular any $\phi \in \mathrm{Mon}(\otimes)$ is equivalent to $(p_1(\phi) \otimes p_2(\phi)) \circ (\mathrm{id} \otimes B(\phi) \otimes \mathrm{id})$, which is a sort of canonical decomposition.

\subsection{DY cohomology of multiplications and tangent braidings}\label{subsectionDYcohomologyTangentBraidings}
The goal of this subsection is to establish the  infinitesimal counterpart of the results in \S\ref{sectionLaxMult} and to build the bridge to the DY cohomology of the tensor product functor~$\otimes$ equipped with a monoidal structure. We assume that the monoidal category $\mathcal{C}$ is $\Bbbk$-linear (where $\Bbbk$ is a field) and that its monoidal product $\otimes$ is $\Bbbk$-bilinear on morphisms. We will make heavy use of the notations and definitions from \S\ref{subsectionDefMonStruct} and in particular of the monoidal category $\mathcal{C}_{\epsilon} = \mathcal{C} \otimes_{\Bbbk} \Bbbk[\epsilon]/(\epsilon^2)$ defined in \eqref{categoryExtendedDualNumbers}--\eqref{compositionExtendedDualNumbers}. Recall that $\otimes_{\epsilon}$ denotes the monoidal product of $\mathcal{C}_{\epsilon}$. Also note that $a + \epsilon b$ is an isomorphism in $\mathcal{C}_{\epsilon}$ if and only if $a$ is an isomorphism in $\mathcal{C}$ and then $(a + \epsilon b)^{-1} = a^{-1} - \epsilon (a^{-1} \circ b \circ a^{-1})$.

\smallskip

\indent Recall the sets $\mathrm{Mon}(\otimes)$ and $\mathrm{UMon}(\otimes)$ from Definition \ref{defUMonOtimes}. If $\phi \in \mathrm{Mon}(\otimes)$ and $f$ is a natural transformation with components $f_{(X_1,Y_1),(X_2,Y_2)} : X_1 \otimes Y_1 \otimes X_2 \otimes Y_2 \to X_1 \otimes X_2 \otimes Y_1 \otimes Y_2$ we denote by $\phi + \epsilon f$ the natural isomorphism in $\mathcal{C}_{\epsilon}$ with components $\phi_{(X_1,Y_1),(X_2,Y_2)} + \epsilon f_{(X_1,Y_1),(X_2,Y_2)}$. Then by item 1 in Lemma \ref{lemmaDeformationMultilinearFunctor}
\[ \phi + \epsilon f \in \mathrm{Mon}(\otimes_{\epsilon}) \quad \iff \quad f \in \mathbf{T}_{\phi}\mathrm{Mon}(\otimes). \] 
When $\phi \in \mathrm{UMon}(\otimes) \subset \mathrm{Mon}(\otimes)$ there is a relevant subspace:
\[ \mathbf{T}_{\phi}\mathrm{UMon}(\otimes) = \bigl\{ f \in \mathbf{T}_{\phi}\mathrm{Mon}(\otimes) \,\big|\, \forall \, X,Y \in \mathcal{C}, \:\: f_{(X,\boldsymbol{1}),(Y,\boldsymbol{1})} = f_{(\boldsymbol{1},X),(\boldsymbol{1},Y)} = 0 \bigr\}. \]
It is clear from \eqref{conditionUnitMult} that
\begin{equation*}
\phi + \epsilon f \in \mathrm{UMon}(\otimes_{\epsilon}) \quad \iff \quad f \in \mathbf{T}_{\phi}\mathrm{UMon}(\otimes).
\end{equation*}
\indent Recall from \eqref{defp1p2} the maps $p_1,p_2 : \mathrm{Mon}(\otimes) \to \mathrm{Mon}(\mathrm{Id}_{\mathcal{C}})$. For the category $\mathcal{C}_{\epsilon}$, these are maps $p_1,p_2 : \mathrm{Mon}(\otimes_{\epsilon}) \to \mathrm{Mon}(\mathrm{Id}_{\mathcal{C}_{\epsilon}})$ which we can write under the form $p_i(\phi + \epsilon f) = p_i(\phi) + \epsilon \,\mathbf{T}_{\phi}p_i(f)$ with 
\[ \forall \, X,Y \in \mathcal{C}, \quad \mathbf{T}_{\phi}p_1(f)_{X,Y} = f_{(X,\boldsymbol{1}),(Y,\boldsymbol{1})}, \qquad \mathbf{T}_{\phi}p_2(f)_{X,Y} = f_{(\boldsymbol{1},X),(\boldsymbol{1},Y)} \]
for $i=1,2$. We have $p_i(\phi) + \epsilon \,\mathbf{T}_{\phi}p_i(f) \in \mathrm{Mon}(\mathrm{Id}_{\mathcal{C}_{\epsilon}})$, which is equivalent to $\mathbf{T}_{\phi}p_i(f) \in \mathbf{T}_{p_i(\phi)}\mathrm{Mon}(\mathrm{Id}_{\mathcal{C}})$. In this way we get linear maps $\mathbf{T}_{\phi}p_i : \mathbf{T}_{\phi}\mathrm{Mon}(\otimes) \to \mathbf{T}_{p_i(\phi)}\mathrm{Mon}(\mathrm{Id}_{\mathcal{C}})$ 
which are the infinitesimal versions of $p_1$ and $p_2$ at the point $\phi$. Similarly the projection map $\mathcal{P} : \mathrm{Mon}(\otimes_{\epsilon}) \to \mathrm{UMon}(\otimes_{\epsilon})$ defined as in \eqref{projMonOnUMon} can be written as $\mathcal{P}(\phi + \epsilon f) = \mathcal{P}(\phi) + \epsilon \mathbf{T}_{\phi}\mathcal{P}(f)$ with
\begin{align}
\begin{split}\label{defTphiP}
\mathbf{T}_{\phi}\mathcal{P}(f) = \bigl( p_1(\phi)^{-1} \otimes p_2(\phi)^{-1} \bigr) \circ f &- \bigl[ \bigl( p_1(\phi)^{-1} \circ \mathbf{T}_{\phi}p_1(f) \circ p_1(\phi)^{-1} \bigr) \otimes p_2(\phi)^{-1} \bigr] \circ \phi\\
&- \bigl[ p_1(\phi)^{-1} \otimes \bigl( p_2(\phi)^{-1} \circ \mathbf{T}_{\phi}p_2(f) \circ p_2(\phi)^{-1} \bigr) \bigr] \circ \phi
\end{split}
\end{align}
and we get a linear projection $\mathbf{T}_{\phi}\mathcal{P} : \mathbf{T}_{\phi}\mathrm{Mon}(\otimes) \to \mathbf{T}_{\mathcal{P}(\phi)}\mathrm{UMon}(\otimes)$ which is the infinitesimal version of $\mathcal{P}$ at the point $\phi$. Now by applying \eqref{decompositionMonOtimes} to the category $\mathcal{C}_{\epsilon}$ we obtain a linear isomorphism
\begin{equation}\label{decompositionTangentMonOtimes}
\fonc{\mathbf{T}_{\phi}\mathcal{F}}{\mathbf{T}_{\phi}\mathrm{Mon}(\otimes)}{\mathbf{T}_{p_1(\phi)}\mathrm{Mon}(\mathrm{Id}_{\mathcal{C}}) \oplus \mathbf{T}_{p_2(\phi)}\mathrm{Mon}(\mathrm{Id}_{\mathcal{C}}) \oplus \mathbf{T}_{\mathcal{P}(\phi)}\mathrm{UMon}(\otimes)}{f}{\bigl( \mathbf{T}_{\phi}p_1(f), \mathbf{T}_{\phi}p_2(f), \mathbf{T}_{\phi}\mathcal{P}(f) \bigr)}
\end{equation}
To compute its inverse write $\mathcal{F}^{-1}\bigl( p_1(\phi) + \epsilon a, p_2(\phi) + \epsilon b, \mathcal{P}(\phi) + \epsilon g \bigr) = \phi + \epsilon \mathbf{T}_{\mathcal{F}(\phi)}(\mathcal{F}^{-1})(a,b,g)$ and note that $\mathbf{T}_{\mathcal{F}(\phi)}(\mathcal{F}^{-1}) = (\mathbf{T}_{\phi}\mathcal{F})^{-1}$. Hence by \eqref{decompositionMonOtimesInverse} we find
\[ (\mathbf{T}_{\phi}\mathcal{F})^{-1}(a,b,g) = \bigl( p_1(\phi) \otimes p_2(\phi) \bigr) \circ g + \bigl( p_1(\phi) \otimes b \bigr) \circ \mathcal{P}(\phi) + \bigl( a \otimes p_2(\phi) \bigr) \circ \mathcal{P}(\phi). \]

\begin{remark}\label{remarkDeformationOfMultIsWeak}
If $\phi \in \mathrm{UMon}(\otimes)$ then $p_1(\phi) = p_2(\phi) = \boldsymbol{e}$ with $\boldsymbol{e}_{X,Y} = \mathrm{id}_{X \otimes Y}$ and $\mathcal{P}(\phi) = \phi$. As a result $\mathbf{T}_{\phi}\mathrm{Mon}(\otimes) \cong \mathbf{T}_{\boldsymbol{e}}\mathrm{Mon}(\mathrm{Id}_{\mathcal{C}})^{\oplus 2} \oplus \mathbf{T}_{\phi}\mathrm{UMon}(\otimes)$ and $\mathbf{T}_{\phi}\mathcal{P}$ is the projection onto $\mathbf{T}_{\phi}\mathrm{UMon}(\otimes)$. In general, if $(\otimes,\phi)$ is a multiplication on $\mathcal{C}$ and $f \in \mathbf{T}_{\phi}\mathrm{Mon}(\otimes)$ then $(\otimes_{\epsilon},\phi + \epsilon f)$ is just a weak-unital multiplication on $\mathcal{C}_{\epsilon}$; it is a multiplication if and only if $\mathbf{T}_{\phi}\mathcal{P}(f) = f$. Briefly: the deformation of a multiplication on $\mathcal{C}$ is in general just a weak-unital multiplication on $\mathcal{C}_{\epsilon}$.
\end{remark}

\indent In order to obtain the tangent version of Lemma \ref{lemmaDecompositionMon}, recall from \eqref{defTangentSpaceQuotient} that by definition
\[ \mathbf{T}_{[\phi]}\bigl( \mathrm{Mon}(\otimes)/{\sim} \bigr) = \bigl( \mathbf{T}_{\phi}\mathrm{Mon}(\otimes) \bigr)/{\equiv}_{\phi}, \qquad \mathbf{T}_{[\alpha]}\bigl( \mathrm{Mon}(\mathrm{Id}_{\mathcal{C}})/{\sim} \bigr) = \bigl( \mathbf{T}_{\alpha}\mathrm{Mon}(\mathrm{Id}_{\mathcal{C}}) \bigr)/{\equiv}_{\alpha} \]
for any $\phi \in \mathrm{Mon}(\otimes)$ and $\alpha \in \mathrm{Mon}(\mathrm{Id}_{\mathcal{C}})$, where $\equiv$ is defined in general in \eqref{defEquivOnTangentSpace}. By \eqref{tangentSpaceAndCohom}, these spaces are isomorphic to the DY cohomology spaces $H^2_{\mathrm{DY}}(\otimes,\phi)$ and $H^2_{\mathrm{DY}}(\mathrm{Id}_{\mathcal{C}},\alpha)$, respectively. Since $\otimes$ and $\mathrm{Id}_{\mathcal{C}}$ are $\Bbbk$-linear in each variable, it will be convenient to use the other definition of $\equiv$ given in item 2 of Lemma \ref{lemmaDeformationMultilinearFunctor}. Also recall the equivalence relation $\sim_{\mathrm{U}}$ on $\mathrm{UMon}(\otimes)$ introduced before Lemma \ref{lemmaDecompositionMon}. For $\omega \in \mathrm{UMon}(\otimes)$ we define
\begin{equation}\label{defTangentSpaceQuotientUMon}
\mathbf{T}_{[\omega]_{\mathrm{U}}}\bigl( \mathrm{UMon}(\otimes)/{\sim}_{\mathrm{U}} \bigr) = \bigl( \mathbf{T}_{\omega} \mathrm{UMon}(\otimes) \bigr)/{\equiv}_{\omega}^{\mathrm{U}}
\end{equation}
where $[\omega]_{\mathrm{U}}$ is the equivalence class of $\omega$ for $\sim_{\mathrm{U}}$ and we declare that $f \equiv_{\omega}^{\mathrm{U}} g$ if there exists $v \in \mathrm{Nat}(\otimes,\otimes)$ such that $\mathrm{id}_{\otimes} + \epsilon v : \otimes_{\epsilon} \Rightarrow \otimes_{\epsilon}$ is a monoidal natural isomorphism $(\otimes_{\epsilon}, \omega + \epsilon f) \Rightarrow (\otimes_{\epsilon}, \omega + \epsilon g)$ in $\mathcal{C}_{\epsilon}$ which is {\em unital}, \textit{i.e.} $v_{(\boldsymbol{1},X)} = v_{(X,\boldsymbol{1})} = 0$ for all $X \in \mathcal{C}$.

\begin{lemma}\label{lemmaDecompositionTMon}
The linear map \eqref{decompositionTangentMonOtimes} descends to an isomorphism
\[ \mathbf{T}_{[\phi]}\bigl( \mathrm{Mon}(\otimes)/{\sim} \bigr) \cong \mathbf{T}_{[p_1(\phi)]}\bigl( \mathrm{Mon}(\mathrm{Id}_{\mathcal{C}})/{\sim}\bigr) \oplus \mathbf{T}_{[p_2(\phi)]}\bigl( \mathrm{Mon}(\mathrm{Id}_{\mathcal{C}})/{\sim} \bigr) \oplus \mathbf{T}_{[\mathcal{P}(\phi)]_{\mathrm{U}}}\bigl(\mathrm{UMon}(\otimes)/{\sim}_{\mathrm{U}} \bigr) \]
which can be rewritten as
\begin{equation}\label{isoDecompositionH2Otimes}
H^2_{\mathrm{DY}}(\otimes,\phi) \cong H^2_{\mathrm{DY}}(\mathcal{C}) \oplus H^2_{\mathrm{DY}}(\mathcal{C}) \oplus \mathbf{T}_{[\mathcal{P}(\phi)]_{\mathrm{U}}}\bigl(\mathrm{UMon}(\otimes)/\!\!\sim_{\mathrm{U}} \bigr).
\end{equation}
\end{lemma}
\begin{proof}
The first isomorphism follows immediately by applying the proof of Lemma \ref{lemmaDecompositionMon} to the category $\mathcal{C}_{\epsilon}$ and looking at tangent spaces. Here are some details for convenience. Let $f,g \in \mathbf{T}_{\phi} \mathrm{Mon}(\otimes)$ be such that $f \equiv_{\phi} g$. Due to item 2 of Lemma \ref{lemmaDeformationMultilinearFunctor} there is a monoidal natural isomorphism $\mathrm{id}_{\otimes} + \epsilon v : (\otimes_{\epsilon}, \phi + \epsilon f) \Rightarrow (\otimes_{\epsilon}, \phi + \epsilon g)$ in $\mathcal{C}_{\epsilon}$. Using the maps $p'_i$ from \eqref{degeneraciesForAutOtimes} we get monoidal natural isomorphisms $p'_i(\mathrm{id}_{\otimes} + \epsilon v) : \bigl( \mathrm{Id}_{\mathcal{C}_{\epsilon}}, p_i(\phi) + \epsilon \mathbf{T}_{\phi}p_i(f) \bigr) \Rightarrow \bigl( \mathrm{Id}_{\mathcal{C}_{\epsilon}}, p_i(\phi) + \epsilon \mathbf{T}_{\phi}p_i(g) \bigr)$, and hence $\mathbf{T}_{\phi}p_i(f) \equiv_{\phi} \mathbf{T}_{\phi}p_i(g)$ for $i=1,2$ again due to item 2 of Lemma \ref{lemmaDeformationMultilinearFunctor}. Although we do not need them, the explicit expressions are of course 
\[ \forall \,X \in \mathcal{C},\qquad p'_1(\mathrm{id}_{\otimes} + \epsilon v)_X = \mathrm{id}_X + \epsilon v_{(X,\boldsymbol{1})}, \qquad p'_2(\mathrm{id}_{\otimes} + \epsilon v)_X = \mathrm{id}_X + \epsilon v_{(\boldsymbol{1},X)}. \]
Similarly, using the map $\mathcal{P}'$ from \eqref{unitalProjectorForAutOtimes} we get a unital monoidal natural isomorphism
\[ \mathcal{P}'(\mathrm{id}_{\otimes} + \epsilon v) : \bigl( \otimes_{\epsilon}, \mathcal{P}(\phi) + \epsilon \mathbf{T}_{\phi}\mathcal{P}(f) \bigr) \Rightarrow \bigl( \otimes_{\epsilon}, \mathcal{P}(\phi) + \epsilon \mathbf{T}_{\phi}\mathcal{P}(g) \bigr) \]
and hence $\mathbf{T}_{\phi}\mathcal{P}(f) \equiv_{\mathcal{P}(\phi)}^{\mathrm{U}} \mathbf{T}_{\phi}\mathcal{P}(g)$. Although we do not need it, here is its explicit expression:
\[ \forall \, X,Y \in \mathcal{C}, \quad \mathcal{P}'(\mathrm{id}_{\otimes} + \epsilon v)_{(X,Y)} =\mathrm{id}_{X \otimes Y} + \epsilon\bigl( v_{(X,Y)} -  v_{(X,\boldsymbol{1})} \otimes \mathrm{id}_Y  - \mathrm{id}_X \otimes v_{(\boldsymbol{1},Y)} \bigr). \]
It follows that $\mathbf{T}_{\phi}\mathcal{F}$ from \eqref{decompositionTangentMonOtimes} descends to a linear map 
\[ \overline{\mathbf{T}_{\phi}\mathcal{F}} : [f]_{\phi} \mapsto \left( [\mathbf{T}_{\phi}p_1(f)]_{p_1(\phi)}, [\mathbf{T}_{\phi}p_2(f)]_{p_2(\phi)}, [\mathbf{T}_{\phi}\mathcal{P}(f)]_{\mathcal{P}(\phi)}^{\mathrm{U}} \right) \]
where $[\:\:]_{\phi}$ (resp. $[\:\:]_{p_i(\phi)}$) denotes the equivalence classes for $\equiv_{\phi}$ (resp. $\equiv_{p_i(\phi)}$) and $[\:\:]_{\mathcal{P}(\phi)}^{\mathrm{U}}$ is the equivalence class for the relation $\equiv_{\mathcal{P}(\phi)}^{\mathrm{U}}$ on the subspace $\mathbf{T}_{\mathcal{P}(\phi)}\mathrm{UMon}(\otimes)$. 
Its inverse is constructed similarly from $(\mathbf{T}_{\phi}\mathcal{F})^{-1}$. This proves the first isomorphism.
\\To prove \eqref{isoDecompositionH2Otimes}, note first that for all $\alpha \in \mathrm{Mon}(\mathrm{Id}_{\mathcal{C}})$ we have an isomorphism of cochain spaces
\begin{equation}\label{translationDYIdentity}
\forall \, n \geq 0, \qquad C^n_{\mathrm{DY}}(\mathrm{Id}_{\mathcal{C}}, \alpha) \overset{\sim}{\to} C^n_{\mathrm{DY}}(\mathrm{Id}_{\mathcal{C}},\boldsymbol{e}), \quad f \mapsto (\alpha^{(n)})^{-1} \circ f
\end{equation}
where $\alpha^{(n)} = \bigl( \alpha^{(n)}_{X_1,\ldots,X_n} \in \mathrm{Aut}_{\mathcal{C}}(X_1 \otimes \ldots \otimes X_n) \bigr)_{X_i \in \mathcal{C}}$ is the iterated monoidal structure as in \eqref{higherMonStruct} with in particular $\alpha^{(2)} = \alpha$ and $\boldsymbol{e}$ is the trivial monoidal structure given by $\boldsymbol{e}_{X_1,X_2} = \mathrm{id}_{X_1 \otimes X_2}$. It is left to the reader to check that \eqref{translationDYIdentity} commutes with the DY coface maps. Combining \eqref{translationDYIdentity} with \eqref{tangentSpaceAndCohom} we get
\[ \mathbf{T}_{[\alpha]}\bigl( \mathrm{Mon}(\mathrm{Id}_{\mathcal{C}})/\!\!\sim \bigr) \cong H^2_{\mathrm{DY}}(F,\alpha) \overset{\sim}{\longrightarrow} H^2_{\mathrm{DY}}(\mathrm{Id}_{\mathcal{C}},\boldsymbol{e}) = H^2_{\mathrm{DY}}(\mathcal{C}). \]
Now \eqref{isoDecompositionH2Otimes} follows from the first claim of the lemma by taking $\alpha = p_1(\phi)$ and $\alpha = p_2(\phi)$.
\end{proof}

\begin{remark}
For $n=2$, the isomorphism \eqref{translationDYIdentity} can be deduced from the fact {\em $\mathrm{Mon}(\mathrm{Id}_{\mathcal{C}})$ has the structure of a group} if we define the composition by $(\alpha \circ \beta)_{X,Y} = \alpha_{X,Y} \circ \beta_{X,Y}$. Hence for any $\alpha \in \mathrm{Mon}(\mathrm{Id}_{\mathcal{C}})$ there is a bijection $L_{\alpha} : \mathrm{Mon}(\mathrm{Id}_{\mathcal{C}}) \overset{\sim}{\to} \mathrm{Mon}(\mathrm{Id}_{\mathcal{C}})$ given by $L_{\alpha}(\beta) = \alpha \circ \beta$. Moreover it is easily seen that if $\beta \sim \beta'$ then $L_{\alpha}(\beta) \sim L_{\alpha}(\beta')$ and thus $L_{\alpha}$ descends to a bijection $\mathrm{Mon}(\mathrm{Id}_{\mathcal{C}})/{\sim} \: \to \mathrm{Mon}(\mathrm{Id}_{\mathcal{C}})/{\sim}$. Taking infinitesimal versions in $\mathcal{C}_{\epsilon}$ at the point $\boldsymbol{e}$ we deduce that $\mathbf{T}_{\boldsymbol{e}}\mathrm{Mon}(\mathrm{Id}_{\mathcal{C}}) \cong \mathbf{T}_{\alpha}\mathrm{Mon}(\mathrm{Id}_{\mathcal{C}})$ and $\mathbf{T}_{[\boldsymbol{e}]}\bigl( \mathrm{Mon}(\mathrm{Id}_{\mathcal{C}})/{\sim} \bigr) \cong \mathbf{T}_{[\alpha]}\bigl( \mathrm{Mon}(\mathrm{Id}_{\mathcal{C}})/{\sim} \bigr)$.
\end{remark}

\indent By \eqref{braidingAssociatedToLaxMult} and Proposition \ref{factoLaxMult} we have a map $B : \mathrm{Mon}(\otimes) \to \mathrm{Br}(\mathcal{C})$. In the category $\mathcal{C}_{\epsilon}$, write as usual $B(\phi + \epsilon f) = B(\phi) + \epsilon \mathbf{T}_{\phi}B(f)$. Then $B(\phi) + \epsilon \mathbf{T}_{\phi}B(f)\in \mathrm{Br}(\mathcal{C}_{\epsilon})$, which is equivalent to $\mathbf{T}_{\phi}B(f) \in \mathbf{T}_{B(\phi)}\mathrm{Br}(\mathcal{C})$. As a result there is a linear map $\mathbf{T}_{\phi}B : \mathbf{T}_{\phi}\mathrm{Mon}(\otimes) \to \mathbf{T}_{B(\phi)}\mathrm{Br}(\mathcal{C})$ given by
\begin{equation}\label{tangentMapBMonBrC}
\forall\, X,Y \in \mathcal{C}, \quad \mathbf{T}_{\phi}B(f)_{X,Y} = \phi^{-1}_{(Y,\boldsymbol{1}),(\boldsymbol{1},X)} \circ \bigl( f_{(\boldsymbol{1},X),(Y,\boldsymbol{1})} - f_{(Y,\boldsymbol{1}),(\boldsymbol{1},X)} \circ B(\phi)_{X,Y} \bigr).
\end{equation}

\begin{lemma}\label{lemmaRelatingUMonAndInfBraidings}
1. Let $\phi \in \mathrm{Mon}(\otimes)$ and $f,g \in \mathbf{T}_{\phi}\mathrm{Mon}(\otimes)$. If $f \equiv_{\phi} g$ then $\mathbf{T}_{\phi}B(f) = \mathbf{T}_{\phi}B(g)$.
\\2. Let $c \in \mathrm{Br}(\mathcal{C})$. For $t \in \mathbf{T}_c\mathrm{Br}(\mathcal{C})$ let $\mathrm{id} \otimes t \otimes \mathrm{id}$ be the natural isomorphism with components $(\mathrm{id} \otimes t \otimes \mathrm{id})_{(X_1,Y_1),(X_2,Y_2)} = \mathrm{id}_{X_1} \otimes t_{Y_1,X_2} \otimes \mathrm{id}_{Y_2}$. Then $\mathrm{id} \otimes t \otimes \mathrm{id} \in \mathbf{T}_{\mathrm{id} \otimes c \otimes \mathrm{id}}\mathrm{Mon}(\otimes)$.
\\3. Let $c \in \mathrm{Br}(\mathcal{C})$ and take $\phi = \mathrm{id} \otimes c \otimes \mathrm{id} \in \mathrm{UMon}(\otimes)$. Then for all $f \in \mathbf{T}_{\phi}\mathrm{UMon}(\otimes)$ we have $\mathrm{id} \otimes \mathbf{T}_{\phi}B(f) \otimes \mathrm{id} \equiv_{\phi}^{\mathrm{U}} f$.
\\4. As a result if $\phi \in \mathrm{UMon}(\otimes)$ there is an isomorphism of vector spaces
\begin{equation}\label{infBraidingsFromUnitalCocycles}
\mathbf{T}_{[\phi]_{\mathrm{U}}}\bigl( \mathrm{UMon}(\otimes)/{\sim}_{\mathrm{U}} \bigr) \to \mathbf{T}_{B(\phi)}\mathrm{Br}(\mathcal{C}), \qquad [f]_{\phi}^{\mathrm{U}} \mapsto \mathbf{T}_{\phi}B(f)
\end{equation}
where $[\:\:]_{\phi}^{\mathrm{U}}$ is the equivalence class for the relation $\equiv_{\phi}^{\mathrm{U}}$ from \eqref{defTangentSpaceQuotientUMon} and $\mathbf{T}_{\phi}B(f)$ is defined in \eqref{tangentMapBMonBrC}.
\end{lemma}
\begin{proof}
1. If $f \equiv_{\phi} g$ then $\phi + \epsilon f \sim \phi + \epsilon g$ in $\mathrm{Mon}(\otimes_{\epsilon})$. Hence by item 1 in Proposition \ref{factoLaxMult} we have $B(\phi + \epsilon f)= B(\phi + \epsilon g)$, which by definition implies $B(\phi) + \epsilon \mathbf{T}_{\phi}B(f) = B(\phi) + \epsilon \mathbf{T}_{\phi}B(g)$.
\\2. By item 2 in Proposition \ref{factoLaxMult} there is a map $M : \mathrm{Br}(\mathcal{C}) \to \mathrm{Mon}(\otimes)$ given by $M(c) = \mathrm{id} \otimes c \otimes \mathrm{id}$. The result follows by applying this map to the category $\mathcal{C}_{\epsilon}$. Namely, if $t \in \mathbf{T}_c\mathrm{Br}(\mathcal{C})$ then $c + \epsilon t \in \mathrm{Br}(\mathcal{C}_{\epsilon})$. Hence $M(c + \epsilon t) = \mathrm{id} \otimes c \otimes \mathrm{id} + \epsilon \, \mathrm{id} \otimes t \otimes \mathrm{id} \in \mathrm{Mon}(\otimes_{\epsilon})$ and then $\mathrm{id} \otimes t \otimes \mathrm{id} \in \mathbf{T}_{\mathrm{id} \otimes c \otimes \mathrm{id}}\mathrm{Mon}(\otimes)$.
\\3. A preliminary remark is in order about our assumption on $\phi$: for an arbitrary choice of $\phi \in \mathrm{UMon}(\otimes)$ the statement would not make sense because $\mathrm{id} \otimes \mathbf{T}_{\phi}B(f) \otimes \mathrm{id}$ belongs to $\mathbf{T}_{\mathrm{id} \otimes B(\phi) \otimes \mathrm{id}}\mathrm{UMon}(\otimes)$ instead of $\mathbf{T}_{\phi}\mathrm{UMon}(\otimes)$. But for $\phi = \mathrm{id} \otimes c \otimes \mathrm{id}$ it holds $B(\phi) = c$ and thus the statement makes sense. Now let us prove it. By assumption, $\phi + \epsilon f$ and $\phi + \epsilon \, \mathrm{id} \otimes \mathbf{T}_{\phi}B(f) \otimes \mathrm{id}$ belong to $\mathrm{UMon}(\otimes_{\epsilon})$. Thus by the proof of item 3 in Proposition \ref{factoLaxMult} applied to $\mathcal{C}_{\epsilon}$ we know that the natural transformation
\[ u_{(X,Y)} = (\phi + \epsilon f)_{(X,\boldsymbol{1}),(\boldsymbol{1},Y)} = \mathrm{id}_{X \otimes Y} + \epsilon f_{(X,\boldsymbol{1}),(\boldsymbol{1},Y)} : X \otimes Y \to X \otimes Y \]
(where we used that $\phi_{(X,\boldsymbol{1}),(\boldsymbol{1},Y)} = \mathrm{id}_{X \otimes Y}$ due to the particular form of $\phi$) is a unital monoidal natural isomorphism $\bigl( \otimes_{\epsilon}, \phi + \epsilon \, \mathrm{id} \otimes \mathbf{T}_{\phi}B(f) \otimes \mathrm{id}\bigr) \Rightarrow \bigl( \otimes_{\epsilon}, \phi + \epsilon f \bigr)$. Hence  $\mathrm{id} \otimes \mathbf{T}_{\phi}B(f) \otimes \mathrm{id} \equiv^{\mathrm{U}}_{\phi} f$ by definition (see below \eqref{defTangentSpaceQuotientUMon}).
\\4. By item 3 in Proposition \ref{factoLaxMult} we have $[\phi]_{\mathrm{U}} = [\mathrm{id} \otimes B(\phi) \otimes \mathrm{id}]_{\mathrm{U}}$. Hence the spaces $\mathbf{T}_{[\phi]_{\mathrm{U}}}\bigl( \mathrm{UMon}(\otimes)/{\sim}_{\mathrm{U}} \bigr)$ and $\mathbf{T}_{[\mathrm{id} \otimes B(\phi) \otimes \mathrm{id}]_{\mathrm{U}}}\bigl( \mathrm{UMon}(\otimes)/{\sim}_{\mathrm{U}} \bigr)$ are canonically isomorphic by the same argument as in \eqref{canonicalIsoTangentSpaceQuotient}. Thus without loss of generality we can assume that $\phi = \mathrm{id} \otimes c \otimes \mathrm{id}$ for some $c \in \mathrm{Br}(\mathcal{C})$. In this case item 3 shows that the inverse to \eqref{infBraidingsFromUnitalCocycles} is given by $t \mapsto [ \mathrm{id} \otimes t \otimes \mathrm{id} ]_{\phi}^{\mathrm{U}}$.
\end{proof}

Combining Lemmas \ref{lemmaDecompositionTMon} and \ref{lemmaRelatingUMonAndInfBraidings} we obtain:

\begin{theorem}\label{thmDYOtimesAndInfBraidings}
For all $\phi \in \mathrm{Mon}(\otimes)$ there is an isomorphism of vector spaces
\begin{equation}\label{isoDecompositionH2DYOtimes}
\flecheIso{H^2_{\mathrm{DY}}(\otimes,\phi)}{H^2_{\mathrm{DY}}(\mathcal{C}) \oplus H^2_{\mathrm{DY}}(\mathcal{C}) \oplus \mathbf{T}_{B(\phi)}\mathrm{Br}(\mathcal{C})}{\,[f]\,}{\bigl( [a^f], [b^f], \mathbf{T}_{\phi}B(f) \bigr)}
\end{equation}
where $[\:\:]$ denotes cohomology classes, $a^f$ and $b^f$ are defined by $a^f_{X_1,X_2} = \phi_{(X_1,\boldsymbol{1}),(X_2, \boldsymbol{1})}^{-1} \circ f_{(X_1,\boldsymbol{1}),(X_2,\boldsymbol{1})}$ and $b^f_{Y_1,Y_2} = \phi_{(\boldsymbol{1},Y_1),(\boldsymbol{1},Y_2)}^{-1} \circ f_{(\boldsymbol{1},Y_1),(\boldsymbol{1},Y_2)}$ for all $X_i,Y_i \in \mathcal{C}$ and $\mathbf{T}_{\phi}B(f)$ is defined in \eqref{tangentMapBMonBrC}.
\end{theorem}
\begin{proof}
Recall that the isomorphism \eqref{isoDecompositionH2Otimes} in Lemma \ref{lemmaDecompositionTMon} is built from the isomorphism \eqref{decompositionTangentMonOtimes} followed by translation isomorphisms of the form \eqref{translationDYIdentity} applied to the first two components, which explains the formulas for $a^f$ and $b^f$. Hence the composition of the isomorphisms \eqref{isoDecompositionH2Otimes} and \eqref{infBraidingsFromUnitalCocycles} is given by $[f] \mapsto \bigl( [a^f], [b^f], \mathbf{T}_{\mathcal{P}(\phi)}B(\mathbf{T}_{\phi}\mathcal{P}(f)) \bigr)$. But we have already noted in the proof of Corollary \ref{coroGeneralizedJSBijection} that $B \circ \mathcal{P} = B$. Hence $(\mathbf{T}_{\mathcal{P}(\phi)}B \circ \mathbf{T}_{\phi}\mathcal{P})(f) = \mathbf{T}_{\phi}(B \circ \mathcal{P})(f) = \mathbf{T}_{\phi} B(f)$.
\end{proof}
\noindent We spell out Theorem~\ref{thmDYOtimesAndInfBraidings} in the case where the monoidal structure $\phi$ of $\otimes$ comes from a braiding $c \in \mathrm{Br}(\mathcal{C})$ as defined in \eqref{MonStructFromBraiding}, \textit{i.e.} $\phi = \mathrm{id} \otimes c \otimes \mathrm{id}$, and deduce a dimension formula for the space of infinitesimal braidings tangent to $c$:
\begin{corollary}\label{coroDimFormulaTangentSpaceBraiding}
Let $c \in \mathrm{Br}(\mathcal{C})$ be a braiding in $\mathcal{C}$.
\\1. There is an isomorphism of vector spaces
\[ \flecheIso{H^2_{\mathrm{DY}}(\otimes,\mathrm{id} \otimes c \otimes \mathrm{id})}{H^2_{\mathrm{DY}}(\mathcal{C}) \oplus H^2_{\mathrm{DY}}(\mathcal{C}) \oplus \mathbf{T}_c\mathrm{Br}(\mathcal{C})}{\,[f]\,}{\bigl( [a^f], [b^f], t^f \bigr)} \]
where $a^f_{X_1,X_2} = f_{(X_1,\boldsymbol{1}),(X_2,\boldsymbol{1})}$, $b^f_{Y_1,Y_2} = f_{(\boldsymbol{1},Y_1),(\boldsymbol{1},Y_2)}$ and $t^f_{X,Y} = f_{(\boldsymbol{1},X),(Y,\boldsymbol{1})} - f_{(Y,\boldsymbol{1}),(\boldsymbol{1},X)} \circ c_{X,Y}$.
\\2. If $\mathcal{C}$ is finite as a $\Bbbk$-linear category we have
\begin{equation}\label{dimTangentSpaceBraidingInTermsDYCohom}
\dim \mathbf{T}_c\mathrm{Br}(\mathcal{C}) = \dim H^2_{\mathrm{DY}}(\otimes, \mathrm{id} \otimes c \otimes \mathrm{id}) - 2\dim H^2_{\mathrm{DY}}(\mathcal{C})
\end{equation}
\end{corollary}
\begin{proof}
1. Immediate from Theorem~\ref{thmDYOtimesAndInfBraidings}. Note that for $\phi = \mathrm{id} \otimes c \otimes \mathrm{id}$ we have $B(\phi) = c$ and the formulas for $a^f$, $b^f$ and $t^f = \mathbf{T}_{\mathrm{id} \otimes c \otimes \mathrm{id}}B(f)$ are simpler because $\phi_{(X_1,\boldsymbol{1}),(X_2, \boldsymbol{1})} = \mathrm{id}_{X_1 \otimes X_2}$, $\phi_{(\boldsymbol{1},Y_1),(\boldsymbol{1},Y_2)} = \mathrm{id}_{Y_1 \otimes Y_2}$ and $\phi_{(Y,\boldsymbol{1}),(\boldsymbol{1},X)} = \mathrm{id}_{Y \otimes X}$.
\\2. By item 1 in Lemma \ref{lemNatTransfoProjGen} the DY cochain spaces are finite-dimensional. Hence we can consider the dimensions of the vector spaces in item 1 and we get the result.
\end{proof}
\noindent In the next subsection we give a result for the computation of $\dim H^2_{\mathrm{DY}}(\otimes, \mathrm{id} \otimes c \otimes \mathrm{id})$ through relative homological algebra.

\subsection{Adjunction theorem for $\otimes$}\label{adjunctionThmMonProduct}
In this subsection we assume that $\mathcal{C}$ is a finite $\Bbbk$-linear tensor category. The monoidal product $\otimes = \otimes_{\mathcal{C}} : \mathcal{C} \times \mathcal{C} \to \mathcal{C}$ is $\Bbbk$-bilinear (by assumption) and exact in each variable (because of rigidity, \cite[\S 4.2]{EGNO}\footnote{In \cite{EGNO} the field $\Bbbk$ is always assumed to be algebraically closed but this assumption is not necessary for this property.}). Hence there exists a right-exact $\Bbbk$-linear functor $P : \mathcal{C} \boxtimes \mathcal{C} \to \mathcal{C}$ such that $\otimes = P \circ \boxtimes$, where $\boxtimes$ is the Deligne product of finite $\Bbbk$-linear categories \cite[\S 5]{deligne}.\footnote{We note as a side remark that if the ground field $\Bbbk$ is perfect then the factorization through the Deligne product of a bifunctor exact in each variable is an exact functor \cite[Prop.\,5.13]{deligne}. Thus if $\Bbbk$ is perfect then $P$ is actually exact.}

\smallskip

\indent The category $\mathcal{C} \boxtimes \mathcal{C}$ has a  monoidal product defined by
\[ (X_1 \boxtimes Y_1) \otimes (X_2 \boxtimes Y_2) = (X_1 \otimes X_2) \boxtimes (Y_1 \otimes Y_2) \]
and extended as a bifunctor right-exact in each variable. Using the universal property of the Deligne product, there is a contravariant endofunctor $(-)^{\vee}$ on $\mathcal{C} \boxtimes \mathcal{C}$ defined by
\[ (X \boxtimes Y)^{\vee} = X^{\vee} \boxtimes Y^{\vee}. \]
Thanks to item 2 in Lemma \ref{lemmaExtensionNatDeligneProd}, there are dinatural transformations $\bigl(\mathrm{ev}_M : M^{\vee} \otimes M \to \boldsymbol{1} \boxtimes \boldsymbol{1} \bigr)_{M \in \mathcal{C} \boxtimes \mathcal{C}}$ and $\bigl( \mathrm{coev}_M : \boldsymbol{1} \boxtimes \boldsymbol{1} \to M \otimes M^{\vee} \bigr)_{M \in \mathcal{C} \boxtimes \mathcal{C}}$ uniquely defined by $\mathrm{ev}_{X \boxtimes Y} = \mathrm{ev}_X \boxtimes \mathrm{ev}_Y$ and $\mathrm{coev}_{X \boxtimes Y} = \mathrm{coev}_X \boxtimes \mathrm{coev}_Y$. This endows $\mathcal{C} \boxtimes \mathcal{C}$ with a left duality. Similar remarks apply for right duality and hence $\mathcal{C} \boxtimes \mathcal{C}$ is rigid. This implies that the monoidal product $\otimes_{\mathcal{C} \boxtimes \mathcal{C}}$ is even exact in each variable \cite[\S 4.2]{EGNO}.

\smallskip

If $\phi$ is a monoidal structure for $\otimes_{\mathcal{C}}$ then item 1 in Lemma \ref{lemmaExtensionNatDeligneProd} (or rather its straightforward generalization for natural transformations with several components) gives a monoidal structure for $P$, which we still denote by $\phi$. It satisfies $\phi_{X_1 \boxtimes Y_1, X_2 \boxtimes Y_2} = \phi_{(X_1,Y_1),(X_2,Y_2)}$ by definition. A DY cochain $f \in C^n_{\mathrm{DY}}(P,\phi)$ is a natural transformation whose source and target look as follows on the factorized objects of $\mathcal{C} \boxtimes \mathcal{C}$
\begin{equation}\label{DYCochainsMonProduct}
f_{X_1 \boxtimes Y_1, \ldots, X_n \boxtimes Y_n} : X_1 \otimes Y_1 \otimes X_2 \otimes Y_2 \otimes \ldots \otimes X_n \otimes Y_n \to X_1 \otimes \ldots \otimes X_n \otimes Y_1 \otimes \ldots \otimes Y_n.
\end{equation}
By the straightforward generalization of item 1 in Lemma \ref{lemmaExtensionNatDeligneProd} for natural transformations with several components, these values determine $f$ uniquely. It follows that:
\begin{proposition}\label{propDYcohomOfMonProductAndLinearization}
The cochain complexes $C^{\bullet}_{\mathrm{DY}}(\otimes_{\mathcal{C}},\phi)$ and $C^{\bullet}_{\mathrm{DY}}(P,\phi)$ are isomorphic.
\end{proposition}
\begin{proof}
There is an isomorphism of vector spaces $J^n : C^n_{\mathrm{DY}}(P,\phi) \to C^n_{\mathrm{DY}}(\otimes,\phi)$ defined by $J^n(f)_{(X_1,Y_1), \ldots, (X_n,Y_n)} = f_{X_1 \boxtimes Y_1, \ldots, X_n \boxtimes Y_n}$ for each $n \in \mathbb{N}$. The compatibility of $(J^n)_{n \in \mathbb{N}}$ with the DY differential is immediate.
\end{proof}

\indent From now on we fix a monoidal structure of the form $\phi = \mathrm{id} \otimes c \otimes \mathrm{id}$ as defined in \eqref{MonStructFromBraiding}, where $c$ is a braiding on $\mathcal{C}$. This fixes a monoidal structure for $P$ entirely defined by
\begin{equation}\label{monoidalStructPFromBraiding}
P^{(2)}_{X_1 \boxtimes Y_1,X_2 \boxtimes Y_2} = \mathrm{id}_{X_1} \otimes c_{Y_1,X_2} \otimes \mathrm{id}_{Y_2}.
\end{equation}
 We are in the situation \eqref{assumptionsFCD} with $F = P$ and $\mathcal{C}$ is replaced by $\mathcal{C} \boxtimes \mathcal{C}$ while $\mathcal{D}$ is replaced by $\mathcal{C}$. Our goal is to apply Adjunction Theorem (Thm.\,\ref{thmChangeOfCoeffDY}) to the functor $P$. The coend
\begin{equation}\label{defCoendBulkA}
\mathscr{A} = \int^{X \in \mathcal{C}} X^{\vee} \boxtimes X \in \mathcal{C} \boxtimes \mathcal{C}
\end{equation}
will play a key role; it exists by Corollary \ref{coendsAsCokerAndConsequence} because the $\Bbbk$-bilinear functor $\boxtimes : \mathcal{C} \times \mathcal{C} \to \mathcal{C} \boxtimes \mathcal{C}$ is exact in each variable \cite[Prop.\,5.13]{deligne}. Let $j_X : X^{\vee} \boxtimes X \to \mathscr{A}$ be the universal dinatural transformation. Define a half-braiding $\lambda^{(+),(-)} : \mathscr{A} \otimes - \Rightarrow - \otimes \mathscr{A}$ by the commutative diagram
\begin{equation}\label{halfBraidingOnBulk}
\xymatrix@C=6em{
{\begin{array}{l}
(X^{\vee} \otimes C_1) \boxtimes (X \otimes C_2)\\
=(X^{\vee} \boxtimes X) \otimes (C_1 \boxtimes C_2)
\end{array}} \ar[r]^{c_{X^{\vee},C_1} \boxtimes \, c^{-1}_{C_2,X}} \ar[d]_{j_X \otimes \mathrm{id}_{C_1 \boxtimes C_2}} & {\begin{array}{l}
(C_1 \otimes X^{\vee}) \boxtimes (C_2 \otimes X)\\
=(C_1 \boxtimes C_2) \otimes (X^{\vee} \boxtimes X)
\end{array}} \ar[d]_{\mathrm{id}_{C_1 \boxtimes C_2} \otimes j_X} \\
\mathscr{A} \otimes (C_1 \boxtimes C_2) \ar[r]_{\exists!\,\lambda^{(+),(-)}_{C_1 \boxtimes C_2}} & (C_1 \boxtimes C_2) \otimes \mathscr{A}
} \end{equation}
for all $X,C_1,C_2 \in \mathcal{C}$. Here we use that $\otimes_{\mathcal{C}\boxtimes\mathcal{C}}$ is exact in each variable, thus $j_X \otimes \mathrm{id}_{C_1 \boxtimes C_2}$ is a universal dinatural transformation for the functor $(X,Y) \mapsto (X^{\vee} \boxtimes Y) \otimes (C_1 \boxtimes C_2)$. The values $\lambda^{(+),(-)}_{C_1 \boxtimes C_2}$ uniquely define $\lambda^{(+),(-)}$ by item 1 in Lemma \ref{lemmaExtensionNatDeligneProd}.\footnote{In the notations of this lemma we take $\mathcal{A} = \mathcal{B} = \mathcal{C}$, $\mathcal{D} = \mathcal{C}$ and $F(C_1,C_2) = \mathscr{A} \otimes (C_1 \boxtimes C_2)$, $G(C_1,C_2) = (C_1 \boxtimes C_2) \otimes \mathscr{A}$, $\widetilde{F}(M) = \mathscr{A} \otimes M$, $\widetilde{G}(M) = M \otimes \mathscr{A}$.}
\begin{remark}\label{rem:L-functor}
Note that there is a functor
\begin{align}\label{embeddingCCinZCC}
L^{(+),(-)} : \mathcal{C} \boxtimes \mathcal{C} \to \mathcal{Z}(\mathcal{C} \boxtimes \mathcal{C}), \quad & C_1 \boxtimes C_2 \mapsto \bigl( C_1 \boxtimes C_2, c_{C_1,-} \boxtimes c^{-1}_{-,C_2} \bigr)\ ,\\
& f_1\boxtimes f_2 \mapsto  f_1\boxtimes f_2\ ,\notag
\end{align}
for any morphisms $f_i: C_i \to C_i'$ due to naturality of the braiding and its inverse in both components. Consider the forgetful functor $\mathcal{U} : \mathcal{Z}(\mathcal{C} \boxtimes \mathcal{C}) \to \mathcal{C} \boxtimes \mathcal{C}$ satisfying 
$$
\mathcal{U} \circ L^{(+),(-)} \circ \boxtimes = \mathrm{Id}_{\mathcal{C} \boxtimes \mathcal{C}} \circ \boxtimes\ .$$ 
We note that the functor $L^{(+),(-)}$ is right-exact, since it comes from the factorization through the Deligne product of a bifunctor right-exact in each variable. Thus $\mathcal{U} \circ L^{(+),(-)}$ is right-exact as well, because $\mathcal{U}(g) = g$ for any morphism $g$. The definition of the Deligne product in the form given in \cite[Def.\,1.11.1]{EGNO} then implies that $\mathcal{U} \circ L^{(+),(-)} = \mathrm{Id}_{\mathcal{C} \boxtimes \mathcal{C}}$, by uniqueness of the factorization through $\boxtimes$. It follows that $\mathcal{U}\bigl( L^{(+),(-)}(f) \bigr) = f$ for any morphism $f$ in $\mathcal{C} \boxtimes \mathcal{C}$, and hence $L^{(+),(-)}(f) = f$. In particular, $L^{(+),(-)}$ is exact and thus by item 2 of Corollary \ref{coendsAsCokerAndConsequence} we have $\bigl( \mathscr{A}, \lambda^{(+),(-)} \bigr) \cong L^{(+),(-)}(\mathscr{A})$.
\end{remark}

\begin{theorem}\label{thmChangeFunctorMonProduct}
Let $\mathcal{C}$ be a finite $\Bbbk$-linear tensor category with monoidal product $\otimes$ and $c$ be a braiding in $\mathcal{C}$. Then 
\[H^{\bullet}_{\mathrm{DY}}(\otimes, \mathrm{id} \otimes c \otimes \mathrm{id}) \cong H^{\bullet}_{\mathrm{DY}}\!\left( \mathcal{C} \boxtimes \mathcal{C}; \boldsymbol{1} \boxtimes \boldsymbol{1}, \bigl( \mathscr{A}, \lambda^{(+),(-)} \bigr) \right) \]
where $\mathrm{id} \otimes c \otimes \mathrm{id}$ is the monoidal structure of $\otimes = \otimes_{\mathcal{C}}$ defined in \eqref{MonStructFromBraiding} and $\bigl( \mathscr{A}, \lambda^{(+),(-)} \bigr) \in \mathcal{Z}(\mathcal{C} \boxtimes \mathcal{C})$ is the object $\mathscr{A} \in \mathcal{C} \boxtimes \mathcal{C}$ defined in \eqref{defCoendBulkA} endowed with the half-braiding \eqref{halfBraidingOnBulk}.
\end{theorem}

\indent Because of Proposition \ref{propDYcohomOfMonProductAndLinearization}, Theorem \ref{thmChangeFunctorMonProduct} is an application of Theorem \ref{thmChangeOfCoeffDY} in the following situation:
\begin{equation}\label{diagramLiftingMonProductOnCentralizers}
\xymatrix@C=4em@R=.7em{
\mathcal{Z}(\mathcal{C} \boxtimes \mathcal{C}) \ar@/^.7em/[dd]^{\mathcal{U}_{\mathcal{C} \boxtimes \mathcal{C}}} \ar@/^.7em/[r]^-{\widetilde{P}}_-{\text{\normalsize \rotatebox{270}{$\dashv$}}} & 
\ar@/^.7em/[l]^-{\widetilde{R}} \mathcal{Z}(P,\mathrm{id} \otimes c \otimes \mathrm{id}) \ar@/^.7em/[dd]^{\mathcal{U}_P}\\
\dashv & \dashv\\
\ar@/^.7em/[uu]^{\mathcal{F}_{\mathcal{C} \boxtimes \mathcal{C}}} \mathcal{C} \boxtimes \mathcal{C} \ar@/^.7em/[r]^{P}_{\text{\normalsize \rotatebox{270}{$\dashv$}}} & \ar@/^.7em/[l]^R \ar@/^.7em/[uu]^{\mathcal{F}_P} \mathcal{C}
} \end{equation}
where $\mathrm{id} \otimes c \otimes \mathrm{id}$ is the monoidal structure on $P$. Therefore we only need to prove that $\widetilde{R}(\boldsymbol{1}) = \bigl( \mathscr{A}, \lambda^{(+),(-)} \bigr)$. Let us first recall the formula for the right adjoint $R$. Note that $P(\mathscr{A}) = \int^X X^{\vee} \otimes X$ by right exactness of $P$, with universal dinatural transformation $P(j_X) : X^{\vee} \otimes X \to P(\mathscr{A})$ (Corollary \ref{coendsAsCokerAndConsequence}); $P(\mathscr{A})$ is known as the Lyubashenko coend of $\mathcal{C}$ \cite{lyu}.
\begin{lemma}\label{lemmaRightAdjointMonProduct}
The right adjoint $R : \mathcal{C} \to \mathcal{C} \boxtimes \mathcal{C}$ of $P$ is
\begin{equation}\label{rightAdjointMonoidalProd}
R(W) = (W \boxtimes \boldsymbol{1}) \otimes \mathscr{A} \cong \int^{X \in \mathcal{C}} (W \otimes X^{\vee}) \boxtimes X
\end{equation}
on objects and $R(f) = (f \boxtimes \mathrm{id}_{\boldsymbol{1}}) \otimes \mathrm{id}_{\mathscr{A}}$ on morphisms.
\end{lemma}
\begin{proof}
The unit $e : \mathrm{Id}_{\mathcal{C} \boxtimes \mathcal{C}} \Rightarrow RP$ is given by 
\begin{align}
\begin{split}\label{unitForRightAdjointMonProduct}
e_{X \boxtimes Y} : X \boxtimes Y &\xrightarrow{(\mathrm{id}_X \otimes \mathrm{coev}_Y)\boxtimes \mathrm{id}_Y} (X \otimes Y \otimes Y^{\vee}) \boxtimes Y = \bigl( (X \otimes Y) \boxtimes \boldsymbol{1} \bigr) \otimes (Y^{\vee} \boxtimes Y)\\
&\xrightarrow{\mathrm{id}_{(X \otimes Y) \boxtimes \boldsymbol{1}} \otimes j_Y} \bigl( (X \otimes Y) \boxtimes \boldsymbol{1} \bigr) \otimes \mathscr{A} = RP(X \boxtimes Y).
\end{split}
\end{align}
By item 1 in Lemma \ref{lemmaExtensionNatDeligneProd} these values entirely determine $e$. The counit $h : PR \Rightarrow \mathrm{Id}_{\mathcal{C}}$ is uniquely determined by the commutative diagram
\begin{equation}\label{counitForRightAdjointMonProduct}
\xymatrix@C=5em{
W \otimes C^{\vee} \otimes C \ar[r]^-{\mathrm{id}_W \otimes P(j_C)} \ar[dr]_{\mathrm{id}_W \otimes \mathrm{ev}_C}& W \otimes P(\mathscr{A}) \cong PR(W) \ar[d]^{\exists!\,h_W}\\
&W
} \end{equation}
for all $C, W \in \mathcal{C}$. Here we use that the dinatural transformation $\mathrm{id}_W \otimes P(j_X)$ is universal since $\otimes$ is exact in each variable. It is not difficult to check that $e$ and $h$ satisfy the zig-zag axioms of an adjunction, which proves that $P \dashv R$ by \cite[\S VI.1]{MLCat}.
\end{proof}

\begin{proof}[Proof of Theorem \ref{thmChangeFunctorMonProduct}]
By Proposition \ref{propDYcohomOfMonProductAndLinearization} we compute $H^{\bullet}_{\mathrm{DY}}(P)$ instead of $H^{\bullet}_{\mathrm{DY}}(\otimes)$. We recall that Lemma \ref{lemmaExtensionNatDeligneProd} ensures that a dinatural transformation labelled with objects in $\mathcal{C} \boxtimes \mathcal{C}$ is entirely determined by its values on the objects of the form $X \boxtimes Y$. Also recall that $P$ has the monoidal structure given by $P^{(2)}_{X_1 \boxtimes Y_1,X_2 \boxtimes Y_2} = \mathrm{id}_{X_1} \otimes c_{Y_1,X_2} \otimes \mathrm{id}_{Y_2}$.  For this proof it is more convenient to use the description of $\mathcal{Z}(\mathcal{C} \boxtimes \mathcal{C})$ and $\mathcal{Z}(P)$ as categories of modules over the monads $Z_{\mathcal{C}\boxtimes \mathcal{C}}$ and $Z_P$ defined as \eqref{defMonadZF}, see \eqref{halfBraidingFromModuleStructure}--\eqref{moduleStructureViaHalfBraiding} for this monadic description. In this proof, we will often omit the monoidal product symbol between objects in $\mathcal{C}$, that is we write $XY$ instead of $X \otimes Y$ for $X,Y \in \mathcal{C}$. We also allow ourselves to not indicate the category to which the variables in a coend notation belong to.
\\Recall that $\bigl( \mathscr{A}, \lambda^{(+),(-)} \bigr) = L^{(+),(-)}(\mathscr{A})$ with the functor $L^{(+),(-)}$ from \eqref{embeddingCCinZCC}. Now under the isomorphism $\mathcal{Z}(\mathcal{C} \boxtimes \mathcal{C}) \cong Z_{\mathcal{C}\boxtimes \mathcal{C}}\text{-}\mathrm{mod}$ the functor $L^{(+),(-)}$ becomes $C_1 \boxtimes C_2 \mapsto(C_1 \boxtimes C_2,q^{C_1 \boxtimes C_2})$ where $q^{C_1 \boxtimes C_2} : Z_{\mathcal{C} \boxtimes \mathcal{C}}(C_1 \boxtimes C_2) = \int^{X \boxtimes Y} (X^{\vee}C_1X) \boxtimes (Y^{\vee}C_2Y) \to C_1 \boxtimes C_2$ is determined by the commutative diagram (see \eqref{moduleStructureViaHalfBraiding})
\[ \xymatrix@C=11em{
\text{ \begin{tabular}{c} $(X^{\vee}C_1X) \boxtimes (Y^{\vee}C_2Y)$\\
{\footnotesize $=(X \boxtimes Y)^{\vee} \otimes (C_1 \boxtimes C_2) \otimes (X \boxtimes Y)$}\end{tabular}} \ar[r]^-{(\mathrm{id}_{X^{\vee}} \otimes c_{C_1,X}) \boxtimes (\mathrm{id}_{Y^{\vee}} \otimes c^{-1}_{Y,C_2})} \ar[d]_-{i^{\mathcal{C} \boxtimes \mathcal{C}}_{X \boxtimes Y}(C_1 \boxtimes C_2)} & (X^{\vee}XC_1) \boxtimes (Y^{\vee}YC_2) \ar[d]^{(\mathrm{ev}_X \otimes \mathrm{id}_{C_1}) \boxtimes (\mathrm{ev}_Y \otimes \mathrm{id}_{C_2})}\\
Z_{\mathcal{C} \boxtimes \mathcal{C}}(C_1 \boxtimes C_2) \ar[r]_-{\exists! \, q^{C_1 \boxtimes C_2}} & C_1 \boxtimes C_2
} \]
for all $X,Y \in \mathcal{C}$. Hence since
\begin{equation}\label{descriptionZCCACoend}
Z_{\mathcal{C} \boxtimes \mathcal{C}}(\mathscr{A}) = \int^{X \boxtimes Y} (X^{\vee} \boxtimes Y^{\vee} ) \otimes \mathscr{A} \otimes (X \boxtimes Y) = \int^{X \boxtimes Y, C} (X^{\vee} \, C^{\vee} \, X) \boxtimes (Y^{\vee} \, C \, Y)
\end{equation}
we have $L^{(+),(-)}(\mathscr{A}) = (\mathscr{A}, q^{\mathscr{A}})$ where $q^{\mathscr{A}} : Z_{\mathcal{C} \boxtimes \mathcal{C}}(\mathscr{A}) \to \mathscr{A}$ is defined by
the following commutative diagram for all $X,Y,C \in \mathcal{C}$
\begin{equation}\label{defActionqA}
\xymatrix@C=11em{
(X^{\vee}C^{\vee}X) \boxtimes (Y^{\vee}CY) \ar[r]^{(\mathrm{id}_{X^{\vee}} \otimes c_{C^{\vee},X}) \boxtimes (\mathrm{id}_{Y^{\vee}} \otimes c^{-1}_{Y,C})} \ar[d]_{i^{\mathcal{C} \boxtimes \mathcal{C}}_{X \boxtimes Y}(C^{\vee} \boxtimes C)} & (X^{\vee}XC^{\vee}) \boxtimes (Y^{\vee}YC) \ar[d]^{(\mathrm{ev}_X \otimes \mathrm{id}_{C^{\vee}}) \boxtimes (\mathrm{ev}_Y \otimes \mathrm{id}_{C})}\\
Z_{\mathcal{C} \boxtimes \mathcal{C}}(C^{\vee} \boxtimes C) \ar[r]_-{q^{C^{\vee} \boxtimes C}} \ar[d]_{Z_{\mathcal{C} \boxtimes \mathcal{C}}(j_C)} & C^{\vee} \boxtimes C \ar[d]^{j_C}\\
Z_{\mathcal{C} \boxtimes \mathcal{C}}(\mathscr{A}) \ar[r]_{\exists! \, q^{\mathscr{A}}} & \mathscr{A}
} \end{equation}
Indeed we recall that $j$ denotes the universal dinatural transformation of $\mathscr{A}$ and thus the left column in \eqref{defActionqA} is the universal dinatural transformation  of the right-hand coend in \eqref{descriptionZCCACoend}. For further use we note that \eqref{ZFonMorphisms} gives another expression for this dinatural transformation:
\begin{equation}\label{univDinatZCCA}
Z_{\mathcal{C} \boxtimes \mathcal{C}}(j_C) \circ i^{\mathcal{C} \boxtimes \mathcal{C}}_{X \boxtimes Y}(C^{\vee} \boxtimes C) = i^{\mathcal{C} \boxtimes \mathcal{C}}_{X \boxtimes Y}(\mathscr{A}) \circ \bigl( \mathrm{id}_{X^{\vee} \boxtimes Y^{\vee}} \otimes j_C \otimes \mathrm{id}_{X \boxtimes Y} \bigr).
\end{equation}
We are ready to compare $(\mathscr{A}, q^{\mathscr{A}})$ with $\widetilde{R}(\boldsymbol{1})$. Note first that the half-braiding on the unit object $\boldsymbol{1} \in \mathcal{Z}(P)$ is just the identity: $\boldsymbol{1} \otimes P(-) \xrightarrow{\:=\:} P(-) \otimes \boldsymbol{1}$. It follows from \eqref{moduleStructureViaHalfBraiding} that the corresponding $Z_P$-module structure $r_{\boldsymbol{1}} : Z_P(\boldsymbol{1}) \to \boldsymbol{1}$ is defined by $r_{\boldsymbol{1}} \circ i^P_M(\boldsymbol{1}) = \mathrm{ev}_{P(M)}$ for all $M \in \mathcal{C} \boxtimes \mathcal{C}$. Now using the definition of $\mathrm{ev}_{P(M)}$ in \eqref{dualityForImagesOfF}, which depends on the monoidal structure \eqref{monoidalStructPFromBraiding} of $P$, we see that $r_{\boldsymbol{1}}$ is uniquely determined by the commutative diagram
\begin{equation}\label{defActionZPonUnit}
\xymatrix@C=5em{
{\begin{array}{c} P(X^{\vee} \boxtimes Y^{\vee}) \otimes P(X \boxtimes Y)\\ =X^{\vee} \otimes Y^{\vee} \otimes X \otimes Y \end{array}} \ar[r]^-{\mathrm{id}_{X^{\vee}} \otimes c_{Y^{\vee},X} \otimes \mathrm{id}_Y}\ar[d]_-{i^P_{X \boxtimes Y}(\boldsymbol{1})} & 
{\begin{array}{c} P\bigl( (X^{\vee} \otimes X) \boxtimes (Y^{\vee} \otimes Y) \bigr)\\= X^{\vee} \otimes X \otimes Y^{\vee} \otimes Y \end{array}}\ar[d]^-{\mathrm{ev}_X \otimes \mathrm{ev}_Y}\\
Z_P(\boldsymbol{1})=\int^{X \boxtimes Y} P(X^{\vee} \boxtimes Y^{\vee}) \otimes P(X \boxtimes Y)  \ar[r]_-{r_ {\boldsymbol{1}}} & \boldsymbol{1}
} \end{equation}
for all $X,Y \in \mathcal{C}$. Then 
\[ \widetilde{R}(\boldsymbol{1},r_{\boldsymbol{1}}) = \bigl( R(\boldsymbol{1}), ((r_{\boldsymbol{1}} \boxtimes \mathrm{id}_{\boldsymbol{1}}) \otimes \mathrm{id}_{\mathscr{A}}) \circ \xi_{\boldsymbol{1}} \bigr) \]
by \eqref{rightAdjointLift} and by definition of $R$ on morphisms (Lemma \ref{lemmaRightAdjointMonProduct}), where $R(\boldsymbol{1}) = \mathscr{A}$ and
\[ \xi_{\boldsymbol{1}} : Z_{\mathcal{C} \boxtimes \mathcal{C}}R(\boldsymbol{1}) = Z_{\mathcal{C} \boxtimes \mathcal{C}}(\mathscr{A}) \to RZ_P(\boldsymbol{1}) \]
is defined in general in \eqref{defIsoNatXi}. We have
\[ RZ_P(\boldsymbol{1}) = \int^{X \boxtimes Y } \bigl( (X^{\vee} \, Y^{\vee} \, X \, Y) \boxtimes \boldsymbol{1} \bigr) \otimes \mathscr{A} = \int^{X \boxtimes Y, C} (X^{\vee} \, Y^{\vee} \, X \, Y \, C^{\vee}) \boxtimes C\ , \]
thus in order to describe $\xi_{\boldsymbol{1}}$ it suffices to compute $\xi_{\boldsymbol{1}} \circ i_{X \boxtimes Y}^{\mathcal{C} \boxtimes \mathcal{C}}(\mathscr{A}) \circ \bigl( \mathrm{id}_{X^{\vee} \boxtimes Y^{\vee}} \otimes j_C \otimes \mathrm{id}_{X \boxtimes Y} \bigr)$ for all $X,Y,C \in \mathcal{C}$ because $i_{X \boxtimes Y}^{\mathcal{C} \boxtimes \mathcal{C}}(\mathscr{A}) \circ \bigl( \mathrm{id}_{X^{\vee} \boxtimes Y^{\vee}} \otimes j_C \otimes \mathrm{id}_{X \boxtimes Y} \bigr)$ is the universal dinatural transformation of $Z_{\mathcal{C} \boxtimes \mathcal{C}}(\mathscr{A})$ as noted in \eqref{univDinatZCCA}. According to \eqref{descriptionXiForCentralMonads}, $\xi_{\boldsymbol{1}} \circ i_{X \boxtimes Y}^{\mathcal{C} \boxtimes \mathcal{C}}(\mathscr{A})$ is equal to
\[ R\bigl(i^P_{X \boxtimes Y}(\boldsymbol{1}) \bigr) \circ R\bigl( \mathrm{id}_{X^{\vee}  Y^{\vee}} \otimes h_{\boldsymbol{1}} \otimes \mathrm{id}_{X Y} \bigr) \circ R\bigl(P^{(3)}_{X^{\vee} \boxtimes Y^{\vee}, \mathscr{A}, X \boxtimes Y}\bigr)^{-1} \circ e_{(X^{\vee} \boxtimes Y^{\vee}) \otimes \mathscr{A} \otimes (X \boxtimes Y)} \]
where $e$ and $h$ are respectively the unit and counit of the adjunction $P \dashv R$ as defined in the proof of Lemma \ref{lemmaRightAdjointMonProduct}. Thus it remains to compute
\[ R\bigl( \mathrm{id}_{X^{\vee} Y^{\vee}} \otimes h_{\boldsymbol{1}} \otimes \mathrm{id}_{X Y} \bigr) \circ R\bigl(P^{(3)}_{X^{\vee} \boxtimes Y^{\vee}, \mathscr{A}, X \boxtimes Y}\bigr)^{-1} \circ e_{(X^{\vee} \boxtimes Y^{\vee}) \otimes \mathscr{A} \otimes (X \boxtimes Y)} \circ \bigl( \mathrm{id}_{X^{\vee} \boxtimes Y^{\vee}} \otimes j_C \otimes \mathrm{id}_{X \boxtimes Y} \bigr) \]
which we do in three steps. First by naturality of $e$, we have a commutative diagram
\[ \xymatrix@C=4.8em{
{\small \begin{array}{c} \bigl(X^{\vee} \boxtimes Y^{\vee} \bigr) \otimes (C^{\vee} \boxtimes C) \otimes (X \boxtimes Y) \\ =\bigl( X^{\vee} C^{\vee} X \bigr) \boxtimes \bigl( Y^{\vee} C Y \bigr) \end{array}} \!\!\ar[r]^-{\mathrm{id} \otimes j_C \otimes \mathrm{id}} \ar[d]_-{e_{( X^{\vee} C^{\vee} X) \boxtimes ( Y^{\vee} C Y)}} & {\small (X^{\vee} \boxtimes Y^{\vee}) \otimes \mathscr{A} \otimes (X \boxtimes Y) } \ar[d]^-{e_{(X^{\vee} \boxtimes Y^{\vee}) \otimes \mathscr{A} \otimes (X \boxtimes Y)}}\\
{\small \begin{array}{c} RP\bigl[(X^{\vee} \boxtimes Y^{\vee} ) \otimes (C^{\vee} \boxtimes C) \otimes (X \boxtimes Y)\bigr] \\ =\bigl(( X^{\vee} C^{\vee} X  Y^{\vee} C Y) \boxtimes \boldsymbol{1} \bigr) \otimes \mathscr{A} \end{array}} \!\!\ar[r]_-{RP(\mathrm{id} \otimes j_C \otimes \mathrm{id})} & {\small \begin{array}{c} RP\bigl[ (X^{\vee}  \boxtimes Y^{\vee} \bigr) \otimes \mathscr{A} \otimes (X \boxtimes Y) \bigr] \end{array}}
} \]
Hence by \eqref{unitForRightAdjointMonProduct} we find
\begin{align*}
&e_{(X^{\vee} \boxtimes Y^{\vee}) \otimes \mathscr{A} \otimes (X \boxtimes Y)} \circ \bigl( \mathrm{id}_{X^{\vee} \boxtimes Y^{\vee}} \otimes j_C \otimes \mathrm{id}_{X \boxtimes Y} \bigr)\\
=\:&RP\bigl( \mathrm{id}_{X^{\vee} \boxtimes Y^{\vee}} \otimes j_C \otimes \mathrm{id}_{X \boxtimes Y} \bigr)  \circ \bigl( \mathrm{id}_{(X^{\vee} C^{\vee} X Y^{\vee} C  Y) \boxtimes \boldsymbol{1}} \otimes j_{Y^{\vee}  C  Y} \bigr)\\
&\circ \bigl( (\mathrm{id}_{X^{\vee} C^{\vee} X} \otimes \mathrm{coev}_{Y^{\vee} C Y}) \boxtimes \mathrm{id}_{Y^{\vee} C Y} \bigr).
\end{align*}
Next by naturality of $P^{(3)}$ we have a commutative diagram
\[ \xymatrix@C=4.8em{
{\small \begin{array}{c} RP\bigl[(X^{\vee} \boxtimes Y^{\vee} ) \otimes (C^{\vee} \boxtimes C) \otimes (X \boxtimes Y)\bigr] \\ =\bigl(( X^{\vee} C^{\vee} X  Y^{\vee} C Y) \boxtimes \boldsymbol{1} \bigr) \otimes \mathscr{A} \end{array}} \!\!\ar[r]^-{RP(\mathrm{id} \otimes j_C \otimes \mathrm{id})} \ar[d]_{R\left(P^{(3)}_{X^{\vee} \boxtimes Y^{\vee}, C^{\vee} \boxtimes C, X \boxtimes Y}\right)^{-1}}& {\small \begin{array}{c} RP\bigl[ (X^{\vee} \boxtimes Y^{\vee}) \otimes \mathscr{A} \otimes (X \boxtimes Y) \bigr] \end{array}} \ar[d]^{R\left( P^{(3)}_{X^{\vee} \boxtimes Y^{\vee}, \mathscr{A}, X \boxtimes Y} \right)^{-1}}\\
{\small \begin{array}{c} R\bigl( X^{\vee}Y^{\vee}C^{\vee}CXY \bigr) \\ =\bigl(( X^{\vee}Y^{\vee}C^{\vee}CXY) \boxtimes \boldsymbol{1} \bigr) \otimes \mathscr{A} \end{array}} \ar[r]_-{R\left(\mathrm{id} \otimes P(j_C) \otimes \mathrm{id} \right)}& {\small \begin{array}{c} R\bigl( X^{\vee} Y^{\vee} P(\mathscr{A})XY \bigr)\\ =\bigl( (X^{\vee} Y^{\vee} P(\mathscr{A})XY) \boxtimes \boldsymbol{1} \bigr) \otimes \mathscr{A} \end{array}}
} \]
By definition of $P^{(3)}$ in \eqref{higherMonStruct}, by definition of the value of the functor $R$ on morphisms and since $P^{(2)} = \mathrm{id} \otimes c \otimes \mathrm{id}$ we thus obtain
\begin{align*}
&R\left( P^{(3)}_{X^{\vee} \boxtimes Y^{\vee}, \mathscr{A}, X \boxtimes Y} \right)^{-1} \circ RP\bigl( \mathrm{id}_{X^{\vee} \boxtimes Y^{\vee}} \otimes j_C \otimes \mathrm{id}_{X \boxtimes Y} \bigr)\\
=\:&R\bigl( \mathrm{id}_{X^{\vee} Y^{\vee}} \otimes P(j_C) \otimes \mathrm{id}_{X Y} \bigr) \circ \left[ \bigl( (\mathrm{id}_{X^{\vee}} \otimes c^{-1}_{Y^{\vee},C^{\vee}} \otimes \mathrm{id}_{C X Y}) \boxtimes \mathrm{id}_{\boldsymbol{1}} \bigr) \otimes \mathrm{id}_{\mathscr{A}} \right]\\
& \circ \bigl[ \bigl( (\mathrm{id}_{X^{\vee}  C^{\vee}} \otimes c^{-1}_{Y^{\vee} C,X} \otimes \mathrm{id}_Y) \boxtimes \mathrm{id}_{\boldsymbol{1}} \bigr) \otimes \mathrm{id}_{\mathscr{A}} \bigr].
\end{align*}
Finally if we apply the functor $R(X^{\vee}Y^{\vee} \otimes - \otimes XY)$ to \eqref{counitForRightAdjointMonProduct} with $W = \boldsymbol{1}$ we obtain the commutative diagram
\[ \xymatrix@C=4.5em{
R(X^{\vee}Y^{\vee}C^{\vee}CXY) \ar[r]^-{R(\mathrm{id} \otimes P(j_C) \otimes \mathrm{id})} \ar[dr]_{\!\!\!\!\!\!R(\mathrm{id}_{X^{\vee}Y^{\vee}} \otimes \mathrm{ev}_C \otimes \mathrm{id}_{XY})\qquad\qquad} & R\bigl( X^{\vee} Y^{\vee} P(\mathscr{A})XY \bigr)\ar[d]^-{R\left(\mathrm{id}_{X^{\vee}Y^{\vee}} \otimes h_{\boldsymbol{1}} \otimes \mathrm{id}_{XY}\right)} \\
& R\bigl( X^{\vee} Y^{\vee}XY \bigr) = \bigl( X^{\vee} Y^{\vee}XY \boxtimes \boldsymbol{1} \bigr) \otimes \mathscr{A}
} \]
so that
\[ R\bigl( \mathrm{id}_{X^{\vee} Y^{\vee}} \otimes h_{\boldsymbol{1}} \otimes \mathrm{id}_{X Y} \bigr) \circ R\bigl( \mathrm{id}_{X^{\vee} Y^{\vee}} \otimes P(j_C) \otimes \mathrm{id}_{X Y} \bigr) = \bigl[ \bigl( (\mathrm{id}_{X^{\vee} Y^{\vee}} \otimes \mathrm{ev}_C \otimes \mathrm{id}_{X Y}) \boxtimes \mathrm{id}_{\boldsymbol{1}} \bigr) \otimes \mathrm{id}_{\mathscr{A}} \bigr]. \]
Putting these computations altogether and using \eqref{defActionZPonUnit}, we conclude that 
\[ \bigl((r_{\boldsymbol{1}} \boxtimes \mathrm{id}_{\boldsymbol{1}}) \otimes \mathrm{id}_{\mathscr{A}} \bigr) \circ \xi_{\boldsymbol{1}} \circ i_{X \boxtimes Y}^{\mathcal{C} \boxtimes \mathcal{C}}(\mathscr{A}) \circ \bigl( \mathrm{id}_{X^{\vee} \boxtimes Y^{\vee}} \otimes j_C \otimes \mathrm{id}_{X \boxtimes Y} \bigr) : (X^{\vee} \, C^{\vee} \, X) \boxtimes (Y^{\vee} \, C \, Y) \to \mathscr{A} \]
is equal to the following composition
\begin{align*}
&\bigl[ \bigl( (\mathrm{ev}_X \otimes \mathrm{ev}_Y ) \boxtimes \mathrm{id}_{\boldsymbol{1}} \bigr) \otimes \mathrm{id}_{\mathscr{A}} \bigr] \circ \bigl[ \bigl( (\mathrm{id}_{X^{\vee}} \otimes c_{Y^{\vee},X} \otimes \mathrm{id}_Y) \boxtimes \mathrm{id}_{\boldsymbol{1}} \bigr) \otimes \mathrm{id}_{\mathscr{A}} \bigr]\\
&\circ \bigl[ \bigl( (\mathrm{id}_{X^{\vee} Y^{\vee}} \otimes \mathrm{ev}_C \otimes \mathrm{id}_{X Y}) \boxtimes \mathrm{id}_{\boldsymbol{1}} \bigr) \otimes \mathrm{id}_{\mathscr{A}} \bigr] \circ \left[ \bigl( (\mathrm{id}_{X^{\vee}} \otimes c^{-1}_{Y^{\vee},C^{\vee}} \otimes \mathrm{id}_{C X Y}) \boxtimes \mathrm{id}_{\boldsymbol{1}} \bigr) \otimes \mathrm{id}_{\mathscr{A}} \right]\\
&\circ \bigl[ \bigl( (\mathrm{id}_{X^{\vee} C^{\vee}} \otimes c^{-1}_{Y^{\vee} C,X} \otimes \mathrm{id}_Y) \boxtimes \mathrm{id}_{\boldsymbol{1}} \bigr) \otimes \mathrm{id}_{\mathscr{A}} \bigr] \circ \bigl[ \mathrm{id}_{(X^{\vee} C^{\vee} Y^{\vee} C  Y) \boxtimes \boldsymbol{1}} \otimes j_{Y^{\vee} C Y} \bigr]\\
&\circ \bigl[ (\mathrm{id}_{X^{\vee} C^{\vee} X} \otimes \mathrm{coev}_{Y^{\vee} C Y}) \boxtimes \mathrm{id}_{Y^{\vee}  C  Y} \bigr].
\end{align*}
By naturality of the monoidal product we can rewrite this as
\begin{align*}
&j_{Y^{\vee} C Y} \circ \biggl( \biggl[ \bigl( \mathrm{ev}_X \otimes \mathrm{ev}_Y \otimes \mathrm{id}_{Y^{\vee} C^{\vee}  Y^{\vee\vee}} \bigr) \circ \bigl(\mathrm{id}_{X^{\vee}} c_{Y^{\vee},X} \otimes \mathrm{id}_{Y Y^{\vee} C^{\vee} Y^{\vee\vee}}\bigr)\\
&\circ \bigl(\mathrm{id}_{X^{\vee} Y^{\vee}} \otimes \mathrm{ev}_C \otimes \mathrm{id}_{X Y Y^{\vee} C^{\vee} Y^{\vee\vee}} \bigr) \circ \bigl( \mathrm{id}_{X^{\vee}} \otimes c^{-1}_{Y^{\vee},C^{\vee}} \otimes \mathrm{id}_{C X Y Y^{\vee} C^{\vee} Y^{\vee\vee}} \bigr)\\
&\circ \bigl( \mathrm{id}_{X^{\vee} C^{\vee}} \otimes c^{-1}_{Y^{\vee} \otimes C,X} \otimes \mathrm{id}_{Y Y^{\vee} C^{\vee} Y^{\vee\vee}} \bigr)  \circ \bigl(\mathrm{id}_{X^{\vee}  C^{\vee}  X} \otimes \mathrm{coev}_{Y^{\vee} C Y} \bigr)  \biggr] \boxtimes \mathrm{id}_{Y^{\vee} C Y} \biggr).
\end{align*}
It is easy to simplify this formula if we represent it diagrammatically:
\begin{center}
\begingroup%
  \makeatletter%
  \providecommand\color[2][]{%
    \errmessage{(Inkscape) Color is used for the text in Inkscape, but the package 'color.sty' is not loaded}%
    \renewcommand\color[2][]{}%
  }%
  \providecommand\transparent[1]{%
    \errmessage{(Inkscape) Transparency is used (non-zero) for the text in Inkscape, but the package 'transparent.sty' is not loaded}%
    \renewcommand\transparent[1]{}%
  }%
  \providecommand\rotatebox[2]{#2}%
  \newcommand*\fsize{\dimexpr\f@size pt\relax}%
  \newcommand*\lineheight[1]{\fontsize{\fsize}{#1\fsize}\selectfont}%
  \ifx\svgwidth\undefined%
    \setlength{\unitlength}{371.25bp}%
    \ifx\svgscale\undefined%
      \relax%
    \else%
      \setlength{\unitlength}{\unitlength * \real{\svgscale}}%
    \fi%
  \else%
    \setlength{\unitlength}{\svgwidth}%
  \fi%
  \global\let\svgwidth\undefined%
  \global\let\svgscale\undefined%
  \makeatother%
  \begin{picture}(1,0.30738079)%
    \lineheight{1}%
    \setlength\tabcolsep{0pt}%
    \put(0,0){\includegraphics[width=\unitlength,page=1]{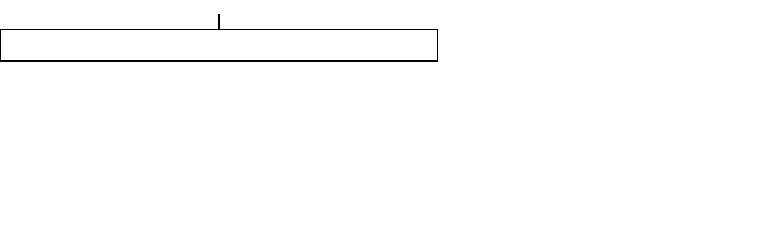}}%
    \put(0.21370556,0.24207336){\color[rgb]{0,0,0}\makebox(0,0)[lt]{\lineheight{1.25}\smash{\begin{tabular}[t]{l}$j_{Y^{\vee} \otimes C \otimes Y}$\end{tabular}}}}%
    \put(0,0){\includegraphics[width=\unitlength,page=2]{endProofThmMonProduct.pdf}}%
    \put(0.00254282,0.02641195){\color[rgb]{0,0,0}\makebox(0,0)[lt]{\lineheight{1.25}\smash{\begin{tabular}[t]{l}$_{X^{\vee}}$\end{tabular}}}}%
    \put(0.05306024,0.02641189){\color[rgb]{0,0,0}\makebox(0,0)[lt]{\lineheight{1.25}\smash{\begin{tabular}[t]{l}$_{C^{\vee}}$\end{tabular}}}}%
    \put(0.10696398,0.02623208){\color[rgb]{0,0,0}\makebox(0,0)[lt]{\lineheight{1.25}\smash{\begin{tabular}[t]{l}$_X$\end{tabular}}}}%
    \put(0,0){\includegraphics[width=\unitlength,page=3]{endProofThmMonProduct.pdf}}%
    \put(0.38214915,0.19804244){\color[rgb]{0,0,0}\makebox(0,0)[lt]{\lineheight{1.25}\smash{\begin{tabular}[t]{l}$\boxtimes$\end{tabular}}}}%
    \put(0.38243639,0.16344856){\color[rgb]{0,0,0}\makebox(0,0)[lt]{\lineheight{1.25}\smash{\begin{tabular}[t]{l}$\boxtimes$\end{tabular}}}}%
    \put(0.38220562,0.12816274){\color[rgb]{0,0,0}\makebox(0,0)[lt]{\lineheight{1.25}\smash{\begin{tabular}[t]{l}$\boxtimes$\end{tabular}}}}%
    \put(0.38206967,0.09243234){\color[rgb]{0,0,0}\makebox(0,0)[lt]{\lineheight{1.25}\smash{\begin{tabular}[t]{l}$\boxtimes$\end{tabular}}}}%
    \put(0.38198471,0.05723289){\color[rgb]{0,0,0}\makebox(0,0)[lt]{\lineheight{1.25}\smash{\begin{tabular}[t]{l}$\boxtimes$\end{tabular}}}}%
    \put(0.38158822,0.02075059){\color[rgb]{0,0,0}\makebox(0,0)[lt]{\lineheight{1.25}\smash{\begin{tabular}[t]{l}$\boxtimes$\end{tabular}}}}%
    \put(0,0){\includegraphics[width=\unitlength,page=4]{endProofThmMonProduct.pdf}}%
    \put(0.42797371,0.01881202){\color[rgb]{0,0,0}\makebox(0,0)[lt]{\lineheight{1.25}\smash{\begin{tabular}[t]{l}$_{Y^{\vee}}$\end{tabular}}}}%
    \put(0.47486154,0.01869155){\color[rgb]{0,0,0}\makebox(0,0)[lt]{\lineheight{1.25}\smash{\begin{tabular}[t]{l}$_C$\end{tabular}}}}%
    \put(0.52728794,0.01893093){\color[rgb]{0,0,0}\makebox(0,0)[lt]{\lineheight{1.25}\smash{\begin{tabular}[t]{l}$_Y$\end{tabular}}}}%
    \put(0.27245406,0.29611445){\color[rgb]{0,0,0}\makebox(0,0)[lt]{\lineheight{1.25}\smash{\begin{tabular}[t]{l}$_{\mathscr{A}}$\end{tabular}}}}%
    \put(0,0){\includegraphics[width=\unitlength,page=5]{endProofThmMonProduct.pdf}}%
    \put(0.81434964,0.19382252){\color[rgb]{0,0,0}\makebox(0,0)[lt]{\lineheight{1.25}\smash{\begin{tabular}[t]{l}$j_C$\end{tabular}}}}%
    \put(0.80682875,0.14687755){\color[rgb]{0,0,0}\makebox(0,0)[lt]{\lineheight{1.25}\smash{\begin{tabular}[t]{l}$\boxtimes$\end{tabular}}}}%
    \put(0.80711602,0.11228366){\color[rgb]{0,0,0}\makebox(0,0)[lt]{\lineheight{1.25}\smash{\begin{tabular}[t]{l}$\boxtimes$\end{tabular}}}}%
    \put(0.80688523,0.07699785){\color[rgb]{0,0,0}\makebox(0,0)[lt]{\lineheight{1.25}\smash{\begin{tabular}[t]{l}$\boxtimes$\end{tabular}}}}%
    \put(0.87241811,0.06931705){\color[rgb]{0,0,0}\makebox(0,0)[lt]{\lineheight{1.25}\smash{\begin{tabular}[t]{l}$_{Y^{\vee}}$\end{tabular}}}}%
    \put(0.91930596,0.06919661){\color[rgb]{0,0,0}\makebox(0,0)[lt]{\lineheight{1.25}\smash{\begin{tabular}[t]{l}$_C$\end{tabular}}}}%
    \put(0.97173234,0.06943596){\color[rgb]{0,0,0}\makebox(0,0)[lt]{\lineheight{1.25}\smash{\begin{tabular}[t]{l}$_Y$\end{tabular}}}}%
    \put(0.81790863,0.24560939){\color[rgb]{0,0,0}\makebox(0,0)[lt]{\lineheight{1.25}\smash{\begin{tabular}[t]{l}$_{\mathscr{A}}$\end{tabular}}}}%
    \put(0,0){\includegraphics[width=\unitlength,page=6]{endProofThmMonProduct.pdf}}%
    \put(0.66029696,0.06931708){\color[rgb]{0,0,0}\makebox(0,0)[lt]{\lineheight{1.25}\smash{\begin{tabular}[t]{l}$_{X^{\vee}}$\end{tabular}}}}%
    \put(0.70718479,0.06919664){\color[rgb]{0,0,0}\makebox(0,0)[lt]{\lineheight{1.25}\smash{\begin{tabular}[t]{l}$_{C^{\vee}}$\end{tabular}}}}%
    \put(0.75961118,0.06943599){\color[rgb]{0,0,0}\makebox(0,0)[lt]{\lineheight{1.25}\smash{\begin{tabular}[t]{l}$_X$\end{tabular}}}}%
    \put(0,0){\includegraphics[width=\unitlength,page=7]{endProofThmMonProduct.pdf}}%
    \put(0.58934388,0.17754906){\color[rgb]{0,0,0}\makebox(0,0)[lt]{\lineheight{1.25}\smash{\begin{tabular}[t]{l}$=$\end{tabular}}}}%
  \end{picture}%
\endgroup%

\end{center}
where we used the zig-zag axiom for ev/coev and dinaturality of $j$. Comparing this result with \eqref{defActionqA} and using \eqref{univDinatZCCA} we see that
\begin{align*}
&\bigl((r_{\boldsymbol{1}} \boxtimes \mathrm{id}_{\boldsymbol{1}}) \otimes \mathrm{id}_{\mathscr{A}} \bigr) \circ \xi_{\boldsymbol{1}} \circ i_{X \boxtimes Y}^{\mathcal{C} \boxtimes \mathcal{C}}(\mathscr{A}) \circ \bigl( \mathrm{id}_{X^{\vee} \boxtimes Y^{\vee}} \otimes j_C \otimes \mathrm{id}_{X \boxtimes Y} \bigr)\\
=\:&q^{\mathscr{A}} \circ i_{X \boxtimes Y}^{\mathcal{C} \boxtimes \mathcal{C}}(\mathscr{A}) \circ \bigl( \mathrm{id}_{X^{\vee} \boxtimes Y^{\vee}} \otimes j_C \otimes \mathrm{id}_{X \boxtimes Y} \bigr)
\end{align*}
for all $X \boxtimes Y \in \mathcal{C} \boxtimes \mathcal{C}$ and $C \in \mathcal{C}$. Since the dinatural transformation $\bigl( i_M^{\mathcal{C} \boxtimes \mathcal{C}}(\mathscr{A}) \circ \bigl( \mathrm{id}_{M^{\vee}} \otimes j_C \otimes \mathrm{id}_M \bigr)  \bigr)_{C \in \mathcal{C}, M \in \mathcal{C} \boxtimes \mathcal{C}}$ is universal we conclude that $\bigl((r_{\boldsymbol{1}} \boxtimes \mathrm{id}_{\boldsymbol{1}}) \otimes \mathrm{id}_{\mathscr{A}} \bigr) \circ \xi_{\boldsymbol{1}} = q^{\mathscr{A}}$.
\end{proof}

Let us explain why the isomorphism in Theorem \ref{thmChangeFunctorMonProduct} is useful in practice. {\em From now on and until the end of this subsection, assume that the ground field $\Bbbk$ is perfect.} Recall first a well-known fact:
\begin{lemma}\label{lemmaDrinfeldCenterOfDeligneProduct}
Let $\mathcal{C}, \mathcal{D}$ be finite $\Bbbk$-linear tensor categories. Then the functor
\begin{align*}
\mathcal{Z}(\mathcal{C}) \boxtimes \mathcal{Z}(\mathcal{D}) &\to \mathcal{Z}(\mathcal{C} \boxtimes \mathcal{D})\\
(V,\rho^V) \boxtimes (W,\rho^W) &\mapsto (V \boxtimes W, \rho^V \boxtimes \rho^W) \quad \text{with }\: (\rho^V \boxtimes \rho^W)_{X \boxtimes Y} = \rho^V_X \boxtimes \rho^W_Y
\end{align*}
is an equivalence of categories.
\end{lemma}
\begin{proof}
We use the isomorphism $\mathcal{Z}(\mathcal{C}) \cong Z_{\mathcal{C}}\text{-}\mathrm{mod}$, see \S\ref{relExtGroupsDYCohomology}. By item 3 in Lemma \ref{lemmaExtensionNatDeligneProd}, see also~\eqref{eq:coend-Deligne}, and exactness of $\boxtimes$ in each variable, we have
\begin{align*}
&Z_{\mathcal{C}\boxtimes \mathcal{D}}(V \boxtimes W) = \int^{M \in \mathcal{C} \boxtimes \mathcal{D}} M^{\vee} \otimes (V \boxtimes W) \otimes M = \int^{X \in \mathcal{C}, Y \in \mathcal{D}} (X^{\vee} \otimes V \otimes X) \boxtimes (Y^{\vee} \otimes W \otimes Y)\\
&\cong\left( \int^{X \in \mathcal{C}} X^{\vee} \otimes V \otimes X \right) \boxtimes \left( \int^{Y \in \mathcal{D}} Y^{\vee} \otimes W \otimes Y\right) = Z_{\mathcal{C}}(V) \boxtimes Z_{\mathcal{D}}(W) = (Z_{\mathcal{C}} \boxtimes Z_{\mathcal{D}})(V \boxtimes W).
\end{align*}
It follows that $Z_{\mathcal{C}\boxtimes \mathcal{D}} = Z_{\mathcal{C}} \boxtimes Z_{\mathcal{D}}$. We conclude thanks to Corollary \ref{coroDeligneProdOfMonads} (which uses the assumption on the ground field $\Bbbk$ to be perfect):
\[ \mathcal{Z}(\mathcal{C}) \boxtimes \mathcal{Z}(\mathcal{D}) \cong (Z_{\mathcal{C}}\text{-}\mathrm{mod}) \boxtimes (Z_{\mathcal{D}}\text{-}\mathrm{mod}) \cong (Z_{\mathcal{C}} \boxtimes Z_{\mathcal{D}})\text{-}\mathrm{mod} = Z_{\mathcal{C} \boxtimes \mathcal{D}}\text{-}\mathrm{mod} \cong \mathcal{Z}(\mathcal{C} \boxtimes \mathcal{D}). \]
It is easy to see that the composition of these equivalences equals the proposed functor.
\end{proof}

\noindent Under the equivalence of Lemma \ref{lemmaDrinfeldCenterOfDeligneProduct}, the object from Theorem \ref{thmChangeFunctorMonProduct} becomes
\[ \mathcal{Z}(\mathcal{C} \boxtimes \mathcal{C}) \ni \bigl( \mathscr{A}, \lambda^{(+),(-)} \bigr) \mapsto \int^{X \in \mathcal{C}} (X^{\vee}, c_{X^{\vee},-}) \boxtimes (X,c_{-,X}^{-1}) \in \mathcal{Z}(\mathcal{C}) \boxtimes \mathcal{Z}(\mathcal{C}) \]
and then (recall Theorem \ref{thmDYCohomRelExt}):
\begin{equation}\label{DYCohomOtimesInTermsOfSeparatedRelExt}
H^{\bullet}_{\mathrm{DY}}(\otimes, \mathrm{id} \otimes c \otimes \mathrm{id}) \cong \mathrm{Ext}^{\bullet}_{\mathcal{Z}(\mathcal{C}) \boxtimes \mathcal{Z}(\mathcal{C}), \mathcal{C} \boxtimes \mathcal{C}}\left(\boldsymbol{1} \boxtimes \boldsymbol{1}, \int^{X \in \mathcal{C}} (X^{\vee}, c_{X^{\vee},-}) \boxtimes (X,c_{-,X}^{-1}) \right).
\end{equation}
These are relative Ext spaces associated to a Deligne product of adjunctions. Such adjunctions are discussed in Appendix \ref{subsectionDeligneResolventPairs}. In particular, by item 1 in Proposition \ref{propDeligneProductOfResolutions}, {\em one can deduce a relatively projective resolution of $\boldsymbol{1} \boxtimes \boldsymbol{1} \in \mathcal{Z}(\mathcal{C}) \boxtimes \mathcal{Z}(\mathcal{C})$ from a relatively projective resolution of $\boldsymbol{1} \in \mathcal{Z}(\mathcal{C})$}.

\begin{remark}
There is also a less straightforward and somewhat shorter proof of Lemma \ref{lemmaDrinfeldCenterOfDeligneProduct} using a result of Laugwitz-Walton \cite[Th.\,4.17]{LW} which states that a factorizable category $\mathcal{A}$ containing a factorizable topologizing (\textit{i.e.}\ closed under direct sums and taking subquotients) subcategory $\mathcal{E}$ is necessarily braided tensor equivalent to $\mathcal{E} \boxtimes \mathcal{X}$ where $\mathcal{X}$ is the Mueger’s centralizer of $\mathcal{E}$ and it is also factorizable. In their paper the ground field $\Bbbk$ is assumed to be algebraically closed (and thus in particular perfect). We apply this theorem to $\mathcal{A} = \mathcal{Z}(\mathcal{C} \boxtimes \mathcal{D})$ which contains mutually transparent copies of $\mathcal{Z}(\mathcal{C})$ and $\mathcal{Z}(\mathcal{D})$, therefore $\mathcal{Z}(\mathcal{C} \boxtimes \mathcal{D})$ is braided tensor equivalent to $\mathcal{Z}(\mathcal{C}) \boxtimes \mathcal{Z}(\mathcal{D}) \boxtimes \mathcal{X}$ for some factorizable $\mathcal{X}$. Comparing the Frobenius--Perron dimensions of both categories we conclude that $\mathcal{Z}(\mathcal{C} \boxtimes \mathcal{D})$ is braided tensor equivalent to $\mathcal{Z}(\mathcal{C}) \boxtimes \mathcal{Z}(\mathcal{D})$.
\end{remark}

Finally, by combining \eqref{dimTangentSpaceBraidingInTermsDYCohom}, \eqref{DYCohomOtimesInTermsOfSeparatedRelExt} and \eqref{isoDYRelExt} we find:
\begin{corollary}\label{coroDimTcBrCInTermsOfRelExt}
Let $\mathcal{C}$ be a finite $\Bbbk$-linear tensor category over a perfect field $\Bbbk$ and $c \in \mathrm{Br}(\mathcal{C})$ be a braiding on $\mathcal{C}$. Then with 
\begin{equation}\label{eq:Gamma-coef}
\Gamma = \int^{X \in \mathcal{C}} (X^{\vee}, c_{X^{\vee},-}) \boxtimes (X,c_{-,X}^{-1}) \in \mathcal{Z}(\mathcal{C}) \boxtimes \mathcal{Z}(\mathcal{C})
\end{equation}
we have
\begin{equation}
    \label{eq:dim-Tc-Br-Ext}
 \dim \mathbf{T}_c\mathrm{Br}(\mathcal{C}) = \dim \Ext^2_{\mathcal{Z}(\mathcal{C}) \boxtimes \mathcal{Z}(\mathcal{C}), \mathcal{C} \boxtimes \mathcal{C}}(\boldsymbol{1} \boxtimes \boldsymbol{1},\Gamma) - 2 \dim \Ext^2_{\mathcal{Z}(\mathcal{C}),\mathcal{C}}(\boldsymbol{1},\boldsymbol{1})\ .
\end{equation}
It is important in practice to note that both relative Ext spaces can be computed from a relatively projective resolution of $\boldsymbol{1} \in \mathcal{Z}(\mathcal{C})$, thanks to item 1 in Proposition \ref{propDeligneProductOfResolutions}.
\end{corollary}

\subsection{The end formula for tangent space to a braiding}\label{sec:end-formula}
Here, we use  the K\"unneth formula in item 1 of Prop.~\ref{propDeligneProductOfResolutions} to
  rewrite the first term in~\eqref{eq:dim-Tc-Br-Ext} in terms of the `standard' adjunction between $\zcat$ and $\cat$, i.e.\ involving only relative $\Ext^n_{\zcat,\cat}$.
  
   We first recall that the  Nakayama functor $\Nak : \mathcal{C} \to \mathcal{C}$ on a finite $\Bbbk$-linear  category $\mathcal{C}=A\text{-}\mathrm{mod}$, for a finite dimensional $\Bbbk$-algebra $A$, is given by $\Nak = A^*\otimes _A -$ where $A^*$ is the co-regular $A$-bimodule.  We note that  up to isomorphism $\Nak$ does not depend on the algebra~$A$ realising $\mathcal{C}$ \cite{FSS}.
    Following~\cite[eq.\,(3.52)]{FSS}, we then have  a lemma turning a coend to an end:

\begin{lemma}[\cite{FSS}]
\label{lem:coend-to-end}
Let $\mathcal{C}$ be a finite $\Bbbk$-linear  category and  $\Nak : \mathcal{C} \to \mathcal{C}$ be the Nakayama functor. We have the following isomorphism in $\mathcal{C}^{\mathrm{op}} \boxtimes \mathcal{C}$:
\begin{equation}\label{eq:coend-to-end}
 \int^{X \in \mathcal{C}} \bar{X} \boxtimes X \cong  \int_{X \in \mathcal{C}} \bar{X} \boxtimes \Nak(X)\ ,
\end{equation}
where by $\bar{X}$ we denote the object $X$ considered in the opposite category $\mathcal{C}^{\mathrm{op}}$.
\end{lemma}

If the category $\cat$ is  furthermore rigid monoidal, we have a natural isomorphism~\cite[Lems.\,4.10\,\&\,4.11]{FSS}:
\begin{equation}\label{eq:Nak-object}
\Nak \cong \dD^\vee \otimes (-)^{\vee\vee}\ ,
\end{equation}
where $\dD$ is the \textit{distinguished invertible} object of $\cat$ defined as the socle of the projective cover $P_{\one}$ of $\one$.
 We call $\cat$ \textit{unimodular} if $\dD\cong \one$. Recall also that $\cat$ is pivotal if the double dual functor $(-)^{\vee\vee}$ is monoidally isomorphic to the identity functor. 
 Furthermore, using~\cite[Prop.\,3.24\,(ii) \& Thm.\,4.14]{FSS}, see also~\cite{Hu} in the Hopf algebra case, one can show that the Nakayama functor $\Nak$ is isomorphic to the identity functor if and only if the finite tensor category~$\mathcal{C}$ is pivotal and unimodular. 
 
By Lemma~\ref{lem:coend-to-end}, applying the equivalence functor $(-)^{\vee}\boxtimes \mathrm{Id}_{\mathcal{C}} :  \mathcal{C}^{\mathrm{op}} \boxtimes \mathcal{C} \to \mathcal{C} \boxtimes \mathcal{C}$ to both sides of~\eqref{eq:coend-to-end}, we see that the object $\mathscr{A}$  in~\eqref{defCoendBulkA}
can be rewritten as the following end
\begin{equation}\label{eq:A-end}
\mathscr{A}  \cong \int_{X \in \mathcal{C}} X^{\vee} \boxtimes \Nak(X)\ ,
\end{equation}
with $\Nak$ given in~\eqref{eq:Nak-object}.
In the case $\cat$ is pivotal and unimodular,  $\mathscr{A}  \cong \int_{X \in \mathcal{C}} X^{\vee} \boxtimes X$.

\begin{proposition}\label{prop:HDY-end}
Let $\mathcal{C}$ be a finite $\Bbbk$-linear tensor category over a perfect field $\Bbbk$ and with a braiding $c \in \mathrm{Br}(\mathcal{C})$, and   $\Nak : \mathcal{C} \to \mathcal{C}$ be the Nakayama functor. Then 
\begin{equation}\label{eq:HDY-end}
H^{n}_{\mathrm{DY}}(\otimes, \mathrm{id} \otimes c \otimes \mathrm{id})
 \cong   \bigoplus_{i+j=n} \int_{X \in \mathcal{C}}  \Ext^i_{\mathcal{Z}(\mathcal{C}),\mathcal{C}}\bigl(\boldsymbol{1},(X^{\vee}, c_{X^{\vee},-})\bigr) \otimes \Ext^j_{\mathcal{Z}(\mathcal{C}),\mathcal{C}}\Bigl(\boldsymbol{1},\bigl(\Nak(X),c_{-,\Nak(X)}^{-1}\bigr) \Bigr) \ ,
 \end{equation}
 with $\Nak$ from~\eqref{eq:Nak-object}.
In particular, if $\mathcal{C}$ is pivotal and unimodular, we have
\begin{align}
\begin{split}\label{eq:HDY2-end}
H^{2}_{\mathrm{DY}}(\otimes, \mathrm{id} \otimes c \otimes \mathrm{id})
 \cong   &\int_{X \in \mathcal{C}}  \Hom_{\mathcal{Z}(\mathcal{C})}\bigl(\boldsymbol{1},(X^{\vee}, c_{X^{\vee},-})\bigr) \otimes \Ext^2_{\mathcal{Z}(\mathcal{C}),\mathcal{C}}\bigl(\boldsymbol{1},(X,c_{-,X}^{-1}) \bigr) \\
 &\oplus \int_{X \in \mathcal{C}}  \Ext^2_{\mathcal{Z}(\mathcal{C}),\mathcal{C}}\bigl(\boldsymbol{1},(X^{\vee}, c_{X^{\vee},-})\bigr) 
 \otimes \Hom_{\mathcal{Z}(\mathcal{C})}\bigl(\boldsymbol{1},(X,c_{-,X}^{-1}) \bigr)\\
  &\oplus \int_{X \in \mathcal{C}}  \Ext^1_{\mathcal{Z}(\mathcal{C}),\mathcal{C}}\bigl(\boldsymbol{1},(X^{\vee}, c_{X^{\vee},-})\bigr) 
 \otimes \Ext^1_{\mathcal{Z}(\mathcal{C}),\mathcal{C}}\bigl(\boldsymbol{1},(X,c_{-,X}^{-1}) \bigr)\ .
 \end{split}
 \end{align}
\end{proposition}
\begin{proof}
We first recall Remark~\ref{rem:L-functor} stating that the DY cohomology coefficient $\Gamma$ from~\eqref{eq:Gamma-coef} is the image of $\mathscr{A}$ under the exact  functor $L^{(+),(-)}$ from 
$\mathcal{C}\boxtimes \mathcal{C}$ to
 $\mathcal{Z}(\mathcal{C})\boxtimes \mathcal{Z}(\mathcal{C})$
which sends the first $\boxtimes$ component using the braiding and the second component using the inverse braiding. It is furthermore easy to check that $L^{(+),(-)}$ is a strict tensor functor.
Then using~\eqref{eq:A-end}, $\Gamma$ can be rewritten as
\begin{equation}
\Gamma \cong \int_{X \in \mathcal{C}} (X^{\vee}, c_{X^{\vee},-}) \boxtimes \bigl(\Nak(X),c_{-,\Nak(X)}^{-1}\bigr)\ .
\end{equation}
 Combining with~\eqref{DYCohomOtimesInTermsOfSeparatedRelExt} and that an end can be pulled out of 2nd argument of the $\Hom$ functor and thus out of 2nd argument of the relative $\mathrm{Ext}$ functor, we get
 \begin{multline}
H^{n}_{\mathrm{DY}}(\otimes, \mathrm{id} \otimes c \otimes \mathrm{id})
 \cong  \int_{X \in \mathcal{C}} \mathrm{Ext}^{n}_{\mathcal{Z}(\mathcal{C}) \boxtimes \mathcal{Z}(\mathcal{C}), \mathcal{C} \boxtimes \mathcal{C}}\left(\boldsymbol{1} \boxtimes \boldsymbol{1}, 
 (X^{\vee}, c_{X^{\vee},-}) \boxtimes \bigl(\Nak(X),c_{-,\Nak(X)}^{-1}\bigr) \right)\\
 \cong   \bigoplus_{i+j=n} \int_{X \in \mathcal{C}}  \Ext^i_{\mathcal{Z}(\mathcal{C}),\mathcal{C}}\bigl(\boldsymbol{1},(X^{\vee}, c_{X^{\vee},-})\bigr) \otimes \Ext^j_{\mathcal{Z}(\mathcal{C}),\mathcal{C}}\Bigl(\boldsymbol{1},\bigl(\Nak(X),c_{-,\Nak(X)}^{-1}\bigr) \Bigr) \ ,
 \end{multline}
 where in the second line we used  the K\"unneth formula (item 2 in Prop.\,\ref{propDeligneProductOfResolutions}). This gives~\eqref{eq:HDY-end}, while~\eqref{eq:HDY2-end} is an immediate consequence of this result and that $\Nak$ is isomorphic to the identity when $\mathcal{C}$ is pivotal and unimodular, recall~\eqref{eq:Nak-object}.
\end{proof}

Let $\mathrm{Proj}(\mathcal{C})$ be the full subcategory of projective objects in $\mathcal{C}$, then by a version of the statement in item 2.\ of Lemma~\ref{lemNatTransfoProjGen} where the projective generator is replaced by $\mathrm{Proj}(\mathcal{C})$  we have an isomorphism
$$
\int^{X\in\mathcal{C}} K(X,X) \cong \int^{P\in\mathrm{Proj}(\mathcal{C})} K(P,P)
$$
where $K : \mathcal{C}^{\mathrm{op}} \times \mathcal{C} \to \mathcal{C}\boxtimes \mathcal{C}$ is any $\Bbbk$-bilinear functor right exact in the second variable. Now, for the choice $K = \boxtimes \circ \bigl((-)^{\vee}\times \mathrm{Id}_{\mathcal{C}}\bigr)$ we have that
\begin{equation}\label{eq:bulk-alg-end-proj}
\mathscr{A} \cong  \int^{P\in\mathrm{Proj}(\mathcal{C})} P^{\vee}\boxtimes P \cong
 \int_{P\in\mathrm{Proj}(\mathcal{C})} P^{\vee}\boxtimes \Nak(P)\ .
\end{equation}

 Let $P_X$ denote the projective cover of a simple object $X$ in $\mathcal{C}$. Recall  that the socle of the projective cover $\projone$ is  isomorphic to the distinguished invertible object $\dD$, or equivalently we have
\begin{equation}\label{eq:P1-dual-D}
\projone^{\vee} \cong P_{\dD^{\vee}}\ .
\end{equation}
It is clear that the socle of $\projone^{\vee}$ is given by $\one^{\vee}\cong \one$. Using that $\projone^{\vee\vee}\cong \projone$, no pivotality is needed here, we see from~\eqref{eq:Nak-object} that $\Nak(\projone)\cong \dD^{\vee}\otimes \projone \cong P_{\dD^{\vee}}\cong \projone^{\vee}$. Combining these facts, we get
\begin{equation}\label{eq:hom-P-NP}
\Hom_{\cat}(\one,P_X^{\vee}) \cong \Bbbk \, \delta_{X= \one} \cong  \Hom_{\cat}(\one,\Nak(P_X)) \ .
\end{equation}
Recall furthermore that for braided $\cat$ the functor $L^{\pm}: \cat \to \zcat$ associating for half-braiding the braiding
 in $\cat$ or its inverse  is faithful, while its image is a full subcategory. Using~\eqref{eq:hom-P-NP} we therefore get 
\begin{equation}\label{eq:homZC-P-NP}
\Hom_{\zcat}\bigl(\one,(P_X^{\vee},c_{P_X^{\vee},-})\bigr) \cong \Bbbk \, \delta_{X= \one} 
\cong  \Hom_{\zcat}\bigl(\one,(\Nak(P_X),c^{-1}_{-,\Nak(P_X)})\bigr) \ .
\end{equation}
Together with~\eqref{eq:bulk-alg-end-proj}, we thus get the following corollary of Prop.~\ref{prop:HDY-end}: 
\begin{corollary}\label{cor:HDY-end-proj}
Let $\mathcal{C}$ be a finite $\Bbbk$-linear tensor category over a perfect field $\Bbbk$  with a braiding $c \in \mathrm{Br}(\mathcal{C})$,
and $\dD$ is the distinguished invertible object of $\cat$.
Then $H^{n}_{\mathrm{DY}}(\otimes, \mathrm{id} \otimes c \otimes \mathrm{id})$ is given by replacing $\int_{X\in\cat}$ by $\int_{X\in \mathrm{Proj}(\mathcal{C})}$ in~\eqref{eq:HDY-end}. In particular, using~\eqref{eq:homZC-P-NP}, \eqref{eq:Nak-object} and~\eqref{eq:P1-dual-D}, we get
 \begin{multline}\label{eq:HDY2-end-proj}
H^{2}_{\mathrm{DY}}(\otimes, \mathrm{id} \otimes c \otimes \mathrm{id})
 \cong    \Ext^2_{\mathcal{Z}(\mathcal{C}),\mathcal{C}}\bigl(\boldsymbol{1},(P_{\dD^\vee},c_{P_{\dD^{\vee}},-})\oplus (P_{\dD^\vee},c^{-1}_{-,P_{\dD^{\vee}}}) \bigr) \\
\oplus \int_{P \in \mathrm{Proj}(\mathcal{C})}  \Ext^1_{\mathcal{Z}(\mathcal{C}),\mathcal{C}}\bigl(\boldsymbol{1},(P^{\vee}, c_{P^{\vee},-})\bigr) 
 \otimes \Ext^1_{\mathcal{Z}(\mathcal{C}),\mathcal{C}}\bigl(\boldsymbol{1},(\dD^\vee\otimes P^{\vee\vee},c^{-1}_{-,\dD^\vee\otimes P^{\vee\vee}}) \bigr)\ .
 \end{multline}
Finally for pivotal $\cat$, the end formula for dimension of the tangent space to the braiding $c$ is
\begin{multline}
    \label{eq:dim-Tc-Br-Ext-end}
\dim \mathbf{T}_c\mathrm{Br}(\mathcal{C}) =
\dim \int_{P \in \mathrm{Proj}(\mathcal{C})}  \Ext^1_{\mathcal{Z}(\mathcal{C}),\mathcal{C}}\bigl(\boldsymbol{1},(P^{\vee}, c_{P^{\vee},-})\bigr) 
 \otimes \Ext^1_{\mathcal{Z}(\mathcal{C}),\mathcal{C}}\bigl(\boldsymbol{1},(\dD^\vee \otimes P,c^{-1}_{-,\dD^\vee \otimes P}) \bigr)\\
+ 
\dim \Ext^2_{\mathcal{Z}(\mathcal{C}),\mathcal{C}}\bigl(\boldsymbol{1},(P_{\dD^\vee},c_{P_{\dD^{\vee}},-})\oplus (P_{\dD^\vee},c^{-1}_{-,P_{\dD^{\vee}}}) \bigr) \\
- 2 \dim \Ext^2_{\mathcal{Z}(\mathcal{C}),\mathcal{C}}(\boldsymbol{1},\one) \ .
\end{multline} 
\end{corollary}

 We note that while the $P$'s involved in the relative $\Ext$'s in~\eqref{eq:HDY2-end-proj}-\eqref{eq:dim-Tc-Br-Ext-end} are projective in $\mathcal{C}$, their images $(P,c_{P,-}) \in \mathcal{Z}(\mathcal{C})$ are in general not projective neither relative projective.

\section{Finite-dimensional Hopf algebras}\label{sectionFinDimHopf}
Let $H$ be a finite-dimensional Hopf $\Bbbk$-algebra where $\Bbbk$ is a field, with unit $1_H$, coproduct $\Delta : H \to H \otimes H$, counit $\varepsilon : H \to \Bbbk$ and antipode $S : H \to H$. Then the category $H\text{-}\mathrm{mod}$ of finite-dimensional $H$-modules is a finite tensor category.

\indent Our first goal is to specialize Theorem \ref{thmChangeFunctorMonProduct} to $\mathcal{C} = H\text{-}\mathrm{mod}$ by describing the object $\bigl(\mathscr{A},\lambda^{(+),(-)}\bigr)$ as a module over $D(H)^{\otimes 2}$ where $D(H)$ is the Drinfeld double, see Proposition~\ref{propAdjunctionTheoremForMonProductInAmod}. Moreover, spaces of infinitesimal braidings in $H\text{-}\mathrm{mod}$ are isomorphic to Zariski tangent spaces of the affine variety of $R$-matrices in $H^{\otimes 2}$. Therefore, the dimension formula~\eqref{eq:dim-Tc-Br-Ext} and the end formulas in \S\ref{sec:end-formula} give dimension formulas for these Zariski tangent spaces, see \eqref{dimensionFormulaForTRMat} and Remark \ref{rem:end-ingr-Hopf}. This allows us to easily compute the dimension of tangent spaces for the example $H = B_k := \Lambda\mathbb{C}^k \rtimes \mathbb{C}[\mathbb{Z}/2\mathbb{Z}]$, see  \S\ref{subsectionExampleBk}.

\indent Our second goal, achieved in \S\ref{subsectionDYTwistedFiber}, is to make explicit the Adjunction Theorem \ref{thmChangeOfCoeffDY} in the case of restriction functors $H\text{-}\mathrm{mod} \to K\text{-}\mathrm{mod}$ for a (possibly twisted) Hopf subalgebra $K$. As an example we compute in \S\ref{subsectionResFunctorBk} the DY cohomology of $B_{k+l}\text{-}\mathrm{mod} \to B_k\text{-}\mathrm{mod}$.

\smallskip

\indent  The following notations and facts will be heavily used in the sequel. First, it is well-known that the antipode is an anti-isomorphism of bialgebras. We use Sweedler's notation without summation sign: $\Delta(h) = h^{(1)} \otimes h^{(2)}$. As usual we write $h^{(1)} \otimes h^{(2)} \otimes h^{(3)}$ instead of $h^{(1)(1)} \otimes h^{(1)(2)} \otimes h^{(2)} = h^{(1)} \otimes h^{(2)(1)} \otimes h^{(2)(2)}$ {\it etc}. The {\em coregular actions} $\triangleright$, $\triangleleft$ of $H$ on the dual vector space $H^*$ are defined by
\begin{equation}\label{defCoregular}
\forall \, h,h' \in H, \:\: \forall\, f \in H^*, \quad (h \triangleright f)(h') = f(h'h), \quad (f \triangleleft h)(h') = f(hh').
\end{equation}

 We denote by $(H^*)^{\mathrm{op}}$ the vector space $H^*$ endowed with the product $\varphi \psi = (\psi \otimes \varphi) \circ \Delta$ and the coproduct defined by $\varphi_{(1)}(x)\varphi_{(2)}(y) = \varphi(xy)$. The {\em Drinfeld double} of $H$ is the vector space $D(H) = (H^*)^{\mathrm{op}} \otimes H$ endowed with the product defined by the following conditions. The subspaces $(H^*)^{\mathrm{op}} \otimes 1$ and $\varepsilon \otimes H$ are subalgebras of $D(H)$; hence we simply write $\varphi$ (resp. $h$) instead of $\varphi \otimes 1$ (resp. $\varepsilon \otimes h$). Then we have $\varphi h = \varphi \otimes h$ and
\begin{equation}\label{definingRelDrinfeldDouble}
\forall \, h \in H, \: \forall \varphi \in H^{*\mathrm{op}}, \quad h\varphi = \bigl(h^{(3)} \triangleright \varphi \triangleleft S(h^{(1)}) \bigr) \, h^{(2)}.
\end{equation}

 There is a well-known isomorphism of categories $\mathcal{Z}(H\text{-}\mathrm{mod}) \cong D(H)\text{-}\mathrm{mod}$: if $(V,\rho) \in \mathcal{Z}(H\text{-}\mathrm{mod})$ then the $H$-module $V$ can be promoted to a $D(H)$-module by letting $\varphi \cdot v = (\varphi \otimes \mathrm{id}_V) \circ \rho_H(v \otimes 1_H)$ for all $\varphi \in H^{*\mathrm{op}}$ and $v \in V$. Conversely if $V$ is a $D(H)$-module then a half-braiding $\rho : V \otimes - \Rightarrow - \otimes V$ is defined by $\rho_X(v \otimes x) = h_i \cdot x \otimes h^i \cdot v$ where $(h_i)$ is a basis of $H$ with dual basis $(h^i)$.

\indent If $A$ is an associative $\Bbbk$-algebra and $B \subset A$ is a subalgebra we have a restriction functor $\mathrm{Res}^A_B : A\text{-}\mathrm{mod} \to B\text{-}\mathrm{mod}$ and an induction functor $\mathrm{Ind}^A_B : B\text{-}\mathrm{mod} \to A\text{-}\mathrm{mod}$, $\mathrm{Ind}^A_B(X) = A \otimes_B X$, where $\text{-}\mathrm{mod}$ denotes the category of finite-dimensional modules. They are of course adjoint functors: $\mathrm{Ind}^A_B \dashv \mathrm{Res}^A_B$. Hence there is a monad $\mathbb{T}^A_B = (T^A_B, \mu^A_B,\eta^A_B)$ on $B\text{-}\mathrm{mod}$ with underlying functor $T^A_B(X) = \mathrm{Res}^A_B \, \mathrm{Ind}^A_B(X)$ which is $A \otimes_B X$ viewed in $B\text{-}\mathrm{mod}$, multiplication 
\[ (\mu^A_B)_X : T^A_BT^A_B(X) = A \otimes_B (A \otimes_B X) \cong (A \otimes_B A) \otimes_B X \xrightarrow{\mathrm{mult}_A \,\otimes_B\, \mathrm{id}_X} A \otimes_B X = T^A_B(X) \]
and unit $(\eta^A_B)_X : X \xrightarrow{1_A \,\otimes_B\,\mathrm{id}_X} A \otimes_B X$ for all $X \in B\text{-}\mathrm{mod}$.

\subsection{Tangent space to an $R$-matrix}\label{subsectionTangentRMatrix}
Recall that an $R$-matrix for the Hopf algebra $H$ is an invertible element $R \in H \otimes H$ such that
\begin{align}
&\forall \, h \in H, \quad R \Delta(h) = \Delta^{\mathrm{op}}(h)R\label{QuasiCocommCondition}\\
&(\Delta \otimes \mathrm{id})(R) = R_{13}R_{23}, \quad (\mathrm{id} \otimes \Delta)(R) = R_{13}R_{12} \label{quasitriangularConditions}
\end{align}
with usual notations (see e.g. \cite[\S VIII.2]{kassel}).

\begin{lemma}
Let $R \in H \otimes H$ be any element which satisfies \eqref{quasitriangularConditions}. Then
\begin{equation}\label{invertibilityRMatrices}
R \text{ is invertible} \quad \iff \quad (\varepsilon \otimes \mathrm{id})(R) = (\mathrm{id} \otimes \varepsilon)(R) = 1_H.
\end{equation}
\end{lemma}
\begin{proof}
For ``$\Rightarrow$'' see e.g. \cite[Th.\,VIII.2.4]{kassel}. For the converse implication, write $R = R^1_i \otimes R^2_i$ with implicit summation. Then
\[ 1_H \otimes 1_H = S(R_i^{1(1)})R_i^{1(2)} \otimes R^2_i = S(R^1_i)R^1_j \otimes R^2_i R^2_j = (S \otimes \mathrm{id})(R)R \]
and thus $R$ has a left inverse. Similarly $1_H \otimes 1_H = R(\mathrm{id} \otimes S^{-1})(R)$ and thus $R$ has a right inverse. Since $H \otimes H$ is associative, the left and right inverses are equal.
\end{proof}

Let $\{h_i\}$ be a basis of $H$ and $R = \sum_{i,j} x_{i,j} h_i \otimes h_j$ with $x_{i,j} \in \Bbbk$. It is clear that the conditions \eqref{QuasiCocommCondition}, \eqref{quasitriangularConditions} and \eqref{invertibilityRMatrices} on $R$ are equivalent to polynomial equations among the $x_{i,j}$ (actually only linear and quadratic equations). Hence
\[ \mathrm{RMat}(H) = \bigl\{ R \in H \otimes H \, \big| \, R \text{ is an } R\text{-matrix} \bigr\} \]
can be seen as an affine algebraic variety in $\Bbbk^{\dim(H)^2}$. Moreover we have the well-known bijection (see e.g. \cite[Prop.\,XIII.1.4]{kassel})
\[ \mathrm{Br}(H\text{-}\mathrm{mod}) \overset{\sim}{\longrightarrow} \mathrm{RMat}(H), \qquad c \mapsto R_c = \tau\bigl( c_{H,H}(1_H \otimes 1_H) \bigr) \]
where $\tau(x \otimes y) = y \otimes x$ and $H$ is the regular module. Conversely given an $R$-matrix $R \in H^{\otimes 2}$ one has the braiding $c^R$ on $H\text{-}\mathrm{mod}$ defined by $c^R_{X,Y}(x \otimes y) = \tau\bigl( R \cdot (x \otimes y) \bigr)$. This bijection yields an identification of tangent spaces
\begin{equation}\label{isoInfBraidingsInfRMat}
\mathbf{T}_c\mathrm{Br}(H\text{-}\mathrm{mod}) \overset{\sim}{\longrightarrow} \mathbf{T}_{R_c}\mathrm{RMat}(H), \qquad t \mapsto \tau\bigl( t_{H,H}(1_H \otimes 1_H) \bigr)
\end{equation}
where $\mathbf{T}_c\mathrm{Br}(H\text{-}\mathrm{mod})$ is from Definition \ref{defTangentBraiding} and
\begin{equation}\label{tangentSpaceRMatrix}
\mathbf{T}_R\mathrm{RMat}(H) = \left\{ T \in H^{\otimes 2} \:\left|\, \begin{array}{l}
\forall \, h \in H, \quad T \Delta(h) = \Delta^{\mathrm{op}}(h)T\\
(\Delta \otimes \mathrm{id})(T) = T_{13}R_{23} + R_{13}T_{23}\\
(\mathrm{id} \otimes \Delta)(T) = T_{13}R_{12} + R_{13}T_{12}
\end{array} \right.\right\}
\end{equation}
is the Zariski tangent space (the conditions $(\varepsilon \otimes \mathrm{id})(T) = (\mathrm{id} \otimes \varepsilon)(T) = 0$ are automatically fulfilled).

\begin{remark}
The elements defined in \eqref{tangentSpaceRMatrix} are in isomorphism with the infinitesimal $R$-matrices considered in \cite{ABSW}: $T \in H^{\otimes 2}$ satisfies \cite[Def.\,2.1]{ABSW} if and only if $RT$ satisfies \eqref{tangentSpaceRMatrix}.
\end{remark}

\indent Let $P : (H\text{-}\mathrm{mod}) \boxtimes H\text{-}\mathrm{mod} = (H \otimes H)\text{-}\mathrm{mod} \to H\text{-}\mathrm{mod}$ be the usual monoidal product, which is the pullback by the coproduct $\Delta : H \to H \otimes H$. Fix $R = R^1_i \otimes R^2_i \in \mathrm{RMat}(H)$. By item 2 in Proposition \ref{factoLaxMult} the natural transformation
\begin{equation}\label{monStructOfPFromRMat}
\fonc{\phi^R_{X_1 \boxtimes Y_1, X_2 \boxtimes Y_2}}{X_1 \otimes Y_1 \otimes X_2 \otimes Y_2}{X_1 \otimes X_2 \otimes Y_1 \otimes Y_2}{x_1 \otimes y_1 \otimes x_2 \otimes y_2}{ x_1 \otimes R^2_i \cdot x_2 \otimes R^1_i \cdot y_1 \otimes y_2}
\end{equation}
is a monoidal structure for $P$. As for braidings on $H\text{-}\mathrm{mod}$, we can describe the DY complex of $(P,\phi^R)$ purely in terms of the algebra $H$. Recall the shape of DY cochains of $P$ in \eqref{DYCochainsMonProduct}. Let $\mathcal{V}^n$ be the subspace of elements $\mathbf{u} \in H^{\otimes 2n}$ such that
\begin{equation}\label{centralizerForDYMonProduct}
\textstyle \forall \, h \in H, \qquad \left(\bigotimes_{i=1}^n h^{(i)} \otimes h^{(n+i)} \right) \mathbf{u} = \mathbf{u} \left( \bigotimes_{i=1}^{2n} h^{(i)} \right)
\end{equation}
where the big tensor products go from left to right, \textit{i.e.} $h^{(1)} \otimes h^{(n+1)} \otimes \ldots \otimes h^{(n)} \otimes h^{(2n)}$. For $\mathbf{u} \in \mathcal{V}^n$ and $X_i,Y_i \in H\text{-}\mathrm{mod}$ we define
\[ \textstyle f^{\mathbf{u}}_{X_1 \boxtimes Y_1,\ldots,X_n \boxtimes Y_n} : \bigotimes_{i=1}^n X_i \otimes Y_i \textstyle \to \bigl(\bigotimes_{i=1}^n X_i \bigr) \otimes \bigl(\bigotimes_{j=1}^n Y_j \bigr), \quad w \mapsto \sigma_n(\mathbf{u} \cdot w) \]
where $\mathbf{u} \cdot w$ is the tensor-wise action of $\mathbf{u}$ on $w$ and
\[ \sigma_n\bigl( x_1 \otimes y_1 \otimes \ldots \otimes x_n \otimes y_n\bigr) = x_1 \otimes \ldots \otimes x_n \otimes y_1 \otimes \ldots \otimes y_n. \]
Condition \eqref{centralizerForDYMonProduct} is equivalent to the $H$-linearity of $f^{\mathbf{u}}_{X_1 \boxtimes Y_1,\ldots,X_n \boxtimes Y_n}$. Moreover the naturality of this collection of morphisms is immediate because tensor products of morphisms in $H\text{-}\mathrm{mod}$ commute with the action of $\mathbf{u}$. Hence we get $f^{\mathbf{u}} \in C^n_{\mathrm{DY}}(P,\phi^R)$. The map
\begin{equation}\label{isoDYCochainsPHmod}
\mathcal{V}^n \to C^n_{\mathrm{DY}}(P,\phi^R), \quad \mathbf{u} \mapsto f^{\mathbf{u}}
\end{equation}
 is an isomorphism of vector spaces whose inverse is $f \mapsto \sigma_n^{-1}\bigl( f_{H \boxtimes H, \ldots, H \boxtimes H}(1^{\otimes 2n}) \bigr)$, where $H$ is the regular module. Under this identification, the DY differentials in degree $1$ and $2$ are
\begin{align*}
\delta^1(\mathbf{u}) &= (1 \otimes R \otimes 1)\cdot(1 \otimes 1 \otimes \mathbf{u}) - \Delta_{H \otimes H}(\mathbf{u})\cdot(1 \otimes R \otimes 1) + (1 \otimes R \otimes 1)\cdot(\mathbf{u} \otimes 1 \otimes 1),\\
\delta^2(\mathbf{u}) &= \bigl( 1 \otimes R^1_i \otimes R_i^{2(1)} \otimes 1 \otimes R_i^{2(2)} \otimes 1 \bigr)\cdot(1 \otimes 1 \otimes \mathbf{u}) - \bigl(\Delta_{H \otimes H} \otimes \mathrm{id}_H^{\otimes 2}\bigr)(\mathbf{u})\cdot (1 \otimes R \otimes 1^{\otimes 3})\\
&\:\:\:\:\:+\bigl( \mathrm{id}_H^{\otimes 2} \otimes \Delta_{H \otimes H} \bigr)(\mathbf{u}) \cdot (1^{\otimes 3} \otimes R \otimes 1) - \bigl( 1 \otimes R_i^{1(1)} \otimes 1 \otimes R_i^{1(2)} \otimes R^2_i \otimes 1 \bigr) \cdot (\mathbf{u} \otimes 1 \otimes 1)
\end{align*}
where $\Delta_{H \otimes H}(x \otimes y) = x^{(1)} \otimes y^{(1)} \otimes x^{(2)} \otimes y^{(2)}$.

\smallskip

\indent Let $\mathcal{Z}\bigl( \Delta^{(n-1)}(H) \bigr)$ be the centralizer of the image of the iterated coproduct $\Delta^{(n-1)} : H \to H^{\otimes n}$. For $\mathbf{h} \in \mathcal{Z}\bigl( \Delta^{(n-1)}(H) \bigr)$ and any $X_1,\ldots,X_n \in H\text{-}\mathrm{mod}$ let $\mathrm{act}^{\mathbf{h}}_{X_1,\ldots,X_n}$ be the representation of $\mathbf{h} \in H^{\otimes n}$ on $X_1 \otimes \ldots \otimes X_n$. This defines a natural transformation $\mathrm{act}^{\mathbf{h}} : (-)^{\otimes n} \Rightarrow (-)^{\otimes n}$. It was observed in \cite[Prop.\,8]{davydov} (also explained in \cite[\S5.3]{FGS}) that the linear map
\begin{equation}\label{isoDYCochainsIdHmod}
\mathcal{Z}\bigl( \Delta^{(n-1)}(H) \bigr) \overset{\sim}{\longrightarrow} C^n_{\mathrm{DY}}(H\text{-}\mathrm{mod}), \quad \mathbf{h} \mapsto \mathrm{act}^{\mathbf{h}}
\end{equation}
is an isomorphism of vector spaces whose inverse is $g \mapsto g_{H,\ldots,H}(1^{\otimes n})$. Under the identifications \eqref{isoDYCochainsPHmod}, \eqref{isoDYCochainsIdHmod} and \eqref{isoInfBraidingsInfRMat} the isomorphism from Corollary \ref{coroDimFormulaTangentSpaceBraiding} becomes
\[ \flecheIso{H^2_{\mathrm{DY}}(P,\phi^R)}{H^2_{\mathrm{DY}}(H\text{-}\mathrm{mod})^{\oplus 2} \oplus \mathbf{T}_R\mathrm{RMat}(H)}{{[}\mathbf{u}{]}}{\bigl( \,\bigl[ (\mathrm{id}_H \otimes \varepsilon)^{\otimes 2}(\mathbf{u}) \bigr], \, \bigl[(\varepsilon \otimes \mathrm{id}_H)^{\otimes 2}(\mathbf{u}) \bigr], \, T^{\mathbf{u}} \bigr)} \]
where $T^{\mathbf{u}} = (\varepsilon \otimes \mathrm{id}_H^{\otimes 2} \otimes \varepsilon)(\mathbf{u}) - \tau\bigl((\mathrm{id}_H \otimes \varepsilon^{\otimes 2} \otimes \mathrm{id}_H)(\mathbf{u})\bigr) R$ with $\tau(x \otimes y) = y \otimes x$. In particular: 
\begin{equation}\label{dimensionFormulaForTRMat}
\dim \mathbf{T}_R\mathrm{RMat}(H) = \dim H^2_{\mathrm{DY}}(P,\phi^R) - 2 \dim H^2_{\mathrm{DY}}(H\text{-}\mathrm{mod}).
\end{equation}

\smallskip

\indent The next result is Theorem \ref{thmChangeFunctorMonProduct} for $\mathcal{C} = H\text{-}\mathrm{mod}$, \textit{i.e.} we describe the coefficient $\bigl( \mathscr{A},\lambda^{(+),(-)} \bigr)$ in this case. It gives an efficient way to compute the dimension of the vector space $H^2_{\mathrm{DY}}(P,\phi^R)$ thanks to homological algebra:
\begin{proposition}\label{propAdjunctionTheoremForMonProductInAmod}
Let $R \in H^{\otimes 2}$ be an $R$-matrix, $P$ be the monoidal product on $H\text{-}\mathrm{mod}$ and $\phi^R$ be the monoidal structure of $P$ defined in \eqref{monStructOfPFromRMat}. Then
\[ H^{\bullet}_{\mathrm{DY}}(P,\phi^R) \cong \mathrm{Ext}^{\bullet}_{D(H) \otimes D(H), H \otimes H}(\Bbbk \boxtimes \Bbbk, H^*) \]
where the vector space $H^*$ is endowed with the $D(H)^{\otimes 2}$-module structure given by
\begin{equation}\label{actionDAAOnADual}
\forall \, \alpha a, \,\beta b \in D(H), \:\: \forall \, \psi \in H^*, \quad (\alpha a \otimes \beta b) \cdot \psi = \ell^{(-)}(\beta)b \triangleright \psi \triangleleft S\bigl(\ell^{(+)}(\alpha) a \bigr)
\end{equation}
where we define 
\begin{align}\label{eq:l-pm-alpha}
\ell^{(+)} : (H^*)^{\mathrm{op}} \to H, \quad \alpha &\mapsto (\mathrm{id} \otimes \alpha)(R) \\
\ell^{(-)} : (H^*)^{\mathrm{op}} \to H, \quad \beta &\mapsto (\beta \otimes \mathrm{id})(R^{-1}) \label{eq:l-pm-beta}
\end{align}
and we use the coregular actions \eqref{defCoregular}.
\end{proposition}
\noindent Recall that $D(H) = H^{*\mathrm{op}} \otimes H$ and we write $\alpha a$ instead of $\alpha \otimes a$, see \eqref{definingRelDrinfeldDouble} for the product. There is an obvious isomorphism $D(H \otimes H) \overset{\sim}{\to} D(H) \otimes D(H)$ given by $ (\alpha \otimes \beta)(a \otimes b) \mapsto \alpha a \otimes \beta b$ where the linear form $\alpha \otimes \beta \in (H \otimes H)^{*\mathrm{op}}$ is defined by $\langle \alpha \otimes \beta, x \otimes y \rangle = \langle \alpha, x \rangle \langle \beta, y\rangle$; in the sequel we use this identification. 
We also note that the maps $\ell^{(\pm)}$ defined in~\eqref{eq:l-pm-alpha}-\eqref{eq:l-pm-beta} are algebra maps which extend to algebra maps $D(H) \to H$ given by $\alpha a \mapsto \ell^{(+)}(\alpha)a$ and $\beta b \mapsto \ell^{(+)}(\beta)b$ respectively.
\begin{proof}
Write $R = R^1_i \otimes R^2_i \in H^{\otimes 2}$ and recall that the braiding in $H\text{-}\mathrm{mod}$ is $c_{X,Y}(x \otimes y) = R^2_i \cdot y \otimes R^1_i \cdot x$. Its inverse is $c_{X,Y}^{-1}(y \otimes x) = S(R^1_i) \cdot x \otimes R^2_i \cdot y$. Also recall that $X^{\vee}$ is the vector space $X^*$ endowed with the $H$-action defined by $(h \cdot f)(x) = f\bigl( S(h) \cdot x \bigr)$ for all $f \in X^*$ and $x \in X$. The coend \eqref{defCoendBulkA} for $\mathcal{C} = H\text{-}\mathrm{mod}$
\[ \mathscr{A} = \int^{X \in H\text{-}\mathrm{mod}} X^{\vee} \boxtimes X \in (H\text{-}\mathrm{mod}) \boxtimes (H\text{-}\mathrm{mod}) = (H \otimes H)\text{-}\mathrm{mod} \]
is the vector space $H^*$ endowed with the $(H \otimes H)$-action $(a \otimes b) \cdot \psi = b \triangleright \psi \triangleleft S(a)$, see e.g. \cite[App.\,A]{FSSmodular}. Its universal dinatural transformation $j_X : X^{\vee} \boxtimes X \to H^*$ is given by $\langle j_X(\chi \boxtimes x),h \rangle = \chi(h \cdot x)$. Note in particular that any $\psi \in H^*$ can be written as $\psi = j_H(\psi \boxtimes 1_H)$ where $H$ is the regular module. Let us compute the half-braiding $\lambda^{(+),(-)}_{V \boxtimes W} : H^* \otimes (V \boxtimes W) \to (V \boxtimes W) \otimes H^*$ defined in \eqref{halfBraidingOnBulk}:
\begin{align*}
\lambda^{(+),(-)}_{V \boxtimes W}(\psi \otimes (v \boxtimes w)) &= \lambda^{(+),(-)}_{V \boxtimes W} \circ (j_H \otimes \mathrm{id}_{V \boxtimes W})\bigl( (\psi \boxtimes 1_H) \otimes (v \boxtimes w) \bigr)\\
&=(\mathrm{id}_{V \boxtimes W} \otimes j_H) \circ (c_{H^{\vee},V} \boxtimes c_{W,H}^{-1})\bigl( (\psi \otimes v) \boxtimes (1_H \otimes w) \bigr)\\
&=(\mathrm{id}_{V \boxtimes W} \otimes j_H)\left[ \bigl( R^2_i \cdot v \otimes (\psi \triangleleft S(R^1_i)) \bigr) \boxtimes \bigl( S(R^1_j) \cdot w \otimes R^2_j \bigr) \right]\\
&= (\mathrm{id}_{V \boxtimes W} \otimes j_H)\left[ \bigl( R^2_i \cdot v \boxtimes S(R^1_j) \cdot w \bigr) \otimes \bigl( (\psi \triangleleft S(R^1_i)) \boxtimes R^2_j \bigr) \right]\\
&=\bigl( R^2_i \cdot v \boxtimes S(R^1_j) \cdot w \bigr) \otimes \bigl( R^2_j \triangleright \psi \triangleleft S(R^1_i) \bigr).
\end{align*}
Through the isomorphism $\mathcal{Z}\bigl( (H \otimes H)\text{-}\mathrm{mod} \bigr) \cong D(H)^{\otimes 2}\text{-}\mathrm{mod}$ recalled at the beginning of \S\ref{sectionFinDimHopf}, the half-braiding $\lambda^{(+),(-)}$ is equivalent to the action of $D(H)^{\otimes 2}$ on $H^*$ given by
\begin{align*}
&(\alpha a \otimes \beta b) \cdot \psi = (\alpha \otimes \beta \otimes \mathrm{id}_{H^*}) \circ \lambda^{(+),(-)}_{H \boxtimes H}\bigl( (a \otimes b) \cdot \psi \otimes (1_H \boxtimes 1_H) \bigr)\\
=\:&  (\alpha \otimes \beta \otimes \mathrm{id}_{H^*}) \circ \lambda^{(+),(-)}_{H \boxtimes H}\bigl( (b \triangleright \psi \triangleleft S(a)) \otimes (1_H \boxtimes 1_H) \bigr) = \alpha(R^2_i)\beta\bigl( S(R^1_j) \bigr) R^2_jb \triangleright \psi \triangleleft S(a)S(R^1_i)
\end{align*}
which is the claimed formula.
\end{proof}

\begin{remark}\label{rem:end-ingr-Hopf}
 Recall that $H^{n}_{\mathrm{DY}}(P,\phi^R)$ can be also computed via the K\"unneth type formula~\eqref{eq:HDY-end} or Corollary~\ref{cor:HDY-end-proj}  if the ground field $\Bbbk$ is perfect. Let us explain their ingredients in the Hopf case:
 \begin{enumerate}
     \item 
 The distinguished invertible object $D$ is the one-dimensional $H$-module given by the character $\alpha^{-1}: H \to \Bbbk$ where $\alpha$ is the so-called modulus defined by
 $$
 c^l \cdot h = \alpha(h)c^l 
 $$
 and $c^l\in H$ is the left co-integral of $H$, see~\cite[Prop.\,6.5.5]{EGNO}.\footnote{We note that our conventions on the distinguished invertible object $D$ are opposite to those in~\cite[Sec.\,6.4]{EGNO}.}
 \item 
 The Nakayama functor $\Nak$ given in~\eqref{eq:Nak-object} is then induced by the automorphism on $H$ that sends $h\in H$ to $\alpha(h^{(1)})S^2(h^{(2)})$.
 \item
 Let $X \in H\text{-}\mathrm{mod}$, then under the equivalence $\mathcal{Z}(H\text{-}\mathrm{mod}) \cong D(H)\text{-}\mathrm{mod}$ recalled just below \eqref{definingRelDrinfeldDouble}, the object $(X^{\vee},c_{X^{\vee},-})$ is the vector space $X^*$ endowed with the action defined by $\langle (\alpha a) \cdot f, x \rangle = \bigl\langle f, S\bigl( \ell^{(+)}(\alpha)a \bigr)x \bigr\rangle$ for all $f \in X^{\vee}$, $\alpha, a \in D(H)$ and $x \in X$, and with $\ell^{(+)}$ defined in~\eqref{eq:l-pm-alpha}. Similarly $(X,c^{-1}_{-,X})$ is the vector space $X$ endowed with the action $(\beta b) \cdot x = \ell^{(-)}(\beta)b \cdot x$ for all $\beta, b \in D(H)$ and $x \in X$, and with $\ell^{(-)}$ defined in~\eqref{eq:l-pm-beta}.
 \end{enumerate}
 The advantage of the formulas~\eqref{eq:HDY-end} or~\eqref{eq:dim-Tc-Br-Ext-end} is that the ends are taken over projective objects, recall Corollary~\ref{cor:HDY-end-proj}. When $H$ has only a few isomorphism classes of simple objects, one can expect that these ends are easy to calculate. This is demonstrated in the example $H=B_k$ in \S \ref{subsectionExampleBk} below.
\end{remark}

\subsection{Example: tangent $R$-matrices for $B_k = \Lambda \mathbb{C}^k \rtimes \mathbb{C}[\mathbb{Z}/2\mathbb{Z}]$}\label{subsectionExampleBk}
Let $k$ be a strictly positive integer and $B_k$ be the $\mathbb{C}$-algebra generated by $g$ and $x_i$ with $1 \leq i \leq k$ modulo the relations
\begin{equation}\label{presentationBk}
x_i x_j = -x_j x_i, \quad gx_i = -x_i g, \quad x_i^2 = 0, \quad g^2 = 1.
\end{equation}
The monomials $x_1^{e_1} \ldots x_k^{e_k}g^{e_{k+1}}$  with $e_i \in \{0,1\}$ form a basis, so that $\dim(B_k) = 2^{k+1}$. It is a Hopf algebra with coproduct $\Delta(x_i) = 1 \otimes x_i + x_i \otimes g$, $\Delta(g) = g \otimes g$, counit $\varepsilon(x_i) = 0$, $\varepsilon(g) = 1$ and antipode $S(x_i) = g x_i$, $S(g) = g$. 
This  algebra can be seen as a generalization of the 4-dimensional Sweedler's Hopf algebra.

It is clear that $B_k$ is pivotal with the pivot given by $g$, i.e.\ $S^2(-) = g(-)g^{-1}$. Furthermore, $B_k$ is unimodular for even $k$, i.e.\ the distinguished invertible object is $D=\boldsymbol{1}$, while for odd $k$ the object $D=D^\vee$ is described by the character $\alpha=\alpha^{-1}$, recall Remark~\ref{rem:end-ingr-Hopf}, where $\alpha$ is zero on all $x_i$'s and $-1$ on $g$. Let us denote the projective cover of such a one-dimensional $B_k$-module, for all $k$, by $P_-$.

The element
\begin{equation}\label{eq:R0-Bk}
R_0 = \boldsymbol{e}_+ \otimes 1 + \boldsymbol{e}_- \otimes g = 1 \otimes \boldsymbol{e}_+ + g \otimes \boldsymbol{e}_- \in B_k \otimes B_k 
\end{equation}
where $\boldsymbol{e}_{\pm} = \frac{1 \pm g}{2}$ satisfies $(R_0)^2 = 1 \otimes 1$ and is a triangular $R$-matrix.

\indent Our main goal is to describe the Zariski tangent space $\mathbf{T}_{R_0}\mathrm{RMat}(B_k)$ defined in~\eqref{tangentSpaceRMatrix}. We first compute its dimension using the end formula~\eqref{eq:dim-Tc-Br-Ext-end}.

\smallskip

Write $(x_1^{e_1} \ldots x_k^{e_k}g^{e_{k+1}})^*$ for the basis elements of $B_k^*$ which are dual to the monomial basis of $B_k$. Recall from \cite[\S 5]{GHS} that $B_k^{*\mathrm{op}}$ is generated by the elements $y_i = x_i^* - (x_ig)^*$ ($1 \leq i \leq k$) and $h = 1^* - g^*$ modulo the relations
\begin{equation}\label{presentationBkDual}
y_i y_j = -y_j y_i, \quad hy_i = -y_i h, \quad y_i^2 = 0, \quad h^2 = 1.
\end{equation}
Hence the Drinfeld double $D(B_k) = B_k^{*\mathrm{op}} \otimes B_k$ has the generators $x_i,g, y_i,h$ modulo the relations \eqref{presentationBk}, \eqref{presentationBkDual} and the mixed relations given in \cite[eq.\,(5.7)]{GHS} that we will not need here.

\smallskip

\indent We will use a `minimal' relatively projective resolution $\mathbf{P}$ of the trivial $D(B_k)$-module $\mathbb{C}$ constructed in~\cite[\S 5]{GHS}, see also~\cite[\S 6.1.2]{FGS}:
\begin{equation}\label{relProjResUnitBk}
\mathbf{P} = \left( 0 \longleftarrow \mathbb{C} \longleftarrow S^0(\mathbb{C}^k) \otimes \mathcal{C}_+ \longleftarrow S^1(\mathbb{C}^k) \otimes \mathcal{C}_- \longleftarrow S^2(\mathbb{C}^k) \otimes \mathcal{C}_+ \longleftarrow \ldots \right)
\end{equation}
where $\mathcal{C}_\pm= B_k^{*\mathrm{op}}\boldsymbol{f}_{\pm}$ are projective covers in $B_k^*\text{-}\mathrm{mod}$ corresponding to the idempotents $\boldsymbol{f}_{\pm} = \frac{1 \pm h}{2}$, equipped additionally with the trivial action of the generators $x_i$ and $g$ acting as $h$, which makes them also $D(B_k)$-modules.
 The spaces  $S^n(\mathbb{C}^k)$ of homogeneous polynomials of degree $n$ in $k$ commuting variables, of dimension $\binom{k+n-1}{n}$, are just multiplicity spaces in~\eqref{relProjResUnitBk}. For what follows, we do not use explicit form of the differentials of $\mathbf{P}$, so we omit them.
\begin{proposition}\label{propDYcohomologyMonProductForBk}
1. We have
\[ \dim \mathbf{T}_{R_0} \mathrm{RMat}(B_k) = k^2. \]
2. Recall that the $R$-matrix $R_0 \in B_k^{\otimes 2}$ gives the monoidal structure $\phi^{R_0}$ for $\otimes$ defined in \eqref{monStructOfPFromRMat}. Then we have 
\[\dim H^n_{\mathrm{DY}}(\otimes,\phi^{R_0}) = \begin{cases}
\binom{2k+n-1}{n} & \text{if } n \text{ is even}\\
0 & \text{if } n \text{ is odd}
\end{cases} \]
\end{proposition}
\begin{proof}
Let
\[ L : B_k\text{-}\mathrm{mod} \to \mathcal{Z}(B_k\text{-}\mathrm{mod}) \cong D(B_k)\text{-}\mathrm{mod}, \quad X \mapsto \bigl(X, c^{R_0}_{X,-} \bigr) = \bigl( X, (c^{R_0}_{-,X})^{-1} \bigr) \]
which by item 3 of Remark~\ref{rem:end-ingr-Hopf} is the pullback along the morphism $\ell^{(+)} = \ell^{(-)} : D(B_k) \to B_k$ from \eqref{eq:l-pm-alpha}, where the equality is due to the symmetric R-matrix~\eqref{eq:R0-Bk}. In particular we see that the $y_i$'s act by $0$ on any $L(X)$ while the $x_i$'s act on $L(X)$ as they do on $X$. The situation is opposite for the modules $\mathcal{C}_{\pm}$ appearing in \eqref{relProjResUnitBk}: the $x_i$'s act by $0$ while the action of the $y_i$'s is free. This immediately shows that
$\Hom_{D(B_k)}(\mathcal{C}_{\pm},L(X)) \cong \Hom_{D(B_k)}(\mathcal{C}_{\pm},\mathrm{Soc}[L(X)])$
where $\mathrm{Soc}$ denotes the socle.
As a consequence, using the resolution~\eqref{relProjResUnitBk} we see that
the relative Ext spaces satisfy:
\[ \Ext_{D(B_k),B_k}^n\bigl(\mathbb{C},L(X)\bigr) \cong \Ext^n_{D(B_k),B_k}\bigl(\mathbb{C},\mathrm{Soc}[L(X)]\bigr)\ . \]
We thus see that the term in the 2nd line of~\eqref{eq:dim-Tc-Br-Ext-end} is canceled by the 3rd line because $\mathrm{Soc}[L(P_{D})]=\mathbb{C}$. For what concerns the end in the 1st line, only one term corresponding to $P=P_-$ gives a non-zero contribution, which is $S^1(\mathbb{C}^k)^{\otimes 2}$. Clearly there are no extra linear relations coming from dinaturality, and we are left with $S^1(\mathbb{C}^k)^{\otimes 2}$ whose dimension is $k^2$.
\\2. Recall Corollary~\ref{cor:HDY-end-proj} and denote by $B_k\text{-}\mathrm{pmod}$ the subcategory of projective $B_k$-modules, which has only two isoclasses of indecomposables (the projective covers of the characters $\varepsilon$ and~$\alpha$). By arguments similar to the previous item we immediately calculate
\begin{align}
H^{n}_{\mathrm{DY}}(\otimes, \mathrm{id} \otimes c \otimes \mathrm{id})
 &\cong   \bigoplus_{i+j=n} \int_{P \in B_k\text{-}\mathrm{pmod}}  \!\!\!\!\!\!\Ext^i_{D(B_k),B_k}\bigl(\mathbb{C}, L(P)\bigr) \otimes \Ext^j_{D(B_k),B_k}\bigl(\mathbb{C},L(D\otimes P) \bigr)\notag\\
&\cong \delta_{n,\mathrm{even}}\, \bigoplus_{i+j=n} S^i(\mathbb{C}^k) \otimes S^j(\mathbb{C}^k) \cong \delta_{n,\mathrm{even}} S^n(\mathbb{C}^{2k})\ ,\label{eq:HDY-Bk-end}
 \end{align}
where we observed that for even $n$, as for $n=2$ above, and each $(i,j)$ only one $P$ term (out of two) in the corresponding end gives a non-zero contribution equal to $S^i(\mathbb{C}^k) \otimes S^j(\mathbb{C}^k)$, while for odd $n$ there is always one tensorand in each end which is zero, so the total space is zero.
\end{proof}

\indent The use of the end formulas (\S\ref{sec:end-formula}) gave a fast proof of the previous proposition. But for completeness we still find it relevant to give another proof which uses the isomorphism $H^n_{\mathrm{DY}}(\otimes,\phi^{R_0}) \cong \mathrm{Ext}^n_{D(B_k)^{\otimes 2}, B_k^{\otimes 2}}(\mathbb{C} \boxtimes \mathbb{C}, B_k^*)$ from Proposition~\ref{propAdjunctionTheoremForMonProductInAmod}. Let us then describe the $D(B_k)^{\otimes 2}$-module structure on $B_k^*$:

\begin{lemma}\label{lemmaBicoregBkDual}
Let $ w_{\pm} = (x_1 \ldots x_k)^* \pm (x_1 \ldots x_kg)^*$. The elements $x_1^{e_1} \ldots x_k^{e_k} \triangleright w_{\pm}$ with $e_i \in \{0,1\}$ form a basis of $B_k^*$. In this basis, the action of $D(B_k)^{\otimes 2}$ on $B_k^*$ associated to $R_0$ (Prop. \ref{propAdjunctionTheoremForMonProductInAmod}) is given by
\begin{align*}
&(x_i \otimes 1) \cdot \bigl( x_1^{e_1} \ldots x_k^{e_k} \triangleright w_{\pm} \bigr) = \mp  x_1^{e_1} \ldots x_k^{e_k}x_i \triangleright w_{\mp},\\
&(1 \otimes x_i) \cdot \bigl( x_1^{e_1} \ldots x_k^{e_k} \triangleright w_{\pm} \bigr) = x_i x_1^{e_1} \ldots x_k^{e_k} \triangleright w_{\pm},\\
&(y_i \otimes 1) \cdot \bigl( x_1^{e_1} \ldots x_k^{e_k} \triangleright w_{\pm} \bigr) = (1 \otimes y_i) \cdot \bigl( x_1^{e_1} \ldots x_k^{e_k} \triangleright w_{\pm} \bigr) =0,\\
&(g \otimes 1) \cdot \bigl( x_1^{e_1} \ldots x_k^{e_k} \triangleright w_{\pm} \bigr) = (h \otimes 1) \cdot \bigl( x_1^{e_1} \ldots x_k^{e_k} \triangleright w_{\pm} \bigr) = \pm (-1)^k x_1^{e_1} \ldots x_k^{e_k} \triangleright w_{\pm},\\
&(1 \otimes g) \cdot \bigl( x_1^{e_1} \ldots x_k^{e_k} \triangleright w_{\pm} \bigr) = (1 \otimes h) \cdot \bigl( x_1^{e_1} \ldots x_k^{e_k} \triangleright w_{\pm} \bigr) = \pm (-1)^{|e|} x_1^{e_1} \ldots x_k^{e_k} \triangleright w_{\pm}
\end{align*}
where $|e| = \sum_{i=1}^k e_i$.
\end{lemma}
\begin{proof}
Note that
\[ x_i^{e_i} \triangleright (x_1^{d_1} \ldots x_i^1 \ldots x_k^{d_k}g^t)^* = (-1)^{e_i\sum_{j=i+1}^k d_j + te_i} (x_1^{d_1} \ldots x_i^{1-e_i} \ldots  x_k^{d_k}g^t)^*.  \]
for all $e_i, d_j, t \in \{0,1\}$. Hence 
\[ x_1^{e_1} \ldots x_k^{e_k} \triangleright (x_1 \ldots x_kg^t)^* = (-1)^{s(e) + t|e|}  (x_1^{1-e_1} \ldots  x_k^{1-e_k}g^t )^* \]
where $s(e) =  \sum_{i=1}^{k-1} e_i\sum_{j=i+1}^k (1- e_j)$ and $|e| = \sum_{i=1}^k e_i$. As a result
\[ x_1^{e_1} \ldots x_k^{e_k} \triangleright w_{\pm} = (-1)^{s(e)} \left( (x_1^{1-e_1} \ldots x_k^{1-e_k})^* \pm (-1)^{|e |} (x_1^{1-e_1} \ldots x_k^{1-e_k}g)^* \right) \]
whence the claim about the basis. The formula for the action of $1 \otimes x_i$ is obvious since $(1 \otimes a) \cdot \psi = a \triangleright \psi$. The formula for the action of $1 \otimes g$ follows from $g \triangleright w_{\pm} = \pm w_{\pm}$. Next we have
\begin{align*}
&(x_1 \ldots x_k)^* \triangleleft S(x_i) = (-1)^{k-i+1} (x_1 \ldots x_i^0 \ldots x_k g)^* = x_i \triangleright (x_1 \ldots x_k g)^*,\\
&(x_1 \ldots x_kg)^* \triangleleft S(x_i) = (-1)^{k-i+1} (x_1 \ldots x_i^0 \ldots x_k)^* = -x_i \triangleright (x_1 \ldots x_k)^*
\end{align*}
and thus $w_{\pm} \triangleleft S(x_i) = \mp x_i \triangleright w_{\mp}$. Since the actions $\triangleright$ and $\triangleleft$ commute, we get the formula for the action of $x_i \otimes 1$. It is easy to check that $g \triangleright w_{\pm} = \pm w_{\pm}$ and $w_{\pm} \triangleleft g = \pm (-1)^k w_{\pm}$, which imply the formulas for the actions of $1 \otimes g$ and $g \otimes 1$ respectively. Finally we have $\ell^{(\pm)}(h) = g$ and $\ell^{(\pm)}(y_i) = 0$, which imply the formulas for the actions of $1 \otimes y_i$, $1 \otimes h$, $y_i \otimes 1$ and $h \otimes 1$. 
\end{proof}

By item 1 in Proposition \ref{propDeligneProductOfResolutions}, the tensor product $\mathbf{P} \boxtimes \mathbf{P}$ of \eqref{relProjResUnitBk} with itself is a relatively projective resolution of $\mathbb{C} \boxtimes \mathbb{C} \in D(B_k)^{\otimes 2}\text{-}\mathrm{mod}$, with the chain spaces in degree $n$ given by
\[ (\mathbf{P} \boxtimes \mathbf{P})_n = \bigoplus_{i=0}^n S^i(\mathbb{C}^k) \otimes S^{n-i}(\mathbb{C}^k) \otimes \bigl(\mathcal{C}_{(-)^i} \boxtimes \mathcal{C}_{(-)^{n-i}}\bigr)\ . \]
We have to compute $\Hom_{D(B_k)^{\otimes 2}}\bigl( (\mathbf{P} \boxtimes \mathbf{P})_n, B_k^* \bigr)$. Any $\varphi \in \Hom_{D(B_k)^{\otimes 2}}(\mathcal{C}_s \boxtimes \mathcal{C}_t, B_k^*)$, for  $s,t =\pm$,  is entirely determined by its value on $\boldsymbol{f}_s \boxtimes \boldsymbol{f}_t$. Since $x_i \otimes x_j$ acts by $0$ on $\mathcal{C}_s \boxtimes \mathcal{C}_t$ for all $i,j$, these generators also act by $0$ on $\varphi(\boldsymbol{f}_s \boxtimes \boldsymbol{f}_t)$, which forces 
$$
\varphi(\boldsymbol{f}_s \boxtimes \boldsymbol{f}_t) = \lambda_+ \, x_1 \ldots x_k \triangleright w_+ +  \lambda_- \, x_1 \ldots x_k \triangleright w_-
$$
 for some $\lambda_{\pm} \in \mathbb{C}$.
  The actions of the generators $g \otimes 1$ and $1 \otimes g$ impose
\[ s \varphi(\boldsymbol{f}_s \boxtimes \boldsymbol{f}_t) = (-1)^k \lambda_+ \, x_1 \ldots x_k \triangleright w_+ -(-1)^k \lambda_- \, x_1 \ldots x_k \triangleright w_- = t \varphi(\boldsymbol{f}_s \boxtimes \boldsymbol{f}_t). \]
Thus $s=t$, $s\lambda_+ = (-1)^k\lambda_+$ and $s\lambda_- = (-1)^{k+1}\lambda_-$. The generators $1 \otimes y_i$ and $y_i \otimes 1$ act freely on $\mathcal{C}_s \boxtimes \mathcal{C}_t$ and they act by $0$ on $B_k^*$, while the generators $1 \otimes h$ and $h \otimes 1$ act as $1 \otimes g$ and $g \otimes 1$ both on $\mathcal{C}_s \boxtimes \mathcal{C}_t$ and on $B_k^*$. Hence these generators do not give further constraints on $\varphi$. As a result
\[ \Hom_{D(B_k)^{\otimes 2}}(\mathcal{C}_s \boxtimes \mathcal{C}_t, B_k^*) = \begin{cases}
\boldsymbol{f}_s \boxtimes \boldsymbol{f}_s \mapsto x_1 \ldots x_k \triangleright w_{(-)^ks} & \text{if } s=t\\
0 & \text{if } s \neq t
\end{cases} \]
where for $s = t$ we mean that the Hom space is one dimensional and spanned by this map. Hence
\[ \Hom_{D(B_k)^{\otimes 2}}\bigl( (\mathbf{P} \boxtimes \mathbf{P})_n, B_k^* \bigr) \cong \begin{cases}
\bigoplus_{i=0}^n S^i(\mathbb{C}^k) \otimes S^{n-i}(\mathbb{C}^k) \cong S^n(\mathbb{C}^{2k}) & \text{if } n \text{ is even}\\
0 & \text{if } n \text{ is odd}
\end{cases} \]
and item 2 in Proposition \ref{propDYcohomologyMonProductForBk} follows. Item 1 in Proposition \ref{propDYcohomologyMonProductForBk} can then be deduced from \eqref{dimensionFormulaForTRMat}, since by \cite[Th.\,5.1]{GHS} we know that $\dim H^2_{\mathrm{DY}}(B_k\text{-}\mathrm{mod}) = \frac{k(k+1)}{2}$.

\smallskip

\indent We note that the space of obstructions $H^3_{\mathrm{DY}}(P,\phi^{R_0})$ is $0$, hence any infinitesimal deformation of the monoidal structure of $(P,\phi^{R_0})$ can be lifted to arbitrary degrees in $h$. By applying Corollary \ref{coroGeneralizedJSBijection} to such deformations, one can expect to get deformations of $R_0$ in arbitrary degrees in $h$ as well. In the next proposition we describe the space of infinitesimal deformations of $R_0$ and their eventual analytic versions, which form a $k^2$-parameter family of genuine $R$-matrices. By \cite{PV} this is the complete list of $R$-matrices on $B_k$ (here we provide a more compact form of their expression).
\begin{proposition}
1. The following elements form a basis of $\mathbf{T}_{R_0}\mathrm{RMat}(B_k)$:
\[ R_0 (x_i \otimes x_j g), \qquad (1 \leq i, j \leq k). \]
2. For all $\boldsymbol{\lambda} = (\lambda_{i,j})_{1 \leq i,j \leq k} \in \mathbb{C}^{k^2}$ the element
\[ R_{\boldsymbol{\lambda}} = R_0 \exp\!\left( \sum_{i,j=1}^k \lambda_{i,j} \, x_i \otimes x_j g \right) = R_0 \prod_{i,j=1}^k \bigl( 1 \otimes 1 + \lambda_{i,j} \, x_i \otimes x_j g \bigr) \]
is an $R$-matrix in $B_k$. They satisfy $R_{\boldsymbol{\lambda} + \boldsymbol{\mu}} = R_{\boldsymbol{\lambda}} R_0 R_{\boldsymbol{\mu}}$.
\end{proposition}
\begin{proof}
1. Let $X_{i,j} = x_i \otimes x_jg$ and $T_{i,j} = R_0X_{i,j}$. It is easily seen that $X_{i,j} \in \mathcal{Z}\bigl( \Delta(A) \bigr)$, and hence that $T_{i,j}\Delta(a) = \Delta^{\mathrm{op}}(a)T_{i,j}$ for all $a \in B_k$. Note that
\begin{equation}\label{easyPropsXij}
(\Delta \otimes \mathrm{id})(X_{i,j}) = g_2 (X_{i,j})_{13} + (X_{i,j})_{23}, \qquad (R_0)_{23} g_2 (X_{i,j})_{13} = (X_{i,j})_{13}(R_0)_{23}
\end{equation}
where $g_2 = 1 \otimes g \otimes 1$ and for the second equality we use $g\boldsymbol{e}_{\pm} = \pm \boldsymbol{e}_{\pm}$. Thus
\[ (\Delta \otimes \mathrm{id})(T_{i,j}) = (R_0)_{13}(R_0)_{23} g_2 (X_{i,j})_{13} + (R_0)_{13}(R_0)_{23} (X_{i,j})_{23} = (T_{i,j})_{13}(R_0)_{23} + (R_0)_{13}(T_{i,j})_{23}. \]
The equality $(\mathrm{id} \otimes \Delta)(T_{i,j}) = (T_{i,j})_{13}(R_0)_{12} + (R_0)_{13}(T_{i,j})_{12}$ is obtained similarly. Hence the $k^2$ elements $T_{i,j}$ satisfy the linear conditions in \eqref{tangentSpaceRMatrix} which define $\mathbf{T}_{R_0} \mathrm{RMat}(B_k)$. Moreover they are linearly independent (because the elements $x_i \otimes x_j g$ are linearly independent and $R_0$ is invertible), and since $\dim \mathbf{T}_{R_0} \mathrm{RMat}(B_k) = k^2$ they form a basis. 
\\2. The elements $X_{i,j} = x_i \otimes x_jg$ commute with each other and are nilpotent of order $2$, which gives the second equality for $R_{\boldsymbol{\lambda}}$. The last claim is also due to the commutation of the elements $X_{i,j}$. Now using \eqref{easyPropsXij} and the fact that $g_2 (X_{i,j})_{13}$ commutes with $(X_{i,j})_{23}$ we get
\begin{align*}
&(\Delta \otimes \mathrm{id})(R_{\boldsymbol{\lambda}}) = (R_0)_{13} (R_0)_{23} \prod_{i,j=1}^k \bigl( 1^{\otimes 3} + \lambda_{i,j}g_2 (X_{i,j})_{13} + \lambda_{i,j}(X_{i,j})_{23} \bigr)\\
=\:&(R_0)_{13} (R_0)_{23} \prod_{i,j=1}^k \bigl( 1^{\otimes 3} + \lambda_{i,j}g_2 (X_{i,j})_{13}\bigr) \bigl( 1^{\otimes 3} + \lambda_{i,j}(X_{i,j})_{23} \bigr)\\
=\:&(R_0)_{13} (R_0)_{23} \prod_{i,j=1}^k \bigl( 1^{\otimes 3} + \lambda_{i,j}g_2 (X_{i,j})_{13}\bigr) \prod_{i,j=1}^k\bigl( 1^{\otimes 3} + \lambda_{i,j}(X_{i,j})_{23} \bigr)\\
=\:&(R_0)_{13} \prod_{i,j=1}^k \bigl( 1^{\otimes 3} + \lambda_{i,j} (X_{i,j})_{13}\bigr) (R_0)_{23}\prod_{i,j=1}^k\bigl( 1^{\otimes 3} + \lambda_{i,j}(X_{i,j})_{23} \bigr) = (R_{\boldsymbol{\lambda}})_{13}(R_{\boldsymbol{\lambda}})_{23}.
\end{align*}
The computation for $(\mathrm{id} \otimes \Delta)(R_{\boldsymbol{\lambda}})$ is similar.
\end{proof}

\begin{remark}
1. We have $\dim H^n_{\mathrm{DY}}(P,\phi^{R_{\boldsymbol{\lambda}}}) = \dim H^n_{\mathrm{DY}}(P,\phi^{R_0})$ for all $\boldsymbol{\lambda} \in \mathbb{C}^{k^2}$. The proof of this equality uses Proposition \ref{propAdjunctionTheoremForMonProductInAmod} exactly as we did above for $\boldsymbol{\lambda}=0$. Let $(B_k^*)_{\boldsymbol{\lambda}}$ be the coefficient for $R_{\boldsymbol{\lambda}}$, which is $B_k^*$ as a vector space for all $\boldsymbol{\lambda}$. The action of $B_k^{\otimes 2}$ on $(B_k^*)_{\boldsymbol{\lambda}}$ equals that on $(B_k^*)_0$ as it does not depend on the $R$-matrix. Note that  $y_m(x_m\boldsymbol{e}_-) = 1$ and $y_m$ is equal to $0$ on all the other elements $x^{e_1} \ldots x_k^{e_k}\boldsymbol{e}_{\pm}$. Hence
$\ell^{(+)}(y_i)  = -\sum_{m=1}^k \lambda_{m,i} x_m$ and $\ell^{(-)}(y_i) = \sum_{m=1}^k \lambda_{m,i} x_m$ where we used the formula $R_{\boldsymbol{\lambda}}^{-1} = (S \otimes \mathrm{id})(R_{\boldsymbol{\lambda}})$. Also $\ell^{(\pm)}(h) = g$. As a result the action of $y_i \otimes 1$ (resp. $1 \otimes y_i$) on $(B_k^*)_{\boldsymbol{\lambda}}$ is equal to the action of $-\sum_{m=1}^k \lambda_{m,i} x_mg \otimes 1$ (resp. $\sum_{m=1}^k \lambda_{m,i}1 \otimes x_m$) and the action of $h \otimes 1$, $1 \otimes h$ are equal to the actions of $g \otimes 1$, $1 \otimes g$. Repeating the arguments which were used for $\boldsymbol{\lambda} = 0$, we find that $\Hom_{D(B_k)^{\otimes 2}}\bigl(\mathcal{C}_s \boxtimes \mathcal{C}_t, (B_k^*)_{\boldsymbol{\lambda}} \bigr)$ is the same for all $\boldsymbol{\lambda}$'s.
\\2. We deduce from the previous item and \eqref{dimensionFormulaForTRMat} that $\dim \mathbf{T}_{R_{\boldsymbol{\lambda}}}\mathrm{RMat}(B_k) =k^2$ for all $\boldsymbol{\lambda} \in \mathbb{C}^{k^2}$. This also immediately follows from \cite{PV} who proved that $\mathrm{RMat}(B_k) \cong M_k(\mathbb{C})$. Nevertheless note that our arguments do not require the knowledge of the whole $\mathrm{RMat}(H)$ to get the dimension of tangent spaces.
\end{remark}

\subsection{DY cohomology of twisted restriction functors}\label{subsectionDYTwistedFiber}
\indent Recall the notations introduced at the beginning of \S\ref{sectionFinDimHopf}. A {\em Drinfeld twist} for $H$ is an invertible element $J = s_i \otimes t_i \in H^{\otimes 2}$ (with implicit summation) such that
\[ s_i s_j^{(1)} \otimes t_i s_j^{(2)} \otimes t_j = s_j \otimes s_i t_j^{(1)} \otimes t_i t_j^{(2)} \quad \text{ and } \quad \varepsilon(s_i)t_i = s_i\varepsilon(t_i) = 1. \] 
Then there is a new Hopf algebra $H_J$, called the {\em twist of $H$}, which is $H$ as an algebra with the same counit but with new coproduct $\Delta_J$ and antipode $S_J$ given by
\[ \Delta_J(h) = J \Delta(h) J^{-1}, \qquad S_J(h) = u_J S(h) u_J^{-1} \]
where
\begin{equation}\label{drinfeldElementTwistJ}
u_J = s_iS(t_i), \qquad u_J^{-1} = S(\overline{s}_i)\overline{t}_i
\end{equation}
with the notation $J^{-1} = \overline{s}_i \otimes \overline{t}_i$. We write
\[ \Delta_J(h) = h^{(1)_J} \otimes h^{(2)_J} \]
to distinguish between Sweedler's notations for $\Delta$ and for $\Delta_J$.

\smallskip

\indent Let $K$ be a Hopf subalgebra of $H_J$; it means in particular that $K$ is endowed with the coproduct $\Delta_J$. As an associative algebra $K$ is a subalgebra of $H$ and thus we have the restriction functor $F = \mathrm{Res}^H_K : H\text{-}\mathrm{mod} \to K\text{-}\mathrm{mod}$. The twist $J$ gives a monoidal structure on $F$:
\begin{equation}\label{monoidalStructFiberFunctor}
F(X) \otimes F(Y) \overset{\sim}{\longrightarrow} F(X \otimes Y), \qquad x \otimes y \mapsto (\overline{s}_i \cdot x) \otimes (\overline{t}_i \cdot y).
\end{equation}
Let $F_J$ be the restriction functor $F$ endowed with this monoidal structure, which we will denote by $F^{(2)}_J$. Our goal is to describe all the arrows appearing in the diagram \eqref{diagramAdjunctionsBetweenCentralizers} for $F_J$ and give the statement of the adjunction theorem (Thm.\,\ref{thmChangeOfCoeffDY}) for the DY cohomology of $F_J$.

\begin{remark}\label{remarkDualsResFunctor} Note that $F_J(X^{\vee})$ and $F_J(X)^{\vee}$ are both $X^*$ as a vector space but they are endowed with different actions of $K$. Indeed, if $\varphi \in F_J(X^{\vee})$ then $\langle k \cdot \varphi, x \rangle = \langle \varphi, S(k) \cdot x \rangle$ while if $\varphi \in F_J(X)^{\vee}$ then $\langle k \cdot \varphi, x \rangle = \langle \varphi, S_J(k) \cdot x \rangle$. It follows from the definition of $S_J$ that there is an isomorphism of $K$-modules
\begin{equation}\label{commentNonTrivialIsoDuals}
\mathsf{d}_X : F_J(X)^{\vee} \overset{\sim}{\longrightarrow} F_J(X^{\vee}), \qquad \varphi \mapsto \varphi(u_J \,\cdot\, ?)
\end{equation}
where $\varphi(u_J \,\cdot\, ?) \in X^*$ is defined by $x \mapsto \varphi(u_J \cdot x)$ with $u_J$ from \eqref{drinfeldElementTwistJ}. The collection $\mathsf{d} = (\mathsf{d}_X)_{X \in H\text{-}\mathrm{mod}}$ is a natural isomorphism. The element $u_J$ satisfies
\begin{equation}\label{monoidalDuality}
s_i u_J^{(1)} \otimes t_i u_J^{(2)} = u_JS(\overline{t}_i) \otimes u_J S(\overline{s}_i)
\end{equation}
which is equivalent to
the commutation of the following diagram for all $X,Y \in H\text{-}\mathrm{mod}$:
\begin{equation*}
\xymatrix@C=5em{
F_J(X)^{\vee} \otimes F_J(Y)^{\vee} \ar[d]_-{\mathsf{d}_X \otimes \mathsf{d}_Y} \ar[r]^-{\sim} & \bigl( F_J(Y) \otimes F_J(X) \bigr)^{\vee} \ar[r]^-{\left((F^{(2)}_J)_{Y,X}\right)^{\vee-1}} & F_J(Y \otimes X)^{\vee}\ar[d]^-{\mathsf{d}_{Y \otimes X}}\\
F_J(X^{\vee}) \otimes F_J(Y^{\vee}) \ar[r]^-{(F^{(2)}_J)_{X^{\vee},Y^{\vee}}} & F_J(X^{\vee} \otimes Y^{\vee}) \ar[r]^-{\sim} & F_J\bigl( (Y \otimes X)^{\vee}) \bigr)
} \end{equation*}
\end{remark}

\indent We first give a representation-theoretic description of the monad $Z_{F_J}$, defined in general in \S\ref{relExtGroupsDYCohomology}. Let $H_J^{*\mathrm{op}}$ be the dual vector space $H^*$ endowed with the product $\varphi\psi = (\psi \otimes \varphi) \circ \Delta_J$ and with the usual coproduct $\langle \Delta(\varphi), h \otimes h' \rangle = \varphi(hh')$. Then the restricted duality pairing $\sigma : K \otimes H_J^{*\mathrm{op}} \to \Bbbk$, $k \otimes \varphi \mapsto \varphi(k)$
is a skew-pairing:
\begin{align*}
&\sigma(kk',\varphi) = \varphi(kk') = \varphi^{(1)}(k) \varphi^{(2)}(k') = \sigma\bigl(k,\varphi^{(1)}\bigr) \sigma\bigl(k',\varphi^{(2)}\bigr),\\
&\sigma(k,\varphi\psi) = (\varphi\psi)(b) = \psi(k^{(1)_J})\varphi(k^{(2)_J}) = \sigma\bigl(\varphi,k^{(2)_J}\bigr) \sigma\bigl( \psi, k^{(1)_J} \bigr).
\end{align*}
Thus we can form the quantum double $D(H_J^{*\mathrm{op}},K)$ of the skew-paired Hopf algebras $K$ and $H_J^{*\mathrm{op}}$, see e.g. \cite[\S 8.2]{KS}. As a vector space it is $H_J^{*\mathrm{op}} \otimes K$. Its product is such that $H_J^{*\mathrm{op}}$ and $K$ are subalgebras; thus we write $\varphi k$ instead of $\varphi \otimes k$. Moreover
\begin{equation}\label{commutationRelationGeneralizedDouble}
k \varphi = \left(k^{(3)_J} \triangleright \varphi \triangleleft S_J(k^{(1)_J})\right) k^{(2)_J}
\end{equation}
where $\triangleright$ and $\triangleleft$ are the coregular actions defined in \eqref{defCoregular}. The algebra $D(H_J^{*\mathrm{op}},K)$ is actually a Hopf algebra with coproduct $\Delta_{D(H_J^{*\mathrm{op}},K)}(\varphi k) = \varphi^{(1)} k^{(1)_J} \otimes \varphi^{(2)} k^{(2)_J}$.

\smallskip

\indent Since $K$ is a subalgebra of $D(H_J^{*\mathrm{op}},K)$ we have the induction-restriction monad $\mathbb{T}^{D(H_J^{*\mathrm{op}},K)}_K$ on $K\text{-}\mathrm{mod}$ as defined in general in the introduction of \S\ref{sectionFinDimHopf}. Note that
\[ T^{D(H_J^{*\mathrm{op}},K)}_K(X) = D(H_J^{*\mathrm{op}},K) \otimes_K X = (H_J^{*\mathrm{op}} \otimes K) \otimes_K X \cong H_J^{*\mathrm{op}} \otimes X \]
where the action of $K$ on the vector space $H^* \otimes X$ is
\begin{equation}\label{actionBOnTDV}
k \cdot (\varphi \otimes x) = \bigl(k^{(3)_J} \triangleright \varphi \triangleleft S_J(k^{(1)_J}) \bigr) \otimes k^{(2)_J} \cdot x
\end{equation}
because of the formula \eqref{commutationRelationGeneralizedDouble}, while the action of $H_J^{*\mathrm{op}}$ is by left multiplication on the first tensorand. The multiplication $H_J^{*\mathrm{op}} \otimes H_J^{*\mathrm{op}} \otimes X \to H_J^{*\mathrm{op}} \otimes X$ of this monad simply reduces to the multiplication in $H_J^{*\mathrm{op}}$, and thus we write $\mu^{H_J^{*\mathrm{op}}}$ instead of $\mu^{D(H_J^{*\mathrm{op}},K)}_K$. Similarly the unit of $\mathbb{T}^{D(H_J^{*\mathrm{op}},K)}_K$ is given by $x \mapsto \varepsilon \otimes x$, so we denote it by $\eta^{H_J^{*\mathrm{op}}}$ instead of $\eta^{D(H_J^{*\mathrm{op}},K)}_K$.

\smallskip

\indent We introduce a notation which will be used in the proof of the next proposition. If $X$ is a (finite-dimensional) $H$-module then for any $x \in X$ and $\varphi \in X^*$ we can form the {\em matrix coefficient} $\overset{X}{M}{^{\varphi}_x} \in H^*$ defined by
\begin{equation}\label{defMatCoeff}
\forall \, h \in H, \quad \overset{X}{M}{^{\varphi}_x}(h) = \varphi(h \cdot x).
\end{equation}
They satisfy $\overset{X}{M}{^{h \cdot \varphi}_x} = \overset{X}{M}{^{\varphi}_x} \triangleleft S(h)$ and $\overset{X}{M}{^{\varphi}_{h \cdot x}} = h \triangleright \overset{X}{M}{^{\varphi}_x}$ because the action of $H$ on $X^{\vee}$ is defined by $(h \cdot \varphi)(x) = \varphi\bigl( S(h)x \bigr)$.
\begin{proposition}\label{propDescriptionMonadSkewDouble}
The monads $Z_{F_J}$ and $\mathbb{T}^{D(H_J^{*\mathrm{op}},K)}_K$ are equal.
\end{proposition}
\begin{proof}
Recall that $Z_{F_J}(V) = \int^{X \in H\text{-}\mathrm{mod}} F_J(X^{\vee}) \otimes V \otimes F_J(X)$ for all $V \in K\text{-}\mathrm{mod}$. Consider the dinatural transformation
\begin{align}
\begin{split}\label{univDinatTransfoZFJ}
i_X^{F_J}(V) : F_J(X^{\vee}) \otimes V \otimes F_J(X) &\to T^{D(H_J^{*\mathrm{op}},K)}_K(V) = H_J^{*\mathrm{op}} \otimes V\\
\varphi \otimes v \otimes x &\mapsto \bigl( \overset{X}{M}{^{\varphi}_x} \triangleleft u_J^{-1}\bigr) \otimes v
\end{split}
\end{align}
with notations from \eqref{defMatCoeff} and \eqref{drinfeldElementTwistJ}. Note by \eqref{actionBOnTDV} that $i_X^{F_J}(V)$ is a $K$-linear morphism:
\begin{align*}
&i_X^{F_J}(V)\bigl( k \cdot (\varphi \otimes v \otimes x) \bigr) = i^{F_J}_X(V)\bigl( k^{(1)}_J \cdot \varphi \otimes k^{(2)_J} \cdot v \otimes k^{(3)_J} \cdot x \bigr)\\
=\:& \bigl( \overset{X}{M}{^{k^{(1)_J} \cdot \varphi}_{k^{(3)_J} \cdot x}} \triangleleft u_J^{-1} \bigr) \otimes k^{(2)_J} \cdot v = \bigl( k^{(3)_J} \triangleright \overset{X}{M}{^{\varphi}_x} \triangleleft S(k^{(1)_J}) u_J^{-1} \bigr) \otimes k^{(2)_J} \cdot v \\
=\:&\bigl( k^{(3)_J} \triangleright \overset{X}{M}{^{\varphi}_x} \triangleleft u_J^{-1} S_J(k^{(1)_J}) \bigr) \otimes k^{(2)_J} \cdot v = k \cdot i_X^{F_J}(V)(\varphi \otimes v \otimes x).
\end{align*}
We claim that $i^{F_J}(V)$ is universal. Indeed take any dinatural transformation $g_X :  F_J(X^{\vee}) \otimes V \otimes F_J(X) \to W$. Let $H$ be the regular $H$-module and define $\overline{g} : H_J^{*\mathrm{op}} \otimes V \to V$ by $\overline{g}(\alpha \otimes v) = g_H\bigl( (\alpha \triangleleft u_J) \otimes v \otimes 1_H \bigr)$. For all $h \in H$ the right multplication by $h$ gives a $H$-linear map $r_h : H \to H$. Then by dinaturality of $g$
\begin{align*}
g_H(\alpha \otimes v \otimes h) &= g_H\bigl(\alpha \otimes v \otimes r_h(1_H)\bigr) = g_H\bigl(r_h^*(\alpha) \otimes v \otimes 1_H\bigr) = g_H\bigl( (h \triangleright \alpha) \otimes v \otimes 1_H \bigr)\\
&= \overline{g}\bigl( (h \triangleright \alpha \triangleleft u_J^{-1}) \otimes v \bigr) = \overline{g}\bigl( (\overset{H}{M}{^{\alpha}_h} \triangleleft u_J^{-1}) \otimes v \bigr) = \overline{g} \circ i^{F_J}_H(V)(\alpha \otimes v \otimes h).
\end{align*}
Hence $g_H = \overline{g} \circ i^{F_J}_H(V)$ and since $H$ is a projective generator of $H\text{-}\mathrm{mod}$ we conclude that $g_X = \overline{g} \circ i^{F_J}_X(V)$ for all $X \in H\text{-}\mathrm{mod}$ by item 2 in Lemma \ref{lemNatTransfoProjGen}. This proves that $Z_{F_J} = T^{D(H_J^{*\mathrm{op}},K)}_K$ as functors. To compute the multiplication of $Z_{F_J}$ note that the general definition of $i^{(2)}$ in  \eqref{defUniversalDinatDoubleCoend} applied to \eqref{univDinatTransfoZFJ} reads
\[ i^{F_J,(2)}_{X,Y}(V)(\varphi \otimes \psi \otimes v \otimes x \otimes y) = (\overset{Y}{M}{^{\varphi}_y} \triangleleft u_J^{-1}) \otimes (\overset{X}{M}{^{\psi}_x} \triangleleft u_J^{-1}) \otimes v. \]
Thus by the general definition of $\mu^{Z_{F_J}}$ in \eqref{defMuMonadZF} we have
\begin{align*}
&\mu^{Z_{F_J}}_V\left(\overset{Y}{M}{^{\varphi}_y} \otimes \overset{X}{M}{^{\psi}_x} \otimes v \right) = \mu^{Z_{F_J}}_V \circ i^{F_J,(2)}_{X,Y}(V)\bigl( \varphi(u_J\cdot ?) \otimes \psi(u_J\cdot ?) \otimes v \otimes x \otimes y\bigr)\\
=\:&i_{X \otimes Y}^{F_J}(V) \circ \left( (F^{(2)}_J)_{Y^{\vee},X^{\vee}} \otimes \mathrm{id}_V \otimes (F^{(2)}_J)_{X,Y} \right)\bigl( \varphi(u_J\cdot ?) \otimes \psi(u_J\cdot ?) \otimes v \otimes x \otimes y\bigr)\\
=\:& i_{X \otimes Y}^{F_J}(V) \bigl( \varphi(u_JS(\overline{s}_i)\cdot ?) \otimes \psi(u_JS(\overline{t}_i)\cdot ?) \otimes v \otimes \overline{s}_j \cdot x \otimes \overline{t}_j \cdot y\bigr)\\
=\:& i_{X \otimes Y}^{F_J}(V) \bigl( \varphi(t_iu_J^{(2)}\cdot ?) \otimes \psi(s_iu_J^{(1)}\cdot ?) \otimes v \otimes \overline{s}_j \cdot x \otimes \overline{t}_j \cdot y\bigr)\\
=\:& \overset{X \otimes Y}{M}{^{\psi(s_i \cdot ?) \otimes \varphi(t_i\cdot ?)}_{\overline{s}_j\cdot x \otimes \overline{t}_j \cdot y}} \otimes v = \left( \bigl( \overset{X}{M}{^{\psi}_x} \otimes \overset{Y}{M}{^{\varphi}_y} \bigr) \circ \Delta_J \right) \otimes v = \bigl( \overset{Y}{M}{^{\varphi}_y}\overset{X}{M}{^{\psi}_x} \bigr) \otimes v
\end{align*}
where the third equality uses \eqref{monoidalStructFiberFunctor} and the fourth equality uses \eqref{monoidalDuality}. If we take $X=Y=H$ and $x=y=1_H$ this yields $\mu^{Z_{F_J}}(\varphi \otimes \psi \otimes v) = \varphi\psi \otimes v = \mu^{H_J^{*\mathrm{op}}}(\varphi \otimes \psi \otimes v)$, as desired. Finally it is readily seen that the units of $Z_{F_J}$ and $\mathbb{T}^{D(H_J^{*\mathrm{op}},K)}_K$ coincide.
\end{proof}

As a byproduct, we have an isomorphism of $\Bbbk$-linear categories
\[ \mathcal{Z}(F_J) \cong Z_{F_J}\text{-}\mathrm{mod} = D(H_J^{*\mathrm{op}},K)\text{-}\mathrm{mod}. \]
Explicitly, for $(V,\rho) \in \mathcal{Z}(F_J)$ then in particular $V$ is a $K$-module and it becomes a $D(H_J^{*\mathrm{op}},K)$-module if we define the action of $H_J^{*\mathrm{op}}$ by $\varphi \cdot v = (\varphi \otimes \mathrm{id}_V)\bigl( \rho_H(v \otimes 1_H)\bigr)$ where $H$ is the regular module. Conversely for $V \in D(H_J^{*\mathrm{op}},K)\text{-}\mathrm{mod}$ then $V$ is a $K$-module by restriction and the formula $\rho_X(v \otimes x) = h_i \cdot x \otimes h^i \cdot v$ for all $X \in H\text{-}\mathrm{mod}$, $v \in V$ and $x \in X$ gives a half-braiding for $V$ relative to $F_J$, where $\{h_i\}$ and $\{h^i\}$ are dual bases of $H$ and $H^*$.

\smallskip

\indent For $K=H$, $F = \mathrm{Id}$ and $J=1$, Proposition \ref{propDescriptionMonadSkewDouble} gives $Z_{H\text{-}\mathrm{mod}} = \mathbb{T}^{D(H)}_H$ where $D(H)$ is the Drinfeld double of $H$ and we simply recover the well-known isomorphism $\mathcal{Z}(H\text{-}\mathrm{mod}) \cong Z_{H\text{-}\mathrm{mod}}\text{-}\mathrm{mod} = D(H)\text{-}\mathrm{mod}$. Hence in the present situation the diagram \eqref{diagramAdjunctionsBetweenCentralizers} becomes
\begin{equation}\label{diagramAdjunctionsSkewDouble}
\xymatrix@C=5em@R=.7em{
\:\:D(H)\text{-}\mathrm{mod}\ar@/^.7em/[dd]^{\mathrm{Res}} \ar@/^.7em/[r]^-{\widetilde{F_J}}_-{\text{\normalsize \rotatebox{270}{$\dashv$}}} & \ar@/^.7em/[l]^-{\widetilde{R}} D(H_J^{*\mathrm{op}},K)\text{-}\mathrm{mod} \ar@/^.7em/[dd]^{\mathrm{Res}}\\
\dashv & \dashv\\
\ar@/^.7em/[uu]^{\mathrm{Ind}}H\text{-}\mathrm{mod} \ar@/^.7em/[r]^{F_J}_{\text{\normalsize \rotatebox{270}{$\dashv$}}} & \ar@/^.7em/[l]^R \ar@/^.7em/[uu]^{\mathrm{Ind}}K\text{-}\mathrm{mod}
} \end{equation}
where Res and Ind are the restriction and induction functors. Recall a well-known easy fact:
\begin{lemma}\label{lemmaRightAdjointPullback}
Let $A,B$ be finite-dimensional $\Bbbk$-algebras and $f : A \to B$ be an algebra morphism. The right adjoint of the pullback functor $f^* : B\text{-}\mathrm{mod} \to A\text{-}\mathrm{mod}$ is $W \mapsto \Hom_A\bigl( f^*(B),W \bigr)$ equipped with the action defined by $(b \cdot \gamma)(b') = \gamma(b'b)$ for all $b \in B$ and $b' \in f^*(B)$.
\end{lemma}
\noindent Hence the right adjoint $R$ of $F_J$ is the coinduction functor
\[ R(W) = \mathrm{Hom}_K(H,W) \]
where $H$ is endowed with the left multiplication of $K$ and the action of $h \in H$ on $g \in R(W)$ is defined by $(h \cdot g)(h') = g(h'h)$. On morphisms $R$  is given by pushforward. The functors $\widetilde{F}_J$ and $\widetilde{R}$ are defined in general in \S\ref{sectionChangeOfFunctor}; in the present situation they take the following form:

\begin{proposition}\label{propLiftAdjunctionForFiberFunctor}
Recall diagram \eqref{diagramAdjunctionsSkewDouble} and the notations $J = s_i \otimes t_i$ and $u_J = s_iS(t_i)$.
\\1. There is a morphism of algebras
\begin{equation}\label{defMorphismGammaBetweenDoubles}
\Upsilon : D(H_J^{*\mathrm{op}},K) \to D(H), \quad \varphi k \mapsto t_i \bigl(\overline{t}_j \triangleright \varphi \triangleleft s_i \bigr) \overline{s}_jk.
\end{equation}
The pullback functor $\Upsilon^* : D(H)\text{-}\mathrm{mod} \to D(H_J^{*\mathrm{op}},K)\text{-}\mathrm{mod}$ is equal to $\widetilde{F_J}$.
\\2. Let $(\mathsf{W},\smallsquare) \in D(H_J^{*\mathrm{op}},K)\text{-}\mathrm{mod}$, where $\smallsquare$ denotes the action. We have $\widetilde{R}(\mathsf{W}) = \mathrm{Hom}_K(H,\mathsf{W})$ equipped with the action of $D(H)$ given by
\begin{equation}\label{actionDAOnHomBAW}
\forall \, h'\in H, \quad \bigl( (\varphi h) \cdot f \bigr)(h') =\left(  t_jt_i^{(2)}h'^{(3)} \triangleright \varphi \triangleleft S\bigl(s_ih'^{(1)}\bigr)u_J^{-1} \right) \!\smallsquare f\!\left(s_jt_i^{(1)}h'^{(2)}h\right)
\end{equation}
for all $\varphi h \in D(H)$ and $f \in \mathrm{Hom}_K(H,\mathsf{W})$. On morphisms $\widetilde{R}$ is given by pushforward.
\end{proposition}
\begin{proof}
1. Recall the general definition of lifted functor in \eqref{liftFunctorAlongMonads}. To describe $\widetilde{F_J}$ we must compute for all $V \in H\text{-}\mathrm{mod}$ the map $\zeta_V : Z_{F_J}F_J(V) \to F_J Z_{H\text{-}\mathrm{mod}}(V)$ defined in \eqref{diagramDefiningZetaForCentralMonads}. Denote by
\[ i_X(V) : X^{\vee} \otimes V \otimes X \to H^{*\mathrm{op}} \otimes V, \quad \varphi \otimes v \otimes x \mapsto \overset{X}{M}{^{\varphi}_x} \otimes v \]
the universal dinatural transformation for $Z_{H\text{-}\mathrm{mod}}(V) = H^{*\mathrm{op}} \otimes V$. Note that $Z_{F_J}F_J(V) = T^{D(H_J^{*\mathrm{op}},K)}_K F_J(V) = H_J^{*\mathrm{op}} \otimes V$ and $F_J Z_{H\text{-}\mathrm{mod}}(V) = F_J T^{D(H)}_H(V) = H^{*\mathrm{op}} \otimes V$ as vector spaces. We have
\begin{align*}
\zeta_V(\varphi \otimes v) &= \zeta_V \circ i_H^{F_J}\bigl( F_J(V) \bigr)\bigl( (\varphi \triangleleft u_J) \otimes v \otimes 1_H \bigr) = F_J\bigl( i_H(V) \bigr) \circ (F_J^{(3)})_{X^{\vee},V,X}\bigl( \mathsf{d}_H(\varphi) \otimes v \otimes 1_H \bigr)\\
&=F_J\bigl( i_H(V) \bigr) \left( \bigl(\varphi \triangleleft u_J S(\overline{s}_i)\bigr) \otimes \overline{t}_i^{(1)}\overline{s}_j\cdot v \otimes \overline{t}_i^{(2)}\overline{t}_j \right) = \bigl( \overline{t}_i^{(2)}\overline{t}_j \triangleright \varphi \triangleleft u_JS(\overline{s}_i) \bigr) \otimes \overline{t}_i^{(1)}\overline{s}_j\cdot v
\end{align*}
where the first equality is a trick based on the regular representation and \eqref{univDinatTransfoZFJ} while the second uses the definition of $\zeta_V$. Hence if $(\mathsf{V},\cdot) \in D(H)\text{-}\mathrm{mod}$ (where $\cdot$ is the action of $D(H)$ on $\mathsf{V}$) then $\widetilde{F_J}(\mathsf{V},\cdot) = (\mathsf{V},\smallsquare)$ where the action $\smallsquare$ of $\varphi k \in D(H_J^{*\mathrm{op}},K)$ on $\mathsf{v} \in \mathsf{V}$ is
\begin{align*}
\varphi k \smallsquare \mathsf{v} &= \left(\overline{t}_i^{(2)}\overline{t}_j \triangleright \varphi \triangleleft u_J S(\overline{s}_i) \right) \overline{t}_i^{(1)}\overline{s}_jk\cdot \mathsf{v} = \overline{t}_i^{(2)}\left( \overline{t}_j \triangleright \varphi \triangleleft u_J S(\overline{s}_i)\overline{t}_i^{(1)} \right) \overline{s}_jk\cdot \mathsf{v}\\
&= t_i \bigl(\overline{t}_j \triangleright \varphi \triangleleft s_i \bigr) \overline{s}_jk \cdot \mathsf{v} = \Upsilon(\varphi k) \cdot \mathsf{v}
\end{align*}
where in the second equality we used \eqref{definingRelDrinfeldDouble} while in the third equality we used that $S(\overline{s}_i)\overline{t}_i^{(1)} \otimes \overline{t}_i^{(2)} = u_J^{-1}s_i \otimes t_i$ which easily follows from the defining properties of $J = s_i \otimes t_i$ and the definition of $u_J$ in \eqref{drinfeldElementTwistJ}. As a result $x \smallsquare \mathsf{v} = \Upsilon(x) \cdot \mathsf{v}$ for all $x \in D(H_J^{*\mathrm{op}},K)$. In particular if $\mathsf{V} = D(H)$ is the regular representation we obtain that $\Upsilon$ is a morphism of algebras:
\begin{align*}
&\Upsilon(xy) = \Upsilon(xy) \cdot 1_{D(H)} = xy \smallsquare 1_{D(H)} = x \smallsquare (y \smallsquare 1_{D(H)})\\
=\:& \Upsilon(x) \cdot (\Upsilon(y) \cdot 1_{D(H)})= \Upsilon(x)\Upsilon(y) \cdot 1_{D(H)} = \Upsilon(x)\Upsilon(y).
\end{align*}
2. One could use \eqref{rightAdjointLift} and \eqref{descriptionXiForCentralMonads} to compute $\widetilde{R}$. But thanks to Lemma \ref{lemmaRightAdjointPullback} we know directly from the previous item that the right adjoint of $\widetilde{F_J}$ is given by
\[ \widetilde{R}(\mathsf{W}) = \Hom_{D(H_J^{*\mathrm{op}},K)}\bigl( \Upsilon^*(D(H)), \mathsf{W} \bigr) \]
where the $D(H)$-module structure on $\widetilde{R}(\mathsf{W})$ is $(x \cdot g)(y) = g(yx)$ for all $g \in \widetilde{R}(\mathsf{W})$ and $x,y \in D(H)$. A straightforward computation shows that $\varphi h = \Upsilon\bigl[ t_jt_i^{(2)} \triangleright \varphi \triangleleft S(s_i)u_J^{-1} \bigr] s_jt_i^{(1)}h$ for any $\varphi h \in D(H)$. Hence if $g \in \widetilde{R}(\mathsf{W})$ we have $g(\varphi h) = \bigl(t_jt_i^{(2)} \triangleright \varphi \triangleleft S(s_i)u_J^{-1} \bigr) \smallsquare g\bigl(s_jt_i^{(1)}h\bigr)$ by $D(H_J^{*\mathrm{op}},K)$-linearity. This shows that $g$ is entirely determined by its values on $H \subset \Upsilon^*(D(H))$ and gives an isomorphism of vector spaces
\[ \Hom_{D(H_J^{*\mathrm{op}},K)}\bigl( \Upsilon^*(D(H)), \mathsf{W} \bigr) \overset{\sim}{\longrightarrow} \mathrm{Hom}_K(H,\mathsf{W}), \quad g \mapsto g|_H. \]
Let $f \in \mathrm{Hom}_K(H,\mathsf{W})$ and write $f = g|_H$; then
\begin{align*}
\bigl( (\varphi h) \cdot f \bigr)(h') &= \bigl( (\varphi h) \cdot g \bigr)(h') = g(h'\varphi h) \overset{\eqref{definingRelDrinfeldDouble}}{=} g\bigl[ \bigl(h'^{(3)} \triangleright \varphi \triangleleft S(h'^{(1)}) \bigr) h'^{(2)}h \bigr]\\
&= \left( t_jt_i^{(2)}h'^{(3)} \triangleright \varphi \triangleleft S(h'^{(1)}) S(s_i)u_J^{-1} \right) \smallsquare g\bigl(s_jt_i^{(1)}h'^{(2)}h\bigr)\\
&= \left(  t_jt_i^{(2)}h'^{(3)} \triangleright \varphi \triangleleft S(s_ih'^{(1)})u_J^{-1} \bigr) \smallsquare f\bigl(s_jt_i^{(1)}h'^{(2)}h\right)
\end{align*}
which is the announced formula.
\end{proof}
\noindent Note by definition of $\Delta_J$ that 
\begin{align*}
&s_ih'^{(1)} \otimes s_j t_i^{(1)}h'^{(2)} \otimes t_j t_i^{(2)}h'^{(3)} = (1 \otimes J) \, (\mathrm{id} \otimes \Delta)(J) \, (\mathrm{id} \otimes \Delta)\bigl( \Delta(h') \bigr)\\
=\:& (\mathrm{id} \otimes \Delta_J)\bigl( \Delta_J(h') \bigr) \, (1 \otimes J) \, (\mathrm{id} \otimes \Delta)(J) = h'^{(1)_J} s_i \otimes h'^{(2)_J} s_j t_i^{(1)} \otimes h'^{(3)_J} t_j t_i^{(2)}.
\end{align*}
Hence using moreover the definition of $S_J$ we can rewrite formula \eqref{actionDAOnHomBAW} as follows:
\[ \bigl( (\varphi h) \cdot f \bigr)(h')  = \left( h'^{(3)_J} t_j t_i^{(2)} \triangleright \varphi \triangleleft S(s_i) u_J^{-1} S_J(h'^{(1)_J} ) \right) \smallsquare f\!\left( h'^{(2)_J} s_j t_i^{(1)} h\right). \]
In this form it is readily seen that $\mathrm{Hom}_K(H,\mathsf{W})$ is stable by the action of $D(H)$.

\begin{remark}
It is non-trivial to check that $\Upsilon$ is a morphism of algebras directly from its definition in \eqref{defMorphismGammaBetweenDoubles}. In the case $K=H$ this is done in \cite[\S 2.2]{MajidOeckl}; they moreover show that, as a Hopf algebra, $D(H_J)$ is isomorphic to a twist of $D(H)$.
\end{remark}

By Theorem \ref{thmChangeOfCoeffDY} (or more precisely \eqref{isosProofThmAdjunction}) and Proposition \ref{propLiftAdjunctionForFiberFunctor}:
\begin{corollary}\label{coroChangeFunctorThmForResFunctor} We have
\[ H^{\bullet}_{\mathrm{DY}}(F_J) \cong \mathrm{Ext}^{\bullet}_{D(H),H}\bigl( \Bbbk, \mathrm{Hom}_K(H,\Bbbk) \bigr) \]
where in its first instance the ground field $\Bbbk$ is endowed with the trivial $D(H)$-module structure while $\mathrm{Hom}_K(H,\Bbbk) = \bigl\{ f \in H^* \,\big|\, \forall\, k \in K, \: f \triangleleft k = \varepsilon(k)f \bigr\}$ has the $D(H)$-module structure given by
\begin{equation}\label{actionDHCoeffResFunctor}
\bigl\langle (\varphi h) \cdot f, h' \bigr\rangle = \varphi\bigl( S(s_ih'^{(1)})u_J^{-1}t_jt_i^{(2)}h'^{(3)} \bigr) \, f\bigl(s_jt_i^{(1)}h'^{(2)}h\bigr)
\end{equation}
for all $f \in H^*$, $\varphi h \in D(H)$ and $h' \in H$.
\end{corollary}

\begin{example}
Assume that $K$ is a {\em Hopf} subalgebra of $H$, so that we can choose $J = 1 \otimes 1$. Then the action of $D(H) = (H^*)^{\mathrm{op}} \otimes H$ on $\Hom_K(H,\Bbbk)$ in Corollary \ref{coroChangeFunctorThmForResFunctor} reduces to
\begin{equation}\label{coeffResFunctorForTrivialTwist}
(\varepsilon \otimes h) \cdot f = h \triangleright f, \qquad (\varphi \otimes 1) \cdot f = \varphi_{(1)} f S(\varphi_{(2)})
\end{equation}
where $\varphi_{(1)} f S(\varphi_{(2)})$ is the product of $\varphi_{(1)}$, $f$ and $S(\varphi_{(2)}) = \varphi_{(2)} \circ S$ in $(H^*)^{\mathrm{op}}$. In particular for $K = \Bbbk$ (the ground field) then $F_J = U : H\text{-}\mathrm{mod} \to \mathrm{vect}_{\Bbbk}$ is the fiber functor and $H^{\bullet}_{\mathrm{DY}}(U) \cong H^{\bullet}_{\mathrm{DY}}\bigl(H\text{-}\mathrm{mod}; \Bbbk, H^* \bigr)$ where the vector space $H^*$ has the $D(H)$-module structure \eqref{coeffResFunctorForTrivialTwist}. This last case was obtained in \cite[\S 4.3]{GHS} by a different method.
\end{example}

\begin{example}\label{exampleComputingCoeffPullbackCoproduct}
Let $R = R^1_i \otimes R^2_i \in H^{\otimes 2}$ be an $R$-matrix for $H$, with implicit summation on $i$. Recall that $R^{-1} = (S \otimes \mathrm{id})(R)$ and consider
\[ J = \underbrace{\bigl( 1 \otimes S(R^1_i) \bigr)}_{s_i} \otimes \underbrace{(R^2_i \otimes 1)}_{t_i}, \qquad J^{-1} = \underbrace{(1 \otimes R^1_i)}_{\overline{s}_i} \otimes \underbrace{(R^2_i \otimes 1)}_{\overline{t}_i}. \]
The element $J$ is easily seen to be a Drinfeld twist for $H \otimes H$, so we have the twisted Hopf algebra $(H \otimes H)_J$. The map $\iota : H \overset{\sim}{\longrightarrow} \Delta(H) \subset (H \otimes H)_J$ identifies $H$ as a Hopf subalgebra of $(H \otimes H)_J$. Indeed $\iota$ is obviously an algebra morphism and is a coalgebra morphism thanks to the twist:
\begin{align*}
\Delta_J\bigl( \iota(h) \bigr) &= J \Delta_{H \otimes H}(h^{(1)} \otimes h^{(2)})J^{-1} = J\bigl( h^{(1)(1)} \otimes h^{(2)(1)} \otimes h^{(1)(2)} \otimes h^{(2)(2)} \bigr) J^{-1} \\
&=J \bigl( h^{(1)} \otimes h^{(3)} \otimes h^{(2)} \otimes h^{(4)} \bigr) J^{-1} = h^{(1)} \otimes h^{(2)} \otimes h^{(3)} \otimes h^{(4)} = \iota(h^{(1)}) \otimes \iota(h^{(2)})
\end{align*}
where the second equality uses that $\Delta_{H \otimes H}(x \otimes y) = x^{(1)} \otimes y^{(1)} \otimes x^{(2)} \otimes y^{(2)}$ by definition of a tensor product of Hopf algebras, the third is by coassociativity and the fourth uses that $R \Delta = \Delta^{\mathrm{op}}R$. This yields a restriction functor $F_J : (H \otimes H)\text{-}\mathrm{mod} \to H\text{-}\mathrm{mod}$ which has the monoidal structure 
\[ \foncIso{(F_J^{(2)})_{X_1 \boxtimes Y_1,X_2 \boxtimes Y_2}}{\begin{array}{c}F_J(X_1 \boxtimes Y_1) \otimes F_J(X_2 \boxtimes Y_2)\\=X_1 \otimes Y_1 \otimes X_2 \otimes Y_2 \end{array}}{\begin{array}{c}F_J\bigl( (X_1 \boxtimes Y_1) \otimes (X_2 \boxtimes Y_2) \bigr)\\=X_1 \otimes X_2 \otimes Y_1 \otimes Y_2 \end{array}}{x_1 \otimes y_1 \otimes x_2 \otimes y_2}{x_1 \otimes (R^2_i \cdot x_2) \otimes (R^1_i \cdot y_1) \otimes y_2}. \]
Let us apply Corollary \ref{coroChangeFunctorThmForResFunctor} to this monoidal functor:
\[ H^{\bullet}_{\mathrm{DY}}(F_J) \cong \mathrm{Ext}^{\bullet}_{D(H) \otimes D(H),H \otimes H}\bigl( \Bbbk \boxtimes \Bbbk,  \mathrm{Hom}_H(H^{\otimes 2},\Bbbk) \bigr) \]
where we identify $D(H \otimes H)$ with $D(H) \otimes D(H)$ and $g \in \mathrm{Hom}_H(H^{\otimes 2},\Bbbk)$ if and only if $g(h^{(1)}x \otimes h^{(2)}y) = \varepsilon(h)g(x \otimes y)$ for all $x,y,h \in H$. The action \eqref{actionDHCoeffResFunctor} of $\alpha a \otimes \beta b \in D(H)^{\otimes 2}$ on $g \in \mathrm{Hom}_H(H^{\otimes 2},\Bbbk)$ is given by
\begin{align}
\begin{split}\label{actionDA2OnAAdual}
\!\!&\bigl\langle (\alpha a \otimes \beta b) \cdot g, x \otimes y \bigr\rangle \\
\!\!=\:& \left\langle \alpha \otimes \beta, S\bigl( s_i(x^{(1)} \otimes y^{(1)}) \bigr) \bigl( R_k^2 \otimes S(R_k^1)  \bigr) t_j t_i^{(2)} (x^{(3)} \otimes y^{(3)}) \right\rangle \left\langle g, s_jt_i^{(1)}(x^{(2)}a \otimes y^{(2)}b) \right\rangle\\
\!\!= \:&\left\langle \alpha, S(x^{(1)})R^2_kR^2_jR^{2\,(2)}_i x^{(3)} \right\rangle  \left\langle \beta, S(y^{(1)}) S\bigl( R^1_k S(R^1_i) \bigr) y^{(3)} \vphantom{R^{2\,(2)}_i} \right\rangle
\left\langle g, R^{2\,(1)}_i x^{(2)} a \otimes S(R^1_j)y^{(2)}b \right\rangle
\end{split}
\end{align}
where in the second equality we used that the element \eqref{drinfeldElementTwistJ} is $u_J = s_i S(t_i) = R^2_i \otimes R^1_i$ with our choice of $J$. For any $g \in \mathrm{Hom}_H(H^{\otimes 2},\Bbbk)$ we can write $g(x \otimes y) = g\bigl(x^{(1)} \otimes x^{(2)}S(x^{(3)})y\bigr) = g(1 \otimes S(x)y)$, whence an isomorphism of vector spaces
\[ \mathrm{Hom}_H(H^{\otimes 2},\Bbbk) \overset{\sim}{\to} \mathrm{Hom}_{\Bbbk}(H,\Bbbk) = H^*, \qquad g \mapsto g(1 \otimes ?). \]
Let us compute the resulting $D(H)^{\otimes 2}$-action on $H^*$. Take $\psi \in H^*$; then $\psi = g(1 \otimes ?)$ for some $g \in \mathrm{Hom}_H(H^{\otimes 2},\Bbbk)$ and from \eqref{actionDA2OnAAdual} we find
\[ \bigl\langle (\alpha a \otimes \beta b) \cdot g, 1 \otimes y \bigr\rangle = \left\langle \alpha, R^2_kR^2_jR^{2\,(2)}_i \right\rangle  \left\langle \beta, S(y^{(1)}) S\bigl( R^1_k S(R^1_i) \bigr) y^{(3)} \vphantom{R^{2\,(2)}_i} \right\rangle
\left\langle \psi, S\bigl(R^1_jR^{2\,(1)}_i a\bigr)y^{(2)}b \right\rangle \]
because $g(x \otimes y) = \psi\bigl(  S(x)y\bigr)$. Elementary properties of $R$-matrices yield
\begin{align*}
&R^2_kR^2_jR^{2\,(2)}_i \otimes R^1_k S(R^1_i) \otimes R^1_jR^{2\,(1)}_i = R^2_k R^{2\,(1)}_i R^2_j \otimes R^1_k S(R^1_i) \otimes R^{2\,(2)}_i R^1_j\\ 
=\:&R^2_k R^2_l R^2_j \otimes R^1_k S(R^1_l) S(R^1_i) \otimes R^2_i R^1_j =R^2_j \otimes S(R^1_i) \otimes R^2_i R^1_j
\end{align*}
and it follows that
\begin{align*}
&\bigl\langle (\alpha a \otimes \beta b) \cdot g, 1 \otimes y \bigr\rangle = \alpha(R^2_j) \bigl\langle \beta, S(y^{(1)})S^2(R^1_i)y^{(3)} \bigr\rangle \left\langle \psi,  S\bigl(R^2_i R^1_j a\bigr)y^{(2)}b \right\rangle\\
=\:&\alpha( R^2_j ) \bigl\langle \beta, S(y^{(1)})S(R^1_i)y^{(3)} \bigr\rangle \left\langle \psi,  S(a) S(R^1_j) R^2_i y^{(2)}b \right\rangle = \alpha(R^2_j) \beta\bigl( S(R^1_i) \bigr) \left\langle \psi,  S(a) S(R^1_j) y R^2_i b \right\rangle.
\end{align*}
As a result $H^{\bullet}_{\mathrm{DY}}(F_J) \cong \mathrm{Ext}^{\bullet}_{D(H)^{\otimes 2},H^{\otimes 2}}\bigl( \Bbbk \boxtimes \Bbbk, H^* \bigr)$ where $H^*$ is endowed with the $D(H)^{\otimes 2}$-module structure given by $\bigl\langle (\alpha a \otimes \beta b) \cdot \psi, y \bigr\rangle = \alpha(R^2_j) \beta\bigl( S(R^1_i) \bigr) \left\langle \psi,  S(a) S(R^1_j) y R^2_i b \right\rangle$. The monoidal functor $F_J$ is equal to $\bigl(P,\phi^R \bigr)$ defined in \eqref{monStructOfPFromRMat}; hence this example gives another proof of Proposition \ref{propAdjunctionTheoremForMonProductInAmod}.
\end{example}

\subsection{Example: restriction functor $B_{l+k}\text{-}\mathrm{mod} \to B_k\text{-}\mathrm{mod}$}\label{subsectionResFunctorBk}
Recall the Hopf algebra $B_k$ from \S\ref{subsectionExampleBk}. Take two strictly positive integers $l,k$ and consider the Hopf algebra $B_{l+k}$, generated by variables $g, x_1, \ldots, x_{l+k}$ modulo the relations \eqref{presentationBk}. The subalgebra generated by $g, x_{l+1}, \ldots, x_{l+k}$ is a Hopf subalgebra isomorphic to $B_k$. Hence the restriction functor
\[ F = \mathrm{Res}^{B_{l+k}}_{B_k} : B_{l+k}\text{-}\mathrm{mod} \to B_k\text{-}\mathrm{mod} \]
is monoidal. Let us use Corollary \ref{coroChangeFunctorThmForResFunctor} and the relatively projective resolution \eqref{relProjResUnitBk} to compute $H^{\bullet}_{\mathrm{DY}}(F)$. Note that here the Drinfeld twist $J = s_i \otimes t_i$ is trivial: $J = 1 \otimes 1$.

\smallskip

\indent If $f \in \Hom_{B_k}(B_{l+k},\mathbb{C})$ then the commutation relations in $B_{l+k}$ and the $B_k$-linearity yield
\begin{align*}
f\bigl( x_1^{e_1} \ldots x_l^{e_l} x_{l+1}^{p_1} \ldots x_{l+k}^{p_k} g^t\bigr) &= (-1)^{(|p| + t) |e|} \varepsilon\bigl( x_{l+1}^{p_1} \ldots x_{l+k}^{p_k} g^t \bigr) f\bigl(x_1^{e_1} \ldots x_l^{e_l}\bigr)\\
&= \delta_{p_1,0}\ldots\delta_{p_k,0}(-1)^{t|e|}f\bigl(x_1^{e_1} \ldots x_l^{e_l}\bigr)
\end{align*}
for all $e_i,p_j,t \in \{0,1\}$, where $|e| = \sum_{i=1}^{l} e_i$, $|p| = \sum_{j=1}^k p_j$ and $\delta$ is the Kronecker symbol. Hence
\[ \Hom_{B_k}(B_{l+k},\mathbb{C}) = \mathrm{span}_{\mathbb{C}}\bigl\{ \bigl(x_1^{e_1} \ldots x_l^{e_l}\bigr)^* + (-1)^{|e|}\bigl(x_1^{e_1} \ldots x_l^{e_l}g\bigr)^* \,\big|\, e_1, \ldots, e_l \in \{0,1\} \bigr\} \]
where $^*$ denotes dual vectors. Recall that the Drinfeld double $D(B_{l+k})$ is generated by $x_i,g, y_i,h$ with $1 \leq i \leq l+k$ modulo the relations given in \cite[\S 5]{GHS}. Write $|e_1,\ldots,e_l\rangle = \bigl(x_1^{e_1} \ldots x_l^{e_l}\bigr)^* + (-1)^{|e|}\bigl(x_1^{e_1} \ldots x_l^{e_l}g\bigr)^*$. According to the computations in the proof of \cite[Lem.\,5.9]{GHS}, the action \eqref{coeffResFunctorForTrivialTwist} of $D(B_{l+k})$ on these basis vectors is given by
\begin{align*}
g \cdot |e_1,\ldots,e_l\rangle &= (-1)^{|e|} \, |e_1,\ldots,e_l\rangle, \qquad h \cdot |e_1,\ldots,e_l\rangle = (-1)^{|e|} \, |e_1,\ldots,e_l\rangle, \\
x_i \cdot |e_1,\ldots,e_l\rangle &= \begin{cases} (-1)^{\sum_{j=i+1}^l e_j} \, |e_1,\ldots,e_{i-1},0,e_{i+1},\ldots, e_l\rangle & \text{if } e_i = 1 \text{ and } 1 \leq i \leq l\\
0 & \text{otherwise}
\end{cases}\\
y_i \cdot |e_1,\ldots,e_l\rangle &= 0.
\end{align*}
In particular $\Hom_{B_k}(B_{l+k},\mathbb{C})$ is generated by $|1,\ldots,1\rangle$. From Corollary \ref{coroChangeFunctorThmForResFunctor} and Theorem \ref{thmDYCohomRelExt} we have
\[ H^{\bullet}_{\mathrm{DY}}\!\left(\mathrm{Res}^{B_{l+k}}_{B_k}\right) \cong \mathrm{Ext}_{D(B_{l+k}),B_{l+k}}^{\bullet}\bigl(\mathbb{C}, \Hom_{B_k}(B_{l+k},\mathbb{C}). \bigr)\]
Consider the relatively projective resolution \eqref{relProjResUnitBk} for $\mathbb{C} \in D(B_{l+k})\text{-}\mathrm{mod}$ and the $D(B_{l+k})$-modules $\mathcal{C}_s = B_{l+k}^{*\mathrm{op}} \boldsymbol{f}_s$ with $s = \pm$. If $\varphi \in \Hom_{D(B_{l+k})}\bigl(\mathcal{C}_s, \Hom_{B_k}(B_{l+k},\mathbb{C})\bigr)$ then since $x_i$ acts by $0$ on $\mathcal{C}_s$ for all $i$ we necessarily have $\varphi(\boldsymbol{f}_s) \in \mathbb{C} \, | 0,\ldots,0 \rangle$. Moreover $h\boldsymbol{f}_s = s\boldsymbol{f}_s$ while $h \cdot | 0,\ldots,0 \rangle = | 0,\ldots,0 \rangle$. Hence 
\[\Hom_{D(B_{l+k})}\bigl(\mathcal{C}_+,\Hom_{B_k}(B_{l+k},\mathbb{C})\bigr) \cong \mathbb{C}, \qquad \Hom_{D(B_{l+k})}\bigl(\mathcal{C}_-,\Hom_{B_k}(B_{l+k},\mathbb{C})\bigr)=0. \] 
We deduce that
\begin{equation}\label{DYCohomResBlk}
\dim H^n_{\mathrm{DY}}\!\left(\mathrm{Res}^{B_{l+k}}_{B_k}\right) = \begin{cases}
\binom{k+l+n-1}{n} &\text{if } n \text{ is even}\\
0 & \text{if } n \text{ is odd}.
\end{cases}
\end{equation}

\appendix

\section{(Di)natural transformations in finite categories}\label{subsectionNatAndCoendsFiniteCat}
\indent Let $\Bbbk$ be a field. This appendix collects some facts about (di)natural transformations, coends and Deligne product in finite $\Bbbk$-linear categories. We include proofs for convenience.

\smallskip

\indent Let $\mathcal{C}$ be a {\em finite} $\Bbbk$-linear category. Recall that any such category is equivalent to the category $A\text{-}\mathrm{mod}$ of finite-dimensional modules over a finite-dimensional $\Bbbk$-algebra \cite[\S 1.8]{EGNO}. A {\em projective generator} is a projective object $P \in \mathcal{C}$ such that for any $X \in \mathcal{C}$ there exists an epimorphism $e : P^n \twoheadrightarrow X$ for some $n \in \mathbb{N}^*$ \cite[\S 1.8]{EGNO}. Such a projective object is not unique but a minimal choice is obtained by taking $P = \bigoplus_{i=1}^l P_i$ where $P_1, \ldots, P_l$ are the indecomposable projective objects of $\mathcal{C}$ up to isomorphism. Note that any $X \in \mathcal{C}$ can be written as a cokernel of a morphism $\alpha : P^m \to P^n$ for some $m,n \in \mathbb{N}$. Indeed since $\mathcal{C}$ is abelian the epimorphism $e : P^n \twoheadrightarrow X$ is the cokernel of some morphism $\beta : Y \to P^n$. Now we cover $Y$ by $P^m \twoheadrightarrow Y$ and let $\alpha : P^m \twoheadrightarrow Y \overset{\beta}{\to} P^n$. Then $\mathrm{coker}(\alpha) = \mathrm{coker}(\beta) = (X,e)$. We call the exact sequence $P^m \overset{\alpha}{\longrightarrow} P^n \overset{e}{\longrightarrow} X \longrightarrow 0$ a {\em projective presentation of $X$}.

\smallskip

\indent A functor is called {\em right exact} if it preserves cokernels. Item 2 in the next lemma is \cite[Prop.\,5.1.7]{KL} where it is stated for bifunctors exact in each variable; the only novelty here is to note that their proof works under a much weaker assumption on the bifunctor.

\begin{lemma}\label{lemNatTransfoProjGen}
Let $\mathcal{C}$ be a finite $\Bbbk$-linear category. Denote by $P$ a projective generator of $\mathcal{C}$ and let $\mathcal{D}$ be a $\Bbbk$-linear abelian category.
\\1. Let $K,L : \mathcal{C} \to \mathcal{D}$ be $\Bbbk$-linear functors and assume that $K$ is right exact. Then the linear map
\begin{align}
\begin{split}\label{natTransfoAndProjGenerator}
\mathrm{Nat}(K,L) &\to \bigl\{  g : K(P) \to L(P) \,\big|\, \forall \, \varphi \in \mathrm{End}_{\mathcal{C}}(P), \:\: g \circ K(\varphi) = L(\varphi) \circ g \bigr\}\\
f &\mapsto f_P
\end{split}
\end{align}
is an isomorphism of vector spaces.
\\2. Let $K : \mathcal{C}^{\mathrm{op}} \times \mathcal{C} \to \mathcal{D}$ be a $\Bbbk$-bilinear functor right exact in the second variable\footnote{Meaning that the functor $K(C,-) : \mathcal{C} \to \mathcal{C}$ is right-exact for all $C \in \mathcal{C}$.}. Then for any $D \in \mathcal{D}$ the linear map
\begin{align}
\begin{split}\label{dinatTransfoAndProjGenerator}
\mathrm{Dinat}(K,D) &\to \bigl\{ g : K(P,P) \to D \,\big|\, \forall \, \varphi \in \mathrm{End}_{\mathcal{C}}(P), \:\: g \circ K(\varphi^{\mathrm{op}},\mathrm{id}_P) = g \circ K(\mathrm{id}_P,\varphi) \bigr\}\\
d &\mapsto d_P
\end{split}
\end{align}
is an isomorphism of vector spaces, where $\varphi^{\mathrm{op}}$ reminds that the morphism $\varphi$ is viewed in $\mathcal{C}^{\mathrm{op}}$.
\end{lemma}
\begin{proof}
1. The inverse map to \eqref{natTransfoAndProjGenerator} is constructed as follows. Let $g : K(P) \to L(P)$ which satisfies the condition above. We first set $f_P = g$. More generally for each $n \geq 1$ the object $P^n$ comes equipped with morphisms $j_i : P \to P^n$ and $\pi_i : P^n \to P$ for all $1 \leq i \leq n$ and we define $f_{P^n} = \sum_{i=1}^n j_i \circ g \circ \pi_i$. The family $(f_{P^n})_{n \in \mathbb{N}^*}$ is natural in the full subcategory $\bigl\{ P^n \, |\, n \in \mathbb{N}^* \bigr\}$ by the required property of $g$ and by the matrix description of morphisms $P^m \to P^n$. Now let $X \in \mathcal{C}$ and take a presentation $P^m \overset{\alpha}{\longrightarrow} P^n \overset{e}{\longrightarrow} X \longrightarrow 0$, so that $(X,e) = \mathrm{coker}(\alpha)$. By the universal property of a cokernel we obtain
\[ \xymatrix@R=.6em{
K(P^m) \ar[rr]^-{K(\alpha)} \ar[dd]_-{f_{P^m}} & & K(P^n)\ar[rr]^-{K(e)} \ar[dd]_-{f_{P^n}} & & K(X) \ar[dd]^{\exists! \, f_X}\\
 & \circlearrowright & & \circlearrowright &\\
L(P^m) \ar[rr]_-{L(\alpha)} && L(P^n) \ar[rr]_-{L(e)} && L(X)
} \]
We must check that $f_X$ does not depend on the presentation of $X$. Let $P^{m'} \overset{\alpha'}{\longrightarrow} P^{n'} \overset{e'}{\longrightarrow} X \longrightarrow 0$ be exact. Then as above this gives a morphism $f'_X : K(X) \to L(X)$ defined by $f'_X \circ K(e') = L(e') \circ f_{P^{n'}}$.  Since $P^n$ is projective there exists $\gamma : P^n \to P^{n'}$ such that $e = e' \circ \gamma$ and a short computation reveals that $f'_X \circ K(e) = f_X \circ K(e)$. But note that $K(e)$ is an epimorphism because $K$ is right exact, and thus $f_X = f'_X$. As a result we have constructed a family of morphisms $(f_X)_{X \in \mathcal{C}}$. Naturality is proven by the same kind of arguments.
\\2. This is completely similar to the previous item; but for convenience of the reader we give full details. The inverse map to \eqref{dinatTransfoAndProjGenerator} is constructed as follows. Let $g : K(P,P) \to D$ which satisfies the condition above. We first set $d_P = g$. More generally for each $n \geq 1$ the object $P^n$ comes equipped with morphisms $j_i : P \to P^n$ and $\pi_i : P^n \to P$ for all $1 \leq i \leq n$ and we define $d_{P^n} = \sum_{i=1}^n g \circ K(j_i,\pi_i)$. The family $(d_{P^n})_{n \in \mathbb{N}^*}$ is dinatural in the full subcategory $\bigl\{ P^n \, |\, n \in \mathbb{N}^* \bigr\}$ by the required property of $g$ and by the matrix description of morphisms $P^m \to P^n$. Now let $X \in \mathcal{C}$ and take a presentation $P^m \overset{\alpha}{\longrightarrow} P^n \overset{e}{\longrightarrow} X \longrightarrow 0$, so that $(X,e) = \mathrm{coker}(\alpha)$. Then $K(\mathrm{id}_X,e) = \mathrm{coker} \bigl( K(\mathrm{id}_X,\alpha) \bigr)$ by right-exactness of $K(X,-)$ and the universal property of a cokernel gives
\[ \xymatrix@C=4em{
K(X,P^m) \ar[d]_{K(e,\mathrm{id}_{P^m})}^{\qquad\qquad\text{\normalsize$\circlearrowright$}} \ar[r]^{K(\mathrm{id}_X,\alpha)} & K(X,P^n) \ar[r]^{K(\mathrm{id}_X,e)} \ar[d]^{K(e,\mathrm{id}_{P^n})} & K(X,X)\ar[dd]^{\exists! \, d_X}_{\text{\normalsize$\circlearrowright$}\qquad}\\
K(P^n,P^m) \ar[d]_{K(\alpha,\mathrm{id}_{P^m})}^{\qquad\qquad\qquad\qquad\text{\normalsize$\circlearrowright$}} \ar[r]_{K(\mathrm{id}_{P^n},\alpha)} & K(P^n,P^n) \ar[dr]^{d_{P^n}} &\\
K(P^m,P^m) \ar[rr]_{d_{P^m}}& & D
} \]
because the commutation of the left squares implies that $d_{P^n} \circ K(e,\mathrm{id}_{P^n})$ vanishes on $K(\mathrm{id}_X,\alpha)$. The morphism $d_X$ does not depend on the projective presentation of $X$. For if $P^{m'} \overset{\alpha'}{\longrightarrow} P^{n'} \overset{e'}{\longrightarrow} X \longrightarrow 0$ is another presentation then we have $d'_X : K(X,X) \to D$ defined by $d'_X \circ K(\mathrm{id}_X,e') = d_{P^{n'}} \circ K(e',\mathrm{id}_{P^{n'}})$. Since $P^n$ is projective there exists $\gamma : P^n \to P^{n'}$ such that $e = e' \circ \gamma$ and we find
\begin{align*}
d'_X \circ K(\mathrm{id}_X,e) &= d'_X \circ K(\mathrm{id}_X,e') \circ K(\mathrm{id}_X,\gamma) = d_{P^{n'}} \circ K(e',\mathrm{id}_{P^{n'}}) \circ K(\mathrm{id}_X,\gamma)\\
&= d_{P^{n'}} \circ K(\mathrm{id}_{P^{n'}},\gamma) \circ K(e',\mathrm{id}_{P^n}) = d_{P^n} \circ K(\gamma,\mathrm{id}_{P^n}) \circ K(e',\mathrm{id}_{P^n})\\
&=d_{P^n} \circ K(e,\mathrm{id}_{P^n}) = d_X \circ K(\mathrm{id}_X,e).
\end{align*}
But $K(\mathrm{id}_X,e)$ is a cokernel, whence an epimorphism and it follows that $d_X = d'_X$. Having constructed a family of morphisms $\bigl( d_X \bigr)_{X \in \mathcal{C}}$, we now prove its dinaturality. Let $f : X_1 \to X_2$ and take projective resolutions $P^{m_i} \overset{\alpha_i}{\longrightarrow} P^{n_i} \overset{e_i}{\longrightarrow} X_i \longrightarrow 0$. By definition $d_{X_i} \circ K(\mathrm{id}_X,e_i) = d_{P^{n_i}} \circ K(e_i,\mathrm{id}_{P^{n_i}})$. Since $P^{n_1}$ is projective, there exists $\omega : P^{n_1} \to P_{n_2}$ such that $f \circ e_1 = e_2 \circ \omega$. As a result
\begin{align*}
&d_{X_2} \circ K(\mathrm{id}_{X_2},f) \circ K(\mathrm{id}_{X_2},e_1) = d_{X_2} \circ K(\mathrm{id}_{X_2},e_2) \circ K(\mathrm{id}_{X_2},\omega)\\
=\:& d_{P^{n_2}} \circ K(e_2,\mathrm{id}_{P^{n_2}}) \circ K(\mathrm{id}_{X_2},\omega) = d_{P^{n_2}} \circ K(\mathrm{id}_{P^{n_2}},\omega) \circ K(e_2,\mathrm{id}_{P^{n_1}})\\
=\:&d_{P^{n_1}} \circ K(\omega, \mathrm{id}_{P^{n_1}}) \circ K(e_2,\mathrm{id}_{P^{n_1}}) = d_{P^{n_1}} \circ K(e_1, \mathrm{id}_{P^{n_1}}) \circ K(f,\mathrm{id}_{P^{n_1}})\\
=\:& d_{X_1} \circ K(\mathrm{id}_{X_1},e_1) \circ K(f,\mathrm{id}_{P^{n_1}}) = d_{X_1} \circ K(f,\mathrm{id}_{X_1})  \circ K(\mathrm{id}_{X_2},e_1)
\end{align*}
and we have the result because $K(\mathrm{id}_{X_2},e_1)$ is a cokernel and hence an epimorphism.
\end{proof}

\begin{corollary}\label{coendsAsCokerAndConsequence}
Let $K : \mathcal{C}^{\mathrm{op}} \times \mathcal{C} \to \mathcal{D}$ be a $\Bbbk$-bilinear functor right exact in the second variable. Under the assumptions of Lemma \ref{lemNatTransfoProjGen}, we have:
\\1. The coend $\int^{X \in \mathcal{C}} K(X,X)$ exists in $\mathcal{C}$.
\\2. Let $F : \mathcal{D} \to \mathcal{E}$ be a $\Bbbk$-linear right exact functor. Denote by $i$ (resp. $l$) the universal dinatural transformation of $K$ (resp. $FK$). Then the comparison morphism $\zeta$ defined by
\[ \xymatrix{
&FK(X,X) \ar[dl]_-{l_X}^{\text{\normalsize $\quad\qquad\qquad\circlearrowleft$}} \ar[dr]^-{F(i_X)}& \\
\int^X FK(X,X) \ar[rr]_-{\exists! \, \zeta} && F\!\left( \int^X K(X,X) \right) } \]
for all $X \in \mathcal{C}$ is an isomorphism. Equivalently, $F(i)$ is the universal dinatural transformation of $FK$.
\end{corollary}
\begin{proof}
1. This is the construction of \cite[Cor.\,5.1.8]{KL}. Recall that $P$ denotes a projective generator and let $\{\alpha_j\}_{1 \leq j \leq n}$ be a basis of $\mathrm{End}_{\mathcal{C}}(P)$. For each $j$ consider $u_j = K(\alpha_j,\mathrm{id}_{P_j}) - K(\mathrm{id}_{P_j},\alpha_j)$ and define
\[ E = \mathrm{coker}\left[ K(P,P)^n \xrightarrow{\sum_{j=1}^n u_j \circ \pi_j} K(P,P) \right] \]
where $\pi_j : K(P,P)^n \to K(P,P)$ are the canonical projections. By construction, the object $E$ comes equipped with an epimorphism $g : K(P,P) \to E$ satisfying $g \circ K(\varphi,\mathrm{id}_P) = g \circ K(\mathrm{id}_P,\varphi)$ for all $\varphi \in \mathrm{End}_{\mathcal{C}}(P)$. Item 2 in Lemma \ref{lemNatTransfoProjGen} then gives a dinatural transformation $i_X : K(X,X) \to E$ such that $i_P = g$. Let $d_X : K(X,X) \to D$ be any other dinatural transformation. Then by dinaturality we have $d_P \circ (u_1,\ldots, u_n) = 0$ and the universal property of a cokernel gives $\delta : E \to D$ such that $d_P = \delta \circ i_P$. By item 2 in Lemma \ref{lemNatTransfoProjGen} we deduce that $d_X = \delta \circ i_X$ for all $X \in \mathcal{C}$, proving that the pair $(E,i)$ is the coend of $K$.
\\2. The functor $F$ preserves cokernels. Thus it immediately follows from the proof of the previous item (where we saw that a coend can be constructed as a cokernel) that $F(i)$ is a universal dinatural transformation. We can define a morphism $\omega : F\left( \int^X K(X,X) \right) \to \int^X FK(X,X)$ using universality of $F(i)$: $\omega_X \circ F(i_X) = l_X$ for all $X \in \mathcal{C}$. Then $\omega = \zeta^{-1}$.
\end{proof}

\indent Now let $\mathcal{A}, \mathcal{B}$ be {\em finite} $\Bbbk$-linear categories and $\mathcal{D}$ be any $\Bbbk$-linear abelian category. We denote by $P$ (resp. $Q$) a projective generator for $\mathcal{A}$ (resp. $\mathcal{B}$).

\begin{corollary}\label{coroNatMultiTransfoProjGen}
1. Let $K,L : \mathcal{A} \times \mathcal{B} \to \mathcal{D}$ be $\Bbbk$-bilinear functors and assume that $K$ is right exact in each variable. Then $f \mapsto f_{P,Q}$ gives an isomorphism of vector spaces between $\mathrm{Nat}(K,L)$ and
\[ \bigl\{  g : K(P,Q) \to L(P,Q) \,\big|\, \forall \, \varphi \in \mathrm{End}_{\mathcal{A}}(P), \: \forall \, \psi \in \mathrm{End}_{\mathcal{B}}(Q), \:\: g \circ K(\varphi,\psi) = L(\varphi,\psi) \circ g \bigr\}. \]
2. Let $K : \mathcal{A}^{\mathrm{op}} \times \mathcal{B}^{\mathrm{op}} \times \mathcal{A} \times \mathcal{B} \to \mathcal{D}$ be a $\Bbbk$-multilinear functor right-exact in each variable and let $D \in \mathcal{D}$. Then $d \mapsto d_{P,Q}$ gives an isomorphism of vector spaces between $\mathrm{Dinat}(K,D)$ and
\[
\left\{ d : K(P, Q, P, Q) \to D \: \left|\!\!\!\!\!\!
\begin{array}{c}
\forall \, \varphi \in \mathrm{End}_{\mathcal{A}}(P), \: \forall\, \psi \in \mathrm{End}_{\mathcal{B}}(Q),\\
\quad d \circ K(\varphi^{\mathrm{op}}, \psi^{\mathrm{op}}, \mathrm{id}_P, \mathrm{id}_Q) = d \circ K(\mathrm{id}_P, \mathrm{id}_Q, \varphi, \psi)
\end{array} \right. \right\} \]
where $\varphi^{\mathrm{op}}$ and $\psi^{\mathrm{op}}$ remind that the morphisms $\varphi$ and $\psi$ are viewed in $\mathcal{C}^{\mathrm{op}}$.
\end{corollary}
\begin{proof}
1. $\mathrm{Nat}(K,L)$ is isomorphic to the subspace $\mathcal{S} \subset \prod_{X \in \mathcal{A}} \mathrm{Nat}\bigl(K(X,-),L(X,-)\bigr)$ consisting of sequences $(h_X)_ {X \in \mathcal{A}}$ such that
\[ \forall \, Y \in \mathcal{B}, \: \forall \, X,X' \in \mathcal{A}, \: \forall \, \alpha \in \mathrm{Hom}_{\mathcal{A}}(X,X'), \qquad h_{X',Y} \circ K(\alpha,\mathrm{id}_Y) = L(\alpha,\mathrm{id}_Y) \circ h_{X,Y}. \]
By Lemma \ref{lemNatTransfoProjGen}, $\mathcal{S}$ is isomorphic to the subspace $\mathcal{S}' \subset \prod_{X \in \mathcal{A}} \mathrm{Hom}_{\mathcal{D}}\bigl( K(X,Q), L(X,Q) \bigr)$ consisting of sequences $(\bar h_X)_ {X \in \mathcal{A}}$ such that
\begin{align*}
&\forall \, X,X' \in \mathcal{A},\: \forall \, \alpha \in \mathrm{Hom}_{\mathcal{A}}(X,X'), \qquad \bar h_{X'} \circ K(\alpha,\mathrm{id}_Q) = L(\alpha,\mathrm{id}_Q) \circ \bar h_{X}\\
&\qquad\text{and } \forall \, X \in \mathcal{A}, \: \forall \, \psi \in \mathrm{End}_{\mathcal{B}}(Q), \qquad \bar h_{X} \circ K(\mathrm{id}_X,\psi) = L(\mathrm{id}_X,\psi) \circ \bar h_{X}.
\end{align*}
Said differently $\mathcal{S}' \subset \mathrm{Nat}\bigl( K(-,Q), L(-,Q) \bigr)$ consists of natural transformations $\bar h$ such that $\bar h_{X} \circ K(\mathrm{id}_X,\psi) = L(\mathrm{id}_X,\psi) \circ \bar h_{X}$ for all $X \in \mathcal{A}$ and $\psi \in \mathrm{End}_{\mathcal{B}}(Q)$. Another application of Lemma \ref{lemNatTransfoProjGen} gives the result.
\\2. Same argument.
\end{proof}
\noindent Of course the statements of Corollary \ref{coroNatMultiTransfoProjGen} generalize to (di)natural transformations with an arbitrary finite number of variables.

\smallskip

\indent Denote by $\boxtimes$ the Deligne product of finite $\Bbbk$-linear categories \cite[\S 5]{deligne}, \cite[\S 1.11]{EGNO}. The indecomposable projective objects in $\mathcal{A} \boxtimes \mathcal{B}$ are the $P_i \boxtimes Q_j$ where $P_i$ (resp. $Q_j$) is an indecomposable projective object in $\mathcal{A}$ (resp. $\mathcal{B}$). This is easily seen using that $\mathcal{A} \cong A\text{-}\mathrm{mod}$ and $\mathcal{B} \cong B\text{-}\mathrm{mod}$, where $A$ and $B$ are finite-dimensional $\Bbbk$-algebras, so that $\mathcal{A} \boxtimes \mathcal{B} \cong (A \otimes B)\text{-}\mathrm{mod}$. Hence if $P = \bigoplus_i P_i$ (resp. $Q = \bigoplus_j Q_j$) is the minimal projective generator of $\mathcal{A}$ (resp. $\mathcal{B}$) then $P \boxtimes Q$ is the minimal projective generator of $\mathcal{A} \boxtimes \mathcal{B}$.

\begin{lemma}\label{lemmaExtensionNatDeligneProd}
1. Let $F,G : \mathcal{A} \times \mathcal{B} \to \mathcal{D}$ be $\Bbbk$-bilinear functors right exact in each variable and let $\widetilde{F}, \widetilde{G} : \mathcal{A} \boxtimes \mathcal{B} \to \mathcal{D}$ be such that $F = \widetilde{F} \circ \boxtimes$ and $G = \widetilde{G} \circ \boxtimes$. The map
\[ J : \mathrm{Nat}(\widetilde{F}, \widetilde{G}) \to \mathrm{Nat}(F,G), \qquad J(f)_{(X,Y)} = f_{X \boxtimes Y} \]
is an isomorphism of vector spaces.
\\2. Let $H : \mathcal{A}^{\mathrm{op}} \times \mathcal{B}^{\mathrm{op}} \times (\mathcal{A} \times \mathcal{B}) \to \mathcal{D}$ be a $\Bbbk$-multilinear functor right-exact in each variable. Let $\widetilde{H} : (\mathcal{A} \boxtimes \mathcal{B})^{\mathrm{op}} \times (\mathcal{A} \boxtimes \mathcal{B}) \to \mathcal{D}$ be the $\Bbbk$-bilinear functor right-exact in each variable such that $H = \widetilde{H} \circ (\boxtimes \times \boxtimes)$. For any $D \in \mathcal{D}$, the map
\[ I : \mathrm{Dinat}(\widetilde{H},D) \to \mathrm{Dinat}(H,D), \qquad I(d)_{(X,Y)} = d_{X \boxtimes Y} \]
is an isomorphism of vector spaces.
\\3. Let $H$ and $\widetilde{H}$ be as in the previous item and let $U = \int^{(X,Y) \in \mathcal{A} \times \mathcal{B}} H\bigl( (X,Y), (X,Y) \bigr)$ with the universal dinatural transformation $u \in \mathrm{Dinat}(H,U)$. Then the object $U$ endowed with the dinatural transformation $I^{-1}(u)$ is the coend of $\widetilde{H}$.
\end{lemma}
\begin{proof}
1. The functor $\widetilde{F}$ is right exact by definition. By item 1 in Lemma \ref{lemNatTransfoProjGen}, any $f \in \mathrm{Nat}(\widetilde{F}, \widetilde{G})$ is uniquely determined by $f_{P \boxtimes Q} \in \Hom_{\mathcal{D}}\bigl( \widetilde{F}(P \boxtimes Q), \widetilde{G}(P \boxtimes Q) \bigr) = \Hom_{\mathcal{D}}\bigl( F(P,Q), G(P,Q) \bigr)$. By item 1 in Corollary \ref{coroNatMultiTransfoProjGen} for natural transformations with two components, any $g \in \mathrm{Nat}(F,G)$ is uniquely determined by $g_{(P,Q)} \in \Hom_{\mathcal{D}}\bigl( F(P,Q), G(P,Q) \bigr)$. The result follows.
\\2. Same argument, but now using item 2 in Lemma \ref{lemNatTransfoProjGen}.
\\3. Let $d \in \mathrm{Dinat}(\widetilde{H},D)$ for some $D \in \mathcal{D}$. Then $I(d) \in \mathrm{Dinat}(H,D)$ and by universality of $u$ there exists $\varphi : U \to D$ such that $I(d)_{(X,Y)} = \varphi \circ u_{(X,Y)}$ for all $X,Y$. Note that $I^{-1}(\varphi \circ u)_{X \boxtimes Y} = \varphi \circ u_{X \boxtimes Y}$, which by the previous item is sufficient to conclude that $I^{-1}(\varphi \circ u) = \varphi \circ I^{-1}(u)$. Hence $d = \varphi \circ I^{-1}(u)$, proving that $I^{-1}(u)$ is the universal dinatural transformation.
\end{proof}

We allow ourselves to write the last item in Lemma \ref{lemmaExtensionNatDeligneProd} as 
\begin{equation}
    \label{eq:coend-Deligne}
 \int^{M \in \mathcal{A} \boxtimes \mathcal{B}} \widetilde{H}(M,M) = \int^{X \in \mathcal{A}, Y \in \mathcal{B}} \widetilde{H}(X \boxtimes Y,X \boxtimes Y)
 \end{equation}
to mean that it is enough to know the value of the universal dinatural transformation on the ``pure tensors'' $X \boxtimes Y \in \mathcal{A} \boxtimes \mathcal{B}$. This property was also proven in \cite[\S 3.4]{FSS} with different arguments.

\section{Normalization of cosimplicial complexes}\label{appendixNormalization}
Let $X$ be a cosimplicial abelian group (or vector space), which we spell out explicitly for the convenience of the reader:
\begin{itemize}[itemsep=-.2em]
\item For each $n \in \mathbb{N}$ we have an abelian group (or vector space) $X^n$.
\item For each $n \geq 0$ and $0 \leq i \leq n+1$ there are morphisms $\partial^n_i : X^n \to X^{n+1}$, called {\em coface maps}, which satisfy $\partial^{n+1}_j \partial^n_i = \partial^{n+1}_i \partial^n_{j-1}$ for all $0 \leq i < j \leq n+2$.
\item For each $n \geq 1$ and $0 \leq i \leq n-1$ there are morphisms $s^n_i : X^n \to X^{n-1}$, called {\em codegeneracy maps}, which satisfy $s^{n-1}_j s^n_i = s^{n-1}_i s^n_{j+1}$ for all $0 \leq i \leq j \leq n-2$.
\item The following equalities are satisfied for all $0 \leq i \leq n+1$ and $0 \leq j \leq n$:
\[ s^{n+1}_j \partial^n_i = \begin{cases}
\partial^{n-1}_i s^n_{j-1} & \text{if } i < j\\
\mathrm{id}_{X_n} &\text{if } i=j \text{ or } i=j+1\\
\partial_{i-1}^{n-1} s_j^n & \text{if } i > j+1
\end{cases} \]
\end{itemize}
The various relations between the maps $\partial$ and $s$ are called the {\em cosimplicial identities}. For readability we often do not write the superscripts on $\partial$ and $s$.

\smallskip

\indent Let $\delta^n = \sum_{i=0}^{n+1} (-1)^i \partial^n_i : X^n \to X^{n+1}$; then $C(X) = \bigl( X^0 \overset{\delta_0}{\longrightarrow} X^1 \overset{\delta_1}{\longrightarrow} X^2 \overset{\delta_2}{\longrightarrow} \ldots \bigr)$ is  the {\em cochain complex of $X$}. Define
\[ N(X^0) = X^0, \qquad N(X^n) = \bigcap_{i=0}^{n-1} \ker(s^n_i) \subset X^n \quad \text{for all } n \geq 1. \]
One checks easily using the cosimplicial identities that $\delta^n\bigl(N(X^n)\bigr) \subset N(X^{n+1})$. Hence $N(X)$ is a subcomplex of $C(X)$, called the {\em normalized cochain complex of $X$}.

\smallskip

\indent It is well-known that $C(X)$ and $N(X)$ have the same cohomology. More precisely there exists a morphism $\mathcal{N} : C(X) \to C(X)$ of complexes which is the projection onto $N(X)$ and induces an isomorphism between the cohomology groups, see e.g. \cite[Th.\,8.3.8]{weibel}, \cite[Th.\,III.2.4]{GJ}. Actually the operation $X \mapsto N(X)$ is part of a famous result called Dold-Kan correspondence. This is mainly explained for simplicial chain complexes in the literature; in this appendix we review the construction of the projector $\mathcal{N}$ for cosimplicial cochain complexes. This is an adaptation of the discussion after Cor.\,III.2.3 in \cite{GJ}. Note that in deformation theory the subcomplex $N(X)$ corresponds to unit constraints on infinitesimal deformations (like unitality of an infinitesimal product on an associative algebra). Applying $\mathcal{N}$ to a cocycle does not change its equivalence class but gives an infinitesimal deformation which satisfies the unit constraints.

\smallskip

\indent For $n \geq 1$ and $0 \leq i \leq n-1$ let
\[ \pi_i^n = \mathrm{id}_{X_n} - \partial^{n-1}_is^n_i. \]
Here are easy consequences of the cosimplicial identities:
\[ 
s_j\pi_i^n = \begin{cases}
\pi_i^{n-1} s_j & \text{if } j > i\\
0 & \text{if } j=i
\end{cases}
\qquad \pi_i^{n+1} \partial_j = \begin{cases}
\partial_j \pi_{i-1}^n & \text{if } i > j\\
0 & \text{if } j=i
\end{cases} \]
They readily imply the following key properties:
\begin{equation}\label{basicPropPi}
s_j \pi_0^n \ldots \pi_i^n = 0 \quad \text{and} \quad \pi_0^n \ldots \pi_i^n \partial_j = 0 \qquad \text{for all } 0 \leq j \leq i \leq n-1.
\end{equation}

\begin{lemma}\label{lemmaCommutationPiDelta}
For all $0 \leq i \leq n-1$ we have $\delta^n \pi^n_0 \ldots \pi^n_i = \pi_0^{n+1} \ldots \pi_i^{n+1} \delta^n$.
\end{lemma}
\begin{proof}
By induction on $i$. The case $i=0$ follows from the cosimplicial identities. Assume that the property is true for some $i-1$, with $i \geq 1$. Then
\[ \delta^n \pi^n_0 \ldots \pi^n_i = \pi^{n+1}_0 \ldots \pi^{n+1}_{i-1}\delta^n \pi^n_i = \pi^{n+1}_0 \ldots \pi^{n+1}_{i-1}\delta^n - \pi^{n+1}_0 \ldots \pi^{n+1}_{i-1} \sum_{j=i}^{n+1} (-1)^j \partial_j \partial_i s_i \]
where we used the definition of $\pi_i^n$ and \eqref{basicPropPi} for the second equality. Using the cosimplicial identities we find
\begin{align*}
\sum_{j=i}^{n+1} (-1)^j \partial_j \partial_i s_i &= \sum_{j=i+2}^{n+1} (-1)^j \partial_j \partial_i s_i = \partial_i s_i \sum_{j=i+2}^{n+1} (-1)^j \partial_j = \partial_i s_i \left(\delta^n - \sum_{j=0}^{i+1} (-1)^j \partial_j \right)\\
& = \partial_i s_i \delta^n - \sum_{j=0}^{i-1} (-1)^j \partial_i s_i\partial_j = \partial_i s_i \delta^n - \sum_{j=0}^{i-1} (-1)^j \partial_j \partial_{i-1} s_{i-1}.
\end{align*}
But note that $\pi^{n+1}_0 \ldots \pi^{n+1}_{i-1} \sum_{j=0}^{i-1} (-1)^j \partial_j = 0$ due to \eqref{basicPropPi}. Hence we obtain $\delta^n \pi^n_0 \ldots \pi^n_i = \pi^{n+1}_0 \ldots \pi^{n+1}_{i-1}\delta^n - \pi^{n+1}_0 \ldots \pi^{n+1}_{i-1}\partial_i s_i \delta^n = \pi^{n+1}_0 \ldots \pi^{n+1}_{i-1} \pi^{n+1}_i \delta^n$.
\end{proof}
\indent Define
\[ \mathcal{N}_0 = \mathrm{id}_{X_0} \quad \text{and} \quad \mathcal{N}^n = \pi_0^n \ldots \pi_{n-1}^n : X^n \to X^n. \]

\begin{proposition}\label{propPropertiesN}
The morphism $\mathcal{N}^n$ satisfy the following properties:
\begin{enumerate}[itemsep=-.2em, topsep=.2em]
\item $\mathcal{N}^n$ is an idempotent, $\mathcal{N}^n \mathcal{N}^n = \mathcal{N}^n$, and is the projection onto $N(X^n)$.
\item The family of maps $\mathcal{N} = (\mathcal{N}^n)_{n \in \mathbb{N}}$ is a morphism of cochain complexes $C(X) \to N(X)$.
\item Let $x \in X^n$ be a cocycle. Then $\mathcal{N}^n(x) = x + \delta^{n-1}(v)$ for some $v \in X^{n-1}$. 
\end{enumerate}
\end{proposition}
\begin{proof}
1. It follows from \eqref{basicPropPi} that $s_i^n \mathcal{N}^n = 0$ for all $0 \leq i \leq n-1$ and hence
\[ \mathcal{N}^n \mathcal{N}^n = \pi_0^n \ldots \pi_{n-1}^n \mathcal{N}^n = \pi_0^n \ldots \pi_{n-2}^n (\mathrm{id} - \partial_{n-1} s_{n-1}) \mathcal{N}^n = \pi_0^n \ldots \pi_{n-2}^n \mathcal{N}^n = \ldots = \mathcal{N}^n. \]
For the same reason we see that $\mathcal{N}^n(x) = x$ if $x \in N(X^n)$.
\\2. Note from the cosimplicial identities that $\pi_{i-1} \partial_i = \partial_i - \partial_{i-1}$. Combining this with Lemma \ref{lemmaCommutationPiDelta} we get
\begin{align*}
\delta^n \mathcal{N}^n = \pi_0^{n+1} \ldots \pi_{n-1}^{n+1} \delta_n &= (-1)^{n+1}\pi_0^{n+1} \ldots \pi_{n-1}^{n+1} (\partial_n - \partial_{n+1})\\
&= (-1)^{n+1}\pi_0^{n+1} \ldots \pi_{n-1}^{n+1} \pi_n^{n+1}\partial_{n+1} = \mathcal{N}^{n+1} \delta^n
\end{align*}
where in the second and fourth equalities we used \eqref{basicPropPi}.
\\3. Note first the following formula, which is easily obtained from the cosimplicial identities:
\[ \delta^{n-1} s_i + s_i \delta^n = (-1)^i \partial_i s_i + \sum_{j=0}^{i-1} (-1)^j \partial_j(s_i + s_{i-1}). \]
Combining it with \eqref{basicPropPi} and Lemma \ref{lemmaCommutationPiDelta}, we find
\begin{align}\label{homotopyProp}
\begin{split}
\pi_0^n \ldots \pi_i^n &= \pi_0^n \ldots \pi_{i-1}^n (\mathrm{id} - \partial_is_i) = \pi_0^n \ldots \pi_{i-1}^n\bigl( \mathrm{id} - (-1)^i(\delta^{n-1} s_i + s_i \delta^n) \bigr)\\
&=\pi_0^n \ldots \pi_{i-1}^n - (-1)^i\delta^{n-1}\pi_0^{n-1} \ldots \pi_{i-1}^{n-1}s_i - (-1)^i \pi_0^n \ldots \pi_{i-1}^ns_i\delta^n.
\end{split}
\end{align}
Now we can prove by induction on $i$ the following property: $\pi_0^n \ldots \pi_i^n(x) = x + \delta^{n-1}(v_i)$ for some $v_i \in X^{n-1}$. The desired result is for $i=n-1$. The case $i=0$ follows directly from \eqref{homotopyProp}: $\pi_0^n(x) = x - \delta^ns_0(x)$. Assume that the property is true for $i-1$ with $i \geq 1$. Then we get from \eqref{homotopyProp}
\begin{align*}
\pi_0^n \ldots \pi_i^n(x) =& \pi_0^n \ldots \pi_{i-1}^n(x) - (-1)^i \delta^{n-1}\pi_0^{n-1} \ldots \pi_{i-1}^{n-1}s_i(x)\\
&= x + \delta^{n-1}\bigl( v_{i-1} - (-1)^i \pi_0^{n-1} \ldots \pi_{i-1}^{n-1}s_i(x) \bigr).
\end{align*}
This gives a recursion formula for $v_i$, with $v_0 = -s_0(x)$.
\end{proof}

\begin{corollary}\label{coroProjQuasiIso}
The morphism of cochain complexes $\mathcal{N}$ is a quasi-isomorphism, i.e. it induces an isomorphism between the cohomology groups: $H^{\bullet}\bigl( C(X) \bigr) \overset{\sim}{\to} H^{\bullet}\bigl( N(X) \bigr)$.
\end{corollary}
\begin{proof}
Let $I^n : N(X^n) \to X^n$ be the canonical embedding. We have $\mathcal{N}^n I^n = \mathrm{id}$ since $\mathcal{N}^n$ is the projection onto $N(X^n)$. By item 3 in Proposition \ref{propPropertiesN} $I^n \mathcal{N}^n(x)$ is cohomologous to $x$. Hence $I^n$ is the inverse of $\mathcal{N}^n$ in cohomology.
\end{proof}

For low-degree cases we get (using the cosimplicial identities):
\begin{align}\label{normalizationProjLowDegrees}
\begin{split}
\mathcal{N}^1 &= \mathrm{id} - \partial_0s_0, \qquad \mathcal{N}^2 = \mathrm{id} - \partial_0 s_0 - \partial_1 s_1 + \partial_0s_1,\\
\mathcal{N}^3 &= \mathrm{id} - \partial_0s_0 - \partial_1s_1 - \partial_2s_2 + \partial_0s_1 + \partial_1s_2 + \partial_0\partial_1s_0s_2 - \partial_0s_2.
\end{split}
\end{align}
A straightforward computation using the cosimplicial identities and \eqref{basicPropPi} reveals that $\mathcal{N}^n = \pi_0^n \ldots \pi_{n-2}^n\bigl( \mathrm{id} + (-1)^ns_n\delta^n \bigr)$. Hence if $x \in X^n$ is a cocycle then $\mathcal{N}^n(x) = \pi_0^n \ldots \pi_{n-2}^n(x)$. In particular, thanks to the recursion formula at the end of the proof of Proposition \ref{propPropertiesN}, we find that if $x \in X^1$, $y \in X^2$, $z \in X^3$ are cocycles then
\begin{equation}\label{normalizationLowDegrees}
\mathcal{N}^1(x) = x, \quad \mathcal{N}^2(y) = y - \delta^1s_0(y), \quad \mathcal{N}^3(z) = z + \delta^2\bigl( -s_0(z) + s_1(z) - \partial_0s_0s_0(z) \bigr).
\end{equation}

\section{Deligne product of resolvent pairs}\label{subsectionDeligneResolventPairs}
\textbf{Assumption:} {\em In this appendix, $\Bbbk$ is a perfect field.} This holds for instance if $\mathrm{char}(\Bbbk) = 0$ or if $\Bbbk$ is algebraically closed.

\smallskip

\indent Recall the definition of a resolvent pair below \eqref{adjunction} and denote by $\boxtimes$ the Deligne product of finite $\Bbbk$-linear categories \cite{deligne}, \cite[\S 1.11]{EGNO}. Let $\mathcal{A}, \mathcal{A}', \mathcal{B}, \mathcal{B}'$ be finite $\Bbbk$-linear categories and 
\[ \xymatrix@R=.7em{
\mathcal{A}\ar@/^.7em/[dd]^{\mathcal{U}}\\
\dashv\\
\ar@/^.7em/[uu]^{\mathcal{F}}\mathcal{B}
} \qquad\quad \xymatrix@R=.7em{
\mathcal{A}'\ar@/^.7em/[dd]^{\mathcal{U}'}\\
\dashv\\
\ar@/^.7em/[uu]^{\mathcal{F}'}\mathcal{B}'
} \]
be resolvent pairs such that $\mathcal{U}$ and $\mathcal{U}'$ are $\Bbbk$-linear. The functors $\mathcal{U},\mathcal{U}'$ are exact by definition of a resolvent pair. Thus the functor $(X,X') \mapsto \mathcal{U}(X) \boxtimes \mathcal{U}'(X')$ is exact in each variable. In particular it is right exact and by universal property of the Deligne product we get a right exact $\Bbbk$-linear functor $\mathcal{U} \boxtimes \mathcal{U}' : \mathcal{A} \boxtimes \mathcal{A}' \to \mathcal{B} \boxtimes \mathcal{B}'$ uniquely defined by $(\mathcal{U} \boxtimes \mathcal{U}')(X \boxtimes X') = \mathcal{U}(X) \boxtimes \mathcal{U}'(X')$. The functor $\mathcal{U} \boxtimes \mathcal{U}'$ is actually {\em exact} due to our assumption on the ground field \cite[Prop.\,5.13]{deligne}. Since $\mathcal{F}$ and $\mathcal{F}'$ have right adjoints, they preserve colimits and in particular they are right exact. Hence there exists a right-exact functor $\mathcal{F} \boxtimes\mathcal{F}' : \mathcal{B} \boxtimes \mathcal{B}' \to \mathcal{A} \boxtimes \mathcal{A}'$ defined in the same way. Let us show that these two functors form a resolvent pair. In particular we must check that $\mathcal{U} \boxtimes \mathcal{U}'$ is faithful, which will follow from this lemma:

\begin{lemma}\label{lemmaCriterionExactFaithful}
Let $\mathcal{X}, \mathcal{Y}$ be abelian categories such that every object in $\mathcal{X}$ has finite length. Let $F : \mathcal{X} \to \mathcal{Y}$ be an exact functor such that $F(S) \neq 0$ for all simple objects $S \in \mathcal{X}$. Then $F$ is faithful.
\end{lemma}
\begin{proof}
Let $V \in \mathcal{X}$ be an arbitrary object. We prove that
\[ \forall \, W \in \mathcal{X}, \:\: \forall \, f \in \Hom_{\mathcal{X}}(V,W), \quad F(f) = 0 \: \implies \: f=0 \]
by induction on the length of $W$. Recall that the image of a morphism $f$ is $\mathrm{im}(f) = \ker\bigl(\mathrm{coker}(f)\bigr)$ and thus $F$ preserves images by exactness. First assume that $W$ has length $1$ and let $f : V \to W$ with $F(f)=0$. Then $W$ is a simple object and $\mathrm{im}(f)$ is a subobject, which by definition forces $\mathrm{im}(f) = 0$ or $\mathrm{im}(f) = W$. In the latter case we get $F(W) = F\bigl(\mathrm{im}(f) \bigr) = \mathrm{im}\bigl(F(f) \bigr) = 0$, contradicting the assumption on $F$. Hence $f=0$. Now assume that $W$ has length $n$ and let $f : V \to W$ such that $F(f) = 0$. Take a simple subobject $S \subset W$ and consider $\overline{f} : V \overset{f}{\to} W \twoheadrightarrow W/S$. Then $F(\overline{f}) = 0$ and since $W/S$ has length $n-1$ the induction hypothesis gives $\overline{f} = 0$. It follows that $\mathrm{im}(f) \subset S$, which forces $\mathrm{im}(f) = 0$ because again $\mathrm{im}(f) = S$ would contradict the assumption $F(S) \neq 0$.
\end{proof}

\begin{proposition}\label{propDeligneProductResPairs}
1. $\mathcal{F} \boxtimes\mathcal{F}'$ is the left adjoint of $\mathcal{U} \boxtimes \mathcal{U}'$ and this adjunction forms a resolvent pair.
\\2. If objects $P \in \mathcal{A}$, $P' \in \mathcal{A}'$ are relatively projective with respect to the adjunctions $\mathcal{F} \dashv \mathcal{U}$ and $\mathcal{F}' \dashv \mathcal{U}'$, then $P \boxtimes P'$ is relatively projective with respect to the adjunction $(\mathcal{F} \boxtimes\mathcal{F}') \dashv (\mathcal{U} \boxtimes \mathcal{U}')$.
\\3. If morphisms $f$ and $f'$ are allowable with respect to the adjunctions $\mathcal{F} \dashv \mathcal{U}$ and $\mathcal{F}' \dashv \mathcal{U}'$, then $f \boxtimes f'$ is allowable with respect to the adjunction $(\mathcal{F} \boxtimes\mathcal{F}') \dashv (\mathcal{U} \boxtimes \mathcal{U}')$.
\end{proposition}
\begin{proof}
1. Let $\eta : \mathrm{Id}_{\mathcal{B}} \Rightarrow \mathcal{U} \mathcal{F}$ and $\varepsilon : \mathcal{F} \mathcal{U} \Rightarrow \mathrm{Id}_{\mathcal{A}}$  be the unit and counit of $\mathcal{F} \dashv \mathcal{U}$, and $\eta', \varepsilon'$ be the unit and counit of $\mathcal{F}' \dashv \mathcal{U}'$. Define
\begin{align*}
&(\eta \boxtimes \eta')_{X \boxtimes X'} : X \boxtimes X' \xrightarrow{\eta_X \boxtimes \eta'_{X'}} (\mathcal{U}\mathcal{F})(X) \boxtimes (\mathcal{U}' \mathcal{F}')(X') = \bigl((\mathcal{U} \boxtimes \mathcal{U}')(\mathcal{F} \boxtimes \mathcal{F}')\bigr)(X \boxtimes X'),\\
&(\varepsilon \boxtimes \varepsilon')_{Y \boxtimes Y'} : \bigl((\mathcal{F} \boxtimes \mathcal{F}')(\mathcal{U} \boxtimes \mathcal{U}')\bigr)(Y \boxtimes Y') = (\mathcal{F}\mathcal{U})(Y) \boxtimes (\mathcal{F}' \mathcal{U}')(Y') \xrightarrow{\varepsilon_Y \boxtimes \varepsilon'_{Y'}} Y \boxtimes Y'.
\end{align*}
By item 1 in Lemma \ref{lemmaExtensionNatDeligneProd} these values define natural transformations $\eta \boxtimes \eta' : \mathrm{Id}_{\mathcal{B} \boxtimes \mathcal{B}'} \Rightarrow (\mathcal{U} \boxtimes \mathcal{U}')(\mathcal{F} \boxtimes \mathcal{F}')$ and $\varepsilon \boxtimes \varepsilon' : (\mathcal{F} \boxtimes \mathcal{F}')(\mathcal{U} \boxtimes \mathcal{U}') \Rightarrow \mathrm{Id}_{\mathcal{A} \boxtimes \mathcal{A}'}$. They are easily seen to satisfy the unit/counit equations of an adjunction, proving that $(\mathcal{F} \boxtimes\mathcal{F}') \dashv (\mathcal{U} \boxtimes \mathcal{U}')$.
\\We already know that $\mathcal{U} \boxtimes \mathcal{U}'$ is $\Bbbk$-linear and exact. It remains to show that it is faithful. The simple objects in $\mathcal{A} \boxtimes \mathcal{A}'$ are $S \boxtimes S'$ where $S$ and $S'$ are a simple objects in $\mathcal{A}$ and $\mathcal{A}'$ (due to the assumption that $\Bbbk$ is perfect, \cite[Lem.\,5.9]{deligne}). Since $\mathcal{U}$ and $\mathcal{U}'$ are faithful, we have $\mathcal{U}(S) \neq 0$ and $\mathcal{U}'(S') \neq 0$. Hence $(\mathcal{U} \boxtimes \mathcal{U}')(S \boxtimes S') = \mathcal{U}(S) \boxtimes \mathcal{U}'(S') \neq 0$ and Lemma \ref{lemmaCriterionExactFaithful} applies.
\\2. By assumption $P$ and $P'$ are direct summands of $\mathcal{F}(Y)$ and $\mathcal{F}'(Y')$ for some $Y \in \mathcal{B}$ and $Y' \in \mathcal{B}'$ respectively.  Hence $P \boxtimes P'$ is a direct summand of $\mathcal{F}(Y) \boxtimes \mathcal{F}'(Y') = (\mathcal{F} \boxtimes \mathcal{F}')(Y \boxtimes Y')$.
\\3. $f : V \to W$ is allowable means that there exists $s : \mathcal{U}(W) \to \mathcal{U}(V)$ such that $\mathcal{U}(f) \circ s \circ \mathcal{U}(f) = \mathcal{U}(f)$. Similarly there is $s'$ such that $\mathcal{U}(f') \circ s' \circ \mathcal{U}(f') = \mathcal{U}(f')$. The morphism $s \boxtimes s'$ satisfies
\[ \bigl( (\mathcal{U} \boxtimes \mathcal{U}')(f \boxtimes f') \bigr) \circ (s \boxtimes s') \circ \bigl( (\mathcal{U} \boxtimes \mathcal{U}')(f \boxtimes f') \bigr) = (\mathcal{U} \boxtimes \mathcal{U}')(f \boxtimes f') \]
which proves that $f \boxtimes f'$ is allowable.
\end{proof}
Recall from \S\ref{sectionLiftingAdjunctions} that we denote by $\mathbb{T}\text{-}\mathrm{mod}$ the category of modules over a monad $\mathbb{T}$. From Proposition \ref{propResolventPairsAreMonadic} about the monadicity of resolvent pairs we deduce the following fact, which for monads defined by algebras in monoidal categories is proven in \cite[Prop.\,3.8]{DSPS} also by a monadicity argument:
\begin{corollary}\label{coroDeligneProdOfMonads}
Let $\mathbb{T}$ and $\mathbb{T}'$ be monads on $\mathcal{B}, \mathcal{B}'$ such that the underlying functors $T,T'$ are $\Bbbk$-linear. Then the functor
\[ (\mathbb{T}\text{-}\mathrm{mod}) \boxtimes (\mathbb{T}'\text{-}\mathrm{mod}) \to (\mathbb{T} \boxtimes \mathbb{T}')\text{-}\mathrm{mod}, \quad (V,r) \boxtimes (V',r') \mapsto (V \boxtimes V', r \boxtimes r')
\]
is an equivalence.
\end{corollary}
\begin{proof}
By item 2 in Proposition \ref{propResolventPairsAreMonadic} and item 1 in Proposition \ref{propDeligneProductResPairs} we have resolvent pairs
\[
\xymatrix@R=.7em@C=4em{
\mathbb{T}\text{-}\mathrm{mod}\ar@/^.7em/[dd]^{\mathcal{U}_{\mathbb{T}}} & \mathbb{T}'\text{-}\mathrm{mod}\ar@/^.7em/[dd]^{\mathcal{U}_{\mathbb{T}'}} & (\mathbb{T}\text{-}\mathrm{mod}) \boxtimes (\mathbb{T}'\text{-}\mathrm{mod})\ar@/^.7em/[dd]^{\mathcal{U}_{\mathbb{T}} \boxtimes \mathcal{U}_{\mathbb{T}'}}\\
\dashv & \dashv & \dashv\\
\ar@/^.7em/[uu]^{\mathcal{F}_{\mathbb{T}}}\mathcal{B} & \ar@/^.7em/[uu]^{\mathcal{F}_{\mathbb{T}'}}\mathcal{B} & \ar@/^.7em/[uu]^{\mathcal{F}_{\mathbb{T}} \boxtimes \mathcal{F}_{\mathbb{T}'}}\mathcal{B} \boxtimes \mathcal{B}'
} \]
It is clear that the monad associated to $(\mathcal{F}_{\mathbb{T}} \boxtimes \mathcal{F}_{\mathbb{T}'}) \dashv (\mathcal{U}_{\mathbb{T}} \boxtimes \mathcal{U}_{\mathbb{T}'})$ is $\mathbb{T} \boxtimes \mathbb{T}'$. Since any resolvent pair is monadic (item 1 in Proposition \ref{propResolventPairsAreMonadic}) we have the result. The proposed functor is just the comparison functor \eqref{comparisonFunctor} in the case under consideration.
\end{proof}
We now discuss the resolutions and relative Ext groups associated to a Deligne product of resolvent pairs. Let $\mathbf{C} = \bigl(C_0 \xleftarrow{d^{\mathbf{C}}_1} C_1 \xleftarrow{d^{\mathbf{C}}_2} \ldots \bigr)$ and $\mathbf{D} = \bigl(D_0 \xleftarrow{d^{\mathbf{D}}_1} D_1 \xleftarrow{d^{\mathbf{D}}_2} \ldots \bigr)$ be chain complexes of $\Bbbk$-vector spaces. Then $\mathbf{C} \otimes \mathbf{D}$ is defined as usual to be the chain complex
\[ C_0 \otimes D_0 \xleftarrow{d^{\mathbf{C} \otimes \mathbf{D}}_1} \ldots \xleftarrow{d^{\mathbf{C} \otimes \mathbf{D}}_n} \bigoplus_{i+j=n} C_i \otimes D_j \xleftarrow{d^{\mathbf{C} \otimes \mathbf{D}}_{n+1}} \ldots \]
where $d_n^{\mathbf{C} \otimes \mathbf{D}}(v \otimes w) = d^{\mathbf{C}}_i(v) \otimes w + (-1)^i v \otimes d^{\mathbf{D}}_j(w)$ for all $v \in C_i$ and $w \in D_j$ and all $i,j \geq 0$ such that $i+j=n$, with the convention that $d_0^{\mathbf{C}} = d_0^{\mathbf{D}} = 0$. The {\em K\"unneth formula} computes the homology of this tensor product from the homologies of the two factors:
\begin{equation}\label{KunnethFormula}
\forall \, n \geq 0, \quad H_n(\mathbf{C} \otimes \mathbf{D}) \cong \bigoplus_{i+j=n} H_i(\mathbf{C}) \otimes H_j(\mathbf{D})
\end{equation}
where $H_0(\mathbf{C}) = \mathrm{coker}(d_0^{\mathbf{C}})$ and similarly for $\mathbf{D}$ and $\mathbf{C} \otimes \mathbf{D}$. For a short proof, see e.g. Problem 7.8.7 in \cite{introRepTheory}. There is of course a completely similar K\"unneth formula for cochain complexes and their cohomology.

\smallskip

\indent If $\mathbf{C}$ and $\mathbf{D}$ are chain complexes in $\mathcal{A}$ and $\mathcal{A}'$ respectively, we can define the chain complex $\mathbf{C} \boxtimes \mathbf{D}$ in $\mathcal{A} \boxtimes \mathcal{A}'$ as above, replacing $\otimes$ by $\boxtimes$ in the definition of the chain spaces and of the differential.

\begin{proposition}\label{propDeligneProductOfResolutions}
1. If $0 \leftarrow V \overset{\delta}{\longleftarrow} \mathbf{P}$ and $0 \leftarrow V' \overset{\,\delta'}{\longleftarrow} \mathbf{P}'$ are relatively projective resolutions of $V \in \mathcal{A}$ and $V' \in \mathcal{A}'$, then $0 \leftarrow V \boxtimes V' \xleftarrow{\:\delta \boxtimes \delta'\:} \mathbf{P} \boxtimes \mathbf{P}'$ is a relatively projective resolution.
\\2. For all $V,W \in \mathcal{A}$ and $V',W' \in \mathcal{A}'$ we have
\[ \Ext^n_{\mathcal{A} \boxtimes \mathcal{A}', \mathcal{B} \boxtimes \mathcal{B}'}(V \boxtimes V', W \boxtimes W') \cong \bigoplus_{i+j=n} \Ext^i_{\mathcal{A},\mathcal{B}}(V,W) \otimes \Ext^j_{\mathcal{A}',\mathcal{B}'}(V',W'). \]
\end{proposition}
\begin{proof}
1. By Proposition \ref{propDeligneProductResPairs}, the sequence $0 \leftarrow V \boxtimes V' \leftarrow \mathbf{P} \boxtimes \mathbf{P}'$ is allowable and the chain objects in $\mathbf{P} \boxtimes \mathbf{P}'$ are relatively projective. For exactness, we can assume that $\mathcal{A} \cong A\text{-}\mathrm{mod}$ and $\mathcal{A}' \cong A'\text{-}\mathrm{mod}$ for some $\Bbbk$-algebras $A$ and $A'$ \cite[\S 1.8]{EGNO}. Then $\mathcal{A} \boxtimes \mathcal{A}' \cong (A \otimes A')\text{-}\mathrm{mod}$. Objects (resp. morphisms) in $\mathcal{A}$, $\mathcal{A}'$ and $\mathcal{A} \boxtimes \mathcal{A}'$ are in particular $\Bbbk$-vector spaces (resp. linear maps), so it suffices to check exactness in $\mathrm{vect}_{\Bbbk}$. This is classical: K\"unneth formula \eqref{KunnethFormula} ensures that $\mathbf{P} \boxtimes \mathbf{P}'$ is exact at strictly positive degrees. The first terms sequence
\begin{equation}\label{firstTermsDeligneProductResolutions}
0 \leftarrow V \boxtimes V' \xleftarrow{d_0 \boxtimes d'_0} P_0 \boxtimes P'_0 \xleftarrow{\left(d_1 \boxtimes \mathrm{id}_{P'_0}\right) \oplus \left(\mathrm{id}_{P_0} \boxtimes d'_1 \right)} (P_1 \boxtimes P'_0) \oplus (P_0 \boxtimes P'_1).
\end{equation}
is exact as well by the formulas for images and kernels of tensor products of linear maps.
\\2. We use the notations from item 1. Recall that $\Ext^{\bullet}_{\mathcal{A},\mathcal{B}}(V,W)$ is the cohomology of the cochain complex $\Hom_{\mathcal{A}}(\mathbf{P},W)$. The Hom's are vector spaces because we work with linear categories. By \cite[Prop.\,1.11.2]{EGNO} we have
\[ \Hom_{\mathcal{A} \boxtimes \mathcal{A}'}(\mathbf{P} \boxtimes \mathbf{P}',W \boxtimes W') \cong \Hom_{\mathcal{A}}(\mathbf{P},W) \otimes \Hom_{\mathcal{A}'}(\mathbf{P}',W') \]
and K\"unneth formula gives the result.
\end{proof}

\end{document}